\documentclass{cambridge7A}
    \usepackage[numbers]{natbib}
    \usepackage{rotating}
    \usepackage{floatpag}
     \rotfloatpagestyle{empty}
   \usepackage{amsmath}
  \usepackage{amsthm}
   
    \usepackage{graphicx}
    \usepackage{txfonts}

 \usepackage{tikz}

\usepackage[all,poly]{xy}
\usepackage{amsfonts}
\usepackage[mathcal]{eucal}
\usepackage{eufrak}
\usepackage{amssymb}
\usepackage{amsmath,amsthm}
\usepackage{mathrsfs}
\usepackage{color}
\usepackage[colorlinks]{hyperref}
\usepackage{enumerate}
\usepackage{makeidx}
\usepackage{mathtools}
\usepackage{empheq}

\newtheorem{theorem}{Theorem}[section]
\newtheorem{proposition}[theorem]{Proposition}
\newtheorem{lemma}[theorem]{Lemma}
\newtheorem{corollary}[theorem]{Corollary}

\theoremstyle{definition}
\newtheorem{example}[theorem]{Example}
\newtheorem{definition}[theorem]{Definition}
\newtheorem{remark}[theorem]{Remark}

\newcommand{\op}{\operatorname{op}}
\newcommand{\M}{\operatorname{\mathbb M}}
\newcommand{\LL}{\operatorname{\mathcal L}}
\newcommand{\PP}{\operatorname{\mathcal P}}

\newcommand{\I}{\operatorname{\mathcal I}}
\newcommand{\IM}{\operatorname{\mathbb I}}
\newcommand{\gr}{\operatorname{gr}}

\newcommand{\sr}{\operatorname{sr}}
\newcommand{\V}{\operatorname{\mathcal V}}

\newcommand\Gr[1][]{{\operatorname{{Gr}^{#1}-}}}

\newcommand{\Gromega}{\operatorname{Gr_{\Omega}-\!}}

\newcommand{\Gro}{\operatorname{Gr_{0}-\!}}
\newcommand{\gro}{\operatorname{gr_{0}-\!}}

\newcommand{\rGr}{\operatorname{\!-Gr}}
\newcommand{\scm[1]}{{\sc \underline{#1}}}

\newcommand{\Cl}{\operatorname{Cl}}
\newcommand{\ann}{\operatorname{ann}}

\newcommand{\grr}{\operatorname{gr-\!}}
\newcommand{\Pgr}[1][]{{\operatorname{{Pgr}^{#1}\!-}}}
\newcommand{\Pgrp}{\operatorname{Pgr-\!}}

\newcommand{\Pgrl}{\operatorname{\!-Pgr}}
\newcommand{\pdim}{\operatorname{pdim}}
\newcommand{\gldim}{\operatorname{grdim}}
\newcommand{\Filt}{\operatorname{Filt-\!}}

\newcommand{\Modd}{\operatorname{Mod-\!}}
\newcommand{\modd}{\operatorname{mod-\!}}

\newcommand{\grs}{\operatorname{gr_S-\!}}

\newcommand{\rModd}{\operatorname{\!-Mod}}

\newcommand{\nPgr}[1][]{{\operatorname{{Pgr}_n^{#1}-}}}
\newcommand{\qPgr}[1][]{{\operatorname{{Pgr}_q^{#1}-}}}

\newcommand{\Prr}{\operatorname{Pr-\!}}
\newcommand{\QGr}{\operatorname{QGr-\!}}
\newcommand{\Fdim}{\operatorname{Fdim-\!}}

\newcommand*\widefbox[1]{\fbox{\hspace{2em}#1\hspace{2em}}}

\newcommand{\SK}{\operatorname{SK}}

\newcommand{\Gal}{\operatorname{Gal}}

\newcommand{\Qcoh}{\operatorname{QCoh-\!}}

\newcommand{\uloopr}[1]{\ar@'{@+{[0,0]+(-4,5)}@+{[0,0]+(0,10)}@+{[0,0] +(4,5)}}^{#1}}
\newcommand{\uloopd}[1]{\ar@'{@+{[0,0]+(5,4)}@+{[0,0]+(10,0)}@+{[0,0]+ (5,-4)}}^{#1}}
\newcommand{\dloopr}[1]{\ar@'{@+{[0,0]+(-4,-5)}@+{[0,0]+(0,-10)}@+{[0, 0]+(4,-5)}}_{#1}}
\newcommand{\dloopd}[1]{\ar@'{@+{[0,0]+(-5,4)}@+{[0,0]+(-10,0)}@+{[0,0 ]+(-5,-4)}}_{#1}}

\newcommand{\luloop}[1]{\ar@'{@+{[0,0]+(-8,2)}@+{[0,0]+(-10,10)}@+{[0, 0]+(2,2)}}^{#1}}

\newcommand{\va}{\varphi}
\newcommand{\Ga}{\Gamma}
\newcommand{\ga}{\gamma}
\newcommand{\al}{\alpha}

\newcommand{\de}{\delta}
\newcommand{\De}{\Delta}
\newcommand{\Om}{\Omega}

\newcommand{\la}{\lambda}

\newcommand{\si}{\sigma}

\newcommand{\GL}{\operatorname{GL}}

\newcommand{\Aut}{\operatorname{Aut}}
\newcommand{\out}{\operatorname{out}}
\newcommand{\Pic}{\operatorname{Pic}}
\newcommand{\Inn}{\operatorname{Inn}}
\newcommand{\Picent}{\operatorname{Picent}}

\newcommand{\Tr}{\operatorname{Tr}}

\newcommand{\G}{\mathcal{G}}

\newcommand{\R}{\mathbb R}
\newcommand{\Z}{\mathbb Z}
\newcommand{\C}{\mathbb C}
\renewcommand{\H}{\mathbb H}
\newcommand{\+}{\oplus}
\newcommand{\e}{\mathbf e}

\newcommand{\ol}{\overline}

\newcommand{\ind}{\operatorname{ind}}

\newcommand{\End}{\operatorname{End}}
\newcommand{\im}{\operatorname{im}}
\newcommand{\Hom}{\operatorname{Hom}}

\newcommand{\id}{\operatorname{id}}

\newcommand{\chr}{\operatorname{char}}
\renewcommand{\Im}{\operatorname{Im}}

\newcommand{\coker}{\operatorname{coker}}

\let\nbd=\nobreakdash
\newcommand{\ie}{\emph{i.e.,~}}

\newcommand{\lmps}{\longmapsto}
\newcommand{\mps}{\mapsto}
\newcommand{\ra}{\rightarrow}
\newcommand{\Ra}{\Rightarrow}

\newcommand{\mi}{\setminus}
\newcommand{\lra}{\longrightarrow}

\newcommand{\conggr}{\cong_{\mathrm{gr}}}

\makeindex

\title{{\bf Graded Rings and\\ Graded Grothendieck Groups}}

\author{{\bf Roozbeh Hazrat}\\
Western Sydney University\\
AUSTRALIA}

\bookkeywords{Graded rings and modules, graded Morita theory, graded $K$-theory, graded Grothendieck group}

\bookabstract{This book is devoted to a comprehensive study of graded rings and 
graded K-theory. 

A bird's eye view of the theory of graded modules over a graded ring
might give the impression that it is nothing but ordinary module theory
with all its statements decorated with the adjective ``graded''. Once
the grading is considered to be trivial, the graded theory reduces to
the usual module theory. From this perspective, the theory of graded
modules can be considered as an extension of module theory. However, this simplistic overview might conceal the point that graded
modules come equipped with a  \emph{shift}, thanks to the
possibility of partitioning the structures and then rearranging the
partitions. This adds an extra layer of structure (and complexity) to
the theory. 

Many of the classical invariants constructed for the category of modules can be constructed,   \emph{mutatis mutandis},  starting from the category of graded modules. 
The general viewpoint of this book is that, once a ring has a natural graded structure, graded invariants capture more information  than the ungraded counterparts. 

This book focuses on graded $K$-theory, in particular, 
the graded Grothendieck group $K^{\gr}_0$, as a powerful invariant to study ring theory.

}
\begin{document}

\maketitle


\setcounter{tocdepth}{2}
\tableofcontents

\chapter*{Introduction} \label{introf}




A bird's eye view of the theory of graded modules over a graded ring
might give the impression that it is nothing but ordinary module theory
with all its statements decorated with the adjective ``graded''. Once
the grading is considered to be trivial, the graded theory reduces to
the usual module theory. From this perspective, the theory of graded
modules can be considered as an extension of module theory. However, this simplistic overview might conceal the point that graded
modules come equipped with a {\em shift}, thanks to the
possibility of partitioning the structures and then rearranging the
partitions. This adds an extra layer of structure (and complexity) to
the theory. This monograph focuses on the theory of the graded
Grothendieck group $K^{\gr}_0$, that provides a sparkling illustration of
this idea. Whereas the usual $K_0$ is an abelian group, the shift
provides $K^{\gr}_0$ with a natural structure of $\Z[\Ga]$-module, where
$\Ga$ is the group used for the grading and $\Z[\Ga]$ its group ring. As we will see throughout this book, this extra structure carries substantial information about the
graded ring.

Let $\Ga$ and $\De$ be abelian groups and $f:\Gamma\rightarrow \Delta$ a group
homomorphism. Then for any $\Ga$\!-graded ring $A$, one can consider a
natural $\De$-grading on $A$ (see~\S\ref{mconfi1}); in the same
way, any $\Ga$\!-graded $A$-module can be viewed as a $\De$-graded
$A$-module. These operations induce functors
\begin{align*}
U_f:\Gr[\Gamma] A \longrightarrow \Gr[\Delta] A,\\
(-)_{\Omega}: \Gr[\Gamma] A \longrightarrow \Gr[\Omega] A_{\Omega},
\end{align*}
(see~\S\ref{bill100}), where $\Gr[\Gamma] A$ is the category of
$\Ga$\!-graded right $A$-modules, $\Gr[\Delta] A$ that of $\De$-graded right
$A$-modules, and $\Gr[\Omega] A$ the category of $\Om$-graded right
$A_\Om$-module with $\Om=\ker(f)$.
 \index{$\Gr[\Gamma] A$ category of $\Gamma$-graded right $A$-modules}

One aim of the theory of graded rings is to investigate the ways in
which these categories relate to one another, and which properties of
one category can be lifted to another. In particular, in the two extreme
cases when the group $\De=0$ or $f\colon\Ga\rightarrow\De$ is the identity, we
obtain the forgetful functors
\begin{align*}
U:\Gr[\Gamma] A \longrightarrow \Modd A,\\
(-)_{0}: \Gr[\Gamma] A \longrightarrow \Modd A_{0}. 
\end{align*}

The category $\Pgr[\Gamma] A$ of graded finitely generated projective
$A$-modules is an exact category. Thus Quillen's $K$-theory machinery~\cite{quillen} defines graded $K$-groups
\[K^{\gr}_i(A):=K_i(\Pgr[\Gamma] A),\] for $i \in \mathbb N$.
On the other hand, the shift operation on modules induces a functor on
$\Gr[\Gamma] A$  that is an auto-equivalence (\S\ref{pokiss}), so
that these $K$-groups also carry a $\Ga$\!-module structure. One can treat
the groups $K_i(A)$ and $K_i(A_0)$ in a similar way. Quillen's
$K$-theory machinery allows us to establish relations between these
$K$-groups. In particular

\begin{description}

\item{\bf Relating $K^{\gr}_*(A)$ to $K_*(A)$ for a positively graded rings~\S\ref{todaytalkbit}.} For a $\mathbb Z$-graded ring with the positive support, there is a $\mathbb  Z[x,x^{-1}]$-module isomorphism, 
\begin{equation*}
K^{\gr}_i(A)\cong K_i(A_0)\otimes_{\mathbb Z} \mathbb Z[x,x^{-1}].
\end{equation*}

\item{\bf Relating $K^{\gr}_*(A)$ to $K_*(A_0)$ for graded Noetherian regular rings~\S\ref{quoihes}.} \index{graded Noetherian ring}
Consider the full subcategory 
$\Gro A$ of $\Gr A$, of all graded modules $M$ as objects such that $M_0=0$. This is a Serre subcategory of $\Gr A$. One can show that 
$\Gr A / \Gro A \cong \Modd A_0$.  If $A$ is a (right) regular Noetherian ring,  the quotient category identity above holds for the corresponding graded finitely generated modules, \ie  $\grr A / \gro A \cong \modd A_0$ and 
the localisation theorem gives a long exact sequence of abelian groups,
\begin{equation*}
\cdots \longrightarrow K_{n+1}(A_0) \stackrel{\delta}{\longrightarrow} K_n(\gro A) \longrightarrow K^{\gr}_n(A) \longrightarrow K_n(A_0) \longrightarrow \cdots.
\end{equation*} \index{quotient category} \index{Serre subcategory}

\item{\bf Relating $K^{\gr}_*(A)$ to $K_*(A)$ for graded Noetherian regular rings~\S\ref{vandengg}.}
For a $\mathbb Z$-graded ring $A$ which is right regular Noetherian, there is a long exact sequence of abelian groups 
\begin{equation*}
\cdots \longrightarrow K_{n+1}(A)  \longrightarrow K^{\gr}_n(A) \stackrel{\ol i}{\longrightarrow} K^{\gr}_n(A) \stackrel{U}{\longrightarrow} K_n(A) \longrightarrow \cdots.
\end{equation*}

\end{description}

The main emphasis of this book is on the group $K^{\gr}_0$ as a powerful invariant in  the classification problems. This group is equipped 
with the extra structure of  the action of the grade group induced by the shift.     In many important examples, in fact this shift is all the difference between the graded Grothendieck group and the usual Grothendieck group, \ie 
\[K^{\gr}_0(A) / \langle [P] -[P(1)] \rangle \cong K_0(A),\] where $P$ is a graded projective $A$-module and $P(1)$ is the  shifted module (\S\ref{waraya}, see Corollary~\ref{ght9laks}).

The  motivation to write this book came  from recent activities that adopt the graded Grothendieck group as an invariant to classify the Leavitt path algebras~\cite{haz,hazgr,smith2}. Surprisingly, not much is recorded about the graded version of the Grothendieck group in the literature, despite the fact that $K_0$ has been used in many occasions as a crucial invariant, and there is a substantial amount of information about the graded version of other invariants such as (co)homology groups, Brauer groups, etc. The other surge of interest on this group stems from the recent activities on (graded) representation theory of Hecke algebras. In particular for a quiver Hecke algebra, its graded Grothendieck group is closely related to its corresponding quantised enveloping algebra. For this line of research see the survey~\cite{kleshchev}.

This book tries to fill this gap, by systematically developing the theory of graded Grothendieck groups. In order to do this, we have to carry over and work out the details of known results in ungraded case to the graded setting, and gather together 
important results on the graded theory scattered in research papers .

The group $K_0$ has been successfully used in the operator theory to classify certain classes of $C^*$-algebras. 
Building on work of Bratteli, Elliott in~\cite{elliot}  used the pointed ordered $K_0$-groups (called dimension groups) as a complete invariant for AF $C^*$-algebras. 
Another cornerstone of using $K$-groups for the classifications of a wider range of $C^*$-algebras was the work of Kirchberg and Phillips~\cite{phillips}, who showed that $K_0$ and $K_1$-groups together are complete invariant for a certain type of $C^*$-algebras. The Grothendieck group considered as a module induced by a group action was used by Handelman and Rossmann~\cite{handelman} to give a complete invariant for  the class of direct limits of finite dimensional, representable dynamical systems.  
Krieger~\cite{krieger} introduced (past) dimension groups as a complete invariant for the shift equivalence of topological Markov chains (shift of finite types) in symbolic dynamics. 
Surprisingly, we will see that, Krieger's groups  are naturally expressed  by graded Grothendieck groups  (\S\ref{symbolii}).

We develop the theory for rings graded by abelian groups rather than an arbitrary groups for two reasons, although most of the results could be carried over to non-abelian grade group. One reason is that using the abelian grading makes the presentation and proofs much more transparent. In addition,  in most applications of graded $K$-theory, the ring has an abelian grading (often a $\mathbb Z$-grading).

In Chapter~\ref{drkissme} we study the basic theory of graded rings. Chapter~\ref{moritanji} concentrates on graded Morita theory. 
In Chapter~\ref{ggg} we compute $K^{\gr}_0$ for certain graded rings, such as graded local rings and (Leavitt) path algebras. We study the pre-ordering available on $K^{\gr}_0$ and determine the action of $\Gamma$ on this group. Chapter~\ref{picle} studies graded Picard groups and in Chapter~\ref{ultriuy} we prove that for the so called graded ultramatricial algebras, the graded Grothendieck group is a complete invariant. Finally in Chapter~\ref{waraya}, we explore the relations between (higher) $K^{\gr}_n$ and $K_n$, for the class of $\Z$-graded rings. We describe a generalisation of the Quillen and van den Bergh theorems. The latter theorem uses the techniques employed in the proof of the fundamental theorem of $K$-theory, where the graded $K$-theory appears. For this reason we present a proof of the fundamental theorem in this chapter. 

\bigskip 

\noindent{\bf Conventions.}
Throughout this book, unless it is explicitly stated, all rings have identities, homomorphisms preserve the identity and all modules are unitary.  
Moreover, all modules are considered right modules. For a ring $A$, the category of 
right $A$-modules is denoted by $\Modd A$. 
A full subcategory of $\Modd A$ consisted of all finitely generated $A$-modules is denoted by $\modd A$. By $\Prr A$ we denote the category of finitely generated projective $A$-modules.

For a set $\Gamma$, we write    
$\bigoplus_\Gamma \Z$ or $\mathbb Z ^{\Gamma}$ to mean 
 $\bigoplus_{\gamma \in \Gamma} \mathbb Z_{\gamma}$, where
$\mathbb Z_{\gamma} = \mathbb Z$ for each $\gamma \in \Gamma$.
We denote the cyclic group $\Z/n\Z$  with $n$ elements by $\Z_n$.  
\index{$\Modd A$, category of right $A$-modules}
\index{$\modd A$, category of finitely generated right $A$-modules}
\index{$\Prr A$, category of finitely generated projective $A$-modules}

\bigskip 

\noindent {\bf Acknowledgement.}  
I learned about the graded techniques in algebra from Adrian Wadsworth. Judith Millar worked with me to study the graded $K$-theory of Azumaya algebras. Gene Abrams was a source of encouragement that the graded techniques would be fruitful in the study of Leavitt path algebras. 
Andrew Mathas told me how graded Grothendieck groups are relevant in  representation theory and pointed me to the relevant literature. The discussions with Zuhong Zhang, who kindly invited me to Beijing Institute of Technology in July 2013 and 2014, helped to improve the presentation. To all of them, I am grateful.

\chapter{Graded Rings and Graded Modules}\label{drkissme}

Graded rings appear in many circumstances, both in elementary and advanced areas. Here are two examples: 

\begin{enumerate}
\item In elementary school when we distribute 10 apples giving 2 apples to each person, we have 
\[\text{{\sc 10  Apples :  2 Apples \, = 5 People}}.\] The psychological problem caused to many kids as to exactly how the word ``People'' appears in the equation can be 
overcome by correcting it to \[\text{{\sc 10 Apples : 2 Apples\ /\ People \, = 5 People}}.\]  This shows that already at the level of elementary school arithmetic, children work in a much more sophisticated structure, \ie the graded ring  \[\mathbb Z[x_1^{\pm1}, \cdots, x_n^{\pm1}]\]
 of Laurent polynomial rings! (see the interesting book of Borovik~\cite[\S4.7]{borovik} on this).

\medskip 

\item If $A$ is a commutative ring generated by a finite number of elements of degree 1, then by  the celebrated work of Serre~\cite{serre}, 
the category of quasicoherent sheaves on the scheme is equivalent to $\QGr A \cong \Gr A /\Fdim A$, where $\Gr A$ is the category of graded modules over $A$ and $\Fdim A$ is the Serre subcategory of (direct limits of) finite dimensional submodules.  In particular when $A=K[x_0,x_1,\dots,x_n]$, where $K$ is a field, then 
$\Qcoh \mathbb P^n$ is equivalent to  $\QGr A[x_0,x_1,\dots,x_n]$ (see~\cite{serre,artin,smith2} for more precise statements and relations with 
noncommutative algebraic geometry). 

\end{enumerate}

This book treats graded rings and the category of graded modules over a graded ring. 
 This category is an abelian category (in fact a 
 Grothendieck category). Many of the classical invariants constructed for the category of modules can be constructed,   \emph{mutatis mutandis},  starting from the category of graded modules. 
The general viewpoint of this book is that, once a ring has a natural graded structure, graded invariants capture more information  than the ungraded counterparts.

In this chapter we give a concise introduction to the theory of  graded rings. We introduce grading on matrices, study graded division rings and introduce gradings on graph algebras that will be the source of many interesting examples.

\section{Graded rings}\label{pregr529}

\subsection{Basic definitions and examples}\label{jdjthu}

 A ring $A$ is called a
\emph{$\Ga$\!-graded ring}, or simply a \emph{graded ring},
if $ A= \textstyle{\bigoplus_{ \ga \in \Ga}} A_{\ga}$, where 
\index{graded ring} $\Ga$ is an (abelian) group, each $A_{\ga}$ is
an additive subgroup of $A$ and $A_{\ga}  A_{\delta} \subseteq
A_{\ga + \delta}$ for all $\ga, \delta \in \Ga$. 

If $A$ is an algebra over a field $K$, then $A$ is called a \emph{graded algebra} if $A$ is a graded ring and  for any $\ga \in \Gamma$, $A_{\ga}$ is a $K$-vector subspace.  \index{graded algebra}

The set $A^{h} =
\bigcup_{\ga \in \Ga} A_{\ga}$ is called the set of \emph{homogeneous elements} of $A$. \index{homogeneous element} 
\index{$A^h$, homogeneous elements of $A$} The additive group $A_\gamma$ is called the $\gamma$-\emph{component} of $A$ \index{$\gamma$-component of a graded ring} 
\index{component of a graded ring} and the nonzero elements of $A_\ga$ are called \emph{homogeneous of degree $\ga$}\index{homogeneous of degree $\delta$}.
We write deg$(a) = \ga$ if $a \in A_{\ga}\backslash \{0\}$. \index{$\deg(a)$, degree of the homogeneous element $a$} We call the set 
\[\boxed{\Gamma_A=\big \{ \, \ga \in \Gamma \mid A_\ga \not = 0 \, \big \}}\] 
the \emph{support}   \index{support of a graded ring} of $A$. \index{$\Gamma_A$, support of $A$} We say the $\Gamma$\!-graded ring $A$ has a \emph{trivial grading}, \index{trivial grading} or $A$ is \emph{concentrated in
degree zero}, \index{grading concentrated in
degree zero} if the support of $A$ is the trivial group, \ie $A_0=A$ and $A_\ga=0$ for $\ga \in \Gamma \backslash \{0\}$.

For  $\Ga$\!-graded rings $A$ and $B$, a $\Gamma$\!-\emph{graded ring
homomorphism}\index{graded ring homomorphism} $f:A \ra B$ is a ring
homomorphism such that $f(A_{\ga}) \subseteq B_{\gamma}$ for all
$\ga \in \Ga$.  A graded homomorphism $f$ is called a \emph{graded isomorphism} \index{graded ring isomorphism} if
$f$ is bijective and, when such a graded isomorphism exists, we
write $A \conggr B$. Notice that if $f$ is a graded ring
homomorphism which is bijective, then its inverse $f^{-1}$ is also a
graded ring homomorphism.

If $A$ is a graded ring and $R$ is a commutative graded ring, then $A$ is called a \emph{graded $R$-algebra} if $A$ is an $R$-algebra and the associated algebra homomorphism $\phi:R\rightarrow A$ is a graded homomorphism. \index{graded $R$-algebra} When $R$ is a field concentrated in degree zero, we retrieve the definition of graded algebra above. 


\begin{proposition} \label{basicsofgradedrings}
Let $A = \bigoplus_{ \ga \in \Ga} A_{\ga}$ be a $\Ga$\!-graded ring.
Then 
\begin{enumerate}[\upshape(1)]
\item $1_A$ is homogeneous of degree $0$;

\item $A_0$ is a subring of $A$;

\item Each $A_{\ga}$ is an $A_0$-bimodule;

\item For an invertible element $a \in A_{\ga}$, its inverse
$a^{-1}$ is homogeneous of degree $-\ga$, \ie $a^{-1} \in A_{-\ga}$. 
\end{enumerate}
\end{proposition}

\begin{proof}
(1) Suppose $1_A = \sum_{\ga \in \Ga} a_{\ga}$ for $a_{\ga} \in
A_{\ga}$. Let $b \in A_{\de}$, $\de \in \Ga$, be an
arbitrary nonzero homogeneous element. Then $b = b 1_A = \sum_{\ga \in \Ga} b
a_{\ga}$, where $b a_{\ga} \in A_{\de +\ga}$ for all $\ga \in \Ga$.
Since the decomposition is unique, $b a_{\ga} = 0$ for all $\ga \in
\Ga$ with $\ga \neq 0$. But as $b$ was an arbitrary homogeneous element, it follows that $ba_{\ga}=0$
for all $b \in A$ (not necessarily homogeneous), and in particular $1_A
a_\ga = a_\ga =0$ if $\ga \neq 0$. Thus $1_A =
a_0 \in  A_0$.

\vspace{3pt}

(2) This follows since $A_0$ is an additive subgroup of $A$ with
$A_0 A_0 \subseteq A_0$ and $1 \in A_0$.

\vspace{3pt}

(3) This is immediate.

\vspace{3pt}

(4) Let $b = \sum_{\de\in \Gamma} b_{\de}$, with $\deg(b_{\de})= \de$, be the
inverse of $a\in A_\ga$, so that $1=ab = \sum_{\de\in\Gamma} a b_{\de}$, where $a
b_{\de} \in A_{\ga+\de}$. By (1), since $1$ is homogeneous of degree $0$ and
the decomposition is unique, it follows that $a b_{\de} = 0$ for all
$\de \neq -\ga$. Since $a$ is invertible, 
$b_{-\ga} \neq 0$, so $b=b_{-\ga}\in A_{-\ga}$ as required.
\end{proof}

The ring $A_0$ is called the \emph{0-component ring} of $A$ and plays a crucial role in the theory of graded rings. The proof of Proposition~\ref{basicsofgradedrings}(4), in fact, shows that if $a\in A_\ga$ has a left (or right) inverse then that inverse is in $A_{-\ga}$. 
In Theorem~\ref{jhby67}, we characterise $\mathbb Z$-graded rings such that $A_1$ has a left (or right) invertible element.  \index{0-component ring of a graded ring}

\begin{example}\label{hy431}\scm[Group rings] \index{group ring}
\vspace{0.2cm}

For a group $\Gamma$,  the group ring $\Z[\Gamma]$ has a natural $\Gamma$\!-grading 
\begin{equation*}
\Z[\Gamma]=\bigoplus_{\ga \in \Ga} \Z[\Gamma]_\gamma, \text{   where   }  \Z[\Gamma]_\gamma= \Z\ga.
\end{equation*}
In~\S\ref{khgfewa1}, we will the crossed product which are graded rings and are generalisations of group rings and skew groups rings. A group ring has a natural involution which makes it an involutary graded ring (see~\S\ref{involudool}). 
\end{example}

In several applications (such as $K$-theory of rings,~\S\ref{waraya}) we deal with $\Z$-graded rings with support in $\mathbb N$, the so called \emph{positively graded rings}. \index{positively graded ring}

\begin{example}\scm[Tensor algebras as positively graded rings]\label{stevevanzandt}
\vspace{0.2cm}

Let $A$ be a commutative ring and $M$ be an $A$-module. Denote by $T_n(M)$, $n\geq 1$, the tensor product of $n$ copies of $M$ over $A$. Set  $T_0(M)=A$. 
Then the natural $A$-module isomorphism $T_n(M) \otimes_A T_m(M)\rightarrow T_{n+m}(M)$, induces a ring structure on 
\[T(M):=\bigoplus_{n \in \mathbb N}  T_n(M).\]
The $A$-algebra $T(M)$ is called the  \emph{tensor algebra} of $M$. \index{tensor algebra} Setting \[T(M)_n:=T_n(M),\] makes $T(M)$ a $\Z$-graded ring with support $\mathbb N$. From the definition, we have $T(M)_0=A$.  \index{$T(M)$, tensor algebra of the module $M$}

If $M$ is a free $A$-module, then $T(M)$ is a free algebra over $A$, generated by a basis of $M$. Thus free rings are $\mathbb Z$-graded rings with the generators being homogeneous elements of degree $1$.  We will systematically study the grading of free rings in~\S\ref{freeme}. 
\end{example}

\begin{example}\label{atsusan}\scm[Formal matrix rings as graded rings] \index{Frobenius algebra} \index{formal matrix ring} \index{formal triangular matrix ring}
\vspace{0.2cm}

Let $R$ and $S$ be rings, $M$ a $R\!-\!S$\!-bimodule and $N$ a $S\!-\!R$-bimodule. Consider the set 
\begin{equation*}
T:=\Bigg \{\, \begin{pmatrix}
r & m \\n & s
\end{pmatrix}   \, \bigr | \,  r\in R, s\in S, m\in M, n\in N \, \Bigg\}.
\end{equation*}

Suppose that there are bimodule homomorphisms $\phi:M\otimes_S N \rightarrow R$ and $\psi:N\otimes_R M \rightarrow S$ such that $(mn)m'=m(nm')$, where we denote $\phi(m,n)=mn$ and $\psi(n,m)=nm$. One can then check that $T$ with the matrix addition and multiplication forms a ring with identity. 
The ring $T$ is called the \emph{formal matrix ring} and denoted also by 
\begin{equation*}
T=\begin{pmatrix}
R & M \\N & S
\end{pmatrix}. 
\end{equation*}
For example the Morita ring of a module is a formal matrix ring (see~\S\ref{meinghto} and~(\ref{messiah})). 

Considering 
\begin{equation*}
T_0=\begin{pmatrix}
R & 0 \\0 & S
\end{pmatrix}, \,\,\,\,\,\,  
T_1 =\begin{pmatrix}
0 & M \\N & 0
\end{pmatrix},
\end{equation*}
it is easy to check that $T$ becomes a $\mathbb Z_2$-graded ring. 
In the cases that the images of $\phi$ and $\psi$ are zero, these rings have been extensively studied (see~\cite{krylov} and references therein).

When $N=0$, the ring $T$ is called a 
\emph{formal triangular matrix ring}. In this case there is no need to consider the homomorphisms $\phi$ and $\psi$. Setting further $T_i=0$ for $i\not = 0,1$, makes $T$ also a $\mathbb Z$-graded ring. 

One specific example of such grading on (subrings of) formal triangular matrix ring is used in representation theory. 
Recall that for a field $K$, a finite dimensional $K$-algebra $R$ is called \emph{Frobenius algebra} if $R\cong R^*$ as  right $R$-modules, where $R^*:=\Hom_K(R,K)$. Note that $R^*$ has a natural $R$-bimodule structure. 

Starting from a finite dimensional $K$-algebra $R$, one constructs the \emph{trivial extension} of $R$ which is a Frobenius algebra and has a natural $\Z$-graded structure as follows. \index{trivial extension} Consider $A:=R\bigoplus R^*$, with addition defined component-wise and multiplication defined as 
\[ (r_1,q_1)(r_2,q_2)=(r_1 r_2, r_1 q_1+ q_2 r_2),\]
where $r_1,r_2\in R$ and $q_1,q_2 \in R^*$. 
Clearly $A$ is a Frobenius algebra with identity $(1,0)$. Moreover, setting 
\begin{align*}
A_0 & =R\oplus 0,\\
A_1 &= 0\oplus R^*,\\
A_i &= 0, \text{   otherwise, }
\end{align*}
makes $A$ into a $\Z$-graded ring with support $\{0,1\}$. In fact this ring is a subring 
of formal triangular matrix ring 
\begin{equation*}
T_0=\begin{pmatrix}
R & R^*\\0 & R
\end{pmatrix},
\end{equation*}
consisting of elements 
$
\begin{pmatrix}
a & q \\0 & a
\end{pmatrix}.
$

These rings appear in representation theory (see~\cite[\S2.2]{happel}). The graded version of this contraction is carried out in Example~\ref{gratsusan}. 
\end{example}

\begin{example}\label{meinmein}\scm[The graded ring $A$ as $A_0$-module]
\vspace{0.2cm}

Let $A$ be a $\Gamma$\!-graded ring. Then $A$ can be considered as a $A_0$-bimodule. In many cases $A$ is a projective $A_0$-module, for example in the case of group rings (Example~\ref{hy431}) or when $A$ is a strongly graded ring (see~\S\ref{scrosshg} and~Theorem~\ref{mozsace}). Here is an example that this is not the case in general. Consider the ring $T$ of formal matrix ring 
\begin{equation*}
T=\begin{pmatrix}
R & M \\0 & 0
\end{pmatrix}, 
\end{equation*}
where $M$ is a left $R$-module, which is not projective $R$-module. Then by Example~\ref{atsusan}, $T$ is a $\Z$-graded ring with $T_0=R$ and $T_1=M$. Now $T$ as a $T_0$-module is $R\oplus M$ as $R$-module. Since $M$ is not projective, $R\oplus M$ is not projective $R$-module. We also get that $T_1$ is not projective $T_0$-module.  
\end{example}

\subsection{Partitioning graded rings} \label{mconfi1} 

Let $A$ be a $\Gamma$\!-graded ring and $f:\Gamma \rightarrow \Delta$ be a group homomorphism. Then one can assign a natural $\Delta$-graded structure to $A$ as follows: $A=\bigoplus_{\delta \in \Delta} A_\delta$, where 
\[
A_{\delta}=
   \begin{cases}
     \bigoplus_{\gamma \in f^{-1}(\delta)}A_\gamma  & \text{if } f^{-1}(\delta) \not =\varnothing;\\
0 & \text{otherwise}.
    \end{cases}
\]
 In particular for a subgroup $\Omega$ of  $\Gamma$ we have the following constructions.

\begin{description}
\item[Subgroup grading:]
 The ring $A_\Omega:=\bigoplus_{\omega \in \Omega}A_\omega$ forms a $\Omega$-graded ring. In particular, $A_0$ corresponds to the trivial subgroup of $\Gamma$.   \index{$A_\Omega$, a $\Omega$-graded subring of $A$}

\medskip 

\item[Quotient grading:] Considering \[A=\bigoplus_{\Omega+\alpha \in \Gamma/\Omega} A_{\Omega+\alpha},\] where 
\[A_{\Omega+\alpha} :=\bigoplus_{\omega \in \Omega} A_{\omega+\alpha},\] makes $A$ a $\Gamma/\Omega$-graded ring. (Note that if $\Gamma$ is not abelian, then for this construction,  $\Omega$ needs to be a normal subgroup.) Notice that with this grading, $A_0=A_\Omega$.  If $\Gamma_A \subseteq \Omega$, then $A$ considered as $\Gamma/\Omega$-graded ring, is concentrated in degree zero. 

This construction induces a \emph{forgetful} functor (or with other interpretations, a \emph{block}, or a \emph{coarsening} functor) from the category of $\Gamma$\!-graded rings $\mathcal R^{\Gamma}$ to the category of $\Gamma/\Omega$-graded rings $\mathcal R^{\Gamma/\Omega}$, \ie  \[U:\mathcal R^{\Gamma}\rightarrow \mathcal R^{\Gamma/\Omega}.\] \index{block functor} \index{coarsening functor}
If $\Omega=\Gamma$, this gives the obvious forgetful functor from the category of 
$\Gamma$\!-graded rings to the category of rings. We give a specific example of this construction in Example~\ref{mnuh4} and others in Examples~\ref{egofgrdivisionrings} and~\ref{july26y}.  \index{forgetful functor} \index{$\mathcal R^{\Gamma}$, category of $\Gamma$-graded rings}

\end{description}

\begin{example}\scm[Tensor product of graded rings]\label{deltam} \index{tensor product of graded rings}
\vspace{0.2cm}

Let $A$ be a $\Gamma$\!-graded and $B$ a $\Omega$-graded rings. Then $A\otimes_{\Z} B$ has a natural $\Gamma\times\Omega$-graded ring structure as follows. 
Since $A_\gamma$ and $B_\omega$, $\gamma \in \Gamma$, $\omega\in \Omega$, are $\Z$-modules then 
$A \otimes _{\Z} B$ can be decomposed as a direct sum 
\[A\otimes_{\Z} B =\bigoplus_{(\gamma,\omega) \in \Gamma\times \Omega} A_\gamma \otimes B_\omega\] (to be precise, $A_\gamma \otimes B_\omega$ is the image of $A_\gamma \otimes_{\Z} B_\omega$ in $A\otimes_{\Z} B$). 

Now, if $\Omega=\Gamma$ and 
\begin{align*}
f:\Gamma\times \Gamma &\longrightarrow \Gamma,\\
 (\gamma_1,\gamma_2) &\longmapsto \gamma_1+\gamma_2,
 \end{align*}
  then we get a natural $\Gamma$\!-graded structure on $A\otimes_{\Z} B$. Namely, 
\[A\otimes_{\Z} B =\bigoplus_{\gamma \in \Gamma} (A\otimes B)_\gamma,\] where 
\begin{equation*}
(A \otimes B)_{\ga} = \Big \{ \sum_i a_i \otimes b_i   \mid   a_i \in
A^h, b_i \in B^h, \deg(a_i)+\deg(b_i) = \ga \, \Big\}.
\end{equation*}
  We give  specific examples of this construction in Example~\ref{hyg61}.  One can replace $\Z$ by a field $K$, if $A$ and $B$ are $K$-algebras and $A_\gamma$, $B_\gamma$ are $K$-modules.

\end{example}

\begin{example}\label{hyg61}
Let $A$ be a ring with identity and $\Gamma$ be a group. We consider $A$ as $\Gamma$\!-graded ring concentrated in degree zero. Then by Example~\ref{deltam},  \[A[\Gamma]\cong A\otimes_{\Z} \Z[\Gamma]\] has a $\Gamma$\!-graded structure, \ie  
$A[\Gamma]=\textstyle{\bigoplus_{\ga\in \Gamma}} A\ga$. If $A$ itself is a (nontrivial) $\Gamma$\!-graded ring
$A=\textstyle{\bigoplus_{\ga\in \Gamma}} A_{\ga}$, then by Example~\ref{deltam}, $A[\Gamma]$ has also a $\Gamma$\!-grading 
\begin{equation}\label{hgogt}
A[\Gamma]=\textstyle{\bigoplus_{\ga\in \Gamma}} A^\ga, \text{  where  } A^\ga=\textstyle{\bigoplus_{\ga=\zeta+\zeta'}}  A_\zeta \zeta'.
\end{equation}

An specific example is when $A$ is a positively graded $\mathbb Z$-graded ring. Then $A[x]\cong A\otimes \Z[x]$ is a $\mathbb Z$-graded ring with support $\mathbb N$, where 
\[
A[x]_n=\bigoplus_{i+j=n}A_i x^j.
\] This graded ring will be used in~\S\ref{bolgona} when we prove the fundamental theorem of $K$-theory. Such constructions are systemically studied in~\cite{natlil} (see also \cite[\S6]{grrings}). 
\end{example}

\begin{example}\label{mnuh4}

Let $A$ be a $\Gamma\times \Gamma$-graded ring. Define a $\Gamma$\!-grading on $A$ as follows. For $\gamma \in \Gamma$, set
\[A'_\gamma=\sum_{\alpha \in \Gamma} A_{\gamma- \alpha,\alpha}.\] It is easy to see that $A=\bigoplus_{\gamma \in \Gamma} A'_\gamma$ is a 
$\Gamma$\!-graded ring. When $A$ is a $\mathbb Z \times \Z$-graded, then the $\mathbb Z$-grading on $A$ is obtained from considering all the homogeneous components on a diagonal together as the following figure shows.   

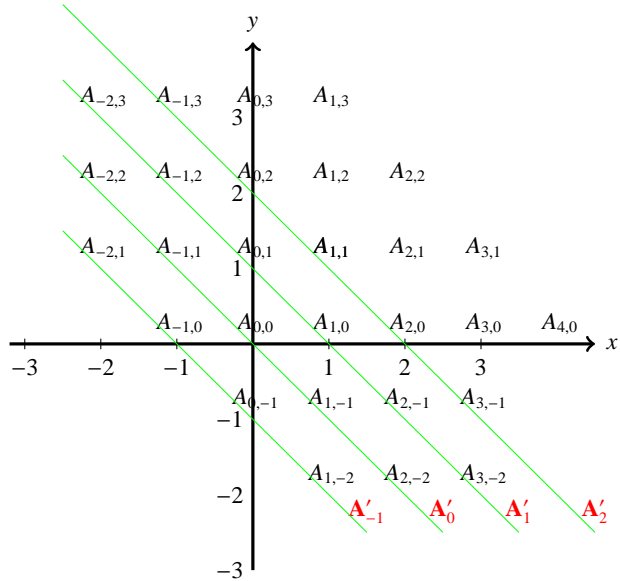
\begin{center}
\begin{tikzpicture}[scale=1]
 \tikzstyle{every node}=[font=\small]
 
     
 \draw[very thick, ->] (-3.2,0) -- (4.5,0) node[right] {$x$};
 \draw[very thick, ->] (0,-3) -- (0,4) node[above] {$y$};

 \foreach \x/\xtext in {-3, -2,-1,  1/1, 2/2, 3/3}
  \draw[shift={(\x,0)}] (0pt,2pt) -- (0pt,-2pt) node[below] {$\xtext$};
  
 \draw[red] (2.5,-2.5) node[above] 
      {$\mathbf A'_0$};

\draw[red] (3.5,-2.5) node[above] 
      {$\mathbf A'_1$};

\draw[red] (4.5,-2.5) node[above] 
      {$\mathbf A'_2$};

\draw[red] (1.5,-2.5) node[above] 
      {$\mathbf A'_{-1}$};

  \foreach \y/\ytext in {-3,-2,-1,1/1, 2/2, 3/3}
    \draw[shift={(0,\y)}]  node[left] {$\ytext$};

 \draw (1,1) node[above] 
      {$\displaystyle\ A_{1,1}$};

       \draw (-1,0) node[above] 
      {$\displaystyle\ A_{-1,0}$};

       \draw (1,-2) node[above] 
      {$\displaystyle\ A_{1,-2}$};
      
       \draw (-2,1) node[above] 
      {$\displaystyle\ A_{-2,1}$};

  \draw (1,-1) node[above] 
      {$\displaystyle\ A_{1,-1}$};
      
     \draw (4,0) node[above] 
      {$\displaystyle\ A_{4,0}$};  
      
     \draw (-1,1) node[above] 
      {$\displaystyle\ A_{-1,1}$};  
      
 \draw (1,2) node[above] 
      {$\displaystyle\ A_{1,2}$};
      
       \draw (-2,2) node[above] 
      {$\displaystyle\ A_{-2,2}$};
      
   \draw (1,3) node[above] 
      {$\displaystyle\ A_{1,3}$};

      \draw (0,2) node[above] 
      {$\displaystyle\ A_{0,2}$};

   \draw (0,3) node[above] 
      {$\displaystyle\ A_{0,3}$};

    \draw (1,1) node[above] 
      {$\displaystyle\ A_{1,1}$};

 \draw (-1,2) node[above] 
      {$\displaystyle\ A_{-1,2}$};
      
   \draw (-1,3) node[above] 
      {$\displaystyle\ A_{-1,3}$};  

     \draw (2,1) node[above] 
      {$\displaystyle\ A_{2,1}$};

 \draw (3,1) node[above] 
      {$\displaystyle\ A_{3,1}$};

  \draw (2,2) node[above] 
      {$\displaystyle\ A_{2,2}$};

       \draw (-2,3) node[above] 
      {$\displaystyle\ A_{-2,3}$};

 \draw (0,1) node[above] 
      {$\displaystyle\ A_{0,1}$};

  \draw (2,-1) node[above] 
      {$\displaystyle\ A_{2,-1}$};
      
       \draw (3,-1) node[above] 
      {$\displaystyle\ A_{3,-1}$};

  \draw (3,-2) node[above] 
      {$\displaystyle\ A_{3,-2}$};

  \draw (2,-2) node[above] 
      {$\displaystyle\ A_{2,-2}$};

 \draw (0,-1) node[above] 
      {$\displaystyle\ A_{0,-1}$};

\draw (0,0) node[above] 
      {$\displaystyle\ A_{0,0}$};

 \draw (1,0) node[above] 
      {$\displaystyle\ A_{1,0}$};

\draw (2,0) node[above] 
      {$\displaystyle\ A_{2,0}$};

\draw (3,0) node[above] 
      {$\displaystyle\ A_{3,0}$};

\draw[very thin,green] (-2.5,2.5) -- (2.5,-2.5);   
\draw[very thin,green] (-2.5,3.5) -- (3.5,-2.5);
\draw[very thin,green] (-2.5,4.5) -- (4.5,-2.5);   
\draw[very thin,green] (-2.5,1.5) -- (1.5,-2.5);

\end{tikzpicture}
\end{center}

In fact this example follows from the general construction given in~\S\ref{mconfi1}. Consider the homomorphism $\Gamma\times\Gamma \rightarrow \Gamma, (\alpha,\beta)\mapsto \alpha+\beta$. Let $\Omega$ be the kernel of this map. Clearly $(\Gamma\times \Gamma) /\Omega  \cong \Gamma$. One can check  that the $(\Gamma\times \Gamma) /\Omega$-graded ring $A$ gives the graded ring constructed in this example (see also Remark~\ref{shanbei}). 
\end{example}

\begin{example}\scm[The direct limit of graded rings]\label{penrith123}
\vspace{0.2cm}

Let $A_i$, $i\in I$, be a direct system of $\Gamma$\!-graded rings, \ie $I$ is a directed partially ordered set and for $i\leq j$, there is a graded homomorphism $\phi_{ij}:A_i \rightarrow A_j$ which is compatible with the ordering. Then $A:=\varinjlim A_i$ is a $\Gamma$\!-graded ring with homogeneous components 
$A_\alpha=\varinjlim {A_i}_\alpha$. For a detailed construction of such direct limits see \cite[II, \S11.3, Remark~3]{bourbaki}. 

As an example, the ring $A=\Z[x_i \mid i\in \mathbb N]$, where $A=\varinjlim_{i \in \mathbb N} \Z[x_1,\dots,x_i]$, with $\deg(x_i)=1$ is a $\Z$-graded ring with support $\mathbb N$.  We give another specific example of this construction in Example~\ref{noisest}.

We will study in detail one type of these graded rings, \ie graded ultramatricial algebras (\S\ref{ultriuy}, Definition~\ref{kjhwal}) and will show that the graded Grothendieck group (\S\ref{ggg}) will classify these graded rings completely. 
\end{example}

\begin{example}\label{noisest}
Let $A= \textstyle{\bigoplus_{\ga \in \Ga}} A_{\ga}$ and $B= \textstyle{\bigoplus_{\ga \in
\Ga}} B_{\ga}$ be $\Gamma$\!-graded rings.  Then $A\times B$ has a natural grading  given by $A\times B  = \textstyle{\bigoplus_{\ga \in \Gamma}} (A \times B)_{\ga}$ where $(A \times B)_{\ga}=A_\ga \times B_\ga$.  
\end{example}

\begin{example}\label{meinbalmain}\scm[Localisation of graded rings] 
\vspace{0.2cm}

Let $S$ be a central multiplicative closed subsets of $\Gamma$\!-graded ring $A$, consisting of homogeneous elements. Then $S^{-1}A$ has a natural a $\Gamma$\!-graded structure. Namely, for $a\in A^h$, define $\deg(a/s)=\deg(a)-\deg(s)$ and for $\gamma \in \Gamma$, 
\[(S^{-1}A)_\gamma=\big \{ \, a/s \mid a\in A^h, \deg(a/s)=\gamma \, \big\}. \] It is easy to see that this is well-defined and make $S^{-1}A$ a $\Gamma$\!-graded ring. 
\end{example}

Many rings have a `canonical' graded structure, among them, crossed products (group rings, skew group rings, twisted group rings),  edge algebras, path algebras, incidence rings, etc. (see~\cite{kelarev} for a review of these ring constructions).  We will study some of these rings in this book. 

\begin{remark}\scm[Rings graded by a category]\label{ggtgt}
\vspace{0.2cm}

The use of groupoids as a suitable language for structures whose operations are partially defined have now been firmly recognised. There is a generalised notion of groupoid graded rings as follows. Recall that a \emph{groupoid} \index{groupoid} is a small category with the property that all morphisms are isomorphisms. As an example, 
let $G$ be a group and $I$ a nonempty set. The set $I\times G \times I$, considered as morphisms, forms a groupoid where the composition defined by \[(i,g,j)(j,h,k)=(i,gh,k).\] One can show that this forms a connected groupoid and any connected groupoid is of this form (\cite[Ch.~3.3, Prop.~6]{lawson}). If $I=\{1,\dots, n\}$, we denote $I\times G \times I$ by $n\times G \times n$.

Let $\Gamma$ be a groupoid and $A$ be a ring. $A$ is called $\Gamma$\!-\emph{groupoid graded} \index{groupoid grading} ring, 
if $ A= \textstyle{\bigoplus_{ \ga \in \Ga}} A_{\ga}$,  where $\ga$ is a morphism of $\Gamma$, each $A_{\ga}$ is
an additive subgroup of $A$ and $A_{\ga}  A_{\delta} \subseteq
A_{\ga \delta}$ if the morphism $\ga \delta$ is defined  and $A_{\ga}  A_{\delta} =0$, otherwise.  For a group $\Gamma$, considering it as a category with one element and $\Gamma$ as the set of morphisms, we recover the $\Gamma$\!-group graded ring $A$ (see Example~\ref{meinbed} for an example of groupoid graded ring). 

One can develop the theory of groupoid graded rings parallel and similar to the group graded rings. See~\cite{lili,oin} for this approach. Since adjoining a zero to a groupoid, gives a semigroup, a groupoid graded ring is a special case of rings graded by semigroups (see Remark~\ref{hgyhv42}). For a general notion of a ring graded by a category see~\cite[\S2]{abme}, where it is shown that the category of graded modules (graded by a category) is a Grothendieck category. 
\end{remark}

\begin{remark}\label{hgyhv42}\scm[Rings graded by a semigroup]
\vspace{0.2cm}

In the definition of a graded ring (\S\ref{jdjthu}), one can replace the group grading with a semigroup. With this setting, the tensor algebras of Example~\ref{stevevanzandt} are $\mathbb N$-graded rings. A number of results on the group graded rings can also be established in the more general setting of rings graded by cancellative monoids or semigroups (see for example~\cite[II, \S11]{bourbaki}). However, in this booke we only consider group graded rings. 
\end{remark}

\begin{remark}\label{parknan} \scm[Graded rings without identity]
\vspace{0.2cm}

For a ring without identity, one defines the concept of the graded ring exactly as when the ring has an identity. The concept of the strongly graded is defined similarly. In several occasions in this book we construct graded rings without identity. For example,  Leavitt path algebras arising from infinite graphs are graded rings without identity~\S\ref{paohdme}. See also~\S\ref{freeme}, the graded free rings. The unitisation of a (non-unital) graded ring has a canonical grading. This is studied in relation with graded $K_0$ of non-unital rings in \S\ref{wedk0} (see~(\ref{thurattila})). 
\end{remark}

\subsection{Strongly graded rings} \label{scrosshg}
Let $A$ be a $\Gamma$\!-graded ring. By Proposition~\ref{basicsofgradedrings}, $1\in A_0$. This implies $A_0 A_\gamma =A_\gamma$ and 
$A_\ga A_0 =A_\gamma$ for  any 
$\gamma\in\Gamma$. If these equality holds for any two arbitrary elements of $\Gamma$, we call the ring a strongly graded ring. Namely, a $\Ga$\!-graded ring $A=\textstyle{\bigoplus_{ \ga \in \Ga}} A_{\ga}$
is called a \emph{strongly graded ring} \index{strongly graded ring} if $A_{\ga} A_{\de} = A_{\ga +\de}$
for all $\ga, \de \in \Ga$. A graded ring $A$ is called a
\emph{crossed product} \index{crossed product ring} if there is an invertible element in every
homogeneous component $A_\ga$ of $A$; that is, $A^* \cap A_\ga \neq
\varnothing$ for all $\ga \in \Ga$, where $A^*$ is the group of all invertible elements of $A$.
We define the support of invertible homogeneous elements of $A$ as 
\begin{equation}\label{henongs}
\boxed{\Gamma_A^*=\{ \, \ga \in \Gamma \mid A_\ga^* \not = \varnothing\,  \}}
\end{equation}
where $A_\ga^*:=A^* \cap A_\ga$. It is easy to see that $\Gamma_A^*$ is a group and $\Gamma_A^* \subseteq \Gamma_A$ (see Proposition~\ref{basicsofgradedrings}(4)). Clearly $A$ is a crossed product 
if and only if $\Gamma_A^*=\Gamma$.   \index{$\Gamma_A^*$, support of invertible elements of $A$}
 
\begin{proposition} \label{crossedproductstronglygradedprop}
 Let $A = \bigoplus_{ \ga \in \Ga} A_{\ga}$ be a $\Ga$\!-graded
ring. Then
\begin{enumerate}[\upshape(1)]
\item $A$ is strongly graded if and only if $1 \in A_{\ga} A_{-\ga}$
for any $\ga \in \Ga$;

\item If $A$ is strongly graded then the support of $A$ is $\Gamma$;

\item Any crossed product ring is strongly graded.

\item If $f:A\rightarrow B$ is a graded homomorphism of graded rings, then B is strongly graded (resp. crossed product) if $A$ is so. 
\end{enumerate}

\end{proposition}

\begin{proof}
(1) If $A$ is strongly graded, then $1\in A_0=A_{\ga} A_{-\ga}$ for any $\ga \in \Ga$. For the converse, the assumption $1 \in  A_{\ga}
A_{-\ga}$  implies that  $A_0=A_{\ga} A_{-\ga}$ for any $\ga \in \Ga$. Then for $\si, \de \in \Ga$,
\[
A_{\si +\de} = A_0 A_{\si +\de} = (A_\si A_{-\si}) A_{\si +\de}  =
A_\si (A_{-\si} A_{\si + \de}) \subseteq A_{\si} A_{\de} \subseteq A_{\si +\de}
\]
proving $A_{\si \de} = A_\si A_\de$, so $A$ is strongly graded.

(2) By (1), $1\in  A_{\ga} A_{-\ga}$ for any $\ga \in \Gamma$. This implies $A_{\ga}\not =0$ for any $\ga$, \ie $\Gamma_A=\Gamma$.


(3) Let $A$ be a crossed product ring. By definition, for $\ga \in \Ga$, there exists $a \in A^* \cap A_\ga$. So
$a^{-1} \in A_{-\ga}$ by Proposition~\ref{basicsofgradedrings}(4) and $1 = a a^{-1} \in A_\ga A_{-\ga}$. Thus $A$ is strongly graded by (1). 

(4) Suppose $A$ is strongly graded. By (1), $1\in  A_{\ga} A_{-\ga}$ for any $\ga \in \Gamma$. Thus \[1 \in f(A_{\ga}) f(A_{-\ga}) \subseteq B_{\ga} B_{-\ga}.\] Again (1) implies $B$ is strongly graded. The case of crossed product follows easily from the definition.  
\end{proof}

The converse of (3) in Proposition~\ref{crossedproductstronglygradedprop} does not hold. One can prove that if $A$ is strongly graded and $A_0$ is a local ring, then 
$A$ is a crossed product algebra (see~\cite[Theorem~3.3.1]{grrings}). In~\S\ref{creekside} we give examples of strongly graded algebra $A$ such that $A$ is crossed product but $A_0$ is not a local ring. We also give an example of strongly $\Z$-graded ring $A$ such that $A_0$ is not local and $A$ is not crossed product (Example~\ref{noncori}). Using graph algebras we will produce large classes of strongly graded rings which are not crossed product (see Theorems~\ref{sthfin3} and ~\ref{sthfin4}). 

If $\Gamma$ is a finitely generated group, generated by the set $\{\gamma_1,\dots,\gamma_n\}$, then (1) in Proposition~\ref{crossedproductstronglygradedprop} can be simplified to the following: $A$ is strongly graded if and only if $1\in A_{\gamma_i}A_{-\gamma_i}$ and $1\in A_{-\gamma_i}A_{\gamma_i}$, where $1\leq i \leq n$.   Thus if $\Gamma=\mathbb Z$, in order that $A$ to be strongly graded, we only need to have $1\in A_1 A_{-1}$ and $1\in A_{-1}A_1$. This will be used, for example, in Proposition~\ref{lanhc888} to show that certain corner skew Laurent
polynomial rings (\S\ref{cornerskew}) are strongly graded.

\begin{example}\label{monster1}\scm[Constructing strongly graded rings via tensor products]
\vspace{0.2cm}

Let $A$ and $B$ be $\Gamma$\!-graded rings. Then by Example~\ref{deltam}, $A\otimes_{\Z}B$ is a $\Gamma$\!-graded ring. If one of the rings is strongly graded (resp. crossed product) then $A\otimes_{\Z} B$ is so. Indeed, suppose $A$ is strongly graded (reap. crossed product). Then the claim follows from Proposition~\ref{crossedproductstronglygradedprop}(4) and the graded homomorphism $A \rightarrow A\otimes_{\Z} B, a \mapsto a\otimes 1$.

As an specific case, suppose $A$ is a $\Z$-graded ring. Then \[A[x,x^{-1}]= A \otimes \Z[x,x^{-1}]\] is a strongly graded ring. Notice that 
with this grading, $A[x,x^{-1}]_0\cong A$.
\end{example}

\begin{example}\label{mominjianguomen}\scm[Strongly graded as a $\Gamma/\Omega$-graded ring]
\vspace{0.2cm}

Let $A$ be a $\Gamma$\!-graded ring. Using Proposition~\ref{crossedproductstronglygradedprop}, it is easy to see that if $A$ is a strongly $\Gamma$\!-graded ring, then it is also strongly $\Gamma/\Omega$-graded ring, where $\Omega$ is a subgroup of $\Gamma$. However the strongly gradedness is not a ``closed'' property, i.e, if $A$ is strongly $\Gamma/\Omega$-graded ring and $A_\Omega$ is strongly $\Omega$-graded ring, it does not follow that $A$ is strongly $\Gamma$\!-graded. 
\end{example}

\subsection{Crossed products}\label{khgfewa1}  \index{crossed product ring}

Natural examples of strongly graded rings are crossed product algebras (see Proposition~\ref{crossedproductstronglygradedprop}(3)).   They cover, as special cases, the skew group rings and twisted groups rings. We briefly describe the construction here. 

Let $A$ be a ring, $\Gamma$ a group (as usual we use the additive notation), and let $\phi:\Gamma\rightarrow \Aut(A)$ and $\psi:\Gamma\times \Gamma\rightarrow A^*$ be maps such that for any $\alpha,\beta,\gamma \in \Gamma$ and $a\in A$, 

\begin{enumerate}[\upshape(i)]

\item ${}^\alpha ({}^\beta a)=\psi(\alpha,\beta)\, {}^{\alpha + \beta} a \, \psi(\alpha,\beta)^{-1},$

\item $\psi(\alpha,\beta)\psi(\alpha + \beta,\gamma)={}^\alpha\psi(\beta,\gamma)\, \psi(\alpha,\beta + \gamma),$

\item $\psi(\alpha,0)=\psi(0,\alpha)=1$

\end{enumerate}

Here for $\alpha\in \Gamma$ and $a\in A$, $\phi(\alpha)(a)$ is denoted by ${}^\alpha a$. The map $\psi$ is called a \emph{2-cocycle map}. \index{cocycle map} Denote by $A^\phi_{\psi}[\Gamma]$ the free left $A$-module with the basis $\Gamma$, and define the multiplication by 
\begin{equation}\label{ppooiiaa23}
(a\alpha)(b\beta)=a\, {}^\alpha b \, \psi(\alpha,\beta )(\alpha + \beta).
\end{equation}

One can show that with this multiplication, $A^\phi_{\psi}[\Gamma]$ is a $\Gamma$\!-graded ring with homogeneous components $A\gamma$, $\gamma\in \Gamma$. In fact $\gamma\in A\gamma$ is invertible, so $A^\phi_{\psi}[\Gamma]$ is a 
crossed product algebra~\cite[Proposition~1.4.1]{grrings}. 

On the other hand any crossed product algebra is of this form (see~\cite[\S 1.4]{grrings}): for any $\ga \in \Gamma$ choose $u_\ga \in A^* \cap A_\ga$ and define $\phi:\Gamma\rightarrow \Aut(A_0)$ by  $\phi(\ga)(a)=u_\ga a u_\ga^{-1}$ for $\ga \in \Gamma$ and $a\in A_0$. Moreover, define the cocycle map 
\begin{align*}
\psi:\Gamma\times \Gamma &\longrightarrow A_0^*,\\
(\zeta,\eta) &\longmapsto u_\zeta u_\eta u_{\zeta + \eta}^{-1}.
\end{align*}
 Then \[A={A_0}_\psi^\phi[\Gamma]=\textstyle{\bigoplus_{\ga \in \Gamma}} A_0 \ga,\] with multiplication \[(a \zeta)(b \eta)=a {}^{\zeta} b \psi(\zeta,\eta) (\zeta+ \eta),\] where ${}^\zeta b$ is defined as $\phi(\zeta)(b)$. 

Note that when $\Gamma$ is cyclic,  one can choose $u_i=u_1^i$ for $u_1 \in A^* \cap A_1$ and thus the cocycle map $\psi$ is trivial, $\phi$ is a homomorphism and the crossed product is a skew group ring. In fact, if $\Gamma=\Z$, then the skew group ring becomes the so called 
\emph{skew Laurent polynomial ring}, denoted by $A_0[x,x^{-1},\phi]$. 
 Moreover if $u_1$ is in the centre of $A$, then $\phi$ is the identity  map and  the crossed product reduces to the group ring $A_0[\Gamma]$. A variant of this construction, namely corner skew polynomial rings is studied in~\S\ref{cornerskew}. 
 \index{skew Laurent polynomial ring} \index{group ring}

{\bf Skew group rings:} If $\psi:\Gamma\times \Gamma\rightarrow A^*$ is a trivial map, \ie $\psi(\alpha,\beta)=1$ for all $\alpha,\beta \in \Gamma$, then Conditions (ii) and (iii) trivially hold, and Condition (i) reduces to ${}^\alpha({}^\beta a)={}^{\alpha+\beta} a$ which means that $\phi:\Gamma\rightarrow \Aut(A)$ becomes a group homomorphism. In this case $A^{\phi}_\psi[\Gamma]$, denoted by $A\star_{\phi} \Gamma$, is a \emph{skew group ring} \index{skew group ring} with multiplication 
\begin{equation}\label{ppooiiaa}
(a \alpha)(b \beta)=a\, {}^\alpha b \, (\alpha+\beta) .
\end{equation}

{\bf Twisted group ring:} If $\phi:\Gamma\rightarrow \Aut(A)$ is trivial, \ie $\phi(\alpha)=1_A$ for all $\alpha\in \Gamma$, then Condition (i) implies that $\psi(\alpha,\beta)\in C(A)\cap A^*$ for any $\alpha,\beta\in \Gamma$. Here $C(A)$ stands for the centre of the ring $A$. \index{$C(A)$, centre of $A$} In this case $A^{\phi}_\psi[\Gamma]$, denoted by $A_{\psi}[\Gamma]$, is a \emph{twisted group ring} \index{twisted group ring} with multiplication 
\begin{equation}\label{ppooiiaa2}
(a \alpha)(b \beta)=a b \psi(\alpha,\beta) (\alpha+\beta) .
\end{equation}

A well-known theorem in the theory of central simple algebras states that if $D$ is a central simple $F$-algebra with a maximal subfield $L$ such that $L/F$ is a Galois extension and $[A:F]=[L:F]^2$, then $D$ is a crossed product, with $\Gamma=\Gal(L/F)$ and $A=L$ (see~\cite[\S12, Theorem~1]{draxl}).


Some of the graded rings we treat in this book are of the form $K[x,x^{-1}]$, where $K$ is a field. This is an example of a graded field. 

A $\Ga$\!-graded ring $A = \bigoplus_{ \ga \in \Ga} A_{\ga}$ is called
a \emph{graded division ring} \index{graded division ring} if every
nonzero homogeneous element has a multiplicative inverse.  If $A$ is also a commutative ring, then $A$ is called a  \index{graded field}  \emph{graded field}.

Let $A$ be a $\Gamma$\!-graded division ring. It follows from Proposition~\ref{basicsofgradedrings}(4) that $\Ga_A$
is a group, so we can write $A = \bigoplus_{ \ga \in \Ga_A}
A_{\ga}$. Then, as a $\Ga_A$-graded ring, $A$ is a crossed product
and it follows from
Proposition~\ref{crossedproductstronglygradedprop}(3) that $A$ is
strongly $\Ga_A$-graded. Note that if $\Gamma_A\not = \Gamma$, then $A$ is not strongly $\Gamma$\!-graded.   Also note that if $A$ is a graded division ring, then $A_0$ is a division ring. 

\begin{remark}\scm[Graded division rings and division rings which are graded]
\vspace{0.2cm}

Note that a graded division ring and a division ring which is graded are different. By definition, $A$ is a graded division ring if and only if $A^h \backslash \{0\}$ is a group. A simple example is the Laurent polynomial ring $D[x,x^{-1}]$, where $D$ is a division ring (Example~\ref{egofgrdivisionrings0}). Other examples show that a graded division ring does not need to be a domain (Example~\ref{goodraof}). However if the grade group is totally ordered, then a domain which is also graded, has to be concentrated in degree zero. Thus a division ring which is graded by a totally ordered grade group $\Gamma$, is of the form $A=\bigoplus_{\gamma \in \Gamma} A_\gamma$ where $A_0$ is a division ring and $A_\gamma=0$ for $\gamma\not = 0$. This will not be the case if $\Gamma$ is not totally ordered (see Example~\ref{egofgrdivisionrings}).  \index{totally ordered abelian group}
\end{remark}

In the following we give some concrete examples of graded division rings. 

\begin{example}\label{egofgrdivisionrings0}\scm[The Veronese subring]
\vspace{0.2cm}

Let $A=\bigoplus_{\gamma \in \Gamma} A_\gamma$ be a $\Gamma$\!-graded ring, where $\Gamma$ is a torsion-free group. For $n\in \Z \backslash \{ 0 \}$, the $n$th-\emph{Veronese subring} \index{Veronese subring} of $A$ is defined as $A^{(n)}=\bigoplus_{\gamma \in \Gamma} A_{n\gamma}$.  This is a $\Gamma$\!-graded ring with $A^{(n)}_\gamma=A_{n\gamma}$. It is easy to see that the support of $A^{(n)}$ is $\Gamma$ if the support of $A$ is $\Gamma$.  Note also that if $A$ is strongly graded, so is $A^{(n)}$. Clearly  $A^{(1)}=A$ and $A^{(-1)}$ is the graded ring with the components ``flipped'', \ie  $A^{(-1)}_\gamma=A_{-\gamma}$. For the case of $A^{(-1)}$ we don't need to require the grade group to be torsion-free. 

Let $D$ be a division ring and let $A= D[x, x^{-1}]$ be the Laurent polynomial ring. The elements of $A$ consist of finite sums $\sum_{i\in \mathbb Z} a_ix^{i}$, where $a_i \in D$. Then $A$ is a $\mathbb Z$-graded division ring 
with $A = \bigoplus_{i \in \mathbb{Z}} A_i$, where $A_i = \{ a x^{i} \mid
a \in D\}$. Consider the $n$th-Veronese subring $A^{(n)}$ which is the ring  $D[x^n, x^{-n}]$. The elements of $A^{(n)}$ consist of finite sums $\sum_{i\in \mathbb Z} a_ix^{in}$, where $a_i \in D$. Then $A^{(n)}$ is a $\mathbb Z$-graded division ring,
with $A^{(n)}= \bigoplus_{i \in \mathbb{Z}} A_{in}$. Here both $A$ and $A^{(n)}$ are strongly graded rings. 

There is also another way to consider the $\Z$-graded ring $B=D[x^n,x^{-n}]$ such that it becomes a graded subring of $A=D[x,x^{-1}]$. Namely, we define $B=\bigoplus_{i\in \Z} B_i$, where $B_i=Dx^{i}$ if $i\in n\Z$ and $B_i=0$ otherwise. This way $B$ is a graded division ring and a graded subring of $A$. 
The support of $B$ is clearly the subgroup $n\mathbb Z$ of $\mathbb Z$. With this definition, $B$ is not strongly graded. 
\end{example}

\begin{example}\label{egofgrdivisionrings}\scm[Different gradings on a graded division ring]
\vspace{0.2cm}

Let $\H = \R \+ \R i \+ \R j \+ \R k$ be the real \emph{quaternion
algebra}, \index{quaternion
algebra} with multiplication defined by $i^2 = -1$, $j^2= -1$ and $i
j = -ji = k$. It is known that $\H$ is a noncommutative division
ring with centre $\R$. We give  $\H$ two different
graded division ring structures, with grade groups $\Z_2 \times \Z_2$ and $\Z_2$ respectively as follows. 
\begin{description}

\item[$\Z_2\times \Z_2$-grading:] Let $\H = R_{(0,0)} \+ R_{(1,0)} \+
R_{(0,1)} \+ R_{(1,1)}$, where
$$
R_{(0,0)} = \, \R ,\;\; R_{(1,0)} = \, \R i,\;\; R_{(0,1)} = \, \R
j,\;\; R_{(1,1)} = \, \R k.
$$
It is routine to check that $\H$ forms a strongly $\Z_2\times \Z_2$-graded division ring.

\item[$\Z_2$-grading:] Let $\H = \C_0 \+ \C_1$,
where $\C_0 = \R \oplus \R i$ and $\C_1 = \C j=\R j \oplus \R k$. One can check that $\C_0\C_0=\C_0$, $\C_0 \C_1 =\C_1 \C_0 =\C_1$ and 
$\C_1 \C_1 =\C_0$. This makes $\H$ a  strongly $\Z_2$-graded division ring.  Note that this grading on $\H$ can be obtain from the first part by considering the quotient grade group $\Z_2\times\Z_2 / 0 \times \Z_2$ (\S\ref{mconfi1}). 
Quaternion algebras are examples of Clifford algebras (see Example~\ref{cliffman}).

\end{description}
\end{example}

The following generalises the above example of
quaternions as a $\Z_2 \times \Z_2$-graded ring.

\begin{example}\scm[Symbol algebras] \index{symbol algebra} \label{goodraof}
\vspace{0.2cm}

Let $F$ be a field, $\xi$ be a primitive $n$-th root of unity and
let $a$, $b \in F^*$. Let
$$
A = \bigoplus_{i=0}^{n-1} \bigoplus_{j=0}^{n-1} F x^i y^j
$$
be the $F$-algebra generated by the elements $x$ and $y$, which are
subject to the relations $x^n = a$, $y^n = b$ and $xy = \xi yx$. By
\cite[Theorem~11.1]{draxl}, $A$ is an $n^2$-dimensional central simple
algebra over $F$. We will show that $A$ forms a graded division ring. Clearly
$A$ can be written as a direct sum
$$
A= \bigoplus_{(i,j) \in \Z_n \+ \Z_n} A_{(i,j)}, \text{ \;\; where
 } A_{(i,j)} = F x^i y^j
$$
and each $A_{(i,j)}$ is an additive subgroup of $A$. Using the fact
that $\xi^{-kj} x^k y^j = y^j x^k$ for each $j,k$, with $0 \leq j,k
\leq n-1$, we can show that \[A_{(i,j)} A_{(k,l)} \subseteq
A_{([i+k], [j+l])},\] for $i,j,k,l \in \Z_n$. A nonzero homogeneous
element $f x^i y^j \in A_{(i,j)}$ has an inverse $$f^{-1} a^{-1}
b^{-1} \xi^{-ij} x^{n-i} y^{n-j},$$ proving $A$ is a graded division
ring. Clearly the support of $A$ is $\Z_n \times \Z_n$, so $A$ is
strongly $\Z_n \times \Z_n$-graded. 
\end{example}

These examples can also be obtained from graded free rings (see Example~\ref{pearlfish}). 

\begin{example}\scm[A good counter-example]
\vspace{0.2cm}

In the theory of graded rings, in many instances it has been established that if the grade group $\Gamma$ is finite (or in some cases, finitely generated), then a graded property implies the corresponding ungraded property of the ring (\ie the property is preserved under the forgetful functor). For example, one can prove that  if a $\mathbb Z$-graded ring is graded Artinian (Noetherian), then the ring is Artinian (Noetherian). One good example which provides counter-examples to such phenomena is the following graded field. \index{Noetherian ring} \index{Artinian ring}  \index{graded Noetherian ring} \index{graded Artinian ring}

Let $K$ be a field and $A=K[x_1^{\pm 1},x_2^{\pm 1},x_3^{\pm 1},\dots]$  a Laurent polynomial ring in a countable many variables. This ring is a graded field with its  `canonical' $\bigoplus_{\infty} \mathbb Z$-grading and thus it is graded Artinian and Noetherian. However $A$ is not Noetherian. 
\end{example}

\subsection{Graded ideals}\label{jghjrye}
Let $A$ be a $\Gamma$\!-graded ring. 
A two-sided ideal $I$ of $A$ is called a  \emph{graded ideal}\index{graded
ideal} (or \emph{homogeneous
ideal}\index{homogeneous ideal}) if
\begin{equation}\label{viewfresh19}
I= \bigoplus_{ \ga \in \Gamma} (I \cap A_{\ga}).
\end{equation}
Thus $I$ is a graded ideal if and only if for any $x\in I$, $x=\sum x_i$, where $x_i \in A^h$, implies that $x_i \in I$. 

The notions of a \emph{graded subring}, a \emph{graded left} and
a \emph{graded right ideal} are defined similarly. \index{graded subring} \index{graded left ideal} \index{graded right ideal}

Let $I$ be a graded ideal of $A$. Then
the quotient ring $A/I$ forms a graded ring, with
\begin{equation}\label{whst8}
\boxed{A/I = \bigoplus_{\ga \in \Ga} (A/I)_{\ga}, \mathrm{ \;\;\; where
\;\;\; } (A/I)_{\ga} = (A_{\ga} +I)/I.}
\end{equation}

With this grading $(A/I)_0 \cong A_0 /I_0$, where $I_0=A_0\cap I$. From~(\ref{viewfresh19}) it follows that an ideal $I$ of $A$ is a graded ideal if and only
if $I$ is generated as a two-sided ideal of $A$ by homogeneous
elements. Also, for a two-sided ideal $I$ of $A$, if~(\ref{whst8}) induces a grading on $A/I$, then $I$ has to be a graded ideal. 
By Proposition~\ref{crossedproductstronglygradedprop}(4), if $A$ is a strongly graded or a crossed product, so is the graded quotient ring $A/I$. 

\begin{example}\scm[Symmetric and exterior algebras as $\Z$-graded rings]
\vspace{0.2cm}

Recall from Example~\ref{stevevanzandt} that for a commutative ring $A$ and  an $A$-module $M$, the tensor algebra $T(M)$ is a $\mathbb Z$-graded ring with support $\mathbb N$. The  \emph{symmetric algebra} \index{symmetric algebra} $S(M)$ is defined as the quotient of $T(M)$ by the ideal generated by elements $x\otimes y -y\otimes x$, $x,y \in M$. Since these elements are homogeneous of degree two, $S(M)$ is a $\mathbb Z$-graded commutative ring. 

Similarly, the  \emph{exterior algebra} \index{exterior algebra} of $M$, denoted by $\bigwedge M$, is defined as the quotient of $T(M)$ by the ideal generated by homogeneous elements $x\otimes x$, $x\in M$. So $\bigwedge M$ is a $\mathbb Z$-graded ring.
\end{example}

Let $I$ be a two-sided ideal of a $\Gamma$\!-graded ring $A$ generated by a subset $\{a_i\}$ of not necessarily homogeneous elements of $A$. If $\Omega$ is a subgroup of $\Gamma$ such that $\{a_i\}$ are homogeneous elements in $\Gamma/\Omega$-graded ring $A$ (see~\S\ref{mconfi1}), then clearly $I$ is a $\Gamma/\Omega$-graded ideal and consequently $A/I$ is a $\Gamma/\Omega$-graded ring.

\begin{example}\scm[Clifford algebras as $\Z_2$-graded rings]  \label{cliffman}
\vspace{0.2cm}

Let $V$ be a $F$-vector space and $q:V\rightarrow F$ be a quadratic form with its
associated non degenerate symmetric bilinear form $B:V\times
V\rightarrow F$. 

 The \emph{Clifford algebra} \index{Clifford algebra} associated to $(V,q)$ is defined as 
\[\Cl(V,q):=T(V)/ \langle   v\otimes v+q(v)  \rangle. \]
Considering $T(V)$ as a $\Z/2\Z$-graded ring (see~\S\ref{mconfi1}), the elements $v\otimes v -q(v)$ are homogeneous of degree zero. This induces a $\Z_2$-graded 
structure on $\Cl(V,q)$.  Identifying $V$ with its image in the Clifford algebra $\Cl(V,q)$, $V$ lies in the odd part of the
Clifford algebra, \ie  $V\subset \Cl(V,q)_1$.

If $\chr(F) \not = 2$, as $B$ is non degenerate, there
exist $x, y\in V$ such that $B(x,y)=1/2$, and thus  \[xy+yx=2B(x,y)=1 \in  \Cl(V,q)_1  \Cl(V,q)_1.\]

Similarly, if $\chr(F) = 2$, there
exist $x, y\in V$ such that $B(x,y)=1$, so \[xy+yx=B(x,y)=1  \in  \Cl(V,q)_1  \Cl(V,q)_1.\]

It follows from Proposition~\ref{crossedproductstronglygradedprop} that Clifford algebras are strongly $\Z_2$-graded rings. 
\end{example}

Recall that for  $\Ga$\!-graded rings $A$ and $B$, a $\Gamma$\!-graded ring
homomorphism $f:A \ra B$ is a ring
homomorphism such that $f(A_{\ga}) \subseteq B_{\gamma}$ for all
$\ga \in \Ga$. It can easily be
shown that $\ker (f)$ is a graded ideal of $A$ and $\im (f)$ is a
graded subring of $B$. It is also easy to see that $f$ is injective (surjective/bijective) if and only if for any $\ga \in \Gamma$, the restriction of $f$ on $A_\ga$ is injective (surjective/bijective).

Note that if $\Ga$ is an abelian group, then
the centre of a graded ring $A$, $C(A)$,  is a graded subring of $A$. More generally, the centraliser of a set of homogeneous elements is a graded subring. 
\index{centre of a ring}  \index{centraliser}

\begin{example}\label{millarexi}\scm[The centre of the graded ring]
\vspace{0.2cm}

If a group $\Ga$ is not abelian, then the centre of a $\Gamma$\!-graded ring may not be a graded
subring. For example, let $\Gamma = S_3 = \{e,a,b,c,d,f \}$ be the symmetric group of order
$3$, where
$$
a = (23),\;\; b=(13) ,\;\; c=(12) ,\;\; d=(123) ,\;\; f=(132).
$$
Let $A$ be a ring, and consider the group ring $R = A[\Gamma]$, which is
a $\Gamma$\!-graded ring by Example~\ref{hy431}. Let $x = 1d+
1f \in R$, where $1= 1_A$, and we note that $x$ is not homogeneous
in $R$. Then $x \in Z( R)$, but the homogeneous components of $x$
are not in the centre of $R$. As $x$ is expressed uniquely as the
sum of homogeneous components, we have $x \notin \bigoplus_{\gamma \in \Gamma}
(Z(R) \cap R_\gamma)$.

This example can be generalised by taking a non-abelian finite group
$\Gamma$ with a subgroup $\Omega$ which is normal and noncentral. Let $A$ be
a ring and consider the group ring $R= A[\Gamma]$ as above. Then $x =
\sum_{\omega \in \Omega} 1 \omega$ is in the centre of $R$, but the homogeneous
components of $x$ are not all in the centre of $R$.
\end{example}

\begin{remark}\label{shanbei}
Let $\Gamma$ and $\Lambda$ be two groups. Let $A$ be a $\Gamma$\!-graded ring and $B$ be a $\Lambda$-graded ring. Suppose $f:A\rightarrow B$ is a ring homomorphism and $g:\Gamma\rightarrow \Lambda$ a group homomorphism such that for any $\ga \in \Gamma$, $f(A_\ga)\subseteq B_{g(\ga)}$. Then $f$ is called a $\Gamma\!-\!\Lambda$-graded homomorphism. In case $\Gamma=\Lambda$ and $g=\id$,  we recover the usual definition of a $\Gamma$\!-graded homomorphism. For example, if $\Omega$ is a subgroup of $\Gamma$, then the identity map $1_A:A\rightarrow A$ is a $\Gamma\!-\!\Gamma/\Omega$-graded homomorphism, where $A$ is considered as $\Gamma$ and $\Gamma/\Omega$-graded rings, respectively (see~\S\ref{mconfi1}).

 Throughout this book, we fix a given group $\Gamma$ and we work with the $\Gamma$\!-graded category and  all our considerations are within this category (See Remark~\ref{bbhiidw} for references to literature where mixed grading is studied). 
\end{remark}

\subsection{Graded prime and maximal ideals} 

A  graded ideal $P$ of $\Gamma$\!-graded ring $A$ is called a \emph{graded prime ideal} \index{graded prime ideal} of $A$ if $P\not = A$ and for any two graded ideals $I,J$, $IJ\subseteq P$, implies $I\subseteq P$ or $J \subseteq P$.   If $A$ is commutative, we obtain the familiar formulation that $P$ is a graded prime ideal if and only if for $x,y\in A^h$, $xy\in P$ implies that $x\in P$ or $y\in P$. Note that a graded prime ideal is not necessarily a prime ideal.

A  graded ideal $P$ is called a  \emph{graded semi-prime ideal} \index{graded semi-prime ideal} if for any graded ideal $I$ in $A$, $I^2\subseteq P$,  implies $I \subseteq P$. A graded ring $A$ is called a \emph{graded prime (graded semi-prime)} ring if the zero ideal is a graded prime 
(graded semi-prime) ideal. \index{graded prime ring} \index{graded semi-prime ring}

A \emph{graded maximal ideal} \index{graded maximal ideal} of a $\Gamma$\!-graded ring $A$ is defined to be a proper graded ideal of $A$ which is maximal among the set of proper graded ideals of $A$. Using Zorn's lemma, one can show that graded maximal ideals exist, and it is not difficult to show that a graded maximal ideal is a graded prime. For a graded commutative ring, a graded ideal is maximal if and only if its quotient ring is graded field. There are similar notions of graded maximal left and right ideals.

Parallel to the ungraded setting, for a $\Gamma$\!-graded ring $A$, the \emph{graded Jacobson radical},  \index{graded Jacobson radical} $J^{\gr}(A)$, is defined as the intersection of all graded left maximal ideals of $A$.  This coincides with the intersection of all graded right maximal ideals and so 
$J^{\gr}(A)$ is a two-sided ideal (see~\cite[Proposition~2.9.1]{grrings}). We denote by $J(A)$ the usual Jacobson radical. 
It is a theorem of G.~Bergman that for a $\mathbb Z$-graded ring $A$, $J(A)$ is a graded ideal  and $J(A) \subseteq J^{\gr}(A)$ (see~\cite{bergman}). 
\index{Jacobson radical} 
\index{$J^{\gr}(A)$, graded Jacobson radical}  \index{$J(A)$, Jacobson radical}

\subsection{Graded simple rings}

A nonzero graded ring $A$ is said to be {\em graded simple}\index{graded simple ring} if the only graded two-sided ideals of $A$ \index{simple ring}
are $\{ 0 \}$ and $A$. The structure of graded simple Artinian rings are known (see Remark~\ref{jussise}). 
Following~\cite{jesper} we prove that a graded ring $A$ is simple if and only if $A$ is graded simple and $C(A)$, the centre of $A$, is a field.  \index{$C(A)$, centre of $A$}

For a $\Gamma$\!-graded ring $A$, recall the support $\Gamma_A$ of $A$, from~\S\ref{jdjthu}. For $a\in A$, writing $a=\sum_{\gamma\in \Gamma} a_\gamma$ where $a_\gamma \in A^h$, define the \emph{support} of $a$, to be \[\Gamma_a=\big  \{\, \gamma \mid a_\gamma \not =0 \, \big \}.\]  \index{support of an element} \index{$\Gamma_a$, support of the element $a$} We also need the notion of minimal support.
A finite set $X$ of $\Gamma$ is called a \emph{minimal support with respect to an ideal} $I$, if $X=\Gamma_a$ for $0\not = a \in I$ and there is no 
$b \in I$ such that $b\not = 0$ and $\Gamma_b \subsetneq \Gamma_a$. \index{minimal support with respect to an ideal}

We start with a lemma.

\begin{lemma}\label{dmjfi23215}
Let $A$ be a $\Gamma$\!-graded simple ring and $I$ an ideal of $A$. Let $0\not = a \in I$ with $\Gamma_a=\{\gamma_1,\dots,\gamma_n\}$. 
Then for any $\alpha \in \Gamma_A$, there is a $0 \not = b \in I$ with $\Gamma_b\subseteq \{\gamma_1-\gamma_n+\alpha,\dots,
\gamma_n-\gamma_n+\alpha\}$. 
\end{lemma}
\begin{proof}
Let $0\not = x\in A_\alpha$, where $\alpha \in \Gamma_A$ and $0\not = a \in I$ with $\Gamma_a=\{\gamma_1,\dots,\gamma_n\}$. 
Write $a=\sum_{i=1}^n a_{\gamma_i}$, where $\deg(a_{\gamma_i})=\gamma_i$. Since $A$ is graded simple, 
\begin{equation}\label{panfhgf}
x=\sum_l r_la_{\gamma_n}s_l,
\end{equation} where $r_l,s_l \in A^h$. Thus there are $r_k,s_k \in A^h$ such that $r_ka_{\gamma_n}s_k \not =0$ which implies that $b:=r_ka s_k\in I$ is not zero.
Comparing the degrees in Equation~(\ref{panfhgf}), it follows $\alpha=\deg(r_k)+\deg(s_k)+\gamma_n$, or 
$\deg(r_k)+\deg(s_k)=\alpha-\gamma_n$. So 
\[\Gamma_b \subseteq \Gamma_a+\deg(r_k)+\deg(s_k) = \{\,\gamma_1-\gamma_n+\alpha,\dots,
\gamma_n-\gamma_n+\alpha\,\}.\qedhere\]
\end{proof}

\begin{theorem}\label{plif63}
Let $A$ be a $\Gamma$\!-graded ring. Then $A$ is a simple ring if and only if $A$ is a graded simple ring and C(A) is a field.
\end{theorem}
\begin{proof}
One direction is straightforward. 

Suppose $A$ is graded simple and $C(A)$ is a field. We will show that $A$ is a simple ring. Suppose $I$ is a nontrivial ideal of $A$ and $0\not = a \in I$ with $\Gamma_a$ a minimal support with respect to $I$.  For any $x\in A^h$, with $\deg(x)=\alpha$ and $\gamma\in \Gamma_a$, we have 
\begin{equation}\label{hyhtgfdi}
\Gamma_{axa_\gamma -a_\gamma x a} \subsetneq \Gamma_a+(\gamma+\alpha).
\end{equation}
Set $b=axa_\gamma -a_\gamma x a\in I$. Suppose $b\not = 0$. By~(\ref{hyhtgfdi}), 
\[\Gamma_b \subsetneq \{\,\gamma_1+\gamma+\alpha,\dots,\gamma_n+\gamma+\alpha \,\}.\]
Applying Lemma~\ref{dmjfi23215} with, say, $\gamma_n+\gamma+\alpha \in \Gamma_b$ and $\gamma_n \in \Gamma_A$, we obtain a $0 \not = c\in I$ such that 
\[\Gamma_c \subseteq \Gamma_b + (\gamma_n-\gamma_n-\gamma-\alpha) \subsetneq \Gamma_a.\]
This is however a contradiction as $\Gamma_a$ was a minimal support. Thus $b=0$, i.e, $axa_\gamma=a_\gamma x a$. It follows 
that for  any  $\gamma_i \in \Gamma_a$
\begin{equation}\label{yhfdreh5}
a_{\gamma_i} x a_\gamma =a_\gamma x a_{\gamma_i}.
\end{equation}
Consider the $R$-bimodule map
\begin{align*}
\phi:R=\langle a_{\gamma_i} \rangle &\longrightarrow \langle a_{\gamma_j} \rangle=R,\\
\sum_l r_l a_{\gamma_i} s_l &\longmapsto \sum_l r_l a_{\gamma_j} s_l .
\end{align*}
To show that $\phi$ is well-defined, since $\phi(t+s)=\phi(t)+\phi(s)$, it is enough to show that if $t=0$ then $\phi(t)=0$, where $t\in  \langle a_{\gamma_i} \rangle$. Suppose $\sum_l r_l a_{\gamma_i} s_l=0$. Then for any $x\in A^h$, using~(\ref{yhfdreh5}) we have 
\[0=a_{\gamma_j}x\ \big(\sum_l r_l a_{\gamma_i} s_l\big)=\sum_l a_{\gamma_j}x r_l a_{\gamma_i} s_l= \sum_l a_{\gamma_i}
xr_la_{\gamma_j}s_l=a_{\gamma_i}x\ \big(\sum_l r_l a_{\gamma_j} s_l\big).\]
Since $A$ is graded simple, $\langle a_{\gamma_i} \rangle =1 $. It follows that $\sum_l r_l a_{\gamma_j} s_l=0$. Thus $\phi$ is well-defined, injective and also clearly surjective. Then $a_{\gamma_j}=\phi(a_{\gamma_i})=a_{\gamma_i}\phi(1)$. But $\phi(1) \in C(A)$. Thus
$a=\sum_j a_{\gamma_j}=a_{\gamma_i}c$ where $c\in C(A)$. But $C(A)$ is a field, so $a_{\gamma_i}=ac^{-1} \in I$. Again since $R$ is graded simple, it follows that $I=R$. This finishes the proof. 
\end{proof}

\begin{remark}
If the grade group is not abelian, in order for Theorem~\ref{plif63} to be valid, the grade group should be hyper-central; A \emph{hyper-central group} is a group that any nontrivial quotient has a nontrivial centre. \index{hyper-central group} If $A$ is strongly graded, and the grade group is torsion-free hyper-central, then $A$ is simple if and only if $A$ is graded simple and $C(A) \subseteq A_0$ (see~\cite{jesper}).
\end{remark}

\begin{remark}\scm[Graded simplicity implying simplicity] 
\vspace{0.2cm}

There are other cases that the graded simplicity of a ring implies that the ring itself is simple. For example if a ring is graded by an ordered group (such as $\Z$), and has a finite support, then graded simplicity implies the simplicity of the ring~\cite[Theorem~3]{bahturin}. 

\end{remark}

\subsection{Graded local rings}\label{merrychirstmass}

\index{graded local ring} 
Recall that a ring is a \emph{local ring} \index{local ring} if the set of noninvertible elements form a two-sided ideal. When $A$ is a commutative ring, then $A$ is local if and only if $A$ has a unique maximal ideal. 

A $\Gamma$\!-graded ring $A$ is called a \emph{graded local ring} \index{graded local ring} if the two-sided ideal $M$ generate by noninvertible homogeneous elements is a proper ideal. One can easily observe that the graded ideal $M$ is the unique graded maximal left, right, and graded two-sided ideal of $A$. When $A$ is a graded commutative ring, then $A$ is graded local if and only if $A$ has a unique graded maximal ideal. 

If $A$ is a graded local ring, then the graded ring $A/M$ is a graded division ring. One can further show that $A_0$ is a local ring with the unique maximal ideal $A_0\cap M$. In fact we have the following proposition.

\begin{proposition}
Let $A$ be a $\Gamma$\!-graded ring. Then $A$ is a graded local ring if and only if $A_0$ is a local ring. 
\end{proposition}
\begin{proof}
Suppose $A$ is a graded local ring. Then by definition, the two-sided ideal $M$ generated by noninvertible homogeneous elements is a proper ideal. Consider $m=A_0 \cap M$ which is a proper ideal of $A_0$. Suppose $x\in A_0 \backslash m$. Then $x$ is a homogeneous element which is not in $M$. Thus $x$ has to be invertible in $A$ and consequently in $A_0$. This shows that $A_0$ is a local ring with the unique maximal ideal $m$. 

Conversely, suppose $A_0$ is a local ring. We first show that any left or right invertible homogeneous element is a two-sided invertible element. Let $a$ be a left invertible homogeneous element. Then there is a homogeneous element $b$ such that $ba=1$. If $ab$ is not right invertible, then $ab\in m$, where $m$ is the unique maximal ideal of the local ring $A_0$. Thus $1-ab \not \in m$ which implies that $1-ab$ is invertible. But $(1-ab)a=a-aba=a-a=0$, and since $1-ab$ is invertible, we get $a=0$ which is a contradiction to the fact that $a$ has a left inverse. Thus $a$ has a right inverse and so is invertible.  A similar argument can be written for right invertible elements. Now let $M$ be the ideal generated by all noninvertible homogeneous elements of $A$. We will show that $M$ is proper, and thus $A$ is a graded local ring. 
Suppose $M$ is not proper. Thus $1=\sum_i r_ia_is_i$, where $a_i$ are noninvertible homogeneous elements and $r_i,s_i$ are homogeneous elements such that $\deg (r_ia_is_i)=0$. If $r_ia_is_i$ is invertible for some $i$, using the fact that right and left invertibles are invertibles, it follows that $a_i$ is invertible which is a contradiction. 
Thus $r_ia_is_i$, for all $i$, are homogeneous elements of degree zero and not invertible. So they are all in $m$. This implies that $1\in m$ which is a contradiction. Thus $M$ is a proper ideal of $A$. 
\end{proof}

For more on graded local rings (graded by a cancellative monoid) see~\cite{huishi}.  In~\S\ref{attila1} we determine the graded Grothendieck group of these rings.

\subsection{Graded von Neumann regular rings}\label{penrith7may}

The von Neumann regular rings constitute an important class of rings. A unital ring $A$ is \emph{von Neumann regular}, \index{von Neumann regular ring} if for any $a \in A$, we have $a \in aAa$. 
There are several equivalent module theoretical definitions, such as $A$ is von Neumann regular if and only if any module over $A$ is flat. 
This gives a comparison with the class of division rings and semisimple rings. A ring is a division ring if and only if any module is free. A semisimple ring is characterised by the property that any module is projective. 
Goodearl's book~\cite{goodearlbook} is devoted to the class of von Neumann regular rings.  The definition extends to non-unital ring in an obvious manner.  

If a ring has a graded structure, one defines the graded version of regularity in a natural way: the graded ring $A$ is called a \emph{graded von Neumann regular}, \index{graded von Neumann regular ring} if for any homogeneous element $a\in A$, we have $a\in aAa$. This means, for any homogeneous element $a\in A$, one can find a homogeneous element $b\in A$ such that $a=aba$.  As an example, a direct sum of graded division rings is a graded von Neumann regular ring. Many of the module theoretic properties established for von Neumann regular rings can be extended to the graded setting; For example, $A$ is graded regular if and only if any graded module is (graded) flat.  \index{graded flat module} \index{flat module} We refer the reader to 
\cite[C, I.5]{grrings2} for a treatment of such rings and~\cite[\S2.2]{balaba} for a concise survey. Several of the graded rings we construct in this book are graded von Neumann regular, such as Leavitt path algebras (Corollary~\ref{hgthusu2}) and corner skew Laurent series (Proposition~\ref{lanhc8}).  

In this section, we briefly give some of the properties of graded von Neumann regular rings. The following proposition is the graded version of 
~\cite[Theorem~1.1]{goodearlbook} which has a similar proof. 

\begin{proposition} \label{pkhti1}
Let $A$ be a $\Gamma$\!-graded ring. The following statements are equivalent. 

\begin{enumerate}[\upshape(1)]
\item  $A$ is a graded von Neumann regular ring;

\item Any finitely generated right (left) graded ideal of $A$ is generated by a homogeneous
idempotent.

 \end{enumerate}
\end{proposition}
\begin{proof}

(1) $\Rightarrow$ (2) First we show that any principal graded ideal is generated by a homogeneous idempotent. So consider the principal ideal  $xA$, where $x\in A^h$. By the assumption, there is $y\in A^h$ such that $xyx=x$. This immediately implies $xA=xyA$. Now note that $xy$ is  homogeneous idempotent. 
 
Next we will prove the claim for graded ideals generated by two elements. The general case follows by an easy induction. So let $xA+yA$ be a graded ideal generated by two homogeneous elements $x,y$. By the previous paragraph, $xA=eA$ for a homogeneous idempotent $e$. Note that $y-ey\in A^h$ and 
$y-ey \in xA+yA$. Thus 
\begin{equation}\label{classigbf}
xA+yA=eA+(y-ey)A.
\end{equation}
Again, the previous paragraph gives us a homogeneous idempotent $f$ such that $(y-ey)A=fA$. Let $g=f-fe \in A_0$. Notice that $ef=0$ which implies that $e$ and $g$ are orthogonal idempotents. Moreover, $fg=g$ and $gf=f$.  It then follows $gA=fA=(y-ey)A$. Now from (\ref{classigbf}), we get
\[xA+yA=eA+gA=(e+g)A.\] 

(2) $\Rightarrow$ (1) Let $x\in A^h$. Then $xA=eA$ for some homogeneous idempotent $e$. Thus $x=ea$ and $e=xy$ for some $a,y\in A^h$. 
Then $x=ea=eea=ex=xyx$. 
\end{proof}

\index{graded Jacobson radical}

\begin{proposition}\label{pkhti12}
Let $A$ be a  $\Gamma$\!-graded von Neumann regular ring. Then

\begin{enumerate}[\upshape(1)]
\item Any graded right (left) ideal of $A$ is idempotent;

\item Any graded ideal is graded semi-prime; \index{graded semi-prime ideal}

\item Any finitely generated right (left) graded ideal of $A$ is a projective module. 

\end{enumerate}

Moreover, if $A$ is a $\mathbb Z$-graded regular ring  then, 

\begin{enumerate}[\upshape(4)]

\item $J(A)=J^{\gr}(A) =0$. 

\end{enumerate}
\end{proposition}
\begin{proof}
The proofs of (1)-(3)  are similar to the ungraded case~\cite[Corollary~1.2]{goodearlbook}. We provide the easy proofs here. 

(1) Let $I$ be a graded right ideal. For any homogeneous element $x \in I$, there is $y\in A^h$ such that $x=xyx$. Thus $x=(xy)x \in I^2$. It follows that $I^2=I$. 

(2) This follows immediately from (1). 

(3) By Proposition~\ref{pkhti1}, any finitely generated right ideal is generated by a homogeneous idempotent. However this latter ideal is a direct summand of the ring, and so is a projective module. 

(4)  By Bergman's observation, for a $\mathbb Z$-graded ring $A$, $J(A)$ is a graded ideal  and $J(A) \subseteq J^{\gr}(A)$ (see~\cite{bergman}).  By Proposition~\ref{pkhti1}, $J^{\gr}(A)$ contains an idempotent, which then forces $J^{\gr}(A)=0$. 
\end{proof}

If the graded ring $A$ is strongly graded then one can show that there is a one-to-one correspondence between the right ideals of $A_0$ and the graded right ideals of $A$ (similarly for the left ideals) (see Remark~\ref{miener18}). This is always the case for the graded regular rings as the following proposition shows.

\begin{proposition}
Let $A$ be a $\Gamma$\!-graded von Neumann regular ring. Then there is a one-to-one correspondence between the right (left) ideals of $A_0$ and the graded right (left) ideals of $A$. 
\end{proposition} 
\begin{proof}
Consider the following correspondences between the graded right ideals of $A$ and the right ideals of $A_0$: For a graded right ideal $I$ of $A$ assign $I_0$ in $A_0$ and for a right ideal $J$ in $A_0$ assign the graded right ideal $JA$ in $A$. Note that $(JA)_0=J$. We show that $I_0A=I$. It is enough to show that any homogeneous element $a$ of $I$ belongs to $I_0A$. Since $A$ is graded regular, $axa=a$ for some $x \in A^h$. But $ax\in I_0$ and thus $a=axa \in I_0A$. A similar proof gives the left ideal correspondence.
\end{proof}

In Theorem~\ref{frhghu} we give yet another characterisation of graded von Neumann regular rings based on the concept of divisible modules. 

Later in Corollary~\ref{vonregu}, we show that if $A$ is a strongly graded ring, then $A$ is graded von Neumann regular if and only if $A_0$ is a von Neumann regular ring. The proof uses the equivalence of suitable categories over the rings $A$ and $A_0$. An element-wise proof of this fact can also be found in~\cite[Theorem~3]{yahya}.

\section{Graded modules}

\subsection{Basic definitions} \label{pokjgtp}
Let $A$ be a $\Gamma$\!-graded ring. A \emph{graded right $A$-module} \index{graded right module} $M$ is defined to be a right $A$-module $M$ 
with a direct sum decomposition $M=\bigoplus_{\ga \in \Gamma}
M_{\ga}$, where each $M_{\ga}$ is an additive subgroup of $M$  such that 
$M_{\la}  A_{\ga} \subseteq M_{\la+\ga}$ for all $\ga, \la \in
\Ga$. 

For  $\Ga$\!-graded right $A$-modules $M$ and $N$, a $\Gamma$\!-\emph{graded module
homomorphism}\index{graded module homomorphism} $f:M \ra N$ is a module
homomorphism such that $f(M_{\ga}) \subseteq N_{\gamma}$ for all
$\ga \in \Ga$.  A graded homomorphism $f$ is called a \emph{graded module isomorphism} \index{graded module isomorphism} if
$f$ is bijective and, when such a graded isomorphism exists, we
write $M \conggr N$. Notice that if $f$ is a graded module 
homomorphism which is bijective, then its inverse $f^{-1}$ is also a
graded module homomorphism.

\subsection{Shift of modules}\label{pokiss}

Let $M$ be a graded right $A$-module. For $\de \in \Ga$, we define the
$\de$-\emph{suspended}, \index{suspended module} or $\de$-\emph{shifted} \index{shifted} graded right $A$-module $M(\de)$ \label{deshiftedmodule} as 
\[\boxed{M(\de) =\bigoplus_{\ga \in \Ga} M(\de)_\ga, \text{ where }M(\de)_\ga =
M_{\de+\ga}.}\] This shift plays a pivotal role in the theory of graded rings. For example, if $M$ is a $\Z$-graded $A$-module, then the following table shows how the shift like ``the tick of the clock'' moves the homogeneous components of $M$ to the left.


\begin{center}
  \begin{tabular}{|l|cccccccc|}\hline
       & \textbf{degrees}  & \textbf{-3} & \textbf{-2}  & \textbf{-1} & \textbf{0} & \textbf{1} &\textbf{2}  &\textbf{3}  \\ \hline
    $M$     &  &      &                 &$M_{-1}$ &  $M_{0\phantom{-}}$        & $M_{1\phantom{-}}$ &  $M_{2\phantom{-}}$ &  $\phantom{M_2}$ \\
    \hline
    $M(1)$ &  &     & $M_{-1}$& $M_{0\phantom{-}}$    &    $M_{1\phantom{-}}$     &  $M_{2\phantom{-}}$   & & \\
    \hline
    $M(2)$ &  &   $M_{-1}$   &   $M_{0\phantom{-}}$ &  $M_{1\phantom{-}}$     &   $M_{2\phantom{-}}$    &   &  &\\
    \hline
   \end{tabular}
\end{center}

Let $M$ be a $\Gamma$\!-graded right $A$-module.  
A submodule $N$ of $M$ is called a  \emph{graded submodule}\index{graded
submodule}  if
$$
N= \bigoplus_{ \ga \in \Gamma} (N \cap M_{\ga}).
$$

\begin{example}\scm[$aA$ as a graded ideal and a graded module] \label{monash1}
\vspace{0.2cm}

Let $A$ be a $\Gamma$\!-graded ring and $a\in A$  a homogeneous element of degree $\alpha$. Then $aA$ is a graded right $A$-module with 
$\gamma \in \Gamma$ homogeneous component defined as \[(aA)_{\gamma}:=aA_{\gamma-\alpha} \subseteq A_\gamma.\] With this grading $aA$ is a graded submodule (and graded right ideal) of $A$. Thus for $\beta\in \Gamma$, $a$ is a homogenous element of the graded $A$-module $aA(\beta)$ of degree $\alpha-\beta$. This will be used throughout the note, for example in Proposition~\ref{hyhfdtegdt}. 

However, note that defining the grading on $aA$ as \[(aA)_\gamma:=aA_\gamma \subseteq A_{\gamma+\alpha}\] makes $aA$ a graded 
submodule of $A(\alpha)$, which is the image of the graded homomorphism $A\rightarrow A(\alpha), r \mapsto ar$. 
\end{example}

There are similar notions of graded left and graded bi-submodules (\S\ref{bimref}). When $N$ is a graded submodule of $M$, the
factor module $M/N$ forms a graded $A$-module, with
\begin{equation}\label{whs21}
\boxed{M/N = \bigoplus_{\ga \in \Ga} (M/N)_{\ga}, \mathrm{ \;\;\; where
\;\;\; } (M/N)_{\ga} = (M_{\ga} +N)/N.}
\end{equation}

\begin{example}\label{monash}
Let $A$ be a $\Gamma$\!-graded ring. Define a grading on the matrix ring $\M_n(A)$ as follows. For $\alpha \in \Gamma$, 
$\M_n(A)_\alpha=\M_n(A_\alpha)$ (for a general theory of grading on matrix rings see~\S\ref{matgrhe}). Let $\e_{ii} \in \M_n(A)$, $1\leq i\leq n$,  be a \emph{matrix unit}, \ie a matrix with $1$ in 
$(i,i)$ position and zero everywhere else, and consider $\e_{ii}\M_n(A)$. By Example~\ref{monash1}, $\e_{ii}\M_n(A)$ is a graded $\M_n(A)$-module and 
\[\bigoplus_{i=1}^n \e_{ii}\M_n(A)=\M_n(A).\] This shows that the graded module $\e_{ii} \M_n(A)$ is a projective module. This is an example of a graded projective module (see~\S\ref{monashre}). \index{$\e_{ij}$, matrix unit}  \index{matrix units} 
\end{example}

\begin{example}\label{qub8}
Let $A$ be a commutative ring. Consider the matrix ring $\M_n(A)$ as a $\mathbb Z$-graded ring concentrated in degree zero. Moreover, consider $\M_n(A)$ as a graded $\M_n(A)$-module with the grading defined as follows: $\M_n(A)_i=\e_{ii} \M_n(A)$ for $1\leq i \leq n$ and zero otherwise. Note that all nonzero homogeneous elements of this module are zero-divisiors, and thus can't constitute a linear independent set. We will use this example to show that a free module which is graded is not necessarily a graded free module (\S\ref{hefeijuly}). 
\end{example}

\begin{example}\scm[Modules with no shift]\label{supinmalnly}
\vspace{0.2cm}

It is easy to construct modules whose shifts don't produce new (nonisomorphic) graded modules. Let $M$ be a graded $A$-module and consider \[N=\bigoplus_{\gamma \in \Gamma} M(\gamma).\] We show that $N\cong_{\gr} N(\alpha)$ for any $\alpha \in \Gamma$. Define the map $f_\alpha:N \rightarrow N(\alpha)$ on homogeneous components as follows and extend it to $N$, 
\begin{align*}
N_\beta=\bigoplus_{\gamma \in \Gamma} M_{\gamma+\beta}&\longrightarrow \bigoplus_{\gamma \in \Gamma} M_{\gamma+\alpha+\beta} = N(\alpha)_\beta\\
\{m_\gamma\}& \longmapsto \{m'_\gamma \},
\end{align*}
where $m'_\gamma=m_{\gamma+\alpha}$ (\ie shift the sequence $\alpha$ ``steps''). It is routine to see that this gives a graded $A$-module homomorphism with inverse homomorphism $f_{-\alpha}$. 
\end{example}

\subsection[Category of graded modules]{The Hom group and the category of graded modules}

For graded right $A$-modules $M$ and $N$, a \emph{graded $A$-module homomorphism of degree $\delta$} \index{graded homomorphism of degree $\delta$} is an $A$-module homomorphism $f:M\rightarrow N$, such that \[f(M_\ga)\subseteq N_{\ga+\delta}\] for any $\ga \in \Ga$. 
Let $\Hom_A(M,N)_{\de}$ denote the subgroup of $\Hom_A(M,N)$ consisting of all graded $A$-module homomorphisms of degree $\de$, \ie
\begin{equation}\label{lmsgoodyear}
\boxed{\Hom_A(M,N)_{\de}=\{ f\in \Hom_A(M,N) \mid f(M_\ga) \subseteq N_{\ga+\de}, \gamma \in \Gamma\}.}
\end{equation}

For graded $A$-modules, $M,N$ and $P$, under the composition of functions, we then have 
\begin{equation}\label{lmsgoodyear3}
\Hom_A(N,P)_{\gamma}\times \Hom_A(M,N)_{\delta} \longrightarrow \Hom(M,P)_{\gamma+\delta}.
\end{equation}

Clearly a graded module homomorphism defined in \S\ref{pokjgtp}  is a graded homomorphism of degree $0$.  

By $\Gr A$ (or $\Gr[\Gamma] A$ to emphasis the grade group of $A$), we denote a category consists of $\Gamma$\!-graded right $A$-modules as objects and  graded homomorphisms as the morphisms. Similarly, $A \rGr$ denotes the category of graded left $A$-modules. Thus  \index{$\Gr A$, category of graded right $A$-modules} 
\index{$A \rGr$, category of graded left $A$-modules}
\index{$\Gr[\Gamma] A$ category of $\Gamma$-graded right $A$-modules} 
\[\boxed{\Hom_{\Gr A}(M,N)=\Hom_A(M,N)_0.}\]
Moreover, for $\alpha \in \Gamma$, as \emph{a set of functions}, one can write 
\begin{equation}\label{googoo1}
\Hom_{\Gr A}\big(M(-\alpha),N\big)=\Hom_{\Gr A}\big(M,N(\alpha)\big)=\Hom_A(M,N)_\alpha.
\end{equation}

A full subcategory of $\Gr A$ consisted of all graded finitely generated $A$-modules is denoted by $\grr A$. 
\index{$\grr A$, category of graded f.g. right $A$-modules} 

For $\alpha \in \Gamma$,  the \emph{$\alpha$-suspension functor} \index{$\alpha$-suspension functor} or \emph{shift functor} \index{shift functor} \index{suspension functor} 
\begin{empheq}[box=\widefbox]{align}\label{RBAday1}
\mathcal T_\alpha:\Gr A&\longrightarrow \Gr A, \\
M &\longmapsto M(\alpha), \notag
\end{empheq}  
is an isomorphism with the property
$\mathcal T_\alpha \mathcal T_\beta=\mathcal T_{\alpha + \beta}$, where $\alpha,\beta\in \Gamma$.

\begin{remark}\label{presipe}
Let $A$ be a $\Gamma$\!-graded ring and $\Omega$ be a subgroup of $\Gamma$ such that $\Gamma_A \subseteq \Omega \subseteq \Gamma$. Then the ring $A$ can be considered naturally as a $\Omega$-graded ring. Similarly, if $A,B$ are $\Gamma$\!-graded rings and $f:A\rightarrow B$ is a $\Gamma$\!-graded homomorphism and $\Gamma_A,\Gamma_B \subseteq \Omega \subseteq \Gamma$, then the homomorphism $f$ can be naturally considered as a $\Omega$-graded homomorphism. In this case, to make a distinction,  we write $\Gr[\Gamma] A$ for the category of $\Gamma$\!-graded $A$-modules and $\Gr[\Omega] A$ for the category of $\Omega$-graded $A$-modules. 
\end{remark}

\begin{theorem}\label{crazhorn}
For graded right $A$-modules $M$ and $N$, such that $M$ is finitely generated, the abelian group $\Hom_{A}(M,N)$ has a natural decomposition 
\begin{equation}\label{hgd543p}
\Hom_{A}(M,N)=\bigoplus_{\gamma \in \Gamma} \Hom_A(M,N)_{\gamma}. 
\end{equation}
Moreover, the endomorphism ring $\Hom_A(M,M)$ is $\Gamma$\!-graded. 
 \end{theorem}
 \begin{proof}
 Let $f\in \Hom_A(M,N)$ and $\la \in \Gamma$. Define a map $f_\la: M \rightarrow N$ as follows: for $m\in M$, 
 \begin{equation}\label{lknaji}
 f_\la(m)=\sum_{\ga \in \Gamma} f(m_{\ga-\la})_\ga,
 \end{equation}
 where $m=\sum_{\ga\in \Gamma}m_\ga$. One can check that $f_\la\in \Hom_A(M,N)$.
 
 Now let $m\in M_\alpha$, $\alpha \in \Gamma$. Then (\ref{lknaji}) reduces to \[f_\la(m)=f(m)_{\alpha+\la}\subseteq M_{\alpha+\la}.\] This shows that  $f_\la\in  \Hom_A(M,N)_\la$. Moreover, $f_\la(m)$ is zero for all but a finite number of $\la\in \Gamma$ and 
 \[\sum_\la f_\la(m)=\sum_\la f(m)_{\alpha+\la}=f(m).\] 
 
Now since $M$ is finitely generated, there is a finite number of homogeneous elements which generate any element $m\in M$. The above argument shows that only a finite number of the $f_\la(m)$ are nonzero and 
 $f=\sum_\la f_\la$. This in turn shows that \[\Hom_{A}(M,N)=\sum_{\gamma \in \Gamma} \Hom_A(M,N)_{\gamma}.\] 
Finally, it is easy to see that $\Hom_A(M,N)_\ga$, $\ga\in \Gamma$ constitutes a direct sum.  

For the second part, replacing $N$ by $M$ in~(\ref{hgd543p}),  we get  \[\Hom_A(M,M)=\bigoplus_{\gamma \in \Gamma} \Hom_A(M,M)_{\gamma}.\] Moreover, by (\ref{lmsgoodyear3}) if 
$f\in \Hom_A(M,M)_{\gamma}$ and $g\in \Hom_A(M,M)_{\la}$ then \[fg \in \Hom_A(M,M)_{\gamma+\la}.\] This shows that when $M$ is finitely generated $\Hom_A(M,M)$ is a $\Gamma$\!-graded ring. 
  \end{proof}

Let $M$ be a graded finitely generated right $A$-module. Then the usual \emph{dual} \index{dual module} of $M$, \ie $M^*=\Hom_A(M,A)$, 
is a left $A$-module. Moreover, using Theorem~\ref{crazhorn}, one can check that  $M^*$ is a graded left $A$-module. Since 
\begin{equation*}
\Hom_A(M,N)(\alpha)=\Hom_A(M(-\alpha),N)=\Hom_A(M,N(\alpha)),
\end{equation*}
we have 
\begin{equation}\label{terrace1}
M(\alpha)^*=M^*(-\alpha).
\end{equation}
This should also make sense:  the dual of ``pushing forward'' $M$ by $\alpha$, is the same as ``pulling back'' the dual $M^*$ by $\alpha$. 

Note that although $\Hom_A(M,N)$ is defined in the category $\Modd A$, the graded structures of $M$ and $N$ are intrinsic in the grading defined on $\Hom_A(M,N)$. Thus if $M$ is isomorphic to $N$ as a ungraded $A$-modules, then $\End_A(M)$ is not necessarily graded isomorphic to $\End_A(N)$. However if $M \cong_{\gr} N(\alpha)$, $\alpha \in \Gamma$, then one can observe that  $\End_A(M)\cong_{\gr}\End_A(N)$ as graded rings.

When $M$ is a free module, $\Hom_A(M,M)$ can be represented as a matrix ring over $A$. Next we define graded free modules. In~\S\ref{matgrhe} we will see that if $M$ is a graded free module, the graded ring $\Hom_A(M,M)$ can be represented as a matrix ring over $A$ with a very concrete grading. 

\begin{example}\label{egofgonrings0}\scm[The Veronese submodule]
\vspace{0.2cm}

For a $\Gamma$\!-graded ring $A$, recall the construction of $n$th-Veronese subring \[A^{(n)}=\bigoplus_{\gamma \in \Gamma} A_{n\gamma}\] (Example~\ref{egofgrdivisionrings0}). In a similar fashion, for a graded $A$-module $M$ and $n\in \Z$, define the $n$th-\emph{Veronese module} \index{Veronese module} of $M$ as \[M^{(n)}=\bigoplus_{\gamma \in \Gamma} M_{n\gamma}.\]  This is a $\Gamma$\!-graded $A^{(n)}$-module. Clearly there is a natural ``forgetful'' functor 
\[U:\Gr A \longrightarrow \Gr A^{(n)},\] which commutes with suspension functors as follows $\mathcal T_\alpha U = U \mathcal T_{n\alpha}$, \ie \[M^{(n)}(\alpha)=M(n\alpha)^{(n)},\] for $\alpha \in \Gamma$ and $n\in \Z$ (see~\S\ref{forgetful} for more on forgetful functors).

\end{example}

\subsection{Graded free modules} \label{hefeijuly}
A $\Gamma$\!-graded (right) $A$-module $F$ is called a \emph{graded free
$A$-module}\index{graded module} if $F$ is a free right $A$-module with a homogeneous base. Clearly a graded free module is a free module but the converse is not correct, namely, a free module which is graded is not necessarily a graded free module.  As an example, for $A=\mathbb R[x]$ considered as a $\Z$-graded ring,
$A\oplus A(1)$ is not a graded free $A\oplus A$-module, whereas $A\oplus A$ is a free $A\oplus A$-module  (see also Example~\ref{qub8}). The definition of free given here is consistent with the categorical definition of free objects over a set of homogeneous elements in the category of graded modules ~(\cite[I, \S7]{hungerford}). \index{graded free module}

Consider a $\Ga$\!-graded $A$-module $\bigoplus_{i \in I} A
(\de_i)$, where $I$ is an indexing set and $\de_i \in
\Ga$. Note that for each $i\in I$, the element $e_i$ of the standard basis (\ie 1 in the $i$th component and zero elsewhere)
is homogeneous of degree $-\de_i$. The set $\{e_i\}_{i \in I}$ forms a base for $\bigoplus_{i \in I} A
(\de_i)$, which by definition makes this a graded free $A$-module. On the other hand,  a graded free $A$-module $F$ with a homogeneous base $\{b_i\}_{i \in I}$, where $\deg (b_i) = -\delta_i$ is graded isomorphic to $\bigoplus_{i \in I} A
(\de_i)$. Indeed one can easily observe that 
 the map induced by 
\begin{align}\label{excampp}
\varphi : \bigoplus\nolimits_{i \in I} A(\de_i) & \lra F \\
e_i & \lmps b_i \notag
\end{align}
is a graded $A$-module isomorphism.  

If the indexing set $I$ is finite, say $I=\{1,\dots,n\}$, then \[\bigoplus\nolimits_{i \in I} A(\de_i) = A(\de_1)\oplus \dots \oplus A(\de_n),\] is also denoted by $A^n(\de_1,\dots,\de_n)$ or  $A^n(\overline \delta)$, where $\overline \delta=(\de_1,\dots,\de_n)$.

In~\S\ref{kjujidw}, we consider the situation when the graded free right $A$-modules $A^n(\overline \delta)$ and  
$A^m (\overline \alpha)$, where $\overline \delta=(\de_1,\dots,\de_n)$ and $\overline \alpha =(\alpha_1,\dots,\alpha_m)$, 
are isomorphic. In~\S\ref{gtr5654}, we will also consider the concept of graded rings with the graded invariant basis numbers. 

\subsection{Graded bimodules}\label{bimref}
The notion of the graded \emph{left} A-modules is developed similarly. \index{graded left module} The category of graded left $A$-modules with graded homomorphisms is denoted by $A \rGr$. In a similar manner for $\Gamma$\!-graded rings $A$ and $B$, we can consider the 
\emph{graded} $A\!-\!B$-\emph{bimodule} \index{graded bimodule} $M$. Namely, $M$ is a $A\!-\!B$-bimodule and additionally $M=\bigoplus_{\gamma \in \Gamma} M_\gamma$ is a graded left $A$-module and a graded right $B$-module, \ie \[A_\alpha M_{\gamma} B_{\beta}  \subseteq M_{\alpha +\gamma+\beta},\] where $\alpha,\gamma,\beta \in \Gamma$. 
The category of graded $A$-bimodules is denoted by $\Gr A \rGr$. \index{$\Gr A \rGr$}

\begin{remark}\scm[Shift of non-abelian group graded modules]
\vspace{0.2cm}

If the grade group $\Gamma$ is not abelian, then in order that the shift of components matches, for a graded left $A$-module $M$ one needs to define 
\[M(\de)_\ga = M_{\ga\de},\] whereas for the graded right $M$-module $A$, shift is defined by \[M(\de)_\ga = M_{\de\ga}.\] With these definitions, for 
$\mathcal T_\alpha, \mathcal T_\beta: \Gr A \rightarrow \Gr A$, we have $\mathcal T_\alpha \mathcal T_\beta =\mathcal T_{\beta \alpha}$, whereas for 
$\mathcal T_\alpha, \mathcal T_\beta: A \rGr  \rightarrow A \rGr$, we have $\mathcal T_\alpha \mathcal T_\beta =\mathcal T_{\alpha \beta}$. For this reason, in the non-abelian grade group setting, several books choose to work with the graded left modules as opposed to the graded right modules we have adopted in this book. 
\end{remark} 

\subsection{Tensor product of graded modules}\label{grtensie} \index{tensor product of graded modules}

Let $A$ be a $\Gamma$\!-graded ring and $M_i$, $i\in I$, be a direct system of $\Gamma$\!-graded $A$-modules, \ie $I$ is a directed partially ordered set and for $i\leq j$, there is a graded $A$-homomorphism $\phi_{ij}:M_i \rightarrow M_j$ which is compatible with the ordering. Then $M:=\varinjlim M_i$ is a $\Gamma$\!-graded $A$-module with homogeneous components 
$M_\alpha=\varinjlim {M_i}_\alpha$ (see Example~\ref{penrith123} for the similar construction for rings). 

In particular,  let $\{M_i \mid i \in I \}$ be $\Gamma$\!-graded right $A$-modules. Then $\bigoplus_{i\in I} M_i$ has a natural graded $A$-module given by 
$(\bigoplus_{i\in I} M_i)_\alpha=\bigoplus_{i\in I} {M_i}_\alpha$, $\alpha \in \Gamma$.

Let $M$ be a graded right $A$-module and $N$ be a graded left $A$-modules. We will observe that the tensor product $M\otimes_A N$ has a natural $\Gamma$\!-graded $\mathbb Z$-module structure. Since each of $M_\gamma$, $\gamma \in \Gamma$, is a right $A_0$-module and similarly $N_\gamma$, $\gamma \in \Gamma$, is a left $A_0$-module, then 
$M\otimes _{A_0} N$ can be decomposed as a direct sum 
\[\boxed{M\otimes_{A_0} N =\bigoplus_{\gamma \in \Gamma} (M\otimes N)_\gamma,}\] where 
\begin{equation*}
\boxed{(M \otimes N)_{\ga} = \Big \{ \sum_i m_i \otimes n_i   \mid   m_i \in
M^h, n_i \in N^h, \deg(m_i)+\deg(n_i) = \ga \, \Big\}.}
\end{equation*}
Now note that $M\otimes_A N \cong (M\otimes_{A_0} N)/ J$, where $J$ is a subgroup of $M\otimes_{A_0}N$  generated by the homogeneous elements \[\{ma\otimes n -m\otimes an \mid m\in M^h, n\in N^h, a\in A^h\}.\] This shows that $M\otimes_A N$ is also a graded module. It is easy to check that, for example, if $N$ is a graded $A$-bimodule, then $M\otimes_A N$ is a graded right $A$-module. It follows from the definition that 
\begin{equation}\label{izmir27june}
M\otimes N(\alpha) =M(\alpha)\otimes N =(M \otimes N)(\alpha).
\end{equation}

Observe that for a graded right $A$-module $M$, the map 
\begin{align}\label{medeltaotel}
M\otimes_A A(\alpha) & \longrightarrow M(\alpha),\\
m \otimes a & \longmapsto ma, \notag
\end{align}
is a graded isomorphism. In particular, for any $\alpha,\beta \in \Gamma$, there is a graded $A$-bimodule isomorphism 
\begin{equation}\label{inenbuild}
\boxed{A(\alpha)\otimes_A A(\beta)\cong_{\gr} A(\alpha+\beta).}
\end{equation}

\begin{example}\label{gratsusan}\scm[Graded formal matrix rings] \index{graded formal matrix ring}
\vspace{0.2cm}

The construction of formal matrix rings (Example~\ref{atsusan}) can be carried over to the graded setting as follows. 
Let $R$ and $S$ be $\Gamma$\!-graded rings, $M$ be a graded $R\!-\!S$\!-bimodule and $N$ be a graded $S\!-\!R$-bimodule. Suppose that there are graded bimodule homomorphisms $\phi:M\otimes_S N \rightarrow R$ and $\psi:N\otimes_R M \rightarrow S$ such that $(mn)m'=n(nm')$, where we denote $\phi(m,n)=mn$ and $\psi(n,m)=nm$. Consider the ring
\begin{equation*}
T=\begin{pmatrix}
R & M \\N & S
\end{pmatrix}, 
\end{equation*}
and define, for any $\gamma \in \Gamma$,
\begin{equation*}
T_\gamma=\begin{pmatrix}
R_\gamma & M_\gamma \\N_\gamma & S_\gamma
\end{pmatrix}.
\end{equation*}
 One checks that $T$ is a $\Gamma$\!-graded ring, called a \emph{graded formal matrix ring}. One specific type of such rings is a Morita ring which appears in graded Morita theory (\S\ref{meinghto}). 

\end{example}

\subsection{Forgetting the grading}\label{forgetful}

Most forgetful functors in algebra tend to have left adjoints, which have a `free' construction. One such example is the forgetful functor from the category of abelian groups to abelian monoids that we will study in~\S\ref{ggg} in relation with the Grothendieck groups. 
However, some of the forgetful functors in the graded setting naturally have right adjoints as we will see below. 

Consider the \emph{forgetful} functor \index{forgetful functor}  
\begin{equation}\label{hgytop}
U:\Gr A \longrightarrow \Modd A,
\end{equation} which simply assigns to any graded module $M$ in $\Gr A$, its underlying module $M$ in $\Modd A$, ignoring the grading. Similarly, the graded homomorphisms are sent to the same homomorphisms, disregarding their graded compatibilities.  

There is a  functor $F:\Modd A \rightarrow \Gr A$ which is a right adjoint to $U$.  
 The construction is as follows: let $M$ be an $A$-module. Consider the abelian group $F(M):=\bigoplus_{\gamma\in \Gamma}M_\gamma$, where $M_\gamma$ is a copy of $M$. Moreover, for $a\in A_\alpha$ and $m\in M_\gamma$ define $m.a =ma \in M_{\alpha+\gamma}$. This defines a graded $A$-module structure on $F(M)$ and makes $F$ an exact functor from $\Modd A$ to $\Gr A$. One can prove that for any $M\in \Gr A$ and $N\in \Modd A$, we have a bijective map 
 \begin{align*}
 \Hom_{\Modd A}\big(U(M),N\big)&\stackrel{\phi}{\longrightarrow} \Hom_{\Gr A}\big(M,F(N)\big),\\
 f&\longmapsto \phi_f,
 \end{align*}
 where $\phi_f(m_\alpha)=f(m_\alpha) \in N_\alpha$. \index{adjoint functor}
 
 \begin{remark}\label{pearlfisher}
 It is not difficult to observe that for any $M \in \Gr A$, \[FU(M)\conggr \bigoplus_{\gamma \in \Gamma}M(\gamma).\] By Example~\ref{supinmalnly}, we have $FU(M)\cong_{\gr} FU(M)(\alpha)$ for any $\alpha \in \Gamma$. We also note that if $\Gamma$ is finite, then $F$ is also a left adjoint functor of $U$. Further, if $U$ has a left adjoint functor, then one can prove that $\Gamma$ is finite (see~\cite[\S2.5]{grrings} for details). 
 \end{remark}

\subsection{Partitioning of the graded modules}\label{bill100}

Let $f:\Gamma \rightarrow \Delta$ be a group homomorphism. Recall from~\S\ref{mconfi1} that there is a functor from the category of $\Gamma$\!-graded rings to the category of $\Delta$-graded rings which gives the natural forgetful functor when 
$\Delta=0$. This functor has a right adjoint functor (see~\cite[Proposition~1.2.2]{grrings} for the case of $\Delta=0$). The homomorphism $f$ induces a forgetful functor on the level of module categories. We describe this here. 

Let $A$ be a $\Gamma$\!-graded ring and consider the corresponding $\Delta$-graded structure induced by the homomorphism $f:\Gamma \rightarrow \Delta$ (\S\ref{mconfi1}). Then one can construct a functor $U_f:\Gr[\Gamma] A \rightarrow \Gr[\Delta] A$ which has a right adjoint. In particular, for a subgroup $\Omega$ of  $\Gamma$, we have the following  canonical `forgetful' functor  (a \emph{block} functor or a \emph{coarsening} functor)  \index{block functor} \index{coarsening functor}
\[U:\Gr[\Gamma] A \rightarrow \Gr[\Gamma/\Omega] A,\] such that when $\Omega=\Gamma$, it gives the functor~(\ref{hgytop}). The construction is as follows. Let $M=\bigoplus_{\alpha \in \Gamma} M_\alpha$ be a $\Gamma$\!-graded $A$-module. Write 
 \begin{equation}\label{hngdurys2}
 M=\bigoplus_{\Omega+\alpha \in \Gamma/\Omega} M_{\Omega+\alpha},
 \end{equation} where 
\begin{equation}\label{hngdurys3}
M_{\Omega+\alpha} :=\bigoplus_{\omega \in \Omega} M_{\omega+\alpha}.
\end{equation} One can easily check that $M$ is a $\Gamma/\Omega$-graded $A$-module. Moreover, \[U(M(\alpha))=M(\Omega+\alpha).\]  We will use this functor to relate the Grothendieck groups of these categories in 
Examples~\ref{poiurp} and~\ref{tonyat11}. 

In a similar manner we have the following functor
\begin{align*}
(-)_{\Omega}: \Gr[\Gamma] A \longrightarrow \Gr[\Omega] A_{\Omega},
\end{align*}
where $\Gr[\Omega] A$ is the category of $\Omega$-graded (right) 
$A_{\Omega}$-modules, where $\Omega=\ker(f)$.

The above construction motivates the following which will establish a relation between the categories $\Gr[\Gamma] A$ and $\Gr[\Omega] A_{\Omega}$. 

Consider the quotient group $\Gamma/\Omega$ and fix a complete set of coset representative 
$\{\alpha_i\}_{i\in I}$. Let $\beta \in \Gamma$ and consider the permutation map $\rho_\beta$ \index{complete set of coset representative}
\begin{align*}
\rho_\beta:\Gamma/\Omega&\longrightarrow \Gamma/\Omega,\\
\Omega+\alpha_i&\longmapsto \Omega+\alpha_i+\beta=\Omega+\alpha_j.
\end{align*}
This defines a bijective map (called $\rho_\beta$ again) $\rho_\beta: \{\alpha_i\}_{i\in I} \rightarrow \{\alpha_i\}_{i\in I}$. Moreover, for any $\alpha_i$, since
\[\Omega+\alpha_i+\beta=\Omega+\alpha_j=\Omega+\rho_\beta(\alpha_i),\] there is a unique $w_i \in \Omega$ such that 
\begin{equation}\label{whathfyr}
\alpha_i+\beta=\omega_i+\rho_\beta(\alpha_i).
\end{equation}

 Recall that if $\mathcal C$ is an additive category, $\bigoplus_I\mathcal C$, where $I$ is a nonempty index set, is defined in the obvious manner, with objects $\bigoplus_{i\in I} M_i$, where $M_i$ are objects of $\mathcal C$ and morphisms accordingly. 
 
Define the functor  
\begin{empheq}[box=\widefbox]{align}\label{ch20ts}
\mathcal P: \Gr[\Gamma] A &\longrightarrow \bigoplus_{\Gamma/\Omega} \Gr[\Omega] A_\Omega,\\
M &\longmapsto \bigoplus_{\Omega+\alpha_i \in \Gamma/\Omega} M_{\Omega+\alpha_i}\notag
\end{empheq}
where 
\[\boxed{M_{\Omega+\alpha_i}=\bigoplus_{\omega \in \Omega} M_{\omega+\alpha_i}.}\]
Since $M_{\Omega+\alpha}$, $\alpha \in \Gamma$, as defined in (\ref{hngdurys3}), can be naturally considered as $\Omega$-graded $A_\Omega$-module, where 
\begin{equation}\label{petriob}
(M_{\Omega+\alpha})_\omega=M_{\omega+\alpha},
\end{equation}
it follows that the functor $\mathcal P$ defined in~(\ref{ch20ts}) is well-defined. 
Note that the homogeneous components defined in~(\ref{petriob}), depend on the coset representation, thus choosing another complete set of coset representative gives a different functor between these categories. 

For any $\beta \in \Gamma$, define a shift functor
\begin{align}\label{hnhdheye}
\overline \rho_\beta:\bigoplus_{\Gamma/\Omega} \Gr[\Omega] A_\Omega &\longrightarrow \bigoplus_{\Gamma/\Omega} \Gr[\Omega] A_\Omega,\\
 \bigoplus_{\Omega+\alpha_i \in \Gamma/\Omega} M_{\Omega+\alpha_i} &\longmapsto \bigoplus_{\Omega+\alpha_i\in \Gamma/\Omega} M(\omega_i)_{\Omega+\rho_\beta(\alpha_i)}, \notag
 \end{align}
 where $\rho_\beta(\alpha_i)$ and $\omega_i$ are defined in (\ref{whathfyr}). The action of $\overline \rho_\beta$ on morphisms are defined accordingly. Note that in the left hand side of~(\ref{hnhdheye}) the graded $A_{\Omega}$-module appears in $\Omega+\alpha_i$ component is denoted by 
$M_{\Omega+\alpha_i}$. When $M$ is a $\Gamma$\!-graded module, then $M_{\Omega+\alpha_i}$ has a $\Omega$-structure as described 
in~(\ref{hngdurys3}).

We are in a position to prove the next theorem. 

\begin{theorem}\label{mhgft42}
Let $A$ be a $\Gamma$\!-graded ring and $\Omega$ a subgroup of $\Gamma$. Then for any $\beta \in \Gamma$, the following diagram is commutative,  
\begin{equation}\label{hhusole}
\xymatrix{
\Gr[\Gamma] A \ar[rr]^-{\mathcal P} \ar[d]_{\mathcal T_\beta}&& \bigoplus_{\Gamma/\Omega} \Gr[\Omega] A_\Omega \ar[d]^{\overline \rho_\beta}\\
\Gr[\Gamma] A \ar[rr]^-{\mathcal P}  && \bigoplus_{\Gamma/\Omega} \Gr[\Omega] A_\Omega,
}
\end{equation}
where the functors $\mathcal P$ and  $\overline \rho_\beta$ are defined in~\eqref{ch20ts} and~\eqref{hnhdheye}, respectively. 
Moreover, if $\Gamma_A \subseteq \Omega$, then the functor $\mathcal P$  induces an equivalence of categories. 
\end{theorem}
\begin{proof}
We first show that Diagram~\ref{hhusole} is commutative. Let $\beta \in \Gamma$ and  $M$ a $\Gamma$\!-graded $A$-module. 
As in~(\ref{whathfyr}), let $\{\alpha_i\}$ be a fixed complete set of coset representative and 
\begin{equation*}
\alpha_i+\beta=\omega_i+\rho_\beta(\alpha_i).
\end{equation*}
Then 
 \begin{equation}\label{hmhfyrhds}
 \mathcal P(\mathcal T_\beta(M))=\mathcal P(M(\beta))=\bigoplus_{\Omega+\alpha_i \in \Gamma/\Omega} M(\beta)_{\Omega+\alpha_i}.
 \end{equation}
 But
 \begin{multline*}
 M(\beta)_{\Omega+\alpha_i}=\bigoplus_{\omega \in \Omega} M(\beta)_{\omega+\alpha_i}=\bigoplus_{\omega \in \Omega} M_{\omega+\alpha_i+\beta}
 =\bigoplus_{\omega \in \Omega} M_{\omega+ \omega_i+\rho_\beta(\alpha_i)}=\\
 \bigoplus_{\omega \in \Omega} M(\omega_i)_{\omega+\rho_\beta(\alpha_i)}=M(\omega_i)_{\Omega+\rho_\beta(\alpha_i)}. 
 \end{multline*}
 Replacing this into equation~\ref{hmhfyrhds} we have
 \begin{equation}\label{meinfg43}
  \mathcal P(\mathcal T_\beta(M))=\bigoplus_{\Omega+\alpha_i \in \Gamma/\Omega}M(\omega_i)_{\Omega+\rho_\beta(\alpha_i)}.
 \end{equation}
  On the other hand by~(\ref{hnhdheye}),
 \begin{equation}\label{ucas4372}
 \overline \rho_\beta \mathcal P (M)=\overline \rho_\beta (\bigoplus_{\Omega+\alpha_i \in \Gamma/\Omega} M_{\Omega+\alpha_i})= \bigoplus_{\Omega+\alpha_i\in \Gamma/\Omega} M(\omega_i)_{\Omega+\rho_\beta(\alpha_i)}.
\end{equation} 
Comparing~(\ref{meinfg43}) and (\ref{ucas4372}) shows that Diagram~\ref{hhusole} is commutative.  
 
 For the last part of theorem, suppose $\Gamma_A \subseteq \Omega$. We construct a functor 
 \[\mathcal P':  \bigoplus_{\Gamma/\Omega} \Gr[\Omega] A_\Omega \longrightarrow \Gr[\Gamma] A,\] 
 which depends on the coset representative
 $\{\alpha_i\}_{i\in I}$  of $\Gamma/\Omega$. First note that any $\alpha \in \Gamma$ can be written uniquely as $\alpha=\alpha_i+\omega$, for some $i\in I$ and $\omega\in \Omega$. Now let 
 \[\bigoplus_{\Omega+\alpha_i \in \Gamma/\Omega} N_{\Omega+\alpha_i} \in  \bigoplus_{\Gamma/\Omega} \Gr[\Omega] A_\Omega,\]
 where $N_{\Omega+\alpha_i}$ is $\Omega$-graded $A_\Omega$-module.  Define a $\Gamma$\!-graded $A$-module $N$ as follows:
 $N=\bigoplus_{\alpha\in \Gamma} N_\alpha$, where $N_\alpha:=(N_{\Omega+\alpha_i})_{\omega}$ and $\alpha=\alpha_i+\omega$. We check that $N$ is a $\Gamma$\!-graded $A$-module, i.e, $N_\alpha A_\gamma \subseteq N_{\alpha+\gamma}$, for $\alpha,\gamma \in \Gamma$. If $\gamma \not \in \Omega$, since $\Gamma_A \subseteq \Omega$,  $A_\gamma=0$  and thus $0=N_\alpha A_\gamma \subseteq N_{\alpha+\gamma}$. Let $\gamma \in \Omega$. Then 
 \[N_\alpha A_\gamma = (N_{\Omega+\alpha_i})_{\omega}A_\gamma\subseteq (N_{\Omega+\alpha_i})_{\omega+\gamma}=N_{\alpha+\gamma},\] as $\alpha+\gamma=\alpha_i+\omega+\gamma$. 
 
 We define $\mathcal P' (\bigoplus_{\Omega+\alpha_i \in \Gamma/\Omega} N_{\Omega+\alpha_i})=N$ for the objects and similarly for the morphisms. 
 It is now not difficult to check that $\mathcal P'$ is an inverse of the functor $\mathcal P$. This finishes the proof. 
 \end{proof}

The above theorem will be used to compare the graded $K$-theories with respect to $\Gamma$ and $\Omega$ (see Example~\ref{tonyat11}).

\begin{corollary}\label{hgfeyes} \index{trivial grading}
Let $A$ be a $\Gamma$\!-graded ring concentrated in degree zero. Then \[\Gr A \approx \bigoplus_{\Gamma} \Modd A.\]  The action 
of $\Gamma$ on $\bigoplus_{\Gamma} \Modd A$ described in~\eqref{hnhdheye} reduces to the following: for $\beta \in \Gamma$, 
\begin{equation}\label{mugeinizmir}
\overline \rho_{\beta}\Big (\bigoplus_{\alpha \in \Gamma} M_\alpha \Big)= \bigoplus_{\alpha \in \Gamma} M(\beta)_{\alpha}= \bigoplus_{\alpha \in \Gamma} M_{\alpha+\beta}.
\end{equation}
\end{corollary}
\begin{proof}
This follows by replacing $\Omega$ by a trivial group in Theorem~\ref{mhgft42}.
\end{proof}

The following corollary, which is a more general case of Corollary~\ref{hgfeyes} with a similar proof, will be used in the proof of Lemma~\ref{dido13}. 

\begin{corollary}\label{zuhoingding}
Let $A$ be a $\Gamma\times \Omega$ graded ring which is concentrated in $\Omega$. Then 
\[\Gr[\Gamma \times \Omega] A \cong \bigoplus_{\Gamma} \Gr[\Omega] A.\]
 The action of $\Ga \times \Omega$ on $\bigoplus_{\Gamma} \Gr[\Omega] A$ described in~\eqref{hnhdheye} reduces to the following: for $(\beta,\omega) \in \Gamma\times \Omega$, 
\[\overline \rho_{(\beta, \omega)}\Big (\bigoplus_{\alpha \in \Gamma} M_\alpha \Big)= \bigoplus_{\alpha \in \Gamma} M_{\alpha+\beta} (\omega).\] 
 \end{corollary}

\subsection{Graded projective modules}\label{monashre}

Graded projective modules play a crucial role in this book. They will appear in the graded Morita theory~\S\ref{moritanji} and will be used to define the graded Grothendieck groups~\S\ref{ggg}. Moreover, the graded higher $K$-theory is constructed from the exact category consisted of graded finitely generated projective modules~\S\ref{waraya}. In this section we define the graded projective modules and give several equivalent criteria for a module to be graded projective. As before, unless stated otherwise, we work in the category of (graded) right modules. 

A graded $A$-module $P$  is called a \emph{graded projective module} \index{graded projective module} if it is a projective object in the abelian category $\Gr A$.
More concretely, $P$ is graded projective if  
for any diagram of graded modules and graded $A$-module homomorphisms
\begin{equation}\label{bitguesthouse14}
\xymatrix{ & P \ar[d]^-{j}  \ar@{.>}[dl]_{h}& \\
     M \ar[r]_{g} & N \ar[r] & 0, }
\end{equation}
there is a graded $A$-module homomorphism 
$h :P \ra M$ with $g h = j$.

In Proposition~\ref{grprojectivethm} we give some equivalent  characterisations of graded projective modules, including the one that shows an $A$-module is graded projective  if and only if it is graded and projective as an $A$-module. By $\Pgrp A$ (or $\Pgr[\Gamma] A$ to emphasis the grade group of $A$) we denote a full subcategory of $\Gr A$, consisting of graded finitely generated projective right $A$-modules. This is the primary category we are interested in. The graded Grothendieck group (\S\ref{ggg}) and higher $K$-groups (\S\ref{waraya}) are constructed from this exact category (see Definition~\ref{hgygtw2}).  \index{$\Pgrp A$, category of graded proj. $A$-modules}

We need the following lemma, which says if a graded map factors into two maps, with one being graded, then one can replace the other one with a graded map as well.

\begin{lemma}\label{grrlinearmaps}
Let $P,M,N$ be graded $A$-modules, with $A$-module homomorphisms $f,g,h$ 
\begin{displaymath}
\xymatrix{ & \; M \ar[dr]^{g}  & \\
    P \ar[ur]^h \ar[rr]_{f} &&N }
\end{displaymath}
such that $f = g  h$, where $f$ is a graded $A$-module homomorphism. 
If $g$ (resp.\ $h$) is a graded $A$-homomorphism then
there exists a graded homomorphism $h' : P \ra M$ (resp.\ $g' :M \ra N$) such that $f = g h'$ (resp.\ $f = g' h$).
\end{lemma}

\begin{proof}
Suppose $g:M\rightarrow N$ is a graded $A$-module homomorphism. Define $h':P\rightarrow M$ as follows: for $p\in P_\alpha$, 
$\alpha\in \Gamma$, let  $h'(p)=h(p)_\alpha$ and extend this linearly to all elements of $P$, \ie for $p\in P$ with $p=\sum_{\alpha\in \Gamma} p_\alpha$,
\[h(p)=\sum_{\alpha \in \Gamma} h(p_\alpha)_\alpha.\]
 One can easily see that $h':P\rightarrow M$ is a graded $A$-module homomorphism. Moreover, for $p\in P_\alpha$, $\alpha \in \Gamma$, we have
\[f(p)=gh(p)=g\Big(\sum_{\gamma \in \Gamma} h(p)_\gamma\Big)=\sum_{\gamma \in \Gamma}  g\big (h(p)_\gamma\big).\]
Since $f$ and $g$ are graded homomorphisms, comparing the degrees of the homogeneous elements of each side of the equation, we get 
\[f(p)=g(h(p)_\alpha)=gh'(p).\] Using the linearity of $f,g,h'$ it follows that $f=gh'$. This proves the lemma for the case $g$. The other case is similar. 
\end{proof}

We are in a position to give  equivalent  characterisations of graded projective modules. 

\begin{proposition}\label{grprojectivethm} 
Let $A$ be a $\Ga$\!-graded ring and $P$
be a graded $A$-module. Then the following are equivalent:
\begin{enumerate}[\upshape(1)]

\item $P$ is graded and projective;

\item  $P$ is graded projective;

\item $\Hom_{\Gr A}(P, -)$ is an exact functor in $\Gr A$;

\item Every short exact sequence of graded $A$-module homomorphisms
\begin{displaymath}
0 \lra L \stackrel{f}{\lra} M \stackrel{g}{\lra} P \lra 0
\end{displaymath}
splits via a (graded) map;

\item $P$ is graded isomorphic to a direct summand of a graded free
$A$-module.
\end{enumerate}
\end{proposition}

\begin{proof}
(1) $\Ra$ (2) Consider the diagram 
\begin{displaymath}
\xymatrix{ & P \ar[d]^-{j}  & \\
     M \ar[r]_{g} & N \ar[r] & 0, }
\end{displaymath}
where $M$ and $N$ are graded modules, $g$ and $j$ are graded homomorphisms and $g$ is surjective. 
Since $P$ is projective,  there is an $A$-module homomorphism $h:P
\ra M$ with $g h = j$. By Lemma~\ref{grrlinearmaps},
there is a graded $A$-module homomorphism $h' :P \ra M$ with $g
h' = j$. This gives that $P$ is a graded projective module. 

\vspace{3pt}

(2) $\Ra$ (3) In exactly the same way as the ungraded setting, 
 we can show (with no assumption on $P$) that  $\Hom_{\Gr A}(P, -)$ is left exact (see \cite[\S IV, Theorem~4.2]{hungerford}). The right exactness follows immediately from the definition of graded projective modules that any diagram of the form~(\ref{bitguesthouse14}) can be completed. 
\vspace{3pt}

(3) $\Ra$ (4) Let 
\begin{equation}\label{syun}
0 \lra L \stackrel{f}{\lra} M \stackrel{g}{\lra} P \lra 0
\end{equation}
be a short exact sequence. Since $\Hom_{\Gr A}(P,-)$ is exact, 
\begin{align*}
\Hom_{\Gr A}(P,M)&\longrightarrow \Hom_{\Gr A}(P,P)\\
h&\longmapsto gh
\end{align*}
is an epimorphism. In particular, there is graded homomorphism $h$ such that $gh=1$, \ie the short exact sequence (\ref{syun}) is spilt.

\vspace{3pt}

(4) $\Ra$ (5) First note that $P$ is a homomorphic image of a graded free $A$-module as follows: 
Let $\{p_i\}_{i \in I}$ be a homogeneous generating
set for $P$, where $\deg(p_i)= \de_i$. Let $\bigoplus_{i \in I}
A(-\de_i)$ be the graded free $A$-module with standard
homogeneous basis $\{e_i\}_{i \in I}$ where $\deg (e_i) = \de_i$.
Then there is an exact sequence
\begin{equation}\label{jeaonghe}
0 \lra \ker (g) \stackrel{\subseteq}{\lra} \bigoplus_{i \in I}
A(-\de_i) \stackrel{g}{\lra} P \lra 0,
\end{equation}
as the map 
\begin{align*}
g : \bigoplus_{i \in I} A(-\de_i) &\longrightarrow P,\\
e_i &\longmapsto p_i,
\end{align*} is a surjective graded $A$-module homomorphism. By the assumption,
there is a $A$-module homomorphism $h :P \ra \bigoplus_{i \in
I} A(-\de_i)$ such that $g  h = \id_P$. By Lemma~\ref{grrlinearmaps} one can assume $h$ is a graded homomorphism.

Since the exact sequence~\ref{jeaonghe} is, in particular, a split exact sequence
of $A$-modules, we know from the ungraded setting
\cite[Proposition~2.5]{magurn} that there is an $A$-module isomorphism
\begin{align*}
\theta: P \oplus \ker (g) & \lra \bigoplus\nolimits_{i \in I}
A(-\de_i)\\
(p , q) & \lmps h(p) + q.
\end{align*}
Clearly this map is also a graded $A$-module homomorphism, so \[P
\oplus \ker (g) \conggr \bigoplus_{i \in I} A(-\de_i).\]

\vspace{3pt}

(5) $\Ra$ (1) Graded free modules are free, so $P$ is isomorphic to
a direct summand of a free $A$-module. From the ungraded setting,
we know that $P$ is projective.
\end{proof}

The proof of Proposition~\ref{grprojectivethm} (see in particular (4) $\Ra$ (5) and (5) $\Ra$ (1)) shows that a graded $A$-module $P$ is a graded finitely generated projective $A$-module if and only if 
\begin{equation}\label{medickiue}
\boxed{
P\oplus Q\conggr A^n(\overline \alpha)}
\end{equation} 
for some $\overline \alpha=(\alpha_1,\dots,\alpha_n)$, $\alpha_i\in \Gamma$. This fact will be used frequently throughout this book. 

\begin{remark}\label{enjoythemomentt}
Recall the functor $\mathcal P$ from~(\ref{ch20ts}). It is easy to see that if $M$ is a $\Gamma$\!-graded projective $A$-module, then $M_{\Omega+\alpha}$ is a $\Omega$-graded projective $A_\Omega$-graded module. Thus the functor $\mathcal P$ restricts to 
\begin{empheq}[box=\widefbox]{align}\label{ch20tsproj}
\mathcal P: \Pgr[\Gamma] A &\longrightarrow \bigoplus_{\Gamma/\Omega} \Pgr[\Omega] A_\Omega,\\
M &\longmapsto \bigoplus_{\Omega+\alpha_i \in \Gamma/\Omega} M_{\Omega+\alpha_i}\notag
\end{empheq}
This will be used later in Examples~\ref{pogbiaue},~\ref{tonyat11} and Lemma~\ref{dido13}. 
\end{remark}

\begin{theorem}[{\sc The dual basis lemma}] Let $A$ be a $\Ga$\!-graded ring and $P$
be a graded $A$-module. Then $P$ is graded projective if and only if 
there exists $p_i \in  P^h$ with $\deg(p_i)=\delta_i$ and $f_i \in  \Hom_{\Gr A} (P,A(-\delta_i))$, for some indexing set $I$, such
that
\begin{enumerate}[\upshape(1)]
\item for every $p \in P$, $f_i (p) =0$ for all but a finite subset
of $i \in I$, 

\item for every $p \in P$, $\sum_{i \in I} f_i (p) p_i = p.$
\end{enumerate}
\end{theorem}
\begin{proof} Since $P$ is  graded projective, by Proposition~\ref{grprojectivethm}(5), there is a graded module $Q$ such that $P\oplus Q \conggr \bigoplus_i A(-\delta_i)$. This gives two graded maps \[\phi: P \rightarrow \bigoplus_i A(-\delta_i), \text{    and    } \, \pi: \bigoplus_i A(-\delta_i)\rightarrow P,\] such that $\pi\phi=1_P$. Let 
\begin{align*}
\pi_i: \bigoplus_i A(-\delta_i) &\longrightarrow A(-\delta_i),\\
\{a_i\}_{i\in I}&\longmapsto a_i,
\end{align*}
be the projection on the $i$th component. So if \[a=\{a_i\}_{i\in I} \in \bigoplus_i A(-\delta_i),\] then 
\[\sum_i \pi_i(a)e_i=a,\] where $\{e_i\}_{i\in I}$ is the standard homogeneous basis of $\bigoplus_i A(-\delta_i)$. Now let $p_i=\pi(e_i)$ and $f_i=\pi_i\phi$. Note that $\deg(p_i)=\delta_i$ and \[f_i \in \Hom_{\Gr A}(P, A(-\delta_i)).\] Clearly $f_i(p)=\pi_i\phi(p)$ is zero for all but a finite number of $i \in I$. This gives (1). Moreover, 
\[\sum_i p_if_i(p)=\sum_i p_i \pi_i\phi(p)=\sum\pi(e_i)\pi_i\phi(p)=\pi\big(\sum_i e_i\pi_i\phi(p)\big)=\pi\phi(p)=p.\] This gives (2).

Conversely, suppose that there exists a dual basis $\{p_i,f_i \mid  i \in I\}$. Consider the maps 
\begin{align*}
\phi:P &\longrightarrow \bigoplus_i A(-\delta_i),\\
p &\longmapsto \{f_i(p)\}_{i \in I}
\end{align*}
and 
\begin{align*}
\pi: \bigoplus_i A(-\delta_i) &\longrightarrow P,\\
\{a_i\}_{i \in I} &\longmapsto \sum_{i} p_ia_i.
\end{align*}
One sees easily that $\phi$ and $\pi$ are graded right $A$-module homomorphism and $\pi\phi=1_P$. Therefore the exact sequence 
\[0\longrightarrow \ker(\pi) \longrightarrow  \bigoplus_i A(-\delta_i) \stackrel{\pi}{\longrightarrow} P \longrightarrow 0\] splits. Thus $P$ is a direct summand of the graded free module $\bigoplus_i A(-\delta_i)$. By Proposition~\ref{grprojectivethm}, $P$ is a graded projective. 
\end{proof}

\begin{remark}\scm[Graded injective modules]

Proposition~\ref{grprojectivethm} shows that an $A$-module $P$ is graded projective if and only if $P$ is graded and projective. However the similar statement is not valid for graded injective modules. A graded $A$-module $I$  is called a \emph{graded injective module} \index{graded injective module} if 
for any diagram of graded modules and graded $A$-module homomorphisms
\begin{displaymath}
\xymatrix{ 
0 \ar[r] & N \ar[r]^{g}  \ar[d]_-{j} & M \ar@{.>}[dl]^{h} \\ 
& I
}
\end{displaymath}
there is a graded $A$-module homomorphism 
$h :M \ra I$ with $hg = j$. 

Using Lemma~\ref{grrlinearmaps} one can show that if a graded module is injective, then it is also graded injective. However a graded injective module is not necessarily injective. The reason for this difference between projective and injective behaviour is that the forgetful functor $U:\Gr A \rightarrow \Modd A$ is a left adjoint functor 
(see Remark~\ref{forgetful}). In details, consider a graded projective module $P$ and the diagram
\begin{displaymath}
\xymatrix{ & P \ar[d]^-{j}  & \\
     M \ar[r]_{g} & N \ar[r] & 0,}
\end{displaymath}
where $M$ and $N$ are $A$-modules. Since the diagram below is commutative 
\begin{displaymath}
\xymatrix{ 
\Hom_{\Modd A}(U(P) ,M) \ar[r]^{\cong} \ar[d] & \Hom_{\Gr A}(P,F(M)) \ar[d]      & \\
\Hom_{\Modd A}(U(P),N) \ar[r]^{\cong}  & \Hom_{\Gr A}(P,F(N))  }
\end{displaymath}
and  there is a graded homomorphism $h':P \rightarrow F(M)$ such that the diagram 
\begin{displaymath}
\xymatrix{ & P \ar[d]^-{j'}  \ar[ld]_{h'} & \\
     F(M) \ar[r]_{F(g)} & F(N) \ar[r] & 0, }
\end{displaymath}
is commutative, there is a homomorphism $h:P\rightarrow M$ such that $gh=j$.  So $P$ is projective (see Proposition~\ref{grprojectivethm} for another proof).

If the grade group is finite, then the forgetful functor is right adjoint as well (see Remark~\ref{forgetful}) and a similar argument as above shows that a graded injective module is injective.   
 \end{remark}

 \subsection{Graded divisible modules}  \index{graded divisible module}
 
 Here we define the notion of graded divisible modules and we give yet another characterisation of graded von Neumann regular rings (see~\S\ref{penrith7may}).

Let $A$ be a $\Gamma$\!-graded ring and $M$ a graded right $A$-module. We say $m\in M^h$ (a homogeneous element of $M$) is divisible by $a\in A^h$ if $m \in Ma$, \ie there is a homogeneous element $n\in M$ such that $m=na$.  We say that $M$ is a \emph{graded divisible module} if for any $m\in M^h$ and any $a\in A^h$, where $\ann_{r}(a) \subseteq \ann(m)$, we have that $m$ is divisible by $a$. Note that for $m\in M^h$, \emph{the annihilator} of $m$, \[\ann(m):=\{\, a\in A \mid ma=0\,\}\] is a graded ideal of $A$. \index{annihilator ideal} \index{$\ann(m)$, annihilator ideal}

\begin{proposition}\label{hyhfdtegdt}
Let $A$ be a $\Gamma$\!-graded ring and $M$ a graded right $A$-module. Then the following are equivalent:

\begin{enumerate}[\upshape(1)]

\item $M$ is a graded divisible module;

\item For any $a\in A^h$, $\gamma \in \Gamma$, and any graded $A$-module homomorphism 
$f:aA(\gamma) \rightarrow M$, the following diagram can be completed. 
\begin{displaymath}
\xymatrix{ 
0 \ar[r] & aA(\gamma) \ar[r]^{\subseteq}  \ar[d]_-{f} & A(\gamma) \ar@{.>}[dl]^{\overline f}\\ 
& M.
}
\end{displaymath}
\end{enumerate}
\end{proposition}
\begin{proof}

(1) $\Ra$ (2)  Let $f:aA(\gamma) \rightarrow M$ be a graded $A$-homomorphism. Set $m:=f(a)$.
Note that $\deg(m)=\alpha-\gamma$, where $\deg(a)=\alpha$ (see Example~\ref{monash1}). 
 If $x\in \ann_r(a)$ then $0=f(ax)=f(a)x=mx$, thus $x \in \ann(m)$. Since $M$ is graded divisible, $m$ is divisible by $a$, \ie $m=na$ for some $n\in M$. It follows that $\deg(n)=-\gamma$. Define $\overline f: A(\gamma) \rightarrow M$ by 
$ \overline f(1)=n$ and extend it to $A(\gamma)$. Thisis a graded $A$-module. Since $\overline f(a)=\overline f(1)a=na=m$, $\overline f$ extends $f$. 

(2) $\Ra$ (1) Let $m\in M^h$ and $a\in A^h$, where $\ann_r(a) \subseteq \ann(m)$. Suppose  $\deg(a)=\alpha$  and $\deg(m)=\beta$. Let $\gamma=\alpha-\beta$ and define the map $f:aA(\gamma)\rightarrow M$, by $f(a)=m$ and extend it to $aA(\gamma)$. The following shows that $f$ is in fact a graded $A$-homomorphism. 
\[f\Big(\big(aA(\gamma)\big)_\delta\Big) \subseteq f\big((aA)_{\gamma+\delta}\big) \subseteq f\big(aA_{\gamma+\delta-\alpha}\big)=f\big(aA_{\delta-\beta}\big) = m A_{\delta-\beta} \subseteq M_{\delta}.\]

Thus there is a $\overline f: A(\gamma) \rightarrow M$ which extends $f$. So $m=f(a)=\overline f(a)=\overline f(1)a=na$, where $f(1)=n$. This means $m$ is divisible by $a$ and the proof is complete.
\end{proof}

\begin{theorem}\label{frhghu}
Let $A$ be a $\Gamma$\!-graded ring. Then $A$ is a graded von Neumann regular if and only if any graded right $A$-module is divisible. \index{graded von Neumann regular ring}
\end{theorem}

\begin{proof}
Let $A$ be a graded regular ring. Consider the exact sequence of graded right $A$-modules
\[ 0\longrightarrow aA(\gamma) \stackrel{\subseteq}{\longrightarrow} A(\gamma) \longrightarrow A(\gamma) /aA(\gamma) \longrightarrow 0.\]
Define 
\begin{align*}
f:A(\gamma) &\longrightarrow aA(\gamma), \\
x&\longmapsto abx.
\end{align*} Since $\deg(ab)=0$ this gives a split graded homomorphism for the exact sequence above. Thus $aA(\gamma)$ is a direct summand of $A(\gamma)$. This shows that any graded $A$-module homomorphism 
$f:aA(\gamma) \rightarrow M$ can be extend to $A(\gamma)$. 

Conversely, consider the diagram 
\begin{displaymath}
\xymatrix{ 
0 \ar[r] & aA \ar[r]^{\subseteq}  \ar[d]_-{\cap} & A \ar@{.>}[dl]^{\overline f}\\ 
& aA.
}
\end{displaymath}
Since $aA$ is divisible, there is $\overline f$ which completes this diagram. Set $\overline f(1)=ab$, where $b\in A^h$. We then have 
$a=\overline f(a)=\overline f(1)a=aba$. Thus $A$ is a graded regular ring. 
\end{proof}

\begin{example}\scm[Graded rings associated to filter rings]\label{filt}
\vspace{0.2cm}

A ring $A$ with identity is called a \emph{filtered ring} \index{filtered ring} if there is an ascending family $\{A_i \mid i \in \Z \}$ of additive subgroup of $A$ such that $1\in A_0$ and $A_i A_j \subseteq A_{i+j}$, for all $i,j \in \Z$. Let $A$ be a filtered ring and $M$ be a right $A$-module. $M$ is called a 
 \emph{filtered module} \index{filtered module} if there is an ascending family $\{M_i \mid i \in \Z \}$ of additive subgroup of $M$ such that  
 $M_i A_j \subseteq M_{i+j}$, for all $i,j \in \Z$. An $A$-module homomorphism $f:M\rightarrow N$ of filtered modules $M$ and $N$ is called a \emph{filtered homomorphism} if $f(M_i) \subseteq N_i$ for $i\in \Z$. A category consisting of filtered right $A$-modules for objects and filtered homomorphisms as morphisms is denoted by $\Filt A$. If $M$ is a filtered $A$-module then 
 \[\gr(M):=\bigoplus_{i\in \Z} M_i/M_{i-1}\] is a $\mathbb Z$-graded 
 $\gr(A):=\bigoplus_{i\in \Z} A_i/A_{i-1}$-module. The operations here are defined naturally. This gives a functor $\gr: \Filt A \rightarrow \Gr \! \gr(A)$. 
 In Example~\ref{grdiviwad} we will use a variation of this construction to associate a graded division algebra to a valued division algebra. 
 \index{$\Filt A$, category of filtered right $A$-modules}
 
 In the theory of filtered rings, one defines the concepts of filtered free and projective modules and under certain conditions the functor $\gr$ sends these objects to the corresponding objects in the category $\Gr \! \gr(A)$. For a comprehensive study of filtered rings see~\cite{filtring}.

\end{example}

\section{Grading on matrices}\label{matgrhe} \index{graded matrix ring}

Let  $A$ be an arbitrary ring and $\Gamma$ an arbitrary group. Then one can consider $\Gamma$\!-gradings on the matrix ring $\M_n(A)$ which, at the first glance, might look somewhat artificial.  However, these types of gradings on matrices appear quite naturally in the graded rings arising from graphs. In this section we study the grading on matrices. We then include a section on graph algebras (including path algebras and Leavitt path algebras, \S\ref{creekside}). These algebras give us a wealth of examples of graded rings and graded matrix rings.

For a free right $A$-module $V$ of dimension $n$, there is a ring isomorphism $\End_A(V)\cong \M_n(A)$. When $A$ is a $\Gamma$\!-graded ring and $V$ is a graded free module of finite rank, by Theorem~\ref{crazhorn}, $\End_A(V)$ has a natural $\Gamma$\!-grading. This induces a graded structure on the matrix ring $\M_n(A)$. In this section we study this grading on matrices. For an $n$-tuple $(\delta_1,\dots,\delta_n)$, $\delta_i\in \Gamma$, we construct a grading on the matrix ring $\M_n(A)$, denoted by $\M_n(A)(\delta_1,\dots,\delta_n)$, and we show that 
\[\End_A\big(A(-\delta_1)\oplus A(-\delta_2) \oplus \dots \oplus A(-\delta_n)\big) \conggr \M_n(A)(\delta_1,\dots,\delta_n).\]

We will see that these graded structures on matrices appear very naturally when studying the graded structure of path algebras in~\S\ref{creekside}. 

\subsection{Graded calculus on matrices}

Let $A$ be a $\Gamma$\!-graded ring and let $M=M_1\oplus \dots \oplus M_n$, where $M_i$ are graded finitely generated  right $A$-modules. Then $M$ is also a graded right $A$-module (see~\S\ref{grtensie}).  Let 

\begin{multline*}
\big(\Hom(M_j,M_i)\big)_{1\leq i,j\leq n}:=\\
\begin{pmatrix}
\Hom_A(M_1,M_1) & \Hom_A(M_2,M_1) & \cdots &
\Hom_A(M_n,M_1) \\
\Hom_A(M_1,M_2) & \Hom_A(M_2,M_2)& \cdots &
\Hom_A(M_n,M_2) \\
\vdots  & \vdots  & \ddots & \vdots  \\
\Hom_A(M_1,M_n) & \Hom_A(M_2,M_n) & \cdots &
\Hom_A(M_n,M_n)
\end{pmatrix}.
\end{multline*}

It is easy to observe that $\big(\Hom(M_j,M_i)\big)_{1\leq i,j\leq n}$ forms a ring with component-wise addition and the matrix multiplication. Moreover, for $\la \in \Gamma$, assigning the additive subgroup 
\begin{equation}\label{hgfreaw}
\begin{pmatrix}
\Hom_A(M_1,M_1)_\la & \Hom_A(M_2,M_1)_\la & \cdots &
\Hom_A(M_n,M_1)_\la \\
\Hom_A(M_1,M_2)_\la & \Hom_A(M_2,M_2)_\la& \cdots &
\Hom_A(M_n,M_2)_\la \\
\vdots  & \vdots  & \ddots & \vdots  \\
\Hom_A(M_1,M_n)_\la & \Hom_A(M_2,M_n)_\la & \cdots &
\Hom_A(M_n,M_n)_\la
\end{pmatrix}
\end{equation}
as a $\lambda$-homogeneous component of $\big(\Hom(M_j,M_i)\big)_{1\leq i,j\leq n}$, using Theorem~\ref{crazhorn}, and (\ref{lmsgoodyear3}) it follows that  $\big(\Hom(M_j,M_i)\big)_{1\leq i,j\leq n}$ is a $\Gamma$\!-graded ring. 

Let $\pi_j:M\rightarrow M_j$ and $\kappa_j:M_j\rightarrow M$ be the (graded) projection and injection homomorphisms. For the next theorem, we need the following identities
\begin{equation}\label{powgt}
\sum_{i=1}^n \kappa_i\pi_i=\id_M \quad \text{  and  } \quad \pi_i \kappa_j =\delta_{ij}\id_{M_j},
\end{equation}
where $\delta_{ij}$ is the Kronecker delta. 

\begin{theorem}\label{nanjite}
Let $A$ be a $\Gamma$\!-graded ring and $M=M_1\oplus \dots \oplus M_n$, where $M_i$ are graded finitely generated  right $A$-modules. Then there is a graded ring isomorphism 
\begin{equation*}
\Phi: \End_A(M)\longrightarrow \big(\Hom(M_j,M_i)\big)_{1\leq i,j\leq n}
\end{equation*} 
defined by 
$f \mapsto (\pi_if \kappa_j)$, $1\leq i,j \leq n$. 
\end{theorem}
\begin{proof}
The map $\Phi$ is clearly well-defined. Since  for $f,g\in \End_A(M)$, 
\begin{multline*}
\Phi(f+g)=\big (\pi_i(f+g)\kappa_j \big)_{1\leq i,j\leq n}=\big (\pi_if\kappa_j+\pi_ig\kappa_j \big)_{1\leq i,j\leq n}=\\ \big (\pi_if\kappa_j \big)_{1\leq i,j\leq n}+\big (\pi_ig\kappa_j \big)_{1\leq i,j\leq n}=\Phi(f)+\Phi(g)
\end{multline*}
and 
\begin{multline*}
\Phi(fg)=\big (\pi_ifg\kappa_j \big)_{1\leq i,j\leq n}=\Big(\pi_if  (\sum_{l=1}^n \kappa_l\pi_l) g\kappa_j \Big)_{1\leq i,j\leq n}=\\
\Big(  \sum_{l=1}^n (\pi_i f   \kappa_l)   (\pi_l g\kappa_j) \Big)_{1\leq i,j\leq n}=\Phi(f)\Phi(g),
\end{multline*}
$\Phi$ is a ring homomorphism.  Moreover, if $f\in \End_A(M)_\la$, $\la \in \Gamma$, then \[\pi_if \kappa_j \in \Hom_A(M_j,M_i)_\la,\] for $1\leq i,j \leq n$.  This (see~(\ref{hgfreaw})) shows that $\Phi$ is a graded ring homomorphism. 
Define the map 
\begin{align*}\label{fgair}
\Psi:  \big(\Hom(M_j,M_i)\big)_{1\leq i,j\leq n} & \longrightarrow \End_A(M)\\
(g_{ij})_{1\leq i,j \leq n} &\longmapsto \sum_{1\leq i,j\leq n}\kappa_ig_{ij}\pi_j
\end{align*} 
Using the identities~\ref{powgt}, 
one can observe that the compositions $\Psi \Phi$ and $\Phi\Psi$ give the identity maps of the corresponding rings. 
Thus $\Phi$ is an isomorphism. 
\end{proof}

For a graded ring $A$, consider $A(\delta_i)$, $1\leq i \leq n$, as graded right $A$-modules and 
observe that 
\begin{equation}\label{lkdeoe}
\Phi_{\de_j,\de_i}: \Hom_A\big (A(\delta_i),A(\delta_j)\big )\cong_{\gr} A(\delta_j-\delta_i),
\end{equation} as graded left $A$-modules such that
\[\Phi_{\de_k,\de_i}(gf)=\Phi_{\de_k,\de_j}(g)\Phi_{\de_j,\de_i}(f),\] where $f\in\Hom(A(\de_i),A(\de_j))$ and $g\in\Hom(A(\de_j),A(\de_k))$
(see~(\ref{hgd543p})). If 
\[V=A(-\delta_1)\oplus A(-\delta_2) \oplus \dots \oplus A(-\delta_n),\] then by Theorem~\ref{nanjite},
\[\End_A(V)\cong_{\gr} \Big (\Hom\big (A(-\delta_j),A(-\delta_i)\big )\Big )_{1\leq i,j\leq n}.\]
Applying $\Phi_{\de_j,\de_i}$ defined in~(\ref{lkdeoe}) to each entry, we have 
\[\End_A(V)\cong_{\gr} \Big (\Hom\big (A(-\delta_j),A(-\delta_i)\big )\Big )_{1\leq i,j\leq n}\cong_{\gr}\big(A(\delta_j-\delta_i)\big)_{1\leq i,j\leq n}.\]
Denoting the graded matrix ring $\big(A(\delta_j-\delta_i)\big)_{1\leq i,j\leq n}$  by $\M_n(A)(\de_1,\dots,\de_n)$, we have 

\begin{equation}\label{pjacko}
\M_n(A)(\de_1,\dots,\de_n) =
\begin{pmatrix}
A(\de_1 - \de_1) & A(\de_2  - \de_1) & \cdots &
A(\de_n - \de_1) \\
A(\de_1 - \de_2) & A(\de_2 - \de_2) & \cdots &
A(\de_n  - \de_2) \\
\vdots  & \vdots  & \ddots & \vdots  \\
A(\de_1 - \de_n) & A(\de_2 - \de_n) & \cdots &
A(\de_n - \de_n)
\end{pmatrix}.
\end{equation}
By~(\ref{hgfreaw}), ${\M_n (A)(\de_1,\dots,\de_n)}_{\la}$, the $\la$-homogeneous elements, are the $n \times n$-matrices over
$A$ with the degree shifted (suspended) as follows:
\begin{equation}\label{mmkkhh}
\boxed{
{\M_n(A)(\de_1,\dots,\de_n)}_{\la} =
\begin{pmatrix}
A_{ \la+\de_1 - \de_1} & A_{\la+\de_2  - \de_1} & \cdots &
A_{\la +\de_n - \de_1} \\
A_{\la + \de_1 - \de_2} & A_{\la + \de_2 - \de_2} & \cdots &
A_{\la+\de_n  - \de_2} \\
\vdots  & \vdots  & \ddots & \vdots  \\
A_{\la + \de_1 - \de_n} & A_{ \la + \de_2 - \de_n} & \cdots &
A_{\la + \de_n - \de_n}
\end{pmatrix}.}
\end{equation}
This also shows that for $x\in A^h$,
 \begin{equation}\label{hogr}
\deg(\e_{ij}(x))=\deg(x)+\delta_i-\delta_j,
\end{equation}
where $\e_{ij}(x)$ is a matrix with $x$ in the $ij$-position and zero elsewhere.

In particular the zero homogeneous component (which is a ring) is of the form
\begin{equation}\label{mmkkhh4}
\boxed{
{\M_n(A)(\de_1,\dots,\de_n)}_{0} =
\begin{pmatrix}
A_{0\phantom{ - \de_1}} & A_{\de_2  - \de_1} & \cdots &
A_{\de_n - \de_1} \\
A_{\de_1 - \de_2} & A_{0\phantom{ - \de_1}} & \cdots &
A_{\de_n  - \de_2} \\
\vdots  & \vdots  & \ddots & \vdots  \\
A_{ \de_1 - \de_n} & A_{ \de_2 - \de_n} & \cdots &
A_{0\phantom{ - \de_1}}
\end{pmatrix}.}
\end{equation}

Setting $\overline \delta=(\de_1, \dots,\de_n)\in \Gamma^n$, one denotes the graded matrix ring~(\ref{pjacko}) by $\M_n(A)(\overline \de)$. To summarise, we have shown that there is a graded ring isomorphism 
\begin{equation}\label{fatloss}
\boxed{\End_A\big(A(-\delta_1)\oplus A(-\delta_2) \oplus \dots \oplus A(-\delta_n)\big) \conggr \M_n(A)(\delta_1,\dots,\delta_n).}
\end{equation}

\begin{remark}\scm[Matrix rings of a non-abelian group grading]\label{fre298}
\vspace{0.2cm}

If the grade group $\Gamma$ is non-abelian, the homogeneous components of the matrix ring take the following form: 
\[
{\M_n(A)(\de_1\dots,\de_n)}_{\la} =
\begin{pmatrix}
A_{\de_1  \la  \de_1^{-1}} & A_{\de_1  \la  \de_2^{-1}} & \cdots &
A_{\de_1  \la  \de_n^{-1}} \\
A_{\de_2  \la  \de_1^{-1}} & A_{\de_2  \la  \de_2^{-1}} & \cdots &
A_{\de_2  \la  \de_n^{-1}} \\
\vdots  & \vdots  & \ddots & \vdots  \\
A_{\de_n  \la  \de_1^{-1}} & A_{\de_n  \la  \de_2^{-1}} & \cdots &
A_{\de_n \la  \de_n^{-1}}
\end{pmatrix}.
\]

\end{remark}

Consider the graded $A$-bimodule $A^n(\overline \de)=A(\de_1)\oplus \dots \oplus A(\de_n).$ Then one can check that $A^n(\overline \de)$ is a graded right $\M_n(A)(\overline \de)$-module and $A^n(-\overline \de)$ is a graded left $\M_n(A)(\overline \de)$-module, where 
$-\overline \de=(-\de_1, \dots,-\de_n)$. These will be used in the graded Morita theory (see Proposition~\ref{instancegg}). 
 
One can easily check the graded ring $R=\M_n(A)(\overline \de)$, where $\overline \de=(\de_1,\dots\de_n)$, $\de_i \in \Gamma$, has the support 
\begin{equation}\label{kkjjhhs}
\Gamma_R=\bigcup_{1\leq i,j\leq n}  \Gamma_A+\de_i-\de_j .
\end{equation}

One can rearrange the shift, without changing the graded matrix ring as the following theorem shows (see also~\cite[pp.~60-61]{grrings}). 

\begin{theorem}\label{ppooahnrq}
Let $A$ be a $\Gamma$\!-graded ring and $\delta_i \in \Gamma$, $1\leq i \leq n$. 
\begin{enumerate}[\upshape(1)]
\item If $\alpha \in \Gamma$, and $\pi \in S_n$ is a permutation then 
\begin{equation} \label{pqow1}
\M_n (A)(\de_1 , \ldots , \de_n)\cong_{\gr} \M_n (A)(\de_{\pi(1)}+\alpha , \ldots , \de_{\pi(n)}+\alpha).
\end{equation}

\item   If $\tau_1,\dots, \tau_n \in \Ga^*_A$, then 
\begin{equation}\label{pqow2}
\M_n (A)(\de_1 , \ldots , \de_n)\cong_{\gr} \M_n (A)(\de_1+\tau_1 , \ldots , \de_n+\tau_n).
\end{equation}
\end{enumerate}
\end{theorem}
 
 \begin{proof}
(1) Let $V$ be a graded free module over $A$ with a homogeneous basis $v_1,\dots,v_n$ of degree $\lambda_1,\dots,\lambda_n$, respectively.  It is easy to see that ((\ref{excampp})) 
\[V\cong_{\gr}A(-\la_1)\oplus\dots\oplus A(-\la_n),\]
and thus 
$\End_A(V)\cong_{\gr} \M_n(A)(\lambda_1,\dots,\lambda_n)$ (see~(\ref{fatloss})). Now let  $\pi \in S_n$. Rearranging the homogeneous basis as $v_{\pi(1)},\dots,v_{\pi(n)}$ and 
defining the $A$-graded isomorphism $\phi:V \rightarrow V$ by $\phi(v_i)=v_{\pi^{-1}(i)}$, we get a graded isomorphism in the level of endomorphism rings, called $\phi$ again  
\begin{equation}\label{hgyhdre4}
 \M_n(A)(\lambda_1,\dots,\lambda_n)\cong_{\gr} \End_A(V) \stackrel{\phi}{\longrightarrow} \End_A(V)\cong_{\gr}  \M_n(A)(\lambda_{\pi(1)},\dots,\lambda_{\pi(n)}).
 \end{equation}
Moreover, (\ref{mmkkhh}) shows that it does not make any difference adding a fixed $\alpha \in \Gamma$ to each of the entries in the shift. This gives us (\ref{pqow1}). 

In fact, the isomorphism $\phi$ in~(\ref{hgyhdre4}) is defined as $\phi(M)=P_{\pi} M P_{\pi}^{-1}$, where $P_{\pi}$ is the $n\times n$ permutation matrix with entries at $(i,\pi(i))$, $1\leq i \leq n$, to be 1 and zero elsewhere. \index{permutation matrix}

(2) For (\ref{pqow2}), let $\tau_i\in \Ga^*_A$, $1\leq i \leq n$, that is, $\tau_i=\deg(u_i)$ for some units $u_i \in A^h$. Consider the basis $v_iu_i$, $1\leq i \leq n$ for $V$. With this basis, 
\[\End_A(V)\cong_{\gr} \M_n (A)(\de_1+\tau_1 , \ldots , \de_n+\tau_n).\] Consider the $A$-graded isomorphism $\id:V\rightarrow V$, by $\id(v_i)= (v_i u_i)u_i^{-1}$. A similar argument as Part (1) now gives (\ref{pqow2}).  The isomorphism is given by $\phi(M)=P^{-1}MP$, where $P=D[u_1,\dots, u_n]$ is the diagonal matrix. 
 \end{proof}

Note that if $A$ has a trivial $\Gamma$\!-grading, \ie $A=\bigoplus_{\gamma\in \Gamma}  A_\gamma$, where 
$A_0=A$ and $A_\gamma=0$, for $0\not = \gamma \in \Gamma$,  this construction induces a \emph{good grading} \index{good grading} on $\M_n(A)$. 
By definition, this is a grading on $\M_n(A)$ such that the matrix unit \index{matrix units} $\e_{ij}$, the
matrix with $1$ in the $ij$-position and zero everywhere else, is homogeneous, for $1\leq i,j\leq n$. \index{$\e_{ij}$, matrix unit}
This particular group gradings on matrix rings have been
studied by D\u{a}sc\u{a}lescu et al.~\cite{dascalescu} (see Remark~\ref{goodig}). Therefore for $x \in A$, 
\begin{equation}\label{oiuytr}
\deg(\e_{ij}(x))=\delta_i - \delta_j.
\end{equation}

One can easily check that for a ring $A$ with trivial $\Gamma$\!-grading,  the graded ring $\M_n(A)(\overline \de)$, where $\overline \de=(\de_1,\dots\de_n)$, $\de_i \in \Gamma$, has the support 
$\{\,\de_i-\de_j \mid 1\leq i,j\leq n\,\}$. (This follows also immediately from~(\ref{kkjjhhs}).) 
  
The grading on matrices appears quite naturally in the graded rings arising from graphs. We will show that the graded structure Leavitt path algebras of acyclic and comet graphs are in effect the graded matrix rings as constructed above (see~\S\ref{creekside}, in particular, Theorems~\ref{gra1} and~\ref{grCn1}).

\begin{example}
Let $A$ be a ring, $\Gamma$ a group and $A$ graded trivially by $\Gamma$, \ie $A$ is concentrated in degree zero (see~\S\ref{jdjthu}). 
Consider the $\Gamma$\!-graded matrix ring \[R=\M_n(A)(0,-1,\dots,-n+1),\] where $n\in \mathbb N.$  By (\ref{kkjjhhs}) the support of $R$ is the set $\{-n+1,-n+2,\dots, n-2,n-1\}$. By (\ref{mmkkhh}), for $k \in \mathbb Z$ we have the following arrangements for the homogeneous elements of $R$, 
\[R_k=\left(\begin{array}{cccc}
A_k & A_{k-1} & \dots & A_{k+1-n}\\
A_{k+1} & A_k &\dots & A_{k+2-n}  \\
 \vdots & \vdots &  \ddots & \vdots \\
A_{k+n-1} & A_{k+n-2} & \dots & A_k  \\
\end{array}\right). 
\]
Thus the 0-component ring is  \index{0-component ring of a graded ring}
\[
R_0=\left(\begin{array}{cccc}
A & 0 & \dots & 0\\
0 & A &\dots & 0  \\
 \vdots & \vdots &  \ddots & \vdots \\
0 & 0 & \dots & A \\
\end{array}\right)\]
and
\begin{align*}
R_{-1}&=\left(\begin{array}{cccc}
0 & 0 & \dots & 0\\
A & 0&\dots & 0  \\
 \vdots & \ddots &  \ddots & \vdots \\
0 &  \dots & A & 0  \\
\end{array}\right), \dots,
R_{-n+1}=\left(\begin{array}{cccc}
0  & 0 & \dots & 0\\
0 & 0 &\dots & 0  \\
 \vdots & \vdots &  \ddots & \vdots \\
A & 0 & \dots &0  \\
\end{array}\right)\\
R_{1}&=\left(\begin{array}{cccc}
0 & A & \dots & 0\\
0 & 0 & \ddots & \vdots  \\
 \vdots & \vdots &  \ddots & A \\
0 & 0 & \dots & 0 \\
\end{array}\right),\dots,\quad 
R_{n-1}=\left(\begin{array}{cccc}
0  & 0 & \dots & A\\
0 & 0 &\dots & 0  \\
 \vdots & \vdots &  \ddots & \vdots \\
0 & 0 & \dots &0  \\
\end{array}\right).
 \end{align*}

In Chapter~\ref{moritanji}, we will see that $R$ is graded Morita equivalent to the trivially graded ring $A$. 

\end{example}

\begin{example} Let $S$ be a ring, $S[x,x^{-1}]$ the $\mathbb Z$-graded Laurent polynomial ring and $A=S[x^3,x^{-3}]$ the $\mathbb Z$-graded subring with support $3\mathbb Z$ (see Example~\ref{egofgrdivisionrings0}). 
Consider the $\mathbb Z$-graded matrix ring \[\M_6(A)(0,1,1,2,2,3).\] By (\ref{mmkkhh}), the homogeneous elements of degree 1 have the form
$$\left(\begin{array}{cccccc}
A_1 & A_0 & A_0 & A_{-1} & A_{-1} & A_{-2} \\
A_2 & A_1 & A_1 & A_0 & A_0 & A_{-1} \\
A_2 & A_1 & A_1 & A_0 & A_0 & A_{-1} \\
A_3 & A_2 & A_2 & A_1 & A_1 & A_0 \\
A_3 & A_2 & A_2 & A_1 & A_1 & A_0 \\
A_4 & A_3 & A_3 & A_2 & A_2 & A_1 \\
\end{array}\right)=
\left(\begin{array}{cccccc}
0 & S & S & 0 & 0 & 0 \\
0 & 0 & 0 & S & S & 0 \\
0 & 0 & 0 & S & S & 0 \\
S x^3 & 0 & 0 & 0 & 0 & S \\
S x^3 & 0 & 0 & 0 & 0 & S \\
0 & S x^3 & S x^3 & 0 & 0 & 0 \\
\end{array}\right).$$
\end{example}

\begin{example}\label{fgfdgs}
Let $K$ be a field. Consider the $\mathbb Z$-graded ring \[R=\M_5(K)(0,1,2,2,3).\] Then the support of this ring is $\{0,\pm1,\pm2\}$ and by~(\ref{mmkkhh4}) the zero homogeneous component (which is a ring) is 
\[
R_0=
\begin{pmatrix}
K & 0 & 0& 0 &0 \\
0 & K & 0& 0 &0 \\
0 & 0 & K& K &0 \\
0 & 0 & K& K &0 \\
0 & 0 & 0& 0 &K
\end{pmatrix}
\cong K\oplus K \oplus \M_2(K)\oplus K. \]
\end{example}

\begin{example}\label{vpnnabil} \scm[$\M_n(A)(\delta_1,\dots,\delta_n)$ with  $\Gamma=\{\delta_1,\dots,\delta_n\}$ is a skew group ring]
\vspace{0.2cm}

Let $A$ be a $\Gamma$\!-graded ring, where $\Gamma=\{\delta_1,\dots,\delta_n\}$ is a finite group. Consider $\M_n(A)(\delta_1,\dots,\delta_n)$, which is a $\Gamma$\!-graded ring with its homogeneous components described by~(\ref{mmkkhh}).  We will show that this graded ring is the skew group ring $\M_n(A)_0\star \Gamma$. In particular, by Proposition~\ref{crossedproductstronglygradedprop}(3), it is a strongly graded ring. Consider the matrix $u_\alpha \in \M_n(A)(\delta_1,\dots,\delta_n)_\alpha$, where in each row $i$, we have $1$ in 
$(i,j)$ position, where $\delta_j-\delta_i+\alpha=0$ and zero everywhere else. One can easily see that $u_\alpha$ is a permutation matrix with exactly one $1$ in each row and column. Moreover, for $\alpha, \beta \in \Gamma$, $u_\alpha u_\beta=u_{\alpha+\beta}$. 
Indeed, consider the $i$th row of $u_\alpha$, with the only $1$ in $j$th column where $\delta_j-\delta_i+\alpha=0$.  Now, consider the $j$th row of $u_\beta$ with a $k$th column such that $\delta_k-\delta_j+\beta=0$ and so with $1$ in $(j,k)$ row. Thus multiplying $u_\alpha$ with $u_\beta$, we have zero everywhere in $i$th row except in $(i,k)$th position. On the other hand since $\delta_k-\delta_i+\alpha+\beta=0$, in $i$th row of $u_{\alpha+\beta}$ we have zero except in $(i,k)$th position. Repeating this argument for each row of $u_\alpha$ shows that $u_\alpha u_\beta=u_{\alpha+\beta}$. 

Now defining $\phi:\Gamma \rightarrow \Aut(\M_n(A)_0)$ by $\phi(\alpha)(a)=u_\alpha a {u_\alpha}^{-1}$, and setting the 2-cocycle $\psi$ trivial, 
by~\S\ref{khgfewa1}, $R= \M_n(A)_0 \star_{\phi} \Gamma$. 

This was observed in~\cite{nats1}, where it was proved that $\M_n(A)(\delta_1,\dots,\delta_n)_0$ is isomorphic to the smash product of 
Cohen and Montgomery~\cite{cohenmont} (see Remark~\ref{bizet}). 
\end{example}


\begin{example}
The following examples (from
\cite[Example~1.3]{dascalescu}) provide two instances of $\mathbb
Z_2$-grading on $\M_2 (K)$, where $K$ is a field. The first grading is a good
grading, whereas the second one is not a good grading.
\begin{enumerate}
\item Let $R = \M_2 (K)$ with the $\mathbb Z_2$-grading defined by
\begin{align*}
R_{0} = \left\{
\begin{pmatrix}
a & 0 \\ 0 & b
\end{pmatrix} \mid a,b \in K \right\}
\mathrm{ \;\;\; and \;\;\; } R_{1} = \left\{
\begin{pmatrix}
0 & c \\d & 0
\end{pmatrix} \mid c,d \in K \right\}.
\end{align*}
Since $e_{11},e_{22} \in R_0$ and   $e_{12},e_{21} \in R_1$, by definition, this is a good grading. Note that $R=\M_2(K)(0,1)$. 

\item Let $S = \M_2 (K)$ with the $\mathbb Z_2$-grading defined by
\begin{align*}
S_{0} = \left\{
\begin{pmatrix}
a & b-a \\ 0 & b
\end{pmatrix} \mid a,b \in K \right\}
\mathrm{ \;\;\; and \;\;\; } S_{1} = \left\{
\begin{pmatrix}
d & c \\d & -d
\end{pmatrix} \mid c,d \in K \right\}.
\end{align*}
Then $S$ is a graded ring, such that the $\mathbb Z_2$-grading is
not a good grading, since $e_{11}$ is not homogeneous. Moreover, comparing $S_0$ with (\ref{mmkkhh4}), shows that the grading on $S$ does not come from the construction given by (\ref{pjacko}). 

Consider the map 
\[
f :R \lra S; \;\;\; 
\begin{pmatrix}
a &b \\ c&d
\end{pmatrix}
\lmps
\begin{pmatrix}
a+c & b+d -a-c \\ c &d-c
\end{pmatrix}.
\] This map is in fact a graded ring isomorphism, and so $R\conggr S$. 
This shows that the good grading is not preserved under graded isomorphisms. 
\end{enumerate}
\end{example}

\begin{remark} \scm[Good gradings on matrix algebras]\label{goodig}
\vspace{0.2cm}

Let $K$ be a field and $\Gamma$ be an abelian group. One can put a $\Gamma$\!-grading on the ring $\M_n(K)$, by assigning a degree (an element of the group $\Gamma$) to each matrix unit $\e_{ij}$, $1\leq i,j\leq n$. This is called a \emph{good grading} or an \emph{elementary grading}. \index{good grading}   \index{matrix units} 
\index{elementary grading} This grading has been studied in~\cite{dascalescu}. In particular it has been shown that a grading on $\M_n(K)$ is good if and only if it can be described as $\M_n(K)(\de_1,\dots,\de_n)$ for some $\delta_i \in \Gamma$. Moreover, any grading on $\M_n(K)$ is a good grading if $\Gamma$ is torsion free. It has also been shown that
 if $R=\M_n(K)$ has a $\Gamma$\!-grading such that $\e_{ij}$ is a homogeneous for some $1\leq i,j\leq n$, then there exists a good grading on $S=\M_n(K)$ with a graded isomorphism $R\cong S$. It is shown that if $\Gamma$ is finite, then the number of good gradings on $\M_n(K)$ is $|\Gamma|^{n-1}$. Moreover, (for a finite $\Gamma$) the class of strongly graded and crossed product good gradings of $\M_n(K)$ have been classified. 
\end{remark}

\begin{remark} 
Let $A$ be a $\Gamma$\!-graded ring and $\Omega$ a subgroup of $\Gamma$. Then $A$ can be considered as $\Gamma/\Omega$-graded ring.  Recall that this gives the forgetful functor $U:\mathcal R^{\Gamma}\rightarrow \mathcal R^{\Gamma/\Omega}$ (\S\ref{mconfi1}). Similarly on the level of modules, one has (again) the forgetful functor   $U:\Gr[\Gamma] A \rightarrow \Gr[\Gamma/\Omega] A$ (\S\ref{bill100}). 

If $M$ is a finitely generated $A$-module, then by Theorem~\ref{crazhorn}, $\End(M)$ is a $\Gamma$\!-graded ring. One can observe that 
\[U(\End(M)) = \End(U(M)).\] In particular, $\M_n(A)(\de_1,\dots,\de_n)$ as a $\Gamma/\Omega$-graded ring coincides with
\[\M_n(A)(\Omega+\de_1,\dots,\Omega+\de_n),\] where $A$ here in the latter case is considered as a $\Gamma/\Omega$ ring.  
\end{remark}

\begin{remark} \scm[Grading on matrix rings with infinite rows and columns]\label{infigoodig}
\vspace{0.2cm}

Let $A$ be a $\Gamma$\!-graded ring and $I$ an index set which can be uncountable. Denote by $\M_I(A)$ the matrix ring with entries indexed by $I\times I$, namely, 
$a_{ij}\in A$, where $i,j\in I$, which are all but a finite number are nonzero. For $i\in I$, choose $\delta_{i} \in \Gamma$ and following the grading on usual matrix rings (see~(\ref{hogr})) for $a\in A^h$, define, 
\begin{equation}\label{infihogr}
\deg(a_{ij})=\deg(a)+\delta_i-\delta_j.
\end{equation}
This makes $\M_I(A)$ a $\Gamma$\!-graded ring. Clearly if $I$ is finite, then this graded ring coincides with $\M_n(A)(\delta_1,\dots,\delta_n)$, where $|I|=n.$
\end{remark}

\subsection{Homogeneous idempotents calculus}\label{idemptis} \index{homogeneous idempotent}

The idempotents naturally arise in relation with decomposition of rings and modules. The following facts about idempotents are well known in the ungraded setting and one can check that they translate into the graded setting with similar proofs (cf.~\cite[\S 21]{lamfc}). Let $P_i$, $1\leq i \leq l$,  be graded right ideals of $A$ such that $A=P_1\oplus\dots \oplus P_l$. Then there are homogeneous orthogonal idempotents $e_i$  (hence of degree zero) such that $1=e_1+\dots+e_l$ and $e_iA=P_i$. 

Let $e$ and $f$ be  homogeneous idempotent elements in the graded ring $A$. (Note that, in general, there are nonhomogeneous idempotents in a graded ring.) Let $\theta:eA \rightarrow fA$ be a 
right $A$-module homomorphism. 
Then $\theta(e)=\theta(e^2)=\theta(e)e=fae$ for some $a\in A$ and for $b \in eA$, $\theta(b)=\theta(eb)=\theta(e)b$. 
This shows that there is a map 
\begin{align}\label{hgwob}
\Hom_A(eA,fA) & \rightarrow fAe,\\
\theta & \mapsto \theta(e) \notag
\end{align}
and one can easily check this is a group isomorphism. We have \[fAe=\bigoplus_{\ga \in \Gamma} fA_\ga e\] and by Theorem~\ref{crazhorn}, 
\[\Hom_A(eA,fA)\cong \bigoplus_{\ga \in \Gamma}\Hom_A(eA,fA)_\ga.\] 
Then one can see that the homomorphism~(\ref{hgwob}) respects the graded decomposition.

 Now if $\theta:eA\rightarrow fA(\alpha)$, where $\alpha \in \Gamma$, is a graded $A$-isomorphism, then $x=\theta(e)\in fA_\alpha e$ and  $y=\theta^{-1}(f)\in eA_{-\alpha}f$, where $x$ and $y$ are homogeneous of degrees $\alpha$ and $-\alpha$, respectively, such that $yx=e$ and $xy=f$.  

Finally, for $f=1$, the map~(\ref{hgwob}) gives that 
\[
\Hom_A(eA, A)  \rightarrow Ae,\\
\]
is a graded left $A$-module isomorphism and for $f=e$, 
\[
\End_A(eA)\rightarrow eAe,\\
\]
is a graded ring isomorphism. In particular, we have a ring isomorphism 
\[\End_A(eA)_0=\End_{\Gr A}(eA) \cong eA_0e.\]

These facts will be used later in Theorem~\ref{catgrhsf}.

\subsection{Graded matrix units}

Let $A$ be a $\Gamma$\!-graded ring. Modelling on the properties of the matrix units $\e_{ij}$, we call a set of homogeneous elements 
$\{\, e_{ij}\in A \mid1\leq i,j\leq n\, \}$,  a set of \emph{graded matrix units} \index{graded matrix units} if 
\begin{equation}\label{maxplanckwed}
e_{ij} e_{kl}=\delta_{jk}e_{il},
\end{equation} 
where $\delta_{jk}$ are the Kronecker deltas. Let $\deg(e_{i1})=\delta_i$. From~(\ref{maxplanckwed}) it follows that 
$\deg(e_{ii})=0$, $\deg(e_{1i})=-\delta_i$ and 
\begin{equation} \label{maxbeannerstr}
\deg(e_{ij})=\delta_i-\delta_j.
\end{equation}

The above set is called a \emph{full set of graded matrix units} \index{full set of matrix units} if $\sum_{i=1}^n e_{ii}=1$. If a graded ring contains a full set of graded matrix units, then the ring is of the form of a matrix ring over an appropriate graded ring (Lemma~\ref{mengfy4}). We can use this to characterise the two-sided ideals of graded matrix rings (Corollary~\ref{mhfy5874}). For this we adopt Lam's presentation~\cite[\S17A]{lam} to the graded setting. 

\begin{lemma}\label{mengfy4}
Let $R$ be a $\Gamma$\!-graded ring. Then $R=\M_n(A)(\delta_1,\dots,\delta_n)$ for some graded ring $A$ if and only if $R$ has a full set of graded matrix units $\{\, e_{ij}\in R \mid1\leq i,j\leq n\, \}$.
\end{lemma}
\begin{proof}
One direction is obvious. Suppose $\{e_{ij}\in R \mid1\leq i,j\leq n\}$ is a full set of graded matrix units in $R$ and $A$ is its centraliser in $R$ which is a graded subring of $R$. We show that $R$ is a graded free $A$-module with the basis $\{e_{ij}\}$. Let $x\in R$ and set 
\[a_{ij}=\sum_{k=1}^{n} e_{ki} x e_{jk} \in R.\]
Since $a_{ij} e_{uv}=e_{ui}xe_{jv}=e_{uv}a_{ij}$, it follows that $a_{ij} \in A$. Let $u=i, v=j$. Then 
$a_{ij}e_{ij}=e_{ii}xe_{jj}$, and since $\{e_{ij}\}$ is full, $\sum_{i,j} a_{ij} e_{ij}=\sum_{ij}e_{ii}xe_{jj}=x $. This shows that $\{e_{ij}\}$ generates $R$ as an $A$-module. It is easy to see that $\{e_{ij} \mid1\leq i,j\leq n\}$ is linearly independent as well. Let $\deg(e_{i1})=\delta_i$. Then (\ref{maxbeannerstr}) shows that the map 
\begin{align*}
R &\longrightarrow \M_n(A)(\delta_1,\dots,\delta_n),\\
ae_{ij} &\longmapsto \e_{ij}(a)
\end{align*}
induces a graded isomorphism. 
\end{proof}

\begin{corollary}
Let $A$, $R$ and $S$ be $\Gamma$\!-graded rings. Suppose \[R=\M_n(A)(\delta_1,\dots,\delta_n)\]  
and there is a graded ring homomorphism $f:R \rightarrow S$. Then \[S=\M_n(B)(\delta_1,\dots,\delta_n),\] for a graded ring $B$ and $f$ is induced by a graded homomorphism $f_0:A\rightarrow B$.  
\end{corollary}
\begin{proof}
Consider the standard full graded matrix units $\{ \e_{ij}\mid 1\leq i,j \leq n\}$ in $R$. Then $\{f(\e_{ij})\}$ is a set of full graded matrix units in $S$. Since $f$ is a graded homomorphism,   
$\deg(f(\e_{ij}))=\deg(e_{ij})=\delta_i-\delta_j$. Let $B$ be the centraliser of this set in $S$. By Lemma~\ref{mengfy4} (and its proof),
\[S=\M_n(B)(0,\delta_2-\delta_1,\dots,\delta_n-\delta_1)=\M_n(B)(\delta_1,\delta_2,\dots,\delta_n). \]
Since $A$ is the centraliser of $\{ \e_{ij}\}$, $f$ sends $A$ to $B$ and thus induces the map on the matrix algebras.  
\end{proof}

\begin{corollary}\label{mhfy5874}
Let $A$ be a $\Gamma$\!-graded ring,  $R=\M_n(A)(\delta_1,\dots,\delta_n)$ and $I$ be a graded ideal of $R$. Then $I=\M_n(I_0)(\delta_1,\dots,\delta_n)$, where $I_0$ is a graded ideal of $A$. 
\end{corollary}
\begin{proof}
Consider the canonical graded quotient homomorphism $f:R\rightarrow R/I$. Set $I_0= \ker (f|_A)$. One can easily see $\M_n(I_0)(\delta_1,\dots,\delta_n)\subseteq I$. By Lemma~\ref{mengfy4}, $R/I=\M_n(B)(\delta_1,\dots,\delta_n)$, where $B$ is the centraliser of the set $\{f(\e_{ij})\}$. Since $A$ is the centraliser of $\{\e_{ij}\}$, $f(A) \subseteq B$. Now for $x\in I$, write $x=\sum_{i,j} a_{ij} \e_{ij}$, $a_{ij} \in A$. Then $0=f(x)=\sum_{i,j} f(a_{ij}) f(\e_{ij})$, which implies $f(a_{ij})=0$ as $f(\e_{ij})$ are linear independent (see the proof of Lemma~\ref{mengfy4}). Thus $a_{i,j} \in I_0$. This shows $I\subseteq \M_n(I_0)(\delta_1,\dots,\delta_n)$, which finishes the proof. 
\end{proof}

Corollary~\ref{mhfy5874} shows that there is a one-to-one inclusion preserving correspondence between the graded ideals of $A$ and the graded ideals of $\M_n(A)(\overline \delta)$, where $\overline \delta=(\delta_1,\dots,\delta_n)$.

\subsection{Mixed shift} \label{kjujidw}

For a $\Ga$\!-graded ring $A$, $\overline \alpha = (\al_1 , \ldots \al_m) \in \Ga^m$ and $\overline \delta=(\de_1, \ldots ,\de_n) \in \Ga^n$, set
\[
\boxed{
\M_{m \times n} (A)[\overline \alpha][\overline \delta] :=
\begin{pmatrix}
A_{\al_1 -\de_1} & A_{\al_1 -\de_2} & \cdots &
A_{\al_1 -\de_n } \\
A_{\al_2 -\de_1} & A_{\al_2 -\de_2} & \cdots &
A_{\al_2 -\de_n  } \\
\vdots  & \vdots  & \ddots & \vdots  \\
A_{\al_m-\de_1} & A_{\al_m -\de_2} & \cdots & A_{\al_m -\de_n}
\end{pmatrix}.}
\]
So $\M_{m\times n} (A)[\overline \alpha][\overline \delta]$ consists of matrices with the
$ij$-entry in $A_{\al_i -\de_j}$.

If $a \in \M_{m \times n} (A)[\overline \alpha][\overline \delta]$, then one can easily check that multiplying $a$ from the left induces a graded right $A$-module homomorphism 
\begin{align}\label{fatso1}
\phi_a:A^n(\overline \delta) & \longrightarrow A^m(\overline \alpha),\\
\begin{pmatrix}
a_1   \\
a_2   \\
\vdots \\
a_n
\end{pmatrix} &\longmapsto a 
\begin{pmatrix}
a_1   \\
a_2   \\
\vdots \\
a_n
\end{pmatrix}. \notag
\end{align}
Conversely, suppose $\phi: A^n(\overline \delta) \rightarrow  A^m (\overline \alpha)$  
 is graded right $A$-module homomorphism. Let $e_j$ denote the standard basis element of $A^n(\overline \delta)$ with  $1$ in the $j$-th entry and zeros elsewhere. Let 
 $\phi(e_j)=(a_{1j},a_{2j},\dots,a_{mj})$, $1\leq j \leq n$. Since $\phi$ is a graded map, comparing the grading of both sides, one can observe that $\deg(a_{ij})=\alpha_i-\delta_j$. So that the map $\phi$ is represented by the left multiplication with the matrix $a=(a_{ij})_{m\times n} \in \M_{m\times n} (A)[\overline \alpha][\overline \delta]$.  
 
 In particular $\M_{m\times m}(A)[\overline \alpha][\overline \alpha]$ represents $\End(A^m(\alpha),A^m(\alpha))_0$. Combining this with~(\ref{fatloss}), we get
\begin{equation}\label{bequitet}
\M_{m\times m}(A)[\overline \alpha][\overline \alpha]=\M_m(A)(-\overline \alpha)_0.
\end{equation}

 The mixed shift will be used in~\S\ref{hhidmi} to describe graded Grothendieck groups by idempotent matrices.  The following simple Lemma comes in handy. 
 
\begin{lemma}\label{smallhandy}
Let $a \in \M_{m \times n} (A)[\overline \alpha][\overline \delta]$ and $b\in \M_{n \times k} (A)[\overline \delta][\overline \beta]$. Then 
$ab \in \M_{m \times k} (A)[\overline \alpha][\overline \beta]$.
\end{lemma}
\begin{proof}
Let 
$\phi_a:  A^n(\overline \delta) \rightarrow  A^m (\overline \alpha)$ and $\phi_b: A^k(\overline \beta) \rightarrow  A^n (\overline \delta)$ be the graded right $A$-module homomorphisms induced by multiplications with $a$ and $b$, respectively (see~\ref{fatso1}). Then 
\[\phi_{ab} =\phi_a \phi_b:  A^k (\overline \beta) \longrightarrow A^m(\overline \alpha).\] This shows that 
$ab \in \M_{m \times k} (A)[\overline \alpha][\overline \beta]$. (This can also be checked directly, by multiplying the matrices $a$ and $b$ and taking into account the shift arrangements.)
\end{proof}

\begin{proposition} \label{rndconggrrna}
Let $A$ be a $\Ga$\!-graded ring and let $\overline \alpha = (\al_1 , \ldots , \al_m) \in \Ga^m$, 
$\overline \delta=(\de_1, \ldots,\de_n)\in \Ga^n$. Then
the following are equivalent: 

\begin{enumerate}[\upshape(1)]

\item $A^m (\overline \alpha)  \cong_{\gr} A^n(\overline \delta)$ as graded right $A$-modules;

\item $A^m (-\overline \alpha) \cong_{\gr} A^n(-\overline \delta)$ as graded left $A$-modules;



\item There exist $a=(a_{ij}) \in \M_{n\times m} (A)[\overline \delta][\overline \alpha]$ and 
$b=(b_{ij}) \in \M_{m\times n} (A)[\overline \alpha][\overline \delta]$ such that $ab=\IM_{n}$ and $ba=\IM_{m}$.

\end{enumerate}
\end{proposition}

\begin{proof}
(1) $\Rightarrow$ (3) Let $\phi: A^m(\overline \al) \rightarrow  A^n (\overline \de)$ and $\psi: A^n(\overline \delta)  \rightarrow  A^m (\overline \alpha)$ 
 be graded right $A$-module isomorphisms such that $\phi\psi=1$ and $\psi\phi=1$. The paragraph prior to Lemma~\ref{smallhandy} shows that the map $\phi$ is represented by the left multiplication with a matrix $a=(a_{ij})_{n\times m} \in \M_{n\times m} (A)[\overline \delta][\overline \alpha]$.   In the same way one can construct $b  \in \M_{m\times n} (A)[\overline \alpha][\overline \delta]$ which induces $\psi$. Now $\phi\psi=1$ and $\psi\phi=1$ translate to $ab=\IM_{n}$ and $ba=\IM_{m}$. 

(3) $\Rightarrow$ (1)  If $a \in  \M_{n\times m} (A)[\overline \delta][\overline \alpha]$, then multiplication from the left, induces  a graded right 
$A$-module homomorphism $\phi_a:A^m(\overline \alpha) \longrightarrow A^n(\overline \delta)$. Similarly $b$ induces $\psi_b: A^n(\overline \delta) \longrightarrow A^m(\overline \alpha)$. Now $ab=\IM_{n}$ and $ba=\IM_{m}$ translate to 
$\phi_a\psi_b=1$ and $\psi_b\phi_a=1$.

(2) $\Longleftrightarrow$ (3) This part is similar to the previous cases by considering the matrix multiplication from the right. Namely, the graded left $A$-module homomorphism $\phi: A^m (-\overline \alpha) \rightarrow A^n(-\overline \delta)  $ represented by a matrix multiplication from the right of the form $\M_{m\times n} (A)[\overline \alpha][\overline \delta]$ and similarly $\psi$ gives a matrix  in 
$\M_{n\times m} (A)[\overline \de][\overline \al]$. The rest follows easily. 
 \end{proof}

The following corollary shows that  $A(\alpha) \cong_{\gr} A$ as graded right $A$-modules if and only if $\alpha \in \Gamma^*_A$. In fact, replacing 
$m=n=1$ in Proposition~\ref{rndconggrrna} we obtain the following. 

\begin{corollary} \label{rndcongcori}
Let $A$ be a $\Ga$\!-graded ring and $\alpha \in \Gamma$.  Then
the following are equivalent: 

\begin{enumerate}[\upshape(1)]

\item $A(\alpha) \cong_{\gr} A$ as graded right $A$-modules;

\item $A(-\alpha) \cong_{\gr} A$ as graded right $A$-modules;

\item $A(\alpha) \cong_{\gr} A$ as graded left $A$-modules;

\item $A(-\alpha) \cong_{\gr} A$ as graded left $A$-modules;

\item There is an invertible homogeneous  element of degree $\alpha$. 





\end{enumerate}

\end{corollary}

\begin{proof}
This follows from Proposition~\ref{rndconggrrna}. 
\end{proof} 

\begin{corollary}\label{andrewmathas}
Let $A$ be a $\Ga$\!-graded ring. Then the following are equivalent: \index{crossed product ring}

\begin{enumerate}[\upshape(1)]
\item $A$ is a crossed product;

\item  $A(\alpha) \cong_{\gr} A$, as graded right $A$-modules, for all $\alpha \in \Gamma$;

\item  $A(\alpha) \cong_{\gr} A$, as graded left $A$-modules, for all $\alpha \in \Gamma$;

\item The shift functor $\mathcal T_\alpha:\Gr A \rightarrow \Gr A$ is isomorphic to identity functor, for all $\alpha \in \Gamma$.

\end{enumerate}

\end{corollary}
\begin{proof}
This follows from Corollary~\ref{rndcongcori}, (\ref{medeltaotel}) and the definition of the crossed product rings (\S\ref{scrosshg}). 
\end{proof}

The Corollary above will be used to show that the action of $\Gamma$ on the graded Grothendieck group of a crossed product algebra is trivial (see Example~\ref{gptwnow}).

\begin{example}\scm[The Leavitt algebra $\LL(1,n)$] \label{levisuji}
\vspace{0.2cm}

In \cite{vitt62} Leavitt considered the free associative ring $A$ with coefficient in $\Z$ generated by symbols $\{x_i,y_i \mid 1\leq i \leq n\}$ subject to relations 
\begin{align}\label{jh54320}
&x_iy_j =\delta_{ij}, \text{ for all } 1\leq i,j \leq n, \\
&\sum_{i=1}^n y_ix_i =1, \notag
\end{align} where $n\geq 2$ and $\delta_{ij}$ is the Kronecker delta. 
The relations guarantee the right $A$-module homomorphism 
\begin{align}\label{is329ho}
\phi:A&\longrightarrow A^n\\
a &\mapsto (x_1a	,x_2a,\dots,x_na)\notag
\end{align}
has an inverse 
\begin{align}\label{is329ho9}
\psi:A^n&\longrightarrow A\\
(a_1,\dots,a_n) &\mapsto  y_1a_1+\dots+y_na_n, \notag 
\end{align}
so $A\cong A^n$ as right $A$-modules. He showed that $A$ is universal with respect to this property, of type $(1,n-1)$ (see \S\ref{gtr5654}) and it is a simple ring. 

Leavitt's algebra constructed in~(\ref{jh54320}) has a natural grading; assigning $1$ to $y_i$ and $-1$ to $x_i$, $1\leq i \leq n$, since the relations are homogeneous (of degree zero), the ring $A$ is a $\mathbb Z$-graded ring (see~\S\ref{freeme} for a general construction of graded rings from free algebras). The isomorphism~(\ref{is329ho}) induces a graded isomorphism   
\begin{align}\label{is329ho22}
\phi:A&\longrightarrow A(-1)^n\\
a &\longmapsto (x_1a	,x_2a,\dots,x_na), \notag
\end{align}
where $A(-1)$ is the suspension of $A$ by $-1$. In fact letting 
\begin{align*}
y& =(y_1,\dots,y_n), \text{  and   }\\ 
x&=\begin{pmatrix}
x_1   \\
x_2   \\
\vdots \\
x_n
\end{pmatrix},
\end{align*}
we have 
$y\in \M_{1\times n}(A)[\overline \alpha][\overline \delta]$  and $x\in \M_{n\times 1}(A)[\overline \delta][\overline \alpha]$, where $\overline \alpha=(0)$ and $\overline \delta=(-1,\dots,-1)$. Thus by Proposition~\ref{rndconggrrna}, 
$A\conggr A(-1)^n$. 

Motivated by this algebra, the Leavitt path algebras were introduced in~\cite{aap05,amp}, which associate to a direct graph a certain algebra. When the graph has one vertex and $n$ loops, the algebra corresponds to this graph is the Leavitt algebra constructed in (\ref{jh54320}) and is denoted by $\LL(1,n)$ or $\LL_n$.  
The Leavitt path algebras will provide a vast array of examples of graded algebras. We will study these algebras in~\S\ref{paohdme}. 
\end{example}

\section{Graded division rings}\label{mochihge} \index{graded division ring}

Graded fields and their noncommutative version, \ie graded division rings, are among the simplest graded rings. With a little effort, we can completely compute the invariants of these algebras which we are interested in, namely, the graded Grothendieck groups (\S\ref{ghgwbsia}) and the graded Picard groups (\S\ref{picle}). 

Recall from~\S\ref{khgfewa1}  that a $\Ga$\!-graded ring $A = \bigoplus_{ \ga \in \Ga} A_{\ga}$ is called
a graded division ring if every
nonzero homogeneous element has a multiplicative inverse.
Throughout this section we consider graded right modules over
graded division rings. Note that we work with the abelian grade groups, however all the results are valid for non-abelian grading as well. 
 We first show that for graded modules over a graded division ring, there is well-defined notion of dimension. 
The proofs follow the standard proofs in
the ungraded setting (see \cite[\S IV,
Theorem~2.4,~2.7,~2.13]{hungerford}), or the graded setting (see~\cite[Proposition~4.6.1]{grrings},
\cite[Chapter~2]{tignolwadsworth}).

\begin{proposition}\label{gradedfree}
Let $A$ be a $\Ga$\!-graded ring. Then $A$ is a graded division ring if and only if any graded $A$-module is graded free. 
If $M$ is a graded module over graded division ring $A$, then any
linearly independent subset of $M$ consisting of homogeneous
elements can be extended to a homogeneous basis of $M$.
\end{proposition}

\begin{proof}
Suppose any graded (right) module is graded free. Let $I$ be a right ideal of $A$. Consider $A/I$ as a right $A$-module, which is graded free by assumption. Thus $I=\ann(A/I)=0$. This shows that the only graded right ideal of $A$ is the zero ideal. This gives that $A$ is a graded division ring.

For the converse, note that if $A$ is a graded division ring (\ie all homogeneous elements are invertible), then for any $m \in M^h$, $\{m \}$ is a linearly
independent subset of $M$. This immediately gives the converse of the statement of the theorem as a consequence of the second part of the theorem. 

 Fix a linearly independent subset $X$ of
$M$ consisting of homogeneous elements. Let
$$
F= \big \{ Q \subseteq M^h \mid X \subseteq Q \textrm{ and $Q$ is
$A$-linearly independent} \big \}.
$$
This is a nonempty partially ordered set with inclusion, and every
chain $Q_1 \subseteq Q_2 \subseteq \dots$ in $F$ has an upper bound
$\bigcup Q_i \in F$. By Zorn's Lemma, $F$ has a maximal element,
which we denote by $P$. If $\left< P \right> \neq M$, then there is
a homogeneous element $m \in M^h \mi \left< P \right>$. We will show
that $P \cup \{m \}$ is a linearly independent set containing $X$,
contradicting the maximality of $P$.

Suppose $m a + \sum  p_i a_i= 0$, where $a, a_i \in A$, $p_i \in P$
with $a \neq 0$.
%
Then there is a homogeneous component of $a$, say
$a_{\la}$, which is also nonzero. Considering the $\la+
\deg(m)$-homogeneous component of this sum, we have \[m= 
 m   a_{\la} a_{\la}^{-1}=- \sum  p_i   a'_i   a_{\la}^{-1}\] for $a'_i$ homogeneous, which
contradicts the assumption $m \in M^h \mi \left< P \right>$. Hence $a=0$, which implies each $a_i
= 0$. This gives the required contradiction, so $M= \left< P
\right>$, completing the proof.
\end{proof}

The following proposition shows in particular that a graded division ring  has graded invariant basis number (we will discuss this type of rings in~\S\ref{gtr5654}). 

\begin{proposition}\label{kj351}
Let $A$ be a $\Ga$\!-graded division ring and $M$ be a graded $A$-module. 
Then any two homogeneous bases of $M$ over $A$ have the same
cardinality.
\end{proposition}

\begin{proof}
By~\cite[\S IV, Theorem~2.6]{hungerford}, if a module $M$ has an infinite basis over a ring, then any other basis of $M$ has the same cardinality. 
This proves the Proposition in the case the homogeneous basis is infinite. 

Now suppose that $M$ has two finite homogeneous bases $X$ and $Y$.  Then $X= \{ x_1, \ldots , x_n\}$ and $Y= \{ y_1 , \ldots
, y_m \}$, for $x_i$, $y_i \in M^h \mi 0$. As $X$ is a basis for $M$,
we can write
$$
y_m =x_1  a_1  + \cdots +  x_n a_n, 
$$
for some $a_i \in A^h$, where $\deg (y_m) = \deg(a_i )+ \deg (x_i)$
for each $1\leq i \leq n$. Since $y_m \neq 0$, we have at least one $a_i \neq 0$.
Let $a_k$ be the first nonzero $a_i$, and we note that $a_k$ is
invertible as it is nonzero and homogeneous in $A$. Then
$$
x_k =  y_m  a_k^{-1} -   x_{k+1} a_{k+1} a_k^{-1} - \cdots - 
 x_n  a_n a_k^{-1} ,
$$
and the set $X' = \{\,y_m ,  x_1, \ldots , x_{k-1} , x_{k+1} , \ldots
, x_n\,\}$ spans $M$ since $X$ spans $M$. So
$$
y_{m-1} =  y_m b_m +   x_1  c_1+ \cdots + x_{k-1}  c_{k-1} + 
x_{k+1} c_{k+1} + \cdots +  x_n c_n, 
$$
for $b_m$, $c_i \in A^h$. There is at least one nonzero $c_i$,
since if all the $c_i$ are zero, then either $y_m$ and $y_{m-1}$ are
linearly dependent or $y_{m-1}$ is zero which are not the case. Let $c_j$ denote the first
nonzero $c_i$. Then $x_j$ can be written as a linear combination of
$y_{m-1}$, $y_m$ and those $x_i$ with $i \neq j, k$. Therefore the
set \[X'' = \{\,y_{m-1} , y_m\, \} \cup \{\,x_i : i \neq j,k\,\}\] spans $M$
since $X'$ spans $M$. 


Continuing this process of adding a $y$ and removing an $x$ gives,
after the $k$-th step, a set which spans $M$ consisting of $y_m ,
y_{m-1} , \ldots , y_{m-k+1}$ and $n-k$ of the $x_i$. If $n <m$,
then after the $n$-th step, we would have that the set $\{ y_m ,
\ldots , y_{m-n+1} \}$ spans $M$. But if $n <m$, then $m-n +1 \geq
2$, so this set does not contain $y_1$, and therefore $y_1$ can be
written as a linear combination of the elements of this set. This
contradicts the linear independence of $Y$, so we must have $m \leq
n$. Repeating a similar argument with $X$ and $Y$ interchanged gives
$n \leq m$, so $n=m$.
\end{proof}

The Propositions~\ref{gradedfree} and~\ref{kj351} above show that for a graded module $M$ over a
graded division ring $A$, $M$ has a homogeneous basis and any two
homogeneous bases of $M$ have the same cardinality. The cardinal
number of any homogeneous basis of $M$ is called the \emph{dimension} \index{dimension of module} of
$M$ over $A$, and it is denoted by $\dim_A (M)$ or $[M:A]$.

\begin{proposition}\label{grdimensionprop}
Let $A$ be a $\Ga$\!-graded division ring and $M$ be a graded $A$-module. If $N$ is a graded submodule of $M$, then
\begin{equation*}
\dim_A (N) + \dim_A(M/N) = \dim_A (M).
\end{equation*}

\end{proposition}

\begin{proof}
By Proposition~\ref{gradedfree}, the submodule $N$ is a graded free $A$-module with a homogeneous basis $Y$ which can be extended to a homogeneous basis $X$ of $M$.  We will show that $U= \{ x + N \mid x \in X \mi Y
\}$ is a homogeneous basis of $M/N$. Note that by~(\ref{whs21}), $U$ consists
of homogeneous elements. Let $t\in  (M/N)^h$. Again by~(\ref{whs21}), $t=m + N$, where $m \in M^h$
and $m = \sum  x_i a_i+ \sum  y_jb_j $ where $a_i$, $b_j \in A$, $y_j
\in Y$ and $x_i \in X \mi Y$. So $m + N = \sum  (x_i +N) a_i$, which
shows that $U$ spans $M/N$. If $\sum (x_i +N) a_i = 0$, for $a_i \in
A$, $x_i \in X \mi Y$, then $\sum  x_i a_i \in N$ which implies that
$\sum  x_i a_i= \sum  y_k b_k$ for $b_k \in A$ and $y_k \in Y$, which implies that $a_i = 0$ and
$b_k = 0$ for all $i$,~$k$. Therefore $U$ is a homogeneous basis for
$M/N$ and as we can construct a bijective map $ X \mi Y \ra U$, we
have $|U| = |X \mi Y|$. Then \[\dim_A M = |X| = |Y| +|X \mi Y| = |Y|
+ |U| = \dim_A N + \dim_A (M/N).\qedhere\]
\end{proof}

The following statement  is the graded version of
a similar statement on simple rings (see \cite[\S IX.1]{hungerford}). This
is required for the proof of Theorem~\ref{clasigr}. \index{graded simple ring}

\begin{proposition}\label{prohungi}
Let $A$ and $B$ be $\Ga$\!-graded division rings. If \[\M_n(A)(\lambda_1,\dots,\lambda_n) \cong_{\gr} \M_{m}(B)(\gamma_1,\dots,\gamma_m)\] as graded
rings, where $\lambda_i,\gamma_j \in \Gamma$, $1\leq i \leq n$, $1\leq j \leq m$, then $n = m$ and $A \cong_{\gr} B$.
\end{proposition}
\begin{proof}
The proof follows the ungraded case (see \cite[\S IX.1]{hungerford}) with an extra attention given to the grading. We refer the reader to ~\cite[\S4.3]{millar} for a detailed proof. 
\end{proof}

We can further determine the relations between the  shift $(\lambda_1,\dots,\lambda_n)$ and $(\gamma_1,\dots,\gamma_m)$ in the above proposition. For this we need to extend ~\cite[Theorem 2.1]{can2} (see also \cite[Theorem 9.2.18]{grrings}) from fields (with trivial grading) to graded division algebras. The following theorem states that two graded matrix algebras over a graded division ring with two shifts are isomorphic if and only if one can obtain one shift from the other by applying (\ref{pqow1}) and (\ref{pqow2}). 

\begin{theorem}\label{clasigr}
Let $A$ be a $\Ga$\!-graded division ring.
Then for $\lambda_i,\gamma_j \in \Gamma$, $1\leq i \leq n$, $1\leq j \leq m$,  
\begin{equation}\label{uytrtyu}
\M_n(A)(\lambda_1,\dots,\lambda_n)\cong_{\gr}\M_m(A)(\gamma_1,\dots,\gamma_m)
\end{equation} if and only if $n=m$ and for a suitable permutation $\pi \in S_n$, we have $\lambda_i=\gamma_{\pi(i)}+\tau_i+\sigma$, $1\leq i \leq n$, where  $\tau_i\in \Ga_A$ and a fixed $\sigma \in \Gamma$, \ie $(\lambda_1,\dots,\lambda_n)$ is obtained from $(\gamma_1,\dots,\gamma_m)$ by applying {\upshape (\ref{pqow1})} and {\upshape(\ref{pqow2})}.
\end{theorem}
\begin{proof}

One direction is Theorem~\ref{ppooahnrq}, noting that since $A$ is a graded division ring, $\Gamma_A=\Gamma^*_A$.

We now prove the converse. That $n=m$ follows from Proposition~\ref{prohungi}.  By~(\ref{wadi}) one can find $\epsilon=(\varepsilon_1,\dots,\varepsilon_1,\varepsilon_2,\dots,\varepsilon_2,\dots,\varepsilon_k,\dots,\varepsilon_k)$ in $\Gamma$  such that  $\M_n(A)(\lambda_1,\dots, \lambda_n)\cong_{\gr}\M_n(A)(\epsilon)$ as in (\ref{wdeild}). 
Now set \[V=A(-\varepsilon_1)\times \dots \times A(-\varepsilon_1) \times  \dots \times A(-\varepsilon_k)\times \dots \times A(-\varepsilon_k)\] and  pick the (standard) homogeneous basis $e_i$, $1\leq i \leq n$  and define $E_{ij} \in \End_A(V)$ by $E_{ij}(e_t)=\delta_{j,t}e_i$, $1\leq i,j,t \leq n$.
One can easily see that $E_{ij}$ is a $A$-graded homomorphism of degree $\varepsilon_{s_i}-\varepsilon_{s_j}$ where $\varepsilon_{s_i}$ and $\varepsilon_{s_j}$ are $i$-th and $j$-th elements in $\epsilon$. Moreover, 
$\End_A(V) \cong_{\gr} \M_n(A)(\epsilon)$ and $E_{ij}$ corresponds to the matrix $e_{ij}$ in $\M_n(A)(\epsilon)$. In a similar manner, one can find \[\epsilon'=(\varepsilon_1',\dots,\varepsilon_1',\varepsilon_2',\dots,\varepsilon_2',\dots,\varepsilon_{k'}',\dots,\varepsilon_{k'}')\] and a graded $A$-vector space $W$ such that \[\M_n(A)(\gamma_1,\dots,\gamma_n)\cong_{\gr}\M_n(A)(\epsilon'),\] and $\End_A(W)\cong_{\gr} \M_n(A)(\epsilon')$. Therefore (\ref{uytrtyu}) provides  a graded ring isomorphism \[\theta:\End_A(V)\rightarrow \End_A(W).\] Define $E_{ij}':=\theta(E_{ij})$ and $E_{ii}'(W)=Q_i$, for $1\leq i,j \leq n$. Since $\{E_{ii} \mid 1\leq i \leq n\}$ is a complete system of orthogonal idempotents, so is $\{E_{ii}' \mid 1\leq i \leq n\}$. It follows that \[W\cong_{\gr} \textstyle{\bigoplus_{1\leq j \leq n}} Q_j.\] Moreover, $E_{ij}'E_{tr}'=\delta_{j,t}E_{ir}'$ and $E_{ii}'$ acts as identity on $Q_i$. These relations show that restricting $E_{ij}'$ on $Q_j$ induces an $A$-graded isomorphism $E_{ij}':Q_j \rightarrow Q_i$ of degree 
$\varepsilon_{s_i}-\varepsilon_{s_j}$ (same degree as $E_{ij}$). 
So   $Q_j\cong_{\gr} Q_1(\varepsilon_{s_1}-\varepsilon_{s_j})$ for any $1\leq j \leq n$.  
Therefore 
\[W\cong_{\gr} \textstyle{\bigoplus_{1\leq j \leq n}} Q_1(\varepsilon_{s_1}-\varepsilon_{s_j}).\] By dimension count (see Proposition~\ref{grdimensionprop}), it follows that $\dim_AQ_1=1$. 

A similar argument for the identity map $\id:\End_A(V)\rightarrow \End_A(V)$ produces \[V\cong_{\gr} \textstyle{\bigoplus_{1\leq j \leq n}} P_1(\varepsilon_{s_1}-\varepsilon_{s_j}),\] where $P_1=E_{11}(V)$, with $\dim_AP_1=1$. 

Since $P_1$ and $Q_1$ are $A$-graded vector spaces of dimension $1$, there is $\sigma \in \Gamma$, such that $Q_1\cong_{\gr}P_1(\sigma)$. Using  the fact that for an $A$-graded module $P$ and $\alpha,\beta \in \Gamma$,  $P(\alpha)(\beta)=P(\alpha+\beta)=P(\beta)(\alpha)$, we have 
\begin{multline}
W\cong_{\gr} \textstyle{\bigoplus_{1\leq j \leq n}} Q_1(\varepsilon_{s_1}-\varepsilon_{s_j})\cong_{\gr}
\textstyle{\bigoplus_{1\leq j \leq n}} P_1(\sigma)(\varepsilon_{s_1}-\varepsilon_{s_j})\cong_{\gr} \\
\textstyle{\bigoplus_{1\leq j \leq n}} P_1(\varepsilon_{s_1}-\varepsilon_{s_j})(\sigma)\cong_{\gr}V(\sigma).
\end{multline}
We denote this $A$-graded isomorphism with $\phi:W \rightarrow V(\sigma)$.
Let $e_i'$, $1\leq i \leq n$  be a (standard) homogeneous basis of degree $\varepsilon_{s_i}'$ in $W$. Then $\phi(e_i')=\sum_{1\leq j \leq n} e_j a_j$, where $a_j \in A^h$ and  $e_j$ are homogeneous of degree $\varepsilon_{s_j}-\sigma$ in $V(\sigma)$.  Since $\deg(\phi(e_i'))=\varepsilon_{s_i}'$, all $e_j$'s with nonzero $a_j$ in the sum have the same degree. For if $\varepsilon_{s_j}-\sigma=\deg(e_j)\not = \deg(e_l)=\varepsilon_{s_l}-\sigma$, then since $\deg(e_j a_j)=\deg(e_l a_l)=\varepsilon_{s_i}'$ it follows that $\varepsilon_{s_j}-\varepsilon_{s_l} \in \Ga_A$ which is a contradiction as $\Ga_A+\varepsilon_{s_j}$ and $\Ga_A+\varepsilon_{s_l}$ are distinct. Thus $\varepsilon_{s_i}'=\varepsilon_{s_j}+\tau_j-\sigma$ where $\tau_j=\deg(a_j) \in \Ga_A$. In the same manner one can show that, $\varepsilon_{s_i}'=\varepsilon_{s_{i'}}'$  in $\epsilon'$ if and only if  $\varepsilon_{s_j}$ and $\varepsilon_{s_{j'}}$ assigned to them by the previous argument are also equal. 
This shows that $\epsilon'$ can be obtained from $\epsilon$ by applying (\ref{pqow1}) and (\ref{pqow2}). Since $\epsilon'$ and $\epsilon$ are also obtained from $\gamma_1,\dots,\gamma_n$ and $\lambda_1,\dots,\lambda_n$, respectively, by applying  (\ref{pqow1}) and (\ref{pqow2}), putting these together shows that $\lambda_1,\dots,\lambda_n$ and $\gamma_1,\dots,\gamma_n$ have the similar relations, \ie   $\lambda_i=\gamma_{\pi(i)}+\tau_i+\sigma$, $1\leq i \leq n$, where  $\tau_i\in \Ga_A$ and a fixed $\sigma \in \Gamma$.
\end{proof}

A \emph{graded division algebra}\index{graded division algebra} $A$
is defined to be a graded division ring with centre $R$ such that 
$[A:R] < \infty$. Note that since $R$ is a graded field, by Propositions~\ref{gradedfree} and~\ref{kj351}, $A$ has a
well-defined dimension over $R$. A graded division algebra $A$
over its centre $R$ is said to be
\emph{unramified}\index{unramified graded division ring} if $\Ga_A = \Ga_R$ and
\emph{totally ramified}\index{totally ramified graded division ring} if $A_0 = R_0$.

Let $A$ be a graded division ring and let $R$ be a graded subfield
of $A$ which is contained in the centre of $A$. It is clear  that $R_0 =
R \cap A_0$ is a field and $A_0$ is a division ring. The group of
invertible homogeneous elements of $A$ is denoted by $A^{h*}$, which is
equal to $A^h \mi 0$. Considering $A$ as a graded $R$-module, since
$R$ is a graded field,  there is a uniquely defined dimension
$[A:R]$ by  Theorem~\ref{gradedfree}. 
The proposition below has been proven in~\cite[Chapter~5]{tignolwadsworth} for  two graded fields $R \subseteq S$ with
a torsion-free abelian grade group.

\begin{proposition}
Let $A$ be a graded division ring  and let $R$ be a graded subfield
of $A$ which is contained in the centre of $A$. Then
$$[A:R] = [A_0 : R_0] | \Ga_A : \Ga_R|.$$
\end{proposition}

\begin{proof}
Since $A$ is a graded division ring, $A_0$ is a division ring. Moreover, $R_0$ is a field.  Let $\{ x_i \}_{i \in I}$ be a basis for $A_0$ over $R_0$. Consider
the cosets of $\Ga_A$ over $\Ga_R$ and let $\{ \de_j
\}_{j \in J}$ be a coset representative, where $\de_j \in \Ga_A$. Take $\{y_j
\}_{j \in J} \subseteq A^{h*}$ such that $\deg (y_j ) = \de_j$ for
each $j$. We will show that $\{ x_i y_j \}$ is a basis for $A$ over
$F$.

Consider the map
\begin{align*}
\psi : A^{h*} & \lra \Ga_A/ \Ga_R,\\
a \; &\lmps \deg(a) + \Ga_R.
\end{align*}
This is a group homomorphism with kernel $A_0 R^{h*}$, since for any
$a \in \ker(\psi)$ there is some $r \in R^{h*}$ with $a r^{-1} \in
A_0$. For $a\in A$,  $a = \sum_{\ga \in \Ga}
a_{\ga}$, where $a_{\ga} \in A_{\ga}$ and $\psi (a_\ga) = \ga + \Ga_R
= \de_j + \Ga_F$ for some $\de_j$ in the coset representative of $\Ga_A$ over
$\Ga_R$. Then there is some $y_j$ with $\deg(y_j) = \de_j$ and
$a_\ga y_j^{-1} \in \ker (\psi)=A_0 R^{h*}$. So 
\begin{equation*}
a_{\ga} y_j^{-1} = (\sum_i  x_i r_i)     g
\end{equation*}
for $g \in R^{h*}$ and $r_i\in R_0$.  Since $R$ is in the centre of $A$, It follows 
$a_{\ga}  = \sum_i  x_iy_j r_ig$. So  $a$ can be written as an
$R$-linear combination of the elements of $\{x_i y_j \}$.

To show linear independence, suppose 
\begin{equation}\label{romoba}
\sum_{i=1}^n  x_i y_i r_i= 0,
\end{equation}
for $r_i \in R$. Write $r_i$ as the sum of its homogeneous components, and  then consider a homogeneous component of the 
sum~(\ref{romoba}), say
$\sum_{k=1}^mx_k y_k  r'_k $, where $\deg(  x_k y_k r'_k) = \al$. Since $x_k\in A_0$, 
$\deg(r'_k) + \deg(y_k) = \al$ for all $k$, so all of the $y_k$ are
the same. This implies that $\sum_k  x_k r'_k =0$, where all of the
$r'_k$ have the same degree. If $r'_k=0$ for all $k$ then $r_i=0$ for all $i$ we are done.
Otherwise, for some $r'_l \neq 0$, we have 
$\sum_k x_k ( r'_k {r'}_l^{-1} ) = 0$. Since $\{x_i\}$ forms a basis for $A_0$ over $R_0$, this
implies $r'_k = 0$ for all $k$ and thus $r_i=0$ for all $1\leq i \leq n$.
\end{proof}

\begin{example}\label{grdiviwad}\scm[Graded division algebras from valued division algebras]
\vspace{0.2cm}

Let $D$ be a division algebra with a valuation. To this 
one associates a graded division algebra 
${\gr(D)= \bigoplus_{\ga \in \Ga_D}\gr(D)_\gamma}$, where 
$\Ga_D$ is the value group of $D$ and 
the summands $\gr(D)_\gamma$ arise from the filtration on $D$ 
induced by the valuation
(see details below and also Example~\ref{filt}).   As it is illustrated in~\cite{tignolwadsworth}, even though computations in the graded setting are 
often easier than working directly with $D$, 
it seems that not  much is lost in passage from~$D$ to its 
corresponding graded division algebra $\gr(D)$.  
This has provided motivation to systematically study this 
correspondence, notably by Hwang, 
Tignol and Wadsworth~\cite {tignolwadsworth}, and to
compare certain functors defined on these objects, 
notably the Brauer group~\cite[Chapter~6]{tignolwadsworth} and the reduced Whitehead group $\SK_1$~\cite[Chapter~11]{tignolwadsworth}. We introduce this correspondence here and in~\S\ref{ggg} we calculate their graded Grothendieck groups (Example~\ref{whyso8}). 

Let $D$ be a division algebra finite dimensional over its 
centre $F$, with a \emph{valuation} \index{valuation function}
$v: D^{\ast} \ra
\Ga$. So $\Ga$ is a totally ordered abelian group,  \index{totally ordered abelian group}
and for any $a, b \in D^{\ast}$, $v$ satisfies the following conditions 
\begin{enumerate}[\upshape (i)]
\item $ v(ab) = v(a) + v(b)$;

\item $v(a+b) \geq \min \{\, v(a),v(b) \,  \}\;\;\;\;\; (b \neq -a).$
\end{enumerate}
Let 
\begin{align*}
V_D   &=   \{\,  a \in D^{\ast} : v(a) \geq 0 \, \}\cup\{0\}, 
\text{ the
valuation ring of $v$};\\ 
M_D   &=   \{ a \, \in D^{\ast} : v(a) > 0 \,
\}\cup\{0\}, \text{ the unique maximal left and right ideal
 of $V_D$}; \\
\overline{D}   &=   V_D / M_D, \text{ the residue
division ring of $v$ on $D$; and} \\
\Ga_D   &=   \mathrm{im}(v), \text{ the value
group of the valuation}. 
\end{align*}
For background on  valued division algebras, 
see~\cite[Chapter~1]{tignolwadsworth}. 
One associates to $D$ a graded division algebra 
as follows:
For each $\gamma\in \Gamma_D$, let
\begin{align*} 
 D^{\ge\ga}   &=   
\{ \, d \in D^{\ast} : v(d) \geq \ga\, \}\cup\{0\}, \text{ an additive 
subgroup of $D$ }; \qquad \qquad\qquad\qquad\qquad \ \\
D^{>\ga}   &=   \{\,  d \in D^{\ast} : v(d) > \ga \, \}\cup\{0\}, 
\text{ a subgroup 
of $D^{\ge\ga}$};   \text{ and}\\
 \gr(D)_\gamma  &= 
D^{\ge\ga}\big/D^{>\ga}. 
\end{align*}
Then define
$$
 \gr(D)   =   \textstyle\bigoplus\limits_{\ga \in \Ga_D} 
\gr(D)_\gamma. \ \ 
$$
Because $D^{>\ga}D^{\ge\de} \,+\, D^{\ge\ga}D^{>\de} 
\subseteq D^{>(\ga +
\de)}$ for all $\ga , \de \in \Ga_D$, the  multiplication on 
$\gr(D)$ induced by multiplication on $D$ is
well-defined, giving that $\gr(D)$ is a $\Gamma$\!-graded  ring, called the 
\emph{associated graded ring} \index{associated graded ring} of $D$. The 
multiplicative property 
(i) of  the valuation $v$ implies that $\gr(D)$ is a graded 
division ring.
Clearly,
we have ${\gr(D)}_0 = \overline{D}$, and $\Ga_{\gr(D)} = \Ga_D$.
For $d\in D^*$, we write $\widetilde d$ for the image 
$d + D^{>v(d)}$ of $d$ in $\gr(D)_{v(d)}$.  Thus, 
the map given by $d\mapsto \widetilde d$ is 
a group epimorphism $ D^* \rightarrow {\gr(D)^*}$ with 
kernel~$1+M_D$.  

The restriction $v|_F$ of the valuation on $D$ to its centre $F$ is
a valuation on $F$, which induces a corresponding graded field  $\gr(F)$.
Then it is clear that $\gr(D)$ is a graded $\gr(F)$-algebra, 
and one can prove that for  
$$
[\gr(D):\gr(F)]  =  
[\overline{D}:\overline{F}] \, |\Ga_D :\Ga_F|
 \ \le  \ [D:F] \ < \infty.
$$

Now let $F$ be a field with a henselian valuation $v$. Then the valuation of $F$ extends uniquely to $D$ (see~\cite[Chapter~1]{tignolwadsworth}), and with respect to this valuation, $D$ 
is said to be  \emph{tame} \index{tame valuation} if $Z(\overline D)$ is 
 separable over $\overline F$ and \[{\chr({\overline F}) \nmid 
\ind(D)\big/\big(\ind(\overline D)[Z(\overline D):\overline F]\big)}.\]
It is known (\cite[Chapter~8]{tignolwadsworth}) that 
$D$ is tame if and only if \[[\gr(D):\gr(F)] = [D:F]\] and 
$Z(\gr(D))=\gr(F)$. 

We will compute the graded Grothendieck group and the graded Picard group of these division algebras in Examples~\ref{whyso8} and~\ref{whyso9}.
\end{example}

\subsection[The zero component ring of graded simple ring]{The zero component ring of a graded central simple ring}\label{wadi}  \index{graded simple ring}

Let $A$ be a $\Ga$\!-graded division ring and $\M_n(A)(\lambda_1,\dots,\lambda_n)$ be a  graded simple ring, where $\lambda_i \in \Gamma$, $1\leq i\leq n$. \index{complete set of coset representative}
Since $A$ is a graded division ring, $\Gamma_A$ is a subgroup of $\Gamma$. Consider the quotient group $\Ga/\Ga_A$ and let $\Ga_A+\varepsilon_1,\dots,\Ga_A+\varepsilon_k$ be the distinct elements in $\Ga/\Ga_A$ representing the cosets $\Ga_A+\lambda_i$, $1\leq i\leq n$, and for each $\varepsilon_l$, let $r_l$ be the number of $i$ with $\Ga_A+\lambda_i=\Ga_A+\varepsilon_l$.  It was observed in \cite[Chapter~2]{tignolwadsworth} that 
\begin{equation}\label{urnha}
\boxed{\M_n(A)(\lambda_1,\dots,\lambda_n)_0 \cong \M_{r_1}(A_0)\times \dots \times \M_{r_k}(A_0)}
\end{equation}
and in particular, $\M_n(A)(\lambda_1,\dots,\lambda_n)_0$ is a simple ring if and only if $k=1$. Indeed, using (\ref{pqow1}) and (\ref{pqow2}) we get 
\begin{equation}\label{wdeild}
\M_n(A)(\lambda_1,\dots,\lambda_n)\cong_{\gr}\M_n(A)(\varepsilon_1,\dots,\varepsilon_1,\varepsilon_2,\dots,\varepsilon_2,\dots,\varepsilon_k,\dots,\varepsilon_k),
\end{equation}
with each $\varepsilon_l$ occurring $r_l$ times.  Now~(\ref{mmkkhh}) for $\lambda=0$ and 
\[(\delta_1,\dots,\delta_n)=(\varepsilon_1,\dots\varepsilon_1,\varepsilon_2,\dots,\varepsilon_2,\dots,\varepsilon_k,\dots,\varepsilon_k)\] immediately gives (\ref{urnha}).

\begin{remark}\scm[The graded Artin-Wedderburn structure theorem] \label{jussise}
\vspace{0.2cm}

The Artin-Wedderburn theorem shows that division rings are the basic ``building blocks'' of ring theory, \ie if a ring $A$ satisfies some finite condition, for example $A$ is right Artinian, then $A/J(A)$ is isomorphic to a finite product of matrix rings over division rings. 
A graded version of Artin-Wedderburn structure \index{graded Artin-Wedderburn Theorem} theorem also holds. We state the statement here without proof. 
We refer the reader to~\cite{grrings,tignolwadsworth} for  proofs of these statements.

A $\Gamma$\!-graded ring $B$ is isomorphic to  
$\M_n(A)(\lambda_1,\dots,\lambda_n)$,  where $A$ is a $\Gamma$\!-graded division ring and $\lambda_i \in \Gamma$, $1\leq i\leq n$, if and only if $B$ is 
\emph{graded right Artinian} \index{graded Artinian ring} (\ie a decreasing chain of graded right ideals becomes stationary) and graded simple.

A $\Gamma$\!-graded ring $B$ is isomorphic to a finite product of matrix rings overs graded division rings (with suitable shifts) if and only if $B$ is graded right Artinian and \emph{graded primitive} \index{graded primitive} (\ie $J^{\gr}(B)=0$).   \index{graded semisimple ring}

\end{remark}

\section{Strongly graded rings and Dade's theorem}\label{dadestmal}  \index{strongly graded ring}

Let $A$ be a $\Gamma$\!-graded ring and $\Omega$ be a subgroup of $\Gamma$. Recall from~\S\ref{mconfi1} that $A$ has a natural $\Gamma/\Omega$-graded structure and $A_{\Omega}=\bigoplus_{\gamma \in \Omega}A_\gamma$ is a $\Omega$-graded ring. If $A$ is a $\Gamma/\Omega$-strongly graded ring, then one can show that the category of $\Gamma$\!-graded $A$-modules, $\Gr[\Gamma] A$, is equivalent to the category of $\Omega$-graded $A_\Omega$-modules, $\Gr[\Omega] A_\Omega$. In fact, the equivalence 
\[\Gr[\Gamma] A \approx \Gr[\Omega] A_\Omega,\]
under the given natural functors (see Theorem~\ref{rhye}) implies that $A$ is a $\Gamma/\Omega$-strongly graded ring. This was first proved by Dade~\cite{dade} in the case of $\Omega=0$, \ie when  $\Gr[\Gamma] A \approx \Modd A_0$. We prove Dade's theorem (Theorem~\ref{dadesthm}) and then state this more general case in Theorem~\ref{rhye}. 

Let $A$ be a $\Gamma$\!-graded ring. For any right $A_0$-module $N$ and any $\ga
\in \Ga$, we identify the right $A_0$-module $N \otimes_{A_0} A_{\ga}$
with its image in $ N \otimes_{A_0} A$. Since $A=\bigoplus_{\gamma \in \Gamma} A_\gamma$ and $A_\gamma$ are $A_0$-bimodules, $N \otimes_{A_0} A$ is a
$\Ga$\!-graded right $A$-module, with 
\begin{equation}\label{ilcasa}
N \otimes_{A_0} A = \bigoplus_{\ga
\in \Ga} \big (N \otimes_{A_0} A_{\ga}\big ).
\end{equation}

 Consider the  \emph{restriction functor} \index{restriction functor}
\begin{empheq}[box=\widefbox]{align*}
\mathcal G:= (-)_0 : \Gr A & \lra \Modd A_0\\
M & \lmps M_0 \\
\psi & \lmps \psi|_{M_0},
\end{empheq}
and the  \emph{induction functor}  \index{induction functor} defined by
\begin{empheq}[box=\widefbox]{align*}
\mathcal I:= -\otimes_{A_0} A : \Modd A_0 & \lra \Gr A \\
N & \lmps N \otimes_{A_0} A \\
\phi & \lmps \phi \otimes \id_A .
\end{empheq}

One can easily check that 
$\G\circ  \I \cong \id_{A_0}$ with the natural transformation,
\begin{align}\label{hgy4nd1}
\G \I (N)= \G (N\otimes_{A_0} A) =N\otimes_{A_0} A_0 & \longrightarrow N,\\
n\otimes a &\mapsto na  \notag.
\end{align}

On the other hand, there is a natural transformation, 
\begin{align}\label{hgy4nd}
\I \G (M)= \I (M_0)= M_0\otimes_{A_0} A & \longrightarrow M,\\
m\otimes a &\mapsto ma  \notag.
\end{align}
 The theorem below shows that $\I \circ \G \cong \id_{A}$ (under (\ref{hgy4nd})), if and only if $A$ is a strongly graded ring. 
Theorem~\ref{dadesthm} was proved by Dade~\cite[Theorem~2.8]{dade} (see also~\cite[Theorem~3.1.1]{grrings}).

\begin{theorem}[{\sc Dade's Theorem}]  \label{dadesthm} \index{Dade's theorem}
Let $A$ be a $\Ga$\!-graded ring. Then $A$ is strongly graded if and only if 
the functors \[(-)_0:\Gr A\rightarrow \Modd A_0,\] and  \[-\otimes_{A_0} A: \Modd A_0 \rightarrow \Gr A\]  form mutually inverse
equivalences of categories.
\end{theorem}

\begin{proof}
One can easily check that (without using the assumption that $A$ is strongly graded) 
$\G\circ  \I \cong \id_{A_0}$ (see~(\ref{hgy4nd1})).  
Suppose $A$ is strongly graded. We show that $\I \circ \G \cong \id_{A}$ .

For a graded $A$-module
$M$,  we have  $\I \circ \G (M) = M_0 \otimes_{A_0} A$. We will show that the natural homomorphism  
\begin{align*}
\phi: M_0 \otimes_{A_0} A \ra M,\\
m \otimes a \mps ma,
\end{align*} is a $\Gamma$\!-graded
$A$-module isomorphism.  The map $\phi$ is clearly graded (see~(\ref{ilcasa})).  Since $A$ is
strongly graded, it follows that for $\ga, \de \in \Ga$,
\begin{equation}\label{robin}
M_{\ga +\de} =  M_{\ga+ \de} A_0 =  M_{\ga +\de} A_{-\ga}A_{\ga}  \subseteq 
 M_{\de}  A_{\ga} \subseteq M_{\ga +\de}.
\end{equation}
Thus $ M_{\de} A_{\ga} = M_{\ga +\de}$. Therefore, $\phi
( M_0 \otimes_{A_0} A_{\ga} ) =  M_0 A_{\ga}= M_{\ga}$, which implies that $\phi$ is
surjective.

Let $N= \ker (\phi)$, which is a graded $A$-submodule of $M_0
\otimes_{A_0} A$, so $N_0 = N \cap (M_0 \otimes_{A_0} A_0$). However the restriction of $\phi$ to 
$M_0 \otimes_{A_0} A_0 \ra M_0$ is
the canonical isomorphism, so $N_0 = 0$. Since $N$ is a graded
$A$-module, a similar argument as~(\ref{robin})  shows $N_\ga = N_0 A_\ga  = 0$ for all $\ga
\in \Ga$. It follows that $\phi$ is injective. Thus $\I \circ \G (M)= M_0 \otimes_{A_0} A\cong M$. Since all the homomorphisms involved are natural, this shows that 
$\I \circ \G \cong \id_{A}$. 

For the converse, suppose $\I$ and $\G$ are mutually inverse (under (\ref{hgy4nd}) and (\ref{hgy4nd1})). For any graded $A$-module $M$, $\I \circ \G (M) \conggr M$, gives that the map 
\begin{align*}
M_0\otimes_ {A_0} A_\alpha & \longrightarrow M_\alpha,\\
m \otimes a &\mapsto ma 
\end{align*}
is bijective, where $\alpha \in \Gamma$.  This immediately implies 
\begin{equation}\label{pohalr}
M_0 A_{\alpha}=M_{\alpha}.
\end{equation}
 Now for any $\beta \in \Gamma$, consider the graded $A$-module $A(\beta)$. Replacing $M$ by $A(\beta)$ in~(\ref{pohalr}), we get 
 $A(\beta)_0 A_{\alpha} = A(\beta)_{\alpha}$, \ie $A_{\beta} A_{\alpha}=A_{\beta+\alpha}$. This shows that $A$ is strongly graded. 
\end{proof}

\begin{corollary}\label{dadesthm3}
Let $A$ be a $\Ga$\!-graded ring and $\Omega$ a subgroup of $\Gamma$ such that $A$ is a $\Gamma/\Omega$-strongly graded ring.  Then  
the functors 
\[(-)_0:\Gr[\Gamma/\Omega] A\longrightarrow \Modd A_\Omega,\] and  
\[-\otimes_{A_\Omega} A: \Modd A_\Omega \longrightarrow \Gr[\Gamma/\Omega] A\]  form mutually inverse
equivalences of categories.
\end{corollary}
\begin{proof}
The result follows from Theorem~\ref{dadesthm}.
\end{proof}

\begin{remark}\label{himifei}
Recall that $\grr A$ denotes the category of graded finitely generated  right $A$-modules and 
$\Pgrp A$ denotes the category of graded finitely generated  projective right $A$-modules.
Note that in general the restriction functor $(-)_0:\Gr A \rightarrow \Modd A_0$ does not induce a functor $(-)_0:\Pgrp A \rightarrow \Prr A_0$. In fact, one can easily produce a graded finitely generated projective $A$-module $P$ such that $P_0$ is not projective $A_0$-module. As an example, consider the $\Z$-graded ring $T$ of Example~\ref{meinmein}. Then $T(1)$ is clearly a graded finitely generated projective $T$-module. However $T(1)_0=M$ is not $T_0=R$-module. 
\end{remark}

\begin{remark}\label{cafejen}
The proof of Theorem~\ref{dadesthm} also shows that $A$ is strongly graded if and only if $\grr A\cong \modd A_0$, if and only if $\Pgrp A \cong \Prr A_0$, (see Remark~\ref{himifei}) via the same functors $(-)_0$ and $-\otimes_{A_0} A$ of the Theorem~\ref{dadesthm}. 
\end{remark}

\begin{remark}\scm[Strongly graded modules]
\vspace{0.2cm}

Let $A$ be a $\Gamma$\!-graded ring and  $M$ be a graded $A$-module. Then $M$ is called \emph{strongly graded} \index{strongly graded module} $A$-module if 
\begin{equation}\label{jenroom}
M_\alpha A_\beta=M_{\alpha+\beta},
\end{equation} for any $\alpha,\beta \in \Gamma$. The proof of Theorem~\ref{dadesthm} shows that $A$ is strongly graded if and only if any graded $A$-module is strongly graded. Indeed if $A$ is strongly graded then~(\ref{robin}) shows that any graded $A$-module is strongly graded. Conversely, if any graded module is strongly graded, then considering $A$ as a graded $A$-module, (\ref{jenroom}) for $M=A$, shows that $A_\alpha A_\beta=A_{\alpha+\beta}$ for any $\alpha,\beta \in \Gamma$.
\end{remark}

\begin{remark}\scm[Ideals correspondence between $A_0$ and $A$]\label{miener18}
\vspace{0.2cm}

The proof of Theorem~\ref{dadesthm} shows that there is a one-to-one correspondence between the right ideals of $A_0$ and the graded right ideals of $A$ (similarly for the left ideals). However, this correspondence does not hold between two-sided ideals. As an example,  
$A=\M_2(K[x^2,x^{-2}])(0,1)$, where $K$ is a field, is a strongly $\mathbb Z$-graded simple ring, whereas $A_0\cong K\otimes K$ is not a simple ring (see~\S\ref{wadi}. Also see Proposition~\ref{balmainmine} for a relation between simplicity of $A_0$ and $A$). 

In the same way, the equivalence $\Gr A \approx \Modd A_0$ of Theorem~\ref{dadesthm} gives a correspondence  between several (one-sided) properties of graded objects in $A$ with objects over $A_0$. For example, one can easily show that $A$ is graded right (left) Noetherian if and only if $A_0$ is right (left) Noetherian (see also Corollary~\ref{vonregu}).  \index{Noetherian ring} \index{simple ring} \index{graded simple ring} \index{graded Noetherian ring}
\end{remark}

Using Theorem~\ref{dadesthm}, we will see that the graded Grothendieck group of a strongly graded ring coincides with the (classical) Grothendieck group of its 0-component ring (see~\S\ref{jijigogo}).

We need a more general version of grading defined in~(\ref{ilcasa}) in order to extend Dade's theorem. Let $A$ be a $\Gamma$\!-graded ring and $\Omega$ a subgroup of $\Gamma$. Let $N$ be a  $\Omega$-graded right $A_\Omega$-module. Then $N\otimes_{A_\Omega} A$ is a $\Gamma$\!-graded right $A$-module, with the grading defined by 
\begin{equation*}
(N \otimes_{A_\Omega}  A)_{\ga} = \Big \{ \sum_i n_i \otimes a_i   \mid   n_i \in
N^h, a_i \in A^h, \deg(n_i)+\deg(a_i) = \ga \, \Big\}.
\end{equation*}
A similar argument as in~\S\ref{grtensie} for tensor products, will show that this grading is well-defined. Note that with this grading, 
\begin{equation*}
(N \otimes_{A_\Omega}  A)_{\Omega} = N \otimes_{A_\Omega}  A_\Omega \cong N,
\end{equation*}
as graded right $A_\Omega$-modules. 

\begin{theorem}\label{rhye}
Let $A$ be a $\Gamma$\!-graded ring and $\Omega$ be a subgroup of \ $\Gamma$. Consider $A$ as a $\Gamma/\Omega$-graded ring. Then $A$ is 
a $\Gamma/\Omega$-strongly graded ring if and only if 
 \begin{align*}
(-)_\Omega : \Gr[\Gamma] A & \lra \Gr[\Omega] A_\Omega \\
M & \lmps M_\Omega \\
\psi & \lmps \psi|_{M_\Omega},
\end{align*}
and 
\begin{align*}
-\otimes_{A_0} A : \Gr[\Omega] A_\Omega & \lra \Gr[\Gamma] A \\
N & \lmps N \otimes_{A_\Omega} A \\
\phi & \lmps \phi \otimes \id_A,
\end{align*}
form mutually inverse equivalences of categories.
\end{theorem}
\begin{proof}
The proof is similar to the proof of Theorem~\ref{dadesthm} and it is omitted. 
\end{proof}

\begin{remark}\label{rhye1}
Compare Theorem~\ref{rhye}, with the following statement. Let $A$ be a $\Gamma$\!-graded ring and $\Omega$ be a subgroup of $\Gamma$. Then $A$ is $\Gamma$\!-strongly graded ring if and only if 
 \begin{align*}
(-)_\Omega : \Gr[\Gamma] A & \lra \Gr[\Omega] A_\Omega \\
M & \lmps M_\Omega \\
\psi & \lmps \psi|_{M_\Omega},
\end{align*}
and 
\begin{align*}
-\otimes_{A_0} A : \Gr[\Omega] A_\Omega & \lra \Gr[\Gamma] A \\
N & \lmps N_0 \otimes_{A_0} A \\
\phi & \lmps \phi_0 \otimes \id_A,
\end{align*}
form mutually inverse equivalences of categories.
\end{remark}

\begin{example}\label{yhyhew45}
Let $A$ be a $\Gamma\times \Omega$-graded ring such that $1\in A_{(\alpha,\Omega)} A_{(-\alpha,\Omega)}$ for any $\alpha \in \Gamma$, where $A_{(\alpha,\Omega)}=\bigoplus_{\omega \in \Omega} A_{(\alpha,\omega)}$. Then by Theorem~\ref{rhye} 
\[\Gr[\Gamma\times \Omega] A  \approx \Gr[\Omega] A_{(0,\Omega)}.\]
This example will be used in~\S\ref{vandengg}. Compare this also with Corollary~\ref{zuhoingding}. 
\end{example}

Another application of Theorem~\ref{dadesthm} is to provide a condition when a strongly graded ring is a graded von Neumann ring (\S\ref{penrith7may}). 
This will be used later in Coroallry~\ref{hgthusu2} to show that the Leavitt path algebras are von Neumann regular rings.  \index{von Neumann regular ring}

\begin{corollary}\label{vonregu}
Let $A$ be a strongly graded ring. Then $A$ a is graded von Neumann regular ring if and only if $A_0$ is a von Neumann regular ring. 
\end{corollary}
\begin{proof}[Sketch of proof]
Since any (graded) flat module is a direct limit of (graded) projective modules, from the equivalence of categories $\Gr A \approx_{\gr} \Modd A_0$ (Theorem~\ref{dadesthm}), it follows that  $A$ is graded von Neumann regular if and only if $A_0$ is von Neumann regular.
\end{proof}

\begin{remark}
An element-wise proof of Corollary~\ref{vonregu} can also be found in~\cite[Theorem~3]{yahya}.
\end{remark}

For a $\Gamma$\!-graded ring $A$, and $\alpha, \beta \in \Gamma$, one has a $A_0$-bimodule homomorphism 
\begin{align}\label{bondstsi}
\phi_{\alpha,\beta}: A_\alpha \otimes_{A_0} A_{\beta} & \longrightarrow A_{\alpha+\beta}\\
a\otimes b &\longmapsto ab. \notag
\end{align}
The following theorem gives another characterisation for  strongly graded rings.

\begin{theorem}\label{mozsace}
Let $A$ be a $\Gamma$\!-graded ring. Then $A$ is a strongly graded ring if and only if for any $\gamma \in \Gamma$, the homomorphism 
\begin{align*}
\phi_{\gamma,-\gamma}: A_\gamma \otimes_{A_0} A_{-\gamma} &\longrightarrow A_0,\\ 
a\otimes b &\longmapsto ab,
\end{align*}
is an isomorphism.  In particular,  if $A$ is strongly graded, then the homogeneous components $A_\gamma$, $\gamma \in \Gamma$,  are finitely generated projective $A_0$-module.

\end{theorem}
\begin{proof}
Suppose that for any $\ga \in \Gamma$, the map $\phi_{\gamma,-\gamma}: A_\gamma \otimes A_{-\gamma} \rightarrow A_0$ is an isomorphism. Thus there are $a_i \in A_\gamma$, $b_i \in A_{-\gamma}$ such that 
\[\sum_i a_i b_i = \phi_{\gamma,-\gamma}\Big (\sum_i a_i \otimes b_i\Big)=1.\] So $1\in A_\gamma A_{-\gamma}$. Now by Proposition~\ref{crossedproductstronglygradedprop}(1) $A$ is strongly graded. 

Conversely, suppose $A$ is a strongly graded ring. We prove that the homomorphism~(\ref{bondstsi}) is an isomorphism. 
The definition of strongly graded implies that $\phi_{\alpha,\beta}$ is surjective. Suppose  
\begin{equation}
\phi_{\alpha,\beta}\Big (\sum_i a_i \otimes b_i \Big)=\sum_i a_i b_i =0. 
\end{equation}
Using Proposition~\ref{crossedproductstronglygradedprop}(1), write $1=\sum_j x_j y_j$, where $x_j \in A_{-\beta}$ and $y_j \in A_{\beta}$. Then
\begin{multline*}
\sum_ia_i \otimes b_i=\Big(\sum_ia_i \otimes b_i\Big)\, \Big(\sum_j x_j y_j\Big)=\sum_i \Big (a_i \otimes \sum_j b_i x_j y_j\Big) =\\
\sum_i \Big(\sum_j (a_ib_i x_j \otimes y_j)\Big)=\sum_j \sum_i \Big(a_ib_i x_j \otimes y_j\Big)=\sum_j \Big (\sum_i (a_ib_i)x_j \otimes y_j)\Big)=0. 
\end{multline*}
This shows that $\phi_{\alpha,\beta}$ is injective. Now setting $\alpha=\ga$ and $\beta=-\ga$  finishes the proof. 

Finally, if $A$ is strongly graded, the above argument shows that the homogeneous components $A_\gamma$, $\gamma \in \Gamma$, are invertible $A_0$-module, which in turn implies $A_\alpha$ are finitely generated projective $A_0$-module.
\end{proof}

\subsection{Invertible components of strongly graded rings}\label{onon123}

Let $A$ and $B$ be rings and $P$ be a $A\!-\!B$-bimodule. Then $P$ is called an \emph{invertible}\index{invertible module} $A\!-\!B$-bimodule, if there is a $B\!-\!A$-bimodule $Q$ such that $P\otimes_B Q\cong A$ as $A\!-\!A$-bimodules and $Q\otimes_A P \cong B$ 
 as $B\!-\!B$-bimodules and the following diagrams are commutative.  

\begin{equation*}
{\def\labelstyle{\displaystyle}
\xymatrix{
P\otimes_B Q\otimes_A P \ar[r] \ar[d]& A \otimes_A P \ar[d]\\
P\otimes_B B \ar[r] & P
}}\qquad 
{\def\labelstyle{\displaystyle}
\xymatrix{
Q\otimes_A P\otimes_B Q \ar[r] \ar[d]& B \otimes_B Q \ar[d]\\
Q\otimes_A A \ar[r] & Q
}}
\end{equation*}

One can prove that $P$ is a finitely generated projective $A$ and $B$-modules. 

Now Theorem~\ref{mozsace} shows that for a strongly $\Gamma$\!-graded ring $A$, the $A_0$-bimodules $A_\ga$, $\ga \in \Gamma$, is an invertible module   and thus is a finitely generated projective $A_0$-module. This in return implies that $A$ is a projective $A_0$-module. Note that in general, one can easily construct a graded ring $A$ where $A$ is not projective over $A_0$ (see Example~\ref{meinmein}) and $A_\gamma$ is not finitely generated $A_0$-module, such as  the $\mathbb Z$-graded ring $\mathbb Z[x_i \mid i \in \mathbb N]$ of Example~\ref{penrith123}. 

\begin{remark}\scm[Other terminologies for strongly graded rings]
\vspace{0.2cm}

The term ``strongly graded'' for such rings was coined by E. Dade in ~\cite{dade} which is now commonly in use. Other terms for these rings are \emph{fully graded} \index{fully graded ring} and \emph{generalised crossed products}. \index{generalised crossed product} See~\cite{dade2} for a history of development of such rings in literature. 
\end{remark}

\section{Grading on graph algebras} \label{creekside}

\subsection{Grading on free rings}\label{freeme} \index{free ring} \index{$R(X)$, free ring}

Let $X$ be a nonempty set of symbols and $\Gamma$ be a group. (As always we assume the groups are abelian, although the entire theory can be written for an arbitrary group.)
Let $d:X\rightarrow \Gamma$ be a map. One can extend 
$d$ in a natural way to a map from the set of finite words on $X$ to $\Gamma$, which is called $d$ again. For example if $x,y,z\in X$ and $xyz$ is a word, then $d(xyz)=d(x)+d(y)+d(z)$. One can easily see that if $w_1,w_2$ are two words, then $d(w_1w_2)=d(w_1)+d(w_2)$. If we allow an empty word, which will be the identity element in the free ring, then we assign the identity of $\Gamma$ to this word. 

Let $R$ be a ring and $R(X)$ be the free ring (with or without identity) on a set $X$ with coefficients in $R$. The elements of $R(X)$ are of the form
$\sum_w r_w w$, where $r_w\in R$ and $w$ stands for a word on $X$. The multiplication is defined by convolution, \ie 
\[\Big (\sum_w r_w w\Big )\, \Big (\sum_v r_v v \Big)=\sum_z \Big(\sum_{\{w,v\mid z=wv, \, r_w,r_v\not =0 \}} r_wr_v\Big) \, z.\]

In order to make $R(X)$ into a graded ring, define \[R(X)_{\gamma}=\Big \{ \, \sum_w r_w w \mid d(w)=\gamma \, \Big \}.\] One can check that 
$R(X)=\bigoplus_{\gamma\in \Gamma} R(X)_{\gamma}$.  Thus $R(X)$ is a $\Gamma$\!-graded ring. Note that if we don't allow the empty word in the construction, then $R(X)$ is a graded ring without identity (see Remark~\ref{parknan}). It is easy to see that $R(X)$ is never  a strongly graded ring. 

\begin{example}\label{july26y}
Let $R$ be a ring and $R(X)$ be the free ring on a set $X$ with a graded structure induced by a map $d:X\rightarrow \Gamma$. Let $\Omega$ be a subgroup of  $\Gamma$ and consider the map 
\begin{align*}
\ol d:X &\longrightarrow \Gamma/\Omega,\\
 x& \longmapsto \Omega+d(x).
\end{align*}
The map $\ol d$ induces a $\Gamma/\Omega$-graded structure on $R(X)$ which coincides with the general construction of quotient grading 
given in~\S\ref{mconfi1}. 
\end{example}

\begin{example}
Let $X=\{x\}$ be a set of symbols with one element and $\mathbb Z_n$ be the cyclic group with $n$ elements.  Assign $1\in \Z_n$ to $x$ and generate the free ring with identity on $X$ with coefficients in a field $F$.  This ring is the usual polynomial ring $F[x]$ which, by the above construction,  is equipped by  a $\Z_n$-grading. Namely, \[F[x]=\bigoplus_{k \in \Z_n} \Big( \sum_{\substack{l\in \mathbb N, \\ \overline l =k}} Fx^l \Big ),\] where $\overline l$ is the image of $l$ in the group $\Z_n$.  For $a\in F$, since the 
polynomial $x^n-a$ is a homogeneous element of degree zero, the ideal $\langle x^n-a \rangle$ is a graded ideal and thus the quotient ring 
$F[x]/\langle x^n-a \rangle$ is also a $\Z_n$-graded ring (see~\S\ref{jghjrye}). In particular if $x^n-a$ is an irreducible polynomial in $F[x]$, then the field $F[x]/\langle x^n-a \rangle$ is a $\Z_n$-graded field as well. 
\end{example}

\begin{example}\label{pearlfish}
Let $\{x,y\}$ be a set of symbols. Assign $1\in \mathbb Z_2$ to $x$ and $y$ and consider the graded free ring $\mathbb R(x,y)$. The ideal generated by homogeneous elements $\{x^2+1,y^2+1,xy+yx\}$ is graded and thus we retrieve the $\mathbb Z_2$-graded Hamilton quaternion algebra of Example~\ref{egofgrdivisionrings}  as follows: 
\[\mathbb H \cong \mathbb R(x,y)/ \langle x^2+1,y^2+1,xy+yx \rangle. \]
Moreover, assigning $(1,0)\in \mathbb Z_2\times \mathbb Z_2$ to $x$ and $(0,1)\in \mathbb Z_2\times \mathbb Z_2$ to $y$ we obtained the $\mathbb Z_2\times \mathbb Z_2$-graded quaternion algebra of  Example~\ref{egofgrdivisionrings}. 
\end{example}

\begin{example}\scm[The Weyl algebra] \label{weyl}
\vspace{0.2cm}

For a (commutative) ring $R$, the \emph{Weyl algebra} \index{Weyl algebra} $R(x,y)/\langle xy-yx-1 \rangle$ can be considered as a $\mathbb Z$-graded ring by assigning $1$ to $x$ and $-1$ to $y$. 
\end{example}

\begin{example}\scm[The Leavitt algebra $\LL(n,k+1)$] \label{levisuji2}
\vspace{0.2cm}

Let $K$ be a field, $n$ and $k$ positive integers and $A$ be the free associative $K$-algebra with identity generated by symbols $\{x_{ij},y_{ji} \mid 1\leq i \leq n+k, 1\leq j \leq n \}$ subject to relations (coming from) 
\[ Y\cdot X=I_{n,n} \qquad \text{ and } \qquad X\cdot Y=I_{n+k,n+k}, \] where 
\begin{equation} \label{breaktr}
Y=\left( 
\begin{matrix} 
y_{11} & y_{12} & \dots & y_{1,n+k}\\ 
y_{21} & y_{22} & \dots & y_{2,n+k}\\ 
\vdots & \vdots & \ddots & \vdots\\ 
y_{n,1} & y_{n,2} & \dots & y_{n,n+k} 
\end{matrix} 
\right), ~~
X=\left( 
\begin{matrix} 
x_{11\phantom{+k}} & x_{12\phantom{+k}} & \dots & x_{1,n\phantom{+k}}\\ 
x_{21\phantom{+k}} & x_{22\phantom{+k}} & \dots & x_{2,n\phantom{+k}}\\ 
\vdots\phantom{+k} & \vdots\phantom{+k} & \ddots & \vdots\phantom{+k}\\ 
x_{n+k,1} & x_{n+k,2} & \dots & x_{n+k,n} 
\end{matrix} 
\right). 
\end{equation} 
To be concrete, the relations are 
\begin{align*}
\sum_{j=1}^{n+k}y_{ij}x_{jl}&=\delta_{i,l}, \qquad 1\leq i,l\leq n\\
\sum_{j=1}^{n}x_{ij}y_{jl}&=\delta_{i,l}, \qquad 1\leq i,l\leq n+k.\\
\end{align*}

In Example~\ref{levisuji} we studied a special case of this algebra when $n=1$ and $k=n-1$.
This algebra was studied by Leavitt in relation with its type in~\cite[p.130]{vitt62}  which is shown that for arbitrary $n$ and $k$ the algebra is of type 
$(n,k)$ (see~\S\ref{gtr5654}) and when $n\geq 2$ they are domains. 
We denote this algebra by $\LL(n,k+1)$. (Cohn's notation in~\cite{cohn11} for this algebra is $V_{n,n+k}$.)

Assigning 
\begin{align*}
\deg(y_{ji})&=(0,\dots,0,1,0\dots,0),\\
\deg(x_{ij})&=(0,\dots,0,-1,0\dots,0),
\end{align*}
 for  $1\leq i \leq n+k$, $1\leq j \leq n$, in $\bigoplus_n \mathbb Z$, where $1$ and $-1$ are in $j$-th entries respectively, makes the free algebra generated by $x_{ij}$ and $y_{ji}$ a graded ring. Moreover, one can easily observe that the relations coming from~(\ref{breaktr}) are all homogeneous with respect to this grading, so that the Leavitt algebra $\LL(n,k+1)$  is a  $\bigoplus_n \mathbb Z$-graded ring. In particular, $\LL(1,k)$ is a $\mathbb Z$-graded ring (Example~\ref{levisuji}).
\end{example}

\subsection{Corner skew Laurent polynomial rings}\label{cornerskew}

Let $R$ be a ring with identity and $p$ an idempotent of $R$. Let $\phi:R\rightarrow pRp$ be a \emph{corner} isomorphism, \index{corner isomorphism} i.e, a ring isomorphism such that $\phi(1)=p$. A \emph{corner skew Laurent polynomial ring}  \index{corner skew Laurent polynomial ring} with coefficients in $R$, denoted by $R[t_{+},t_{-},\phi]$, is a unital ring which is constructed as follows:  The elements of $R[t_{+},t_{-},\phi]$ are formal expressions
\[t^j_{-}r_{-j} +t^{j-1}_{-}r_{-j+1}+\dots+t_{-}r_{-1}+r_0 +r_1t_{+}+\dots +r_it^i_{+},\]
where $r_{-n} \in p_n R$ and $r_n \in R p_n$, for all $n\geq 0$, where $p_0 =1$ and $p_n =\phi^n(p_0)$. The addition is component-wise, and the multiplication is determined by the distribution law and the following rules:
\begin{equation}\label{oiy53} 
t_{-}t_{+} =1, \qquad t_{+}t_{-} =p, \qquad rt_{-} =t_{-}\phi(r),\qquad  t_{+}r=\phi(r)t_{+}.
\end{equation}

The corner skew Laurent polynomial rings are studied in~\cite{arabrucom}, where their $K_1$-groups are calculated. This construction is a special case of a so called fractional skew monoid rings constructed in~\cite{arafrac}. Assigning $-1$ to $t_{-}$ and $1$ to $t_{+}$ makes $A:=R[t_{+},t_{-},\phi]$ a $\mathbb Z$-graded ring with $A=\bigoplus_{i\in \mathbb Z}A_i$, where 
\begin{align*}
A_i& = Rp_it^i_{+}, \text{ for  } i>0,\\
A_i&=t^i_{-}p_{-i}R, \text{ for } i<0,\\
A_0& =R,
\end{align*}
(see~\cite[Proposition~1.6]{arafrac}). 
 Clearly, when $p=1$ and $\phi$ is the identity map, then $R[t_{+},t_{-},\phi]$ reduces to the familiar ring $R[t,t^{-1}]$. 

In the next three propositions we will characterise those corner skew Laurent polynomials which are strongly graded rings (\S\ref{scrosshg}), crossed products (\S\ref{khgfewa1}) and graded von Neumann regular rings (\S\ref{penrith7may}).

Recall that an idempotent element $p$ of the ring $R$ is called a \emph{full idempotent} \index{full idempotent} if $RpR=R$.

\begin{proposition}\label{lanhc888}
Let $R$ be a ring with identity and $A=R[t_{+},t_{-},\phi]$  a corner skew Laurent polynomial ring. Then $A$ is strongly graded if and only if $\phi(1)$ is a full idempotent. 
\end{proposition}
\begin{proof}
First note that $A_1=R\phi(1)t_{+}$ and $A_{-1}=t_{-}\phi(1)R$. Moreover, since $\phi(1)=p$, we have  
\begin{equation*}
r_1\phi(1)t_{+} t_{-}\phi(1)r_2=r_1\phi(1)p\phi(1)r_2=r_1pppr_2=r_1\phi(1)r_2.
\end{equation*}
Suppose $A$ is strongly graded. Then $1\in A_{1}A_{-1}$. That is 
\begin{equation}\label{ghgfri6}
1=\sum_i\Big(r_i\phi(1)t_{+}\Big)\, \Big(t_{-}\phi(1)r'_i\Big)=\sum_i r_i\phi(1)r'_i,
\end{equation}
where $r_i,r'_i\in R$. So $R\phi(1)R=R$, that is $\phi(1)$ is a full idempotent. 

On the other hand suppose $\phi(1)$ is a full idempotent. Since $\mathbb Z$ is generated by $1$, in order to prove that $A$ is strongly graded, it is enough to show that $1\in A_1A_{-1}$ and $1\in A_{-1}A_1$ (see~\S\ref{scrosshg}). But 
\[t_{-}\phi(1) \phi(1)t_{+}=t_{-1}\phi(1)t_{+}=1t_{-}t_{+}=1,\] shows that $1\in A_{-1}A_1$. Since $\phi(1)$ is a full idempotent, there are $r_i,r'_i\in R$, $i\in I$ such that 
$\sum r_i\phi(1)r'_i=1$. Then Equation~\ref{ghgfri6} shows that $1\in A_1A_{-1}$.  
\end{proof}

Recall that a ring $R$ is called \emph{Dedekind finite} if any one sided invertible element is two-sided invertible. Namely, if $ab=1$, then $ba=1$, where $a,b\in R$. \index{Dedekind finite ring} For example, left (right) Noetherian rings are Dedekind finite.

\begin{proposition}\label{dungziflunch}
Let $R$ be a ring with identity which is Dedekind finite and $A=R[t_{+},t_{-},\phi]$ a corner skew Laurent polynomial ring. Then $A$ is crossed product if and only if $\phi(1)=1$.  \index{crossed product ring}
\end{proposition}
\begin{proof}
If $\phi(1)=1$, then from relations~(\ref{oiy53}) follows that $t_{-}t_{+}=t_{+}t_{-}=1$. Therefore all homogeneous components contain invertible elements and thus $A$ is crossed product. 

Suppose $A$ is crossed product. Then there are $a,b\in R$ such that $(t_{-}a)(bt_{+})=1$ and $(bt_{+})(t_{-}a)=1$. Using relations~(\ref{oiy53}), the first equality gives $ab=p$ and the second one gives $bpa=1$, where $\phi(1)=p$. Now 
\[1=bpa=bppa=babpa=ba.\]
Since $R$ is Dedekind finite, it follows $ab=1$ and thus $p=\phi(1)=1$. 
\end{proof}

\begin{proposition}\label{lanhc8} \index{von Neumann regular ring}
Let $R$ be a ring with identity and $A=R[t_{+},t_{-},\phi]$ a corner skew Laurent polynomial ring. Then $A$ is a graded von Neumann regular ring if and only if $R$ is a von Neumann regular ring. 
\end{proposition}
\begin{proof}
If a graded ring is graded von Neumann regular, then it is easy to see that its zero component ring is von Neumann regular. This proves one direction of the theorem. For the converse, suppose $R$ is regular.  Let $x \in A_i$, where $i>0$. So $x=rp_it_{+}^i$, for some $r\in R$, where $p_i=\phi^i(1)$.  By relations (\ref{oiy53}) and induction, we have $t_{+}^i t_{-}^i=\phi^i(p_0)=p_i$. Since $R$ is regular, there is an $s\in R$ such that $rp_i s rp_i =rp_i$. Then choosing 
$y=t^i_{-}p_is$, we have 
\[xyx=(rp_it_{+}^i)(t^i_{-}p_is)(rp_it_{+}^i)=(rp_it_{+}^i t^i_{-}p_is)(rp_it_{+}^i)=  rp_i p_i p_i s rp_it_{+}^i= rp_i t_{+}^i=x.\]
A similar argument shows that for $x \in A_i$, where $i<0$, there is a $y$ such that $xyx=x$. This shows that $A$ is a graded von Neumann regular ring. 
\end{proof}

Note that in a corner skew Laurent polynomial ring $R[t_{+},t_{-},\phi]$, $t_{+}$ is a left invertible element with a right inverse  $t_{-}$ (see the relations~(\ref{oiy53})).  
In fact this property characterises such rings. Namely, a graded ring $A=\bigoplus_{i\in \mathbb Z}A_i$ such that $A_1$ has a left invertible element is a corner skew Laurent polynomial ring as the following theorem shows. The following theorem (first established in~\cite{arafrac}) will be used to realise Leavitt path algebras (\S\ref{paohdme}) as corner skew Laurent polynomial rings (Example~\ref{berlinbielefeld}). 

\begin{theorem}\label{jhby67}
Let $A$ be a $\mathbb Z$-graded ring which has  a left invertible element $t_{+} \in A_1$. Then $t_{+}$ has a right inverse $t_{-}\in A_{-1}$, and $A=A_0[t_{+},t_{-},\phi]$, where 
\begin{align}\label{liyang}
\phi:A_0 &\longrightarrow t_{+}t_{-}A_0 t_{+}t_{-},\\
 a &\longmapsto t_{+}at_{-}.\notag
\end{align} 
 
\end{theorem}
\begin{proof}
Since $t_{+}$ has a right inverse, it follows easily that there is a $t_{-}\in A_{-1}$ with $t_{-}t_{+}=1$. Moreover $t_{+}t_{-}=t_{+}t_{-} t_{+}t_{-}$ is a homogeneous idempotent of degree zero. Observe that the map~\ref{liyang}
  is a (unital) ring 
isomorphism. Consider the corner skew Laurent polynomial ring 
$\widetilde{A}=A_0[\widetilde t_{+},\widetilde t_{-},\phi]$. Since $\phi(a)=t_{+}a t_{-}$, it follows that $t_{-}\phi(a)=at_{-}$ and 
$\phi(a)t_{+}=t_{+}a$. Thus $t_{+}$ and $t_{-}$ satisfy all the relations in (\ref{oiy53}). Therefore there is a well-defined map $\psi:\widetilde A\rightarrow A$, such that $\psi(\widetilde t_{\pm})=t_{\pm}$ and the restriction of $\psi$ on $A_0$ is the identity and 
\[ \psi\Big(\sum_{k=1}^j \widetilde{t_{-}^k} a_{-k}+a_0+\sum_{k=1}^ia_i \widetilde{{t}_{+}^i}\Big)=\sum_{k=1}^j t_{-}^k a_{-k}+a_0+\sum_{k=1}^ia_i t_{+}^i.\] This also shows that $\psi$ is a graded homomorphism. In order to show that $\psi$ is an isomorphism, it suffices to show that  its restriction to each homogeneous component $\psi:\widetilde A_i \rightarrow A_i$ is a bijection.  Suppose $x\in \widetilde A_i$, $i>0$ such that $\psi(x)=0$. Then $x=d\widetilde {t_{+}^i}$ for some $d\in A_0p_i$ where $p_i=\phi^i(1)$ and $\psi(x)=d t_{+}^i$. Note that $\phi^i(1)= t_{+}^i t_{-}^i$. Thus $d\phi^i(1)=dt_{+}^i t_{-}^i=\psi(x) t_{-}^i=0$. It now follows $x=d\widetilde{t_{+}^i}=d\phi^i(1)\widetilde{t_{+}^i}=0$ in $\widetilde A_i$. This shows $\psi$ is injective. Suppose $y \in A_i$. Then $yt^i_{-}\in A_0$ and $yt^i_{-} t_{+}^i t_{-}^i=yt^i_{-}\phi^i(1) \in A_0\phi^i(1)=A_0p_i$. This shows $yt^i_{-} t_{+}^i t_{-}^i \widetilde{t_{+}^i} \in \widetilde A_i$. But $\psi(yt^i_{-} t_{+}^i t_{-}^i \widetilde{t_{+}^i})=yt^i_{-} t_{+}^i t_{-}^i  t_{+}^i=y$. This shows that $\psi:\widetilde A_i \rightarrow A_i$, $i>0$  is a bijection. A similar argument can be written for the case of $i<0$. The case $i=0$ is obvious. This completes the proof.
\end{proof}

\subsection{Graphs}\label{paohdme3}

In this subsection we gather some  graph-theoretic definitions which is needed for the construction of path algebras in~\S\ref{paohdme}. 

A \emph{directed graph} \index{directed graph} \index{graph} $E=(E^0,E^1,r,s)$ consists of two countable sets $E^0$, $E^1$ and maps $r,s:E^1\rightarrow E^0$. The elements of $E^0$ are called \emph{vertices} \index{vertex of a graph} and the elements of $E^1$ \emph{edges}. \index{edge of a graph} If $s^{-1}(v)$ is a finite set for every $v \in E^0$, then the graph is called \emph{row-finite}. \index{row-finite graph} In this book we will only consider row-finite graphs. In this setting, if the number of vertices, \ie  $|E^0|$,  is finite, then the number of edges, \ie  $|E^1|$, is finite as well and we call $E$ a \emph{finite} graph. \index{finite graph} 

 For a graph $E=(E^0,E^1,r,s)$, a vertex $v$ for which $s^{-1}(v)$ is empty is called a \emph{sink}, \index{sink vertex} while a vertex $w$ for which $r^{-1}(w)$ is empty is called a \emph{source}. \index{source vertex} An edge with the same source and range is called a \emph{loop}. \index{loop edge} A path $\mu$ in a graph $E$ is a sequence of edges $\mu=\mu_1\dots\mu_k$, such that $r(\mu_i)=s(\mu_{i+1}), 1\leq i \leq k-1$. In this case, $s(\mu):=s(\mu_1)$ is the \emph{source} \index{source of an edge} of $\mu$, $r(\mu):=r(\mu_k)$ is the \emph{range} \index{range of an edge} of $\mu$, and $k$ is the \emph{length} \index{length of a path} of $\mu$ which is  denoted by $|\mu|$. We consider a vertex $v\in E^0$ as a \emph{trivial} \index{trivial path} path of length zero with $s(v)=r(v)=v$. 
If $\mu$ is a nontrivial path in $E$, and if $v=s(\mu)=r(\mu)$, then $\mu$ is called a \emph{closed path based at} $v$. \index{closed path} If $\mu=\mu_1\dots\mu_k$ is a closed path based at $v=s(\mu)$ and $s(\mu_i) \not = s(\mu_j)$ for every $i \not = j$, then $\mu$ is called a \emph{cycle}. \index{cycle}  

For two vertices $v$ and $w$, the existence of a path with the source $v$ and the range $w$ is denoted by $v\geq w$. Here we allow paths of length zero. By $v\geq_n w$, we mean there is a path of length $n$ connecting these vertices. Therefore $v\geq_0 v$ represents the vertex $v$. Also, by $v>w$, we mean a path from $v$ to $w$ where $v\not = w$. In this book, by $v\geq w' \geq w$, it is understood that there is a path connecting $v$ to $w$ and going through $w'$ (\ie $w'$ is on the path connecting $v$ to $w$). For $n\geq 2$, we define $E^n$ to be the set of paths of length $n$ and $E^*=\bigcup_{n\geq 0} E^n$, the set of all paths. 

For a graph $E$, let $n_{v,w}$ be the number of edges with the source $v$ and range $w$. Then the \emph{adjacency matrix} \index{adjacency matrix} of the graph $E$ is $A_E=(n_{v,w})$. Usually one orders the vertices and then writes $A_E$ based on this ordering. Two different ordering of vertices give different adjacency matrices. However if $A_E$ and $A'_E$ are two adjacency matrices of $E$, then there is a permutation matrix $P$ such that $A'_E=P A_E P^{-1}$. 

A graph $E$ is called \emph{essential} \index{essential graph} if $E$ does not have sinks and sources. Moreover, a graph is called \emph{irreducible} \index{irreducible graph} if for every ordered pair vertices $v$ and $w$, there is a path from $v$ to $w$.  

\subsection{Leavitt path algebras}\label{paohdme} \index{$\LL(E)$, Leavitt path algebra}

A path algebra, with coefficients in the field $K$, is constructed as follows: consider a $K$-vector space with 
finite paths as the basis and define the multiplication by concatenation of paths. A path algebra has a natural graded structure by assigning paths as homogeneous elements of degree equal to their lengths. A formal definition of path algebras with coefficients in a ring $R$ is given below. 

\begin{definition}\label{llkas1} \index{path algebra}
For a graph $E$ and a ring $R$ with identity, we define the \emph{path algebra of $E$},  denoted by $\PP_R(E)$, to be the algebra generated by the sets $\{v \mid v \in E^0\}$, $\{ \alpha \mid \alpha \in E^1 \}$  with the coefficients in $R$, subject to the relations 

\begin{enumerate}
\item $v_iv_j=\delta_{ij}v_i \textrm{ for every } v_i,v_j \in E^0$.

\item $s(\alpha)\alpha=\alpha r(\alpha)=\alpha   \textrm{ for all } \alpha \in E^1$.

\end{enumerate}
\end{definition}
Here the ring $R$ commutes with the generators $\{v,\alpha \mid v \in E^0,\alpha \in E^1\}$. When the coefficient ring $R$ is clear from the context, we simply write $\PP(E)$ instead of $\PP_R(E)$. When $R$ is not commutative, then we consider $\PP_R(E)$ as a left $R$-module. Using the above two relations, it is easy to see that when the number of vertices is finite, then $\PP_R(E)$ is a ring with identity $\sum_{v\in E^0} v$. 

When the graph has one vertex and $n$-loops, the path algebra associated to this graph is isomorphic to $R\langle x_1,\dots,x_n\rangle$, \ie a free associative unital algebra over $R$ with $n$ noncommuting variables.

Setting $\deg(v)=0$ for $v\in E^0$ and  $\deg(\alpha)=1$ for $\alpha \in E^1$, we obtain a natural $\mathbb Z$-grading on the free $R$-ring generated by  $\{v,\alpha \mid v \in E^0,\alpha \in E^1\}$ (\S\ref{freeme}). Since the relations in Definition~\ref{llkas1} are all homogeneous, the ideal generated by these relations is homogeneous and thus we have a natural $\mathbb Z$-grading on $\PP_R(E)$. Note that $\PP(E)$ is positively graded, and for any $m,n \in \mathbb N$,
\[\PP(E)_m \PP(E)_n=\PP(E)_{m+n}.\] However by Proposition~\ref{crossedproductstronglygradedprop}(2), $\PP(E)$ is not an strongly $\Z$-graded ring.

The theory of Leavitt path algebras were introduced in~\cite{aap05,amp} which associate to directed graphs certain type of algebras. These algebras were motivated by Leavitt's construction of universal non-IBN rings~\cite{vitt62}. Leavitt path algebras are quotients of path algebras by relations resembling those in construction of algebras studied by Leavitt (see Example~\ref{levisuji}).

\begin{definition}\label{llkas}  \index{Leavitt path algebra} 
For a row-finite graph $E$ and a ring $R$ with identity, we define the \emph{Leavitt path algebra of $E$},  denoted by $\LL_R(E)$, to be the algebra generated by the sets $\{v \mid v \in E^0\}$, $\{ \alpha \mid \alpha \in E^1 \}$ and $\{ \alpha^* \mid \alpha \in E^1 \}$ with the coefficients in $R$, subject to the relations 

\begin{enumerate}
\item $v_iv_j=\delta_{ij}v_i \textrm{ for every } v_i,v_j \in E^0$.

\item $s(\alpha)\alpha=\alpha r(\alpha)=\alpha \textrm{ and }
r(\alpha)\alpha^*=\alpha^*s(\alpha)=\alpha^*  \textrm{ for all } \alpha \in E^1$.

\item $\alpha^* \alpha'=\delta_{\alpha \alpha'}r(\alpha)$, for all $\alpha, \alpha' \in E^1$.

\item $\sum_{\{\alpha \in E^1, s( \alpha)=v\}} \alpha \alpha^*=v$ for every $v\in E^0$ for which $s^{-1}(v)$ is nonempty.

\end{enumerate}
\end{definition}
Here the ring $R$ commutes with the generators $\{v,\alpha, \alpha^* \mid v \in E^0,\alpha \in E^1\}$. When the coefficient ring $R$ is clear from the context, we simply write $\LL(E)$ instead of $\LL_R(E)$. When $R$ is not commutative, then we consider $\LL_R(E)$ as a left $R$-module. The elements $\alpha^*$ for $\alpha \in E^1$ are called \emph{ghost edges}. \index{ghost edge} One can show that $\LL_R(E)$ is a ring with identity if and only if the graph $E$ is finite (otherwise, $\LL_R(E)$ is a ring with local identities, see~\cite[Lemma~1.6]{aap05}).

Setting $\deg(v)=0$, for $v\in E^0$, $\deg(\alpha)=1$ and $\deg(\alpha^*)=-1$ for $\alpha \in E^1$, we obtain a natural $\mathbb Z$-grading on the free $R$-ring generated by  $\{v,\alpha, \alpha^* \mid v \in E^0,\alpha \in E^1\}$. Since the relations in Definition~\ref{llkas} are all homogeneous, the ideal generated by these relations is homogeneous and thus we have a natural $\mathbb Z$-grading on $\LL_R(E)$. 

If $\mu=\mu_1\dots\mu_k$, where $\mu_i \in E^1$, is an element of $\LL(E)$, then we denote by $\mu^*$ the element $\mu_k ^*\dots \mu_1^* \in \LL(E)$. Further we define $v^*=v$ for any $v\in E^0$. Since $\alpha^* \alpha'=\delta_{\alpha \alpha'}r(\alpha)$, for all $\alpha, \alpha' \in E^1$, any word in the generators $\{v, \alpha, \alpha^* \mid v\in E^0, \alpha \in E^1   \}$ in $\LL(E)$ can be written as $\mu \gamma ^*$ where $\mu$ and $\gamma$ are paths in $E$ (recall that vertices were considered paths of length zero).  The elements of the form $\mu\gamma^*$ are called \emph{monomials}.  \index{monomials} 

If the graph $E$ is infinite, $\LL_R(E)$ is a graded ring without identity (see Remark~\ref{parknan}). 

Taking the grading into account, one can write \[\LL_R(E) =\textstyle{\bigoplus_{k \in \mathbb Z}} \LL_R(E)_k,\] where,
\[\boxed{\LL_R(E)_k=  \Big \{\,  \sum_i r_i \alpha_i \beta_i^*\mid \alpha_i,\beta_i \textrm{ are paths}, r_i \in R, \textrm{ and } |\alpha_i|-|\beta_i|=k \textrm{ for all } i \, \Big\}.}\] 
For simplicity we denote $\LL_R(E)_k$, the homogeneous elements of degree $k$, by $\LL_k$.

\begin{example}\scm[A graded ring whose modules are all graded]
\vspace{0.2cm}

Consider the infinite line graph 

\begin{equation*}
\xymatrix{
E: &  \ar@{.>}[r] & u_{-1} \ar[r]^{e_0} & u_{0} \ar[r]^{e_1} &  u_{1}  \ar@{.>}[r] &   \\
}
\end{equation*}
Then the Leavitt path algebra $\LL(E)$ is a $\Z$-graded ring. Let $X$ be a right $\LL(E)$-module. Set
\[X_i=X u_i, i \in \Z,\] and observe that $X=\bigoplus_{i\in \Z} X_i$. It is easy to check that $X$ becomes a graded $\LL(E)$-module. Moreover, any module homomorphism is a graded homomorphism. Note however that the module category 
$\Modd \LL(E)$ is not equivalent to $\Gr \! \LL(E)$. Also notice that although any ideal is a graded module over $\LL(E)$, but they are not graded ideals of $\LL(E)$.

\end{example}

The following theorem was proved in~\cite{haz} which determines the finite graphs whose associated Leavitt path algebras are strongly graded.

\begin{theorem}\label{sthfin} \index{strongly graded Leavitt path algebra}
Let $E$ be a finite graph and $K$ a field.  Then $\LL_K(E)$ is strongly graded if and only if  $E$ does not have sinks.
\end{theorem}

The proof of this theorem is quite long and does not fit for purpose of this book. However, we can realise the Leavitt path algebras of finite graphs with no source in terms of corner skew Laurent polynomial rings (see~\S\ref{cornerskew}). Using this representation, we can provide a short proof for the above theorem when the graph has no sources. 

\begin{example}\scm[Leavitt path algebras as corner skew Laurent rings]\label{berlinbielefeld}
\vspace{0.2cm}

Let $E$ be a finite graph with no source and $E^0=\{v_1,\dots,v_n\}$ the set of all vertices of $E$. For each $1\leq i \leq n$, we choose an edge $e_i$ such that $r(e_i)=v_i$ and consider $t_{+}=e_1+\dots+e_n \in \LL(E)_1$. Then $t_{-}= e^*_1+\dots+e^*_n$ is its right inverse.  
Thus by Theorem~\ref{jhby67}, $\LL(E)=\LL(E)_0[t_{+},t_{-},\phi]$, where 
\begin{align*}
\phi:\LL(E)_0 &\longrightarrow t_{+}t_{-}\LL(E)_0 t_{+}t_{-}\\
 a &\longmapsto t_{+}at_{-}
\end{align*}

\end{example}

Using this interpretation of Leavitt path algebras we are able to prove the following theorem.

\begin{theorem}\label{sthfin3}
Let $E$ be a finite graph with no source and $K$ a field.  Then $\LL_K(E)$ is strongly graded if and only if  $E$ does not have sinks.
\end{theorem}
\begin{proof}
Write $\LL(E)=\LL(E)_0[t_{+},t_{-},\phi]$, where $\phi(1)=t_{+}t_{-}$ (see Example~\ref{berlinbielefeld}). The theorem now follows from an easy to prove observation that $t_{+}t_{-}$ is a full idempotent if and only if $E$ does not have sinks along with Proposition~\ref{lanhc888}, that $\phi(1)$ is a full idempotent if and only if $\LL(E)_0[t_{+},t_{-},\phi]$ is strongly graded.
\end{proof}

In the following theorem, we use the fact that $\LL(E)_0$ is an \emph{ultramatricial algebra}, \ie it is isomorphic to the union of an increasing chain of a finite product of matrix algebras over a field $K$ (see~\S\ref{sarahbright}). \index{ultramatricial algebra}

\begin{theorem}\label{sthfin4}
Let $E$ be a finite graph with no source and $K$ a field.  Then $\LL_K(E)$ is crossed product if and only if  $E$ is a cycle. \index{crossed product ring}
\end{theorem}
\begin{proof}
Suppose $E$ is a cycle with edges $\{e_1,e_2,\dots,e_n\}$. It is straightforward to check that $e_1+e_2+\dots+e_n$ is an invertible element of degree $1$. It then follows that each homogeneous component contains invertible elements and thus $\LL(E)$ is crossed product. 

Suppose now $\LL(E)$ is crossed product. Write $\LL(E)=\LL(E)_0[t_{+},t_{-},\phi]$, where $\phi(1)=t_{+}t_{-}$  and $t_{+}=e_1+\dots+e_n \in \LL(E)_1$ (see Example~\ref{berlinbielefeld}). Since $\LL(E)_0$ is an ultramatricial algebra, it is Dedekind finite, and thus by Proposition~\ref{dungziflunch}, \[\phi(1)=e_1e_1^*+e_2e_2^*+\dots+e_ne_n^*=v_1+v_2+\dots+v_n.\] From this it follows that, (after suitable permutation), 
$e_ie_i^*=v_i$, for all $1\leq i \leq n$. This in turn shows that there is only one edge emits from each vertex, \ie $E$ is a cycle. 
\end{proof}

As a consequence of Theorem~\ref{sthfin}, we can show that Leavitt path algebras associated to finite graphs with no sinks are graded regular von Neumann rings (\S\ref{penrith7may}).  \index{von Neumann regular ring}

\begin{corollary}\label{hgthusu2}
Let $E$ be a finite graph with no sinks and $K$ a field.  Then $\LL_K(E)$ is a graded von Neumann regular ring. 
\end{corollary}
\begin{proof}
Since $\LL(E)$ is strongly graded (Theorem~\ref{sthfin}), by Corollary~\ref{vonregu}, $\LL(E)$ is von Neumann regular if $\LL(E)_0$ is a von Neumann regular ring. But we know that the zero component ring $\LL(E)_0$ is an ultramatricial algebra which is von Neumann regular (see the proof of~\cite[Theorem~5.3]{amp}).
This finishes the proof. 
\end{proof}

\begin{example}\scm[Leavitt path algebras are not graded unit regular rings]\label{mehui}
\vspace{0.2cm}

By analogy with the ungraded case, a graded ring is \emph{graded von Neumann unit regular} (or \emph{graded unit regular} for short) if for any homogeneous element $x$, there is an invertible homogeneous element $y$ such that $xyx=x$. \index{graded von Neumann unit regular} \index{graded unit regular}Clearly any graded unit regular ring is von Neumann regular. However the converse is not the case. For example, Leavitt path algebras are not in general unit regular as the following example shows. Consider the graph:
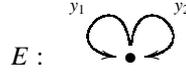
\[
\xymatrix{
E: & \bullet\ar@(u,l)_{y_1} \ar@(u,r)^{y_2}}
\]
Then it is easy to see that there is no homogeneous invertible element $x$ such that $y_1 x y_1=y_1$ in $\LL(E)$. 
\end{example}

The following theorem determines the graded structure of  Leavitt path algebras associated to acyclic graphs. It turns out that such algebras are natural examples of graded matrix rings (\S\ref{matgrhe}). \index{graded matrix ring}

\begin{theorem}\label{gra1}
Let $K$ be a field and $E$ a finite acyclic graph with sinks $\{v_1,\ldots,v_t\}$.  For any sink $v_s$, let $R(v_s)=\{p_1^{v_s},\ldots,p_n^{v_s}\}$ denote the set of all paths ending at $v_s$.  Then there is a $\mathbb{Z}$-graded isomorphism
\begin{equation}
\LL_K(E)\cong_{\gr}\bigoplus_{s=1}^t \M_{n(v_s)}(K)(|p_1^{v_s}|,\ldots,|p_{n(v_s)}^{v_s}|).
\end{equation}
\end{theorem}

\begin{proof}[Sketch of proof]
Fix a sink $v_s$ and denote $R(v_s)=\{p_1,\ldots,p_n\}$. The set \[I_{v_s}=\Big \{\, \sum kp_i p_j^*|k\in K, p_i,p_j \in R(v_s)\,\Big \}\] is an ideal of $\LL_K(E)$, and we have an isomorphism 
\begin{align*}
\phi:I_{v_s}&\longrightarrow \M_{n(v_s)}(K),\\
kp_ip_j^*&\longmapsto k(\e_{ij}),
\end{align*}
 where $k\in K$, $p_i, p_j\in R(v_s)$ and $\e_{ij}$ is the standard matrix unit.  Now, considering the grading on $\M_{n(v_s)}(K)(|p_1^{v_s}|,\ldots,|p_{n(v_s)}^{v_s}|)$, we show that $\phi$ is a graded isomorphism.  Let $p_ip_j^*\in I_{v_s}$.  Then
$$\textrm{deg}(p_ip_j^*)=|p_i|-|p_j|=\textrm{deg}(\e_{ij})=\textrm{deg}(\phi(p_ip_j^*)).$$
So $\phi$ respects the grading.  Hence $\phi$ is a graded isomorphism. One can check that 
\begin{equation*}
\LL_K(E)=\bigoplus_{s=1}^t I_{v_s}\cong_{\gr}\bigoplus_{s=1}^t \M_{n(v_s)}(K)(|p_1^{v_s}|,\ldots,|p_{n(v_s)}^{v_s}|).\qedhere
\end{equation*}
\end{proof}

\begin{example} \label{niroi}
Theorem~\ref{gra1} shows that the Leavitt path algebras of the graphs $E_1$ and $E_2$ below with coefficients from the field $K$ are graded isomorphic to 
$\M_{5} (K )\big (0,1,1,2,2)$ and thus $\LL(E_1) \cong_{\gr} \LL(E_2)$. However \[\LL(E_3)\cong_{\gr} \M_5(K)(0,1,2,2,3).\]


\[\xymatrix@=13pt{
& & \bullet \ar[d] \\
& & \bullet \ar[d] & &  & \bullet \ar[d] & \bullet \ar[d] &  & &   \bullet \ar[d] &  \bullet \ar[d]\\
E_1: \, \,  \bullet \ar[r] & \bullet \ar[r]  & \bullet  & & E_2: \, \, \bullet \ar[r] & \bullet \ar[r] &\bullet && 
E_3: & \bullet \ar[r] & \bullet \ar[r] &\bullet
}\]
\end{example}

\begin{theorem}\label{grCn1}
Let $K$ be a field and $E$ a $C_n$-comet with the cycle $C$ of length $n\geq 1$.   Let $v$ be a vertex on the cycle $C$ with $e$ be the edge in the cycle with $s(e)=v$.  Eliminate the edge $e$ and consider the set $\{p_i|1\leq i\leq m\}$ of all paths with end in $v$.  Then
\begin{equation}
\LL_K(E)\cong_{\gr}\M_m(K[x^n,x^{-n}])(|p_1|,\ldots,|p_m|).
\end{equation}
\end{theorem}

\begin{proof}[Sketch of proof]
One can show that the set of monomials $\{p_iC^kp_j^*|1\leq i,j\leq n, k\in\mathbb{Z}\}$ are a basis of $\LL_K(E)$ as a $K$-vector space.  Define a map
$$\phi:\LL_K(E)\rightarrow \M_m(K[x^n,x^{-n}])(|p_1|,\ldots,|p_m|),\ {\text by } \quad  \phi(p_iC^kp_j^*)=\e_{ij}(x^{kn}),$$
where $\e_{ij}(x^{kn})$ is a matrix with $x^{kn}$ in the $ij$-position and zero elsewhere.  Extend this linearly to $\LL_K(E)$.  We have
\begin{equation*}
\begin{split}
\phi((p_iC^kp_j^*)(p_rC^tp_s^*)) &= \phi(\delta_{jr}p_iC^{k+t}p_s^*)\\
&= \delta_{jr}\e_{is}(x^{(k+t)n})\\
&= (\e_{ij}x^{kn})(\e_{rs}x^{tn})\\
&=\phi(p_iC^kp_j^*)\phi(p_rC^tp_s^*).
\end{split}
\end{equation*}
Thus $\phi$ is a homomorphism.  Also, $\phi$ sends the basis to the basis, so $\phi$ is an isomorphism.\\
We now need to show that $\phi$ is graded.  We have
$$\textrm{deg}(p_iC^kp_j^*)=|p_iC^kp_j^*|=nk+|p_i|-|p_j|$$
and
$$\textrm{deg}(\phi(p_iC^kp_j^*))=\textrm{deg}(\e_{ij}(x^{kn}))=nk+|p_i|-|p_j|.$$
Therefore $\phi$ respects the grading. This finishes the proof.
\end{proof}

\begin{example}\label{noncori}
Consider the Leavitt path algebra $\LL_K(E)$, with coefficients in a field $K$, associated to the following graph. 
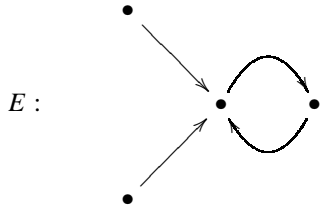
\begin{equation*}
\xymatrix{
& \bullet \ar[dr] &\\
E:  &  & \bullet \ar@/^1.5pc/[r] & \bullet \ar@/^1.5pc/[l] &\\
& \bullet \ar[ur] &&\\}
\end{equation*}
By Theorem~\ref{sthfin},  $\LL_K(E)$ is strongly graded. Now by Theorem~\ref{grCn1}, 
\begin{equation}\label{wokj}
\LL_K(E) \cong_{\gr}\M_4(K[x^2,x^{-2}])(0,1,1,1).
\end{equation} However this algebra is not  a crossed product. Set $B=K[x,x^{-1}]$ with the grading $B=\textstyle{\bigoplus_{n\in \mathbb Z}} Kx^n$ and consider $A=K[x^2,x^{-2}]$ as a graded subring of $B$ with 
$A_n=Kx^n$ if $n\equiv 0 \pmod{2}$, and $A_n=0$ otherwise.  
Using the graded isomorphism of~(\ref{wokj}), by (\ref{mmkkhh})  a homogeneous element of degree $1$ in $\LL_K(E)$ has the form 
\begin{equation*}
\begin{pmatrix}
A_{1} & A_{2} & A_2 & A_2\\
A_{0} & A_{1} & A_1 & A_1\\
A_{0} & A_{1} & A_1 & A_1\\
A_{0} & A_{1} & A_1 & A_1
\end{pmatrix}.
\end{equation*}
Since $A_1=0$, the determinants of these matrices are zero, and thus no homogeneous   element of degree $1$ is  invertible. Thus $\LL_K(E)$ is not crossed product (see~\S\ref{scrosshg}). 

Now consider the following graph,

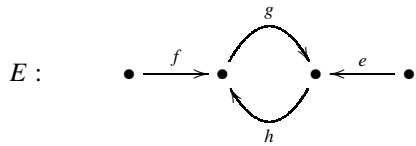
\begin{equation*}
\xymatrix{
E:  & \bullet \ar[r]^f & \bullet \ar@/^1.5pc/[r]^g & \bullet \ar@/^1.5pc/[l]^h & \bullet \ar[l]_e& \\
}
\end{equation*}
By Theorem~\ref{grCn1}, 
\begin{equation}\label{wokj5}
\LL_K(E) \cong_{\gr}\M_4(K[x^2,x^{-2}])(0,1,1,2).
\end{equation}
Using the graded isomorphism of~(\ref{wokj5}), by (\ref{mmkkhh4})   homogeneous elements of degree $0$ in $\LL_K(E)$ have the form 
\begin{equation*}\LL_K(E)_0=
\begin{pmatrix}
A_{0\phantom{-}} & A_{1\phantom{-}} & A_{1\phantom{-}} & A_{2\phantom{-}}\\
A_{-1} & A_{0\phantom{-}} & A_{0\phantom{-}} & A_{1\phantom{-}}\\
A_{-1} & A_{0\phantom{-}} & A_{0\phantom{-}} & A_{1\phantom{-}}\\
A_{-2} & A_{-1} & A_{-1} & A_{0\phantom{-}}
\end{pmatrix}=
\begin{pmatrix}
K\phantom{x^{-2}} & 0 & 0 & Kx^2\\
0\phantom{x^{-2}} & K & K & 0\phantom{x^{2}}\\
0\phantom{x^{-2}} & K & K & 0\phantom{x^{2}}\\
Kx^{-2} & 0 & 0 & K\phantom{x^{2}}
\end{pmatrix}.
\end{equation*}
In the same manner, homogeneous elements of degree one have the form,
\begin{equation*}\LL_K(E)_1=
\begin{pmatrix}
0 & Kx^2 & Kx^2 & 0\phantom{x^{2}}\\
K & 0\phantom{x^{2}} & 0\phantom{x^{2}} & Kx^2\\
K & 0\phantom{x^{2}} & 0\phantom{x^{2}} & Kx^2\\
0 & K\phantom{x^{2}} & K\phantom{x^{2}} & 0\phantom{x^{2}}
 \end{pmatrix}.
\end{equation*}
Choose 
\begin{equation*}u=
\begin{pmatrix}
0 & 0 & x^2 & 0\phantom{^{2}}\\
0 & 0 & 0\phantom{^{2}} & x^2\\
1 & 0 & 0\phantom{^{2}} & x^2\\
0 & 1 & 0\phantom{^{2}} & 0\phantom{^{2}}
 \end{pmatrix} \in \LL(E)_1
\end{equation*}
and observe that $u$ is invertible; this matrix corresponds to the element \[g+h+fge^*+ehf^* \in \LL_K(E)_1.\]

Thus $\LL_K(E)$ is crossed product and therefore a skew group ring as the grading is cyclic, (see~\S\ref{khgfewa1}), \ie \[\LL_K(E)\cong_{\gr} \bigoplus_{i\in \mathbb Z}\LL_K(E)_0 u^i\] and a simple calculation shows that one can describe this algebra as follows:
\[\LL_K(E)_0 \cong \M_2(K)\times \M_2(K),\] and 
\begin{equation}\label{coiu}
\LL_K(E) \cong_{\gr} \big ( \M_2(K)\times \M_2(K) \big ) \star_\tau \mathbb Z,
\end{equation}
where, 
 \begin{equation*}
\tau\Bigr(\left( 
\begin{matrix} 
a_{11} & a_{12}\\ 
a_{21} & a_{22}  
\end{matrix} 
\right) ,
\left( 
\begin{matrix} 
b_{11} & b_{12}\\ 
b_{21} & b_{22}  
\end{matrix} 
\right)\Bigl)=\Bigr(
\left( 
\begin{matrix} 
b_{22} & b_{21}\\ 
b_{12} & b_{11}  
\end{matrix} 
\right) , 
\left( 
\begin{matrix} 
a_{22} & a_{21}\\ 
a_{12} & a_{11}  
\end{matrix} 
\right) \Bigl).
\end{equation*}
\end{example}

\begin{remark}\label{remnhgfeyq}\scm[Noncanonical gradings on Leavitt path algebras]
\vspace{0.2cm}

For a graph $E$, the Leavitt path algebra $\LL_K(E)$ has  a canonical $\mathbb Z$-graded structure. This grading was obtained by assigning $0$ to vertices, $1$ to edges and $-1$ to ghost edges. However, one can equip $\LL_K(E)$ with other graded structures as well. Let $\Gamma$ be an arbitrary group with the identity element $e$. Let $w:E^1\rightarrow \Gamma$ be a \emph{weight} \index{weight of an edge} map and further define $w(\alpha^*)=w(\alpha)^{-1}$, for any edge $\alpha \in E^1$ and $w(v)=e$ for $v\in E^0$.  The free $K$-algebra generated by the vertices, edges and ghost edges is a $\Gamma$\!-graded $K$-algebra (see \S\ref{freeme}). Moreover, the Leavitt path algebra is the quotient of this algebra by relations in  Definition~\ref{llkas} which are all homogeneous. Thus $\LL_K(E)$ is a $\Gamma$\!-graded $K$-algebra. One can write the Theorems~\ref{gra1} and~\ref{grCn1} with this general grading. 

As an example, consider the graphs 
 
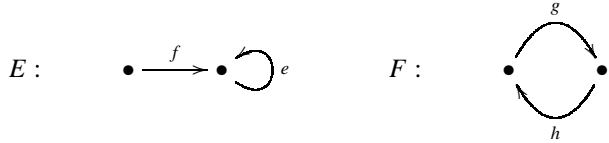
\begin{equation*}
\xymatrix{
E: & \bullet \ar[r]^f &  \bullet \ar@(rd,ru)_{e} & & F:  &  \bullet \ar@/^1.5pc/[r]^g & \bullet \ar@/^1.5pc/[l]^h & \\
}
\end{equation*}
and assign $1$ for the degree of  $f$ and $2$ for the degree of $e$ in $E$ and $1$ for the degrees of $g$ and $h$ in $F$. Then the proof of Theorem~\ref{grCn1} shows that
\[\LL_K(E)\cong \M_2(K[x^2,x^{-2}])(0,1)\] and \[\LL_K(F)\cong \M_2(K[x^2,x^{-2}])(0,1)\] as $\Z$-graded rings. So with these gradings, $\LL_K(E)\cong_{\gr} \LL_K(F)$. 

\end{remark}

\begin{example}\label{nhatrang}\scm[Leavitt path algebras are strongly $\Z_2$-graded] \index{strongly graded Leavitt path algebra}
\vspace{0.2cm}

Let $E$ be a (connected) row-finite graph with at least one edge. By Remark~\ref{remnhgfeyq}, $A=\LL_K(E)$ has a $\Z_2$-grading induced by assigning 
$0$ to vertices and $1\in \Z_2$ to edges and ghost edges. Since the defining relations of Leavitt path algebras guarantee that for any $v\in E^0$,  $v\in A_1 A_1$, one can easily check that $\LL_K(E)$ is strongly $\Z_2$-graded for any graph (compare this with Theorem~\ref{sthfin}). In contrast to the canonical grading, in this case the 0-component ring is not necessarily an ultramatricial ring (see~\S\ref{sarahbright}). 

\end{example}

\section{The graded IBN and graded type}\label{gtr5654} 

A ring $A$ with identity has  \emph{invariant basis number} (IBN) or   \emph{invariant basis property} \index{invariant basis number} \index{invariant basis property}\index{IBN} if any two bases of a free (right) $A$-module have the same cardinality, \ie if $A^n \cong A^{m}$ as $A$-modules, then $n=m$. When $A$ does not have IBN, the \emph{type} of $A$ \index{type of a ring} is defined as a pair of positive integers $(n,k)$ such that $A^n \cong A^{n+k}$ as $A$-modules and these are the smallest number with this property, that is, $(n,k)$ is the minimum under the usual lexicographic order. This means any two bases of a free $A$-module have the unique cardinality if one of the basis has the cardinality less than $n$ and further if a free module has rank $n$, then a free module with the smallest cardinality (other than $n$) isomorphic to this module is of rank $n+k$. Another way to describe a type $(n,k)$ is that, $A^n \cong A^{n+k}$ is the first repetition in the list $A,A^2,A^3,\dots$. 

It was shown that if $A$ has type $(n,k)$, then $A^m\cong A^{m'}$ if and only if $m=m'$ or $m,m' \geq n$ and $m \equiv m' \pmod{k}$ (see \cite[p. 225]{cohn11}, \cite[Theorem~1]{vitt62}).

One can show that a (right) noetherian ring has IBN. Moreover, if there is a ring homomorphism $A\rightarrow B$, (which preserves 1), and $B$ has IBN then $A$ has IBN as well. Indeed, if $A^m\cong A^n$ then 
\begin{equation}\label{gftgfrt}
B^m\cong A^m\otimes_A B \cong A^n\otimes_A B \cong B^n,
\end{equation}
 so $n=m$. One can describe the type of a ring by using the monoid of isomorphism classes of finitely generated projective modules (see Example~\ref{ppurdeu}). For a nice discussion about these rings see~\cite{berrickkeating,cohn11,magurn}.   \index{Noetherian ring}

A graded ring $A$ has a \emph{graded invariant basis number} (gr-IBN) \index{graded invariant basis number} \index{gr-IBN} if any two homogeneous bases of a graded free (right) $A$-module have the same cardinality, \ie if $ A^m(\overline \alpha)\cong_{\gr}  A^n(\overline \delta) $, where  $\overline \alpha=(\alpha_1,\dots,\alpha_m)$ and $\overline \delta=(\delta_1,\dots,\delta_n)$, then $m=n$. Note that in contrast to the ungraded case, this does not imply that two graded free modules with bases of the same cardinality are graded isomorphic (see Proposition~\ref{rndconggrrna}). A graded ring $A$ has \emph{IBN in} \index{IBN in category} $\mbox{gr-}A$, if $A^m \cong_{\gr} A^n$ then $m=n$. If $A$ has IBN in $\grr A$, then $A_0$ has IBN. Indeed, if $A_0^m \cong A_0^n$ as $A_0$-modules, then similair to (\ref{gftgfrt}) \[A^m\cong_{\gr} A_0^m \otimes_{A_0} A \cong A_0^n \otimes_{A_0} A \cong_{\gr}A^n,\] so $n=m$ (see \cite[p. 215]{grrings}).

When the graded ring $A$ does not have gr-IBN, the \emph{graded type} \index{graded type of a ring} of $A$ is defined as a pair of positive integers $(n,k)$ such that $A^n(\overline \delta) \cong_{\gr} A^{n+k}(\overline \alpha)$ as $A$-modules, for some 
$\overline \delta=(\delta_1,\dots,\delta_n)$ and $\overline \alpha=(\alpha_1,\dots,\alpha_{n+k})$ and these are the smallest number with this property. 
In Proposition~\ref{grtypel1} we show that the Leavitt algebra $\LL(n,k+1)$ (see Example~\ref{levisuji2}) has graded type $(n,k)$. 

Parallel to the ungraded setting, one can show that a graded (right) noetherian ring has gr-IBN. Moreover, if there is a graded ring homomorphism $A\rightarrow B$, (which preserves 1), and $B$ has gr-IBN then $A$ has gr-IBN as well. Indeed, if 
$A^m(\overline \alpha) \cong_{\gr} A^n(\overline \delta)$, where $\overline \alpha=(\alpha_1,\dots,\alpha_m)$ and $\overline \delta=(\delta_1,\dots,\delta_n)$, then \[B^m(\overline \alpha) \cong_{\gr} A^m(\overline \alpha)   \otimes_A B \cong A^n(\overline \delta) \otimes_A B \cong_{\gr}  B^n(\overline \delta),\] which implies $n=m$. Using this, one can show that any graded commutative ring has gr-IBN. For, there exists a graded maximal ideal and its quotient ring is a graded field which has gr-IBN (see \S\ref{jghjrye} and Proposition~\ref{kj351}).

Let $A$ be a $\Gamma$\!-graded ring such that $A^m(\overline \alpha) \cong_{\gr} A^n(\overline \delta)$, where $\overline \alpha=(\alpha_1,\dots,\alpha_m)$ and $\overline \delta=(\delta_1,\dots,\delta_n)$. Then there is a universal $\Gamma$\!-graded ring $R$ such that \[ R^m(\overline \alpha)  \cong_{\gr} R^n(\overline \delta)\] and a graded ring homomorphism $R\rightarrow A$ which induces the graded isomorphism 
\[A^m(\overline \alpha) \cong_{\gr} R^m(\overline \alpha) \otimes_R A \cong_{\gr} R^n(\overline \de) \otimes_R A \cong_{\gr} A^n(\overline \de).\] Indeed, by Proposition~\ref{rndconggrrna}, there are matrices 
$a=(a_{ij}) \in \M_{n\times m} (A)[\overline \delta][\overline \alpha]$ and 
$b=(b_{ij}) \in \M_{m\times n} (A)[\overline \alpha][\overline \delta]$ such that $ab=\IM_{n}$ and $ba=\IM_{m}$.
The free ring generated by symbols in place of $a_{ij}$ and $b_{ij}$ subject to relations imposed by $ab=\IM_{n}$ and $ba=\IM_{m}$ is the desired universal graded ring. In detail, let $F$ be a free ring generated by $x_{ij}$, $1\leq i \leq n$, $1\leq j \leq m$  and $y_{ij}$, $1\leq i \leq m$, $1\leq j \leq n$. Assign the degrees $\deg(x_{ij})=\delta_i-\alpha_j$ and $\deg(y_{ij})=\alpha_i-\delta_j$ (see ~\S\ref{freeme}). This makes $F$ a $\Gamma$\!-graded ring. 
Let $R$ be a ring $F$ modulo the relations $\sum_{s=1}^m x_{is}y_{sk}=\delta_{ik}$, $1\leq i,k \leq n$ and $\sum_{t=1}^n y_{it}x_{tk}=\delta_{ik}$, $1\leq i,k \leq m$, where $\delta_{ik}$ is the Kronecker delta. 
Since all the relations are homogeneous, $R$ is a $\Gamma$\!-graded ring. Clearly the map sending $x_{ij}$ to $a_{ij}$ and $y_{ij}$ to $b_{ij}$ induces a graded ring homomorphism $R \rightarrow A$. Again Proposition~\ref{rndconggrrna} shows that $R^m(\overline \alpha) \cong_{\gr} R^n(\overline \delta)  $.

\begin{proposition}\label{grtypel1} 
Let $R=\LL(n,k+1)$ be the  Leavitt algebra of type $(n,k)$. Then 
\begin{enumerate}[\upshape(1)]

\item $R$ is a universal $\bigoplus_n \mathbb Z$-graded ring  which does not have gr-IBN;

\item $R$ has  graded type $(n,k)$;

\item For $n=1$, $R$ has IBN in  $\grr R$. 

\end{enumerate}
\end{proposition}
\begin{proof}
(1) Consider the algebra $\LL(n,k+1)$ constructed in Example~\ref{levisuji2}, which is a $\bigoplus_n \mathbb Z$-graded ring and is universal. Moreover (\ref{breaktr}) combined by Proposition~\ref{rndconggrrna}(3) shows that $R^n \cong_{\gr} R^{n+k}(\overline \alpha)$. Here $\overline \alpha =(\alpha_1,\dots,\alpha_{n+k})$, where $\alpha_i=(0,\dots,0,1,0\dots,0)$ and $1$ is in $i$-th entry. This shows that $R=\LL(n,k+1)$ does not have gr-IBN. 

(2) By \cite[Theorem~6.1]{cohn11}, $R$ is of  type $(n,k)$. This immediately implies the graded type of 
$R$ is also $(n,k)$. 

(3) Suppose $R^n \cong_{\gr} R^m$ as graded $R$-modules. Then $R_0^n \cong R_0^m$ as $R_0$-modules. But $R_0$ is an ultramatricial algebra, \ie direct limit of matrices over a field. Since IBN respects direct limits (\cite[Theorem~2.3]{cohn11}), $R_0$ has IBN. Therefore, $n=m$.
\end{proof}

\begin{remark}
Assignment of $\deg(y_{ij})=1$ and $\deg(x_{ij})=-1$, for all $i,j$, makes $R=\LL(n,k+1)$ a $\mathbb Z$-graded algebra of graded type $(n,k)$ with $R^n\cong_{\gr} R^{n+k}(1)$. 
\end{remark}

\begin{remark}
Let $A$ be a $\Gamma$\!-graded ring. In~\cite[Proposition~4.4]{ibno}, it was shown that if $\Gamma$ is finite then $A$ has  gr-IBN if and only if $A$ has IBN. 
\end{remark}

\section{The graded stable rank}

The notion of stable rank was defined by H. Bass~\cite{bass64} to study $K_1$-group of rings that are finitely generated over commutative rings with finite Krull dimension. For a concise introduction to the stable rank, we refer the reader to~\cite{lamsr,lamfc}, and for its applications to $K$-theory to~\cite{bass,bass64}. It seems that the natural notion of graded stable rank in the context of graded ring theory has not yet been investigated in the literature. In this section we propose a definition for the graded stable rank and study the important case of graded rings with graded stable rank $1$. This will be used later in \S\ref{ggg} in relation with the graded Grothendieck groups.

A row $(a_1,\dots,a_n)$ of homogeneous elements of a $\Gamma$\!-graded ring $A$ is called a \emph{graded left unimodular row} \index{graded unimodular row} if the graded left ideal generated by $a_i$, $1\leq i \leq n$, is $A$.

\begin{lemma}\label{runorun}
Let $(a_1,\dots,a_n)$ be a row of  homogeneous elements of a $\Gamma$\!-graded ring $A$. The following are equivalent.

\begin{enumerate}[\upshape(1)]

\item $(a_1,\dots,a_n)$  is a  left unimodular row;

\item $(a_1,\dots,a_n)$  is a graded left unimodular row;

\item The graded homomorphism 
\begin{align*}
\phi_{(a_1,\dots,a_n)} : A^n(\overline \alpha)  & \longrightarrow A, \\
(x_1,\dots,x_n) & \longmapsto \sum_{i=1}^n x_ia_i, 
\end{align*}
where $\overline \alpha=(\alpha_1,\dots,\alpha_n)$ and  $\alpha_i=-\deg(a_i)$, 
is surjective.

\end{enumerate}
\end{lemma}
\begin{proof}
The proof is straightforward. 
\end{proof}

When $n\geq 2$, a graded left unimodular row $(a_1,\dots,a_n)$ is called \emph{stable} \index{stable unimodular row} if there exists homogeneous elements $b_1,\dots,b_{n-1}$ of $A$ such that the graded left ideal generated by homogeneous elements $a_i+b_ia_n$, $1\leq i \leq n-1$, is $A$. 

The \emph{graded left stable rank} \index{graded stable rank} of a ring $A$ is defined to be $n$, denoted $\sr^{\gr}(A) = n$, if any graded unimodular row of length $n + 1$ is stable, but there exists a unstable unimodular row of length $n$. If such $n$ does not exist (\ie there are unstable unimodular rows of arbitrary length) we say that the graded stable rank of $A$ is infinite.

In order that this definition would be well-defined, one needs to show that if any graded unimodular row of fixed length $n$ is stable, so is any unimodular row of a greater length. This can be proved similar to the ungraded case and we omit the proof (see for example~\cite[Proposition~1.3]{lamsr}). 

When the grade group $\Gamma$ is a trivial group, the above definitions reduce to the standard definitions of unimodular rows and stable ranks.

The case of graded stable rank $1$ is of special importance. Suppose $A$ is a $\Gamma$\!-graded with $\sr^{\gr}(A)=1$. Then from the definition it follows that, if $a,b \in A^h$ such that $Aa+Ab=A$, then there is a homogeneous element $c$ such that the homogeneous element $a+cb$ is left invertible. When $\sr^{\gr}(A)=1$, any left invertible homogeneous element is in fact invertible. For, suppose $c\in A^h$ is left invertible, \ie there is $a\in A^h$ such that $ac=1$. Then the row $(a, 1-ca)$ is graded left unimodular. Thus, there is a $s\in A^h$ such that $u:=a+s(1-ca)$ is left invertible. But $uc=1$. Thus $u$ is (left and right) invertible and consequently, $c$ is an invertible homogeneous element. 

The graded stable rank $1$ is quite a strong condition. In fact if $\sr^{\gr}(A)=1$ then $\Gamma_A=\Gamma_A^*$. For, if $a\in A_\gamma$ is a nonzero element, then since $(a,1)$ is unimodular and $\sr^{\gr}(A)=1$, there is $c\in A^h$, such that $a+c$ is an invertible homogeneous element, necessarily of  degree $\gamma$. Thus $\gamma \in \Gamma_A^*$. 

\begin{example}\scm[Graded division rings have graded stable rank $1$]
\vspace{0.2cm}

Since any nonzero homogeneous element of a graded division ring is invertible, one shows easily that its graded stable rank is $1$. Thus for a field $K$, 
$\sr^{\gr}(K[x,x^{-1}])=1$, whereas $\sr(K[x,x^{-1}])=2$. 
\end{example}

\begin{example}\scm[For a strongly graded ring $A$,  $\sr^{\gr}(A) \not = \sr(A_0)$]
\vspace{0.2cm}

Let $A=\LL(1,2)$ be the Leavitt algebra generated by $x_1,x_2,y_1,y_2$ (see Example~\ref{levisuji}). Then relations~(\ref{jh54320}) show that $y_1$ is left invertible but it is not invertible. This shows that $\sr^{\gr}(A)\not =1$. On the other hand, since $A_0$ is an ultramatricial algebra, $\sr(A_0)=1$ 
(see~\S\ref{sarahbright},~\cite[Corollary~5.5]{lamsr} and~\cite{goodearlbook}). 
\end{example}

We have the following theorem which is a graded version of the Cancellation Theorem with a similar proof (see~\cite[Theorem~20.11]{lamfc}).
\begin{theorem}[{\sc Graded Cancellation Theorem}]\label{jussiej4}
Let $A$ be a $\Gamma$\!-graded ring  and let $M,N,P$ be graded right $A$-modules, with $P$ being finitely generated. If the graded ring $\End_A(P)$ has graded left stable rank $1$, then
$P\oplus M \cong_{\gr} P\oplus N$ as $A$-modules implies $M \cong_{\gr} N$ as $A$-modules. 
\end{theorem}
\begin{proof}
Set $E:=\End_A(P)$. Let $h:P\oplus M \rightarrow P \oplus N$ be a graded $A$-module isomorphism. Then the composition of the maps
\begin{align*}
P\stackrel{i_1}{\longrightarrow} P\oplus M \stackrel{h}\longrightarrow P\oplus N \stackrel{\pi_1}\longrightarrow P,\\
M\stackrel{i_2}\longrightarrow P\oplus M \stackrel{h}\longrightarrow P\oplus N \stackrel{\pi_1}\longrightarrow P
\end{align*}
induces a graded split epimorphism of degree zero, denoted by $(f,g):P\oplus M \rightarrow P$. Here $(f,g)(p,m)=f(p)+g(m)$, where 
$f=\pi_1 h i_1$ and $g=\pi_1 h i_2$. It is clear that  $\ker(f,g)\cong_{\gr} N$. 
Let 
$\left( 
\begin{matrix} 
f' \\ 
g'
\end{matrix} 
\right): P\rightarrow P\oplus M$ be the split homomorphism.  Thus 
\[1=(f,g)\left( 
\begin{matrix} 
f' \\ 
g'
\end{matrix} 
\right)=ff' +gg'.\]
This shows that the left ideal generated by $f'$ and $gg'$ is $E$. Since $E$ has graded stable rank $1$, 
it follows there is $e\in E$ of degree $0$ such that $u:=f'+e (gg')$ is an invertible element of $E$. Writing $u=(1,eg) \left( 
\begin{matrix} 
f' \\ 
g'
\end{matrix} 
\right)$ implies that both $\ker(f,g)$ and $\ker(1,eg)$ are graded isomorphic to 
\[P\oplus M/ \Im \left( 
\begin{matrix} 
f' \\ 
g'
\end{matrix} 
\right).\]
Thus $\ker(f,g)\cong_{\gr}\ker(1,eg)$. But $\ker(f,g)\cong_{\gr} N$ and $\ker(1,eg) \cong_{\gr} M$. Thus $M \cong_{\gr} N$. 
\end{proof}

The following corollary will be used in~\S\ref{ggg} to show that for a graded ring with graded stable rank $1$, the monoid of graded finitely generated projective modules injects into the graded Grothendieck group (see Corollary~\ref{hghtt1}).

\begin{corollary}\label{jussiej}
Let $A$ be a $\Gamma$\!-graded ring with graded left stable rank $1$ and $M,N,P$ be graded right $A$-modules. If $P$ is a graded finitely generated projective $A$-module, then
$P\oplus M \cong_{\gr} P\oplus N$ as $A$-modules implies $M \cong_{\gr} N$ as $A$-modules. 
\end{corollary}
\begin{proof}
Suppose $P\oplus M \cong_{\gr} P\oplus N$ as $A$-modules. Since $P$ is a graded finitely generated $A$-module, there is a graded $A$-module $Q$ such that $P\oplus Q \cong_{\gr} A^n(\ol \alpha)$ (see~(\ref{medickiue})). It follows that 
\[ A^n(\ol \alpha)\oplus M \cong_{\gr} A^n(\ol \alpha) \oplus N.\]
We prove that if 
\begin{equation}\label{gbfgt34}
A(\alpha) \oplus M \cong_{\gr} A(\alpha) \oplus N,
\end{equation} 
then $M\cong_{\gr} N$. The corollary then follows by an easy induction. 

By~(\ref{fatloss}) there is a graded ring isomorphism $\End_A(A(\alpha))\cong_{\gr} A$. Since $A$ has graded stable rank $1$, so does $\End_A(A(\alpha))$. Now by Theorem~\ref{jussiej4}, from~(\ref{gbfgt34}) it follows that $M\cong_{\gr} N$. This finishes the proof. 
\end{proof}

The stable rank imposes other finiteness properties on rings such as the IBN property (see~\cite[Exercise~I.1.5(e)]{weibelk}). 

\begin{theorem}\label{mabasa}
Let $A$ be a $\Gamma$\!-graded ring such that $A \cong_{\gr} A^r(\overline \alpha)$ as  left $A$-modules, for some $\overline \alpha=(\alpha_1,\dots,\alpha_r)$, $r>1$. Then the graded stable rank of $A$ is infinite. 
\end{theorem}
\begin{proof}
Suppose the graded stable rank of $A$ is $n$. Then one can find $\overline \alpha=(\alpha_1,\dots,\alpha_r)$, where $r> n$ and $A^r(\overline \alpha) \cong_{\gr} A$. Suppose $\phi:  A^r(\overline \alpha) \rightarrow A$ is this given graded isomorphism. Set $a_i=\phi(e_i)$, $1\leq i\leq r$, where $\{e_i \mid 1\leq i\leq r\}$ are the standard (homogeneous) basis of $A^r$. Then for any $x\in A^r(\overline \alpha)$, 
\[ \phi(x)=\phi(\sum_{i=1}^r x_ie_i)=\sum_{i=1}^r x_i a_i=\phi_{(a_1,\dots,a_r)}(x_1,\dots,x_r).\] Since $\phi$ is an isomorphism, by Lemma~\ref{runorun}, the row $(a_1,\dots,a_r)$ is graded left unimodular. Since $r>\sr^{\gr}(A)$, there is a homogeneous row $(b_1,\dots,b_{r-1})$ such that 
$(a_1+b_1a_r,\dots, a_{r-1}+b_{r-1}a_r)$ is also left unimodular. Note that $\deg(b_i)=\alpha_r -\alpha_i$. Consider the graded left $A$-module homomorphism 
\begin{align*}
\psi: A^{r-1}(\alpha_1,\dots,\alpha_{r-1})  & \longrightarrow A^{r}(\alpha_1,\dots,\alpha_{r-1},\alpha_r), \\
(x_1,\dots,x_{r-1}) & \longmapsto (x_1,\dots,x_{r-1},\sum_{i=1}^{r-1} x_ib_i) 
\end{align*}
and the commutative diagram 
\begin{equation*}
\xymatrix{
A^{r-1}(\alpha_1,\dots,\alpha_{r-1}) \ar[rr]^{\psi} \ar[rd]_{\phi_{(a_1+b_1a_r,\dots, a_{r-1}+b_{r-1}a_r)}~~~~} &&  A^{r}(\alpha_1,\dots,\alpha_{r-1},\alpha_r) \ar[ld]^{\phi_{(a_1,\dots,a_r)}},\\
& A 
}
\end{equation*}
Since $\phi_{(a_1,\dots,a_r)}$ is an isomorphism and $\phi_{(a_1+b_1a_r,\dots, a_{r-1}+b_{r-1}a_r)}$ is an epimorphism, $\psi$ is also an epimorphism. Thus  there is $(x_1,\dots,x_{r-1})$ such that $\psi(x_1,\dots,x_{r-1})=(0,\dots,0,1)$ which immediately gives a contradiction. 
\end{proof}

\begin{corollary}
The graded stable rank of the Leavitt algebra $\LL(1,n)$ is infinite. 
\end{corollary}
\begin{proof}
This follows from Proposition~\ref{grtypel1} and Theorem~\ref{mabasa}.
\end{proof}

\begin{example}\scm[Graded von Neumann regular rings with the stable rank 1]

One can prove, similar to the ungraded case~\cite[Proposition~4.12]{goodearlbook}, that a graded von Neumann regular ring has a stable rank $1$ if an only of it is a graded von Neumann unit regular.  \index{graded von Neumann regular ring}

\end{example}

\section{Graded rings with involution}\label{involudool}   \index{involutary graded ring}   
\index{graded $*$-ring} \index{graded ring with involution}

Let $A$ be a ring with an involution denoted by  ${}^*$, \ie ${}^*:A \rightarrow A$, $a\mapsto a^*$, is 
an anti-automotphism of order two. Throughout this book we call $A$ also a $*$-ring. If $M$ is a right $A$-module, then $M$ can be given a left $A$-module  structure by defining 
\begin{equation}\label{rightm} \index{*-ring}
am:=m a^*.
\end{equation}
This gives an equivalent  
\begin{equation}\label{ontheair}
\Modd A  \approx A \rModd,
\end{equation}
where $\Modd A$ is the category of right $A$-modules and $A\rModd$ is the category of left $A$-modules. 

Now let $A=\bigoplus_{\gamma \in \Gamma} A_\gamma$ be a $\Gamma$\!-graded ring.  
We call $A$ a \emph{graded $*$-ring} if  there is an involution on $A$ such that for $a\in A_\gamma$, 
$a^* \in A_{-\gamma}$, where $\gamma \in \Gamma$. It follows that $A_\gamma ^*= A_{-\gamma}$, for any $\gamma \in \Gamma$. 

\begin{remark}
Depending on the circumstances, one can also set another definition that $A_\gamma^*= A_\gamma$, where $\gamma \in \Gamma$. \end{remark}

If $A$ is a graded $*$-ring, and $M$ is a graded right $A$-module, then the multiplication in~(\ref{rightm}) makes $M$ a graded left $A^{(-1)}$-module and 
makes $M^{(-1)}$ a graded left $A$-module, where $A^{(-1)}$ and $M^{(-1)}$ are Veronese rings and modules (see Examples~\ref{egofgrdivisionrings0} 
and~\ref{egofgonrings0}). \index{Veronese subring}   \index{Veronese module} These give graded  equivalences
\begin{align}\label{ontheair2}
\mathcal I: \Gr A  &\longrightarrow A^{(-1)} \rGr,\\
M &\longmapsto M \notag
\end{align}
and 
\begin{align}\label{ontheair3}
\mathcal J: \Gr A  &\longrightarrow A \rGr,\\
M &\longmapsto M^{(-1)}. \notag
\end{align}
Here for $\alpha \in \Gamma$, $\mathcal I(M(\alpha))=\mathcal I(M)(\alpha)$ (\ie $\mathcal I$ is a graded functor, see Definition~\ref{grdeffsa}), whereas 
$\mathcal J(M(\alpha))=\mathcal J(M)(-\alpha)$.

Clearly if the grade group $\Gamma$ is trivial, the equivalences~(\ref{ontheair2}) and~(\ref{ontheair3}) both reduce to~(\ref{ontheair}). 
 
 Let $A$ be a graded $*$-field and $R$ a graded $A$-algebra with involution denoted by ${}^*$ again. Then $R$ is a graded $*$-$A$-algebra if 
 $(a r)^*=a^* r^*$ (\ie the graded homomorphism $A\rightarrow R$ is a $*$-homomorphism).    \index{graded $*$-algebra}

 \begin{example}\scm[Group rings]  \index{group ring}
\vspace{0.2cm}

 For a group $\Gamma$ (denoted multiplicatively here),  the group ring $\Z[\Gamma]$ with a natural $\Gamma$\!-grading 
\begin{equation*}
\Z[\Gamma]=\bigoplus_{\ga \in \Ga} \Z[\Gamma]_\gamma, \text{   where   }  \Z[\Gamma]_\gamma= \Z\ga, 
\end{equation*}
and the natural involution $*:\Z[\Gamma] \rightarrow \Z[\Gamma], \gamma \mapsto \gamma^{-1}$ is a graded $*$-ring. 
\end{example}

\begin{example}\scm[Hermitian transpose]
\vspace{0.2cm}

If $A$ is a graded $*$-ring, then for $a=(a_{ij}) \in \M_n(A)(\de_1,\dots,\de_n)$,  the \emph{Hermitian transpose} $a^*=(a_{ji}^*)$, makes $\M_n(A)(\de_1,\dots,\de_n)$ a 
graded $*$-ring (see~\ref{mmkkhh}). \index{Hermitian transpose}
\end{example}

 \begin{example}\scm[Leavitt path algebras are *-graded algebras]\label{bfghrtd2}
\vspace{0.2cm}

A Leavitt path algebra has a natural $*$-involution. Let $K$ be a $*$-field. Define a homomorphism from the free $K$-algebra generated by the vertices, edges and ghost edges of the graph $E$ to $\LL(E)^{\op}$, by $k \mapsto k^*$, $v \mapsto v$, $\alpha \mapsto \alpha^*$ and $\alpha^* \mapsto \alpha$, where $k\in K$, $v\in E^0$, $\alpha \in E^1$ and $\alpha^*$ the ghost edge.  The relations in definition of a Leavitt path algebra, Definition~\ref{llkas}, show that this homomorphism induces an isomorphism from $\LL(E)$ to $\LL(E)^{\op}$. This makes $\LL(E)$ a $*$-algebra. Moreover, considering the grading, it is easy to see that in fact, $\LL(E)$ is a graded $*$-algebra. 
\end{example}

 \begin{example}\scm[Corner skew Laurent rings as *-graded algebras]\label{bfghrtd5}
 \vspace{0.2cm}
 
Recall the corner skew Laurent polynomial ring $A=R[t_{+},t_{-},\phi]$, where $R$ is a ring with identity and $\phi:R\rightarrow pRp$ a corner isomorphism (see~\S\ref{cornerskew}). Let $R$ be a $*$-ring, $p$  a \emph{projection} (\ie $p=p^*=p^2$), and $\phi$  a $*$-isomorphism. Then $A$ has a $*$-involution defined 
on generators by $(t^j_{-}r_{-j})^*=r_{-j}^*t^j_{+}$ and $(r_it^i_{+})^*=t^i_{-}r_i^*$. With this involution $A$ becomes a 
graded $*$-ring.  \index{projection} A $*$-ring is called $*$-\emph{proper}, if $xx^*=0$ implies $x=0$. It is called \emph{positive-definite} if $\sum_i x_i{x_i}^*=0$ implies $x_i=0$ for all $i$. 
A graded $*$-ring is called \emph{graded $*$-proper}, if $x\in A^h$ and $xx^*=0$ then $x=0$.
The following Lemma is easy to observe and we leave it to the reader.  \index{proper} \index{graded proper} 
\index{positive-definite}  \index{*-proper ring}  \index{graded $*$-proper ring}
\end{example}  \index{corner skew Laurent polynomial ring}

\begin{lemma}

Let $R$ be a $*$-ring and $A=R[t_{+},t_{-},\phi]$ a $*$-corner skew Laurent polynomial ring. We have 
\begin{enumerate}

\item $R$ is positive-definite if and only if $A$ is positive-definite.

\item $R$ is proper if and only if $A$ is graded $*$-proper. 

\end{enumerate}

\end{lemma}

\begin{proof}

(1) Since $R$ is $*$-subring of $A$, if $A$ is positive-definite, then $R$ is so. 

For the converse, suppose $R$ is positive-definite and 
\begin{equation}\label{adel15}
\sum_{k=1}^l x_kx^*_k=0,
\end{equation}
 where $x_k \in A$. Write 
\[x_k=t^{j_k}_{-}r^k_{-{j_k}} +t^{{j_k}-1}_{-}r^k_{-{j_k}+1}+\dots+t_{-}r^k_{-{1}}+r^k_{0} +r^k_{1}t_{+}+\dots +r^k_{i_k}t^{i_k}_{+}.\]
It is easy to observe that the constant term of $x_kx^*_k$ is 
\[\phi^{-j_k}(r^k_{-{j_k}} {r^k}^*_{-{j_k}})+\dots +\phi^{-1}(r^k_{-1} {r^k}^*_{-1})+r^k_{0} {r^k}^*_{0} +r^k_{1} {r^k}^*_{1}+\dots+r^k_{i_k}{r^k}^*_{i_k}.\]
Now Equation~\ref{adel15} implies that the sum of these constant terms are zero. Since $R$ is positive-definite and $\phi$ is an $*$-isomorphism, it follows that all $r^k_{-j_k}$ and $r^k_{i_k}$ for $1\leq k \leq l$ are zero and thus $x_k=0$. This finishes the proof of (1). 

(2) The proof is similar to (1) and is left to the reader. 
\end{proof}




\chapter{Graded Morita Theory}\label{moritanji}

Starting from a right $A$-module $P$ one can construct a $6$-tuple $(A,P,P^*,B;\phi,\psi)$, where $P^*=\Hom_A(P,A)$, $B=\Hom_A(P,P)$, $\phi:P^*\otimes_B P \rightarrow A$ and $\psi:P \otimes_A P^* \rightarrow B$, which have appropriate bimodule structures. This is called the \emph{Morita context}  \index{Morita context} associated with $P$. When $P$ is a progenerator (\ie finitely generated projective and a generator), then one can show that 
$\Modd A \approx \Modd B$, i.e, the category of (right) $A$-modules is equivalent to the category of (right) $B$-modules. Conversely, if for two rings $A$ and $B$, $\Modd A\approx \Modd B$, then one can construct a $A$-progenerator $P$ such that $B\cong \End_A(P)$. The feature of this theory, called the Morita theory, is that the whole process involves working with $\Hom$ and tensor functors, both of which respect the grading. Thus starting from a graded ring $A$ and a graded progenerator $P$, and carrying out the Morita theory, one can naturally extending the equivalence from  $\Gr A\approx \Gr B$ to $\Modd A \approx \Modd B$ and vice versa. 

Extending the equivalence from the (sub)categories of graded modules to the categories of modules is not at the first glance obvious. Recall that two categories $\mathcal C$ and $\mathcal D$ are equivalent if and only if there is a functor $\phi:\mathcal C \rightarrow \mathcal D$ which is fully faithful and for any $D \in \mathcal D$, there is $C\in \mathcal C$ such that $\phi(C)=D$. 

Suppose $\phi:\Gr A\rightarrow \Gr B$ is a graded equivalence (see Definition~\ref{grdeffsa}). Then for two graded $A$-modules $M$ and $N$ (where $M$ is finitely generated) we have 
\begin{multline*}
\Hom_A(M,N)_0= \Hom_{\Gr A}(M,N) \cong \\ \Hom_{\Gr B}(\phi(M),\phi(N)) =\Hom_B(\phi(M),\phi(N))_0.
\end{multline*}

The fact that, this can be extended to \[\Hom_A(M,N) \cong \Hom_B(\phi(M),\phi(N)),\] is not immediate. We will show that this is indeed the case.

  In this chapter we study the equivalences between the categories of graded modules and their relations with the categories of modules. The main theorem of this chapter shows that for two graded rings $A$ and $B$, if $\Gr A\approx_{\gr}  \Gr B$ then $\Modd A\approx \Modd B$ (Theorem~\ref{grmorim}).  This was first studied by Gordon and Green~\cite{greengordon} in the setting of $\mathbb Z$-graded rings. 

Throughout this chapter, $A$ is a $\Gamma$\!-graded ring unless otherwise stated. Moreover, all functors are \emph{additive} functors. \index{additive functor}  
For the theory of (ungraded) Morita theory we refer the reader to~\cite{lam} and~\cite{anderson}

\section{First instance of the graded Morita equivalence}

Before describing the general graded Morita equivalence, we treat a special case of matrix algebras, which, in the word of T.Y. Lam, is the ``first instance'' of equivalence between module categories~\cite[\S17B]{lam}. This special case is quite explicit and will help us in calculating the graded Grothendieck group of matrix algebras. Recall from~\S\ref{matgrhe} that for a $\Gamma$\!-graded ring $A$,  and $\ol \delta = (\de_1 , \ldots ,\de_n)$
\begin{equation}\label{communicollges}
A^n(\overline \de)=A(\de_1)\oplus \dots \oplus A(\de_n),
\end{equation} 
is a graded free $A$-bimodule. Moreover, $A^n(\overline \de)$ is a graded right $\M_n(A)(\overline \de)$-module and $A^n(-\overline \de)$ is a graded left $\M_n(A)(\overline \de)$-module. Here  
$-\overline \de=(-\de_1, \dots,-\de_n)$.
\index{graded matrix ring}

\begin{proposition}\label{instancegg}
Let $A$ be a $\Gamma$\!-graded ring and let $\ol \delta = (\de_1 , \ldots ,
\de_n)$, where $\de_i \in \Ga$, $1\leq i \leq n$. Then the functors
\begin{align*}
\psi : \Gr \M_n(A)(\ol \delta)  & \longrightarrow \Gr A, \\
P & \longmapsto P  \otimes_{\M_n(A)(\ol \delta)} A^n(-\ol \delta) 
\end{align*}
and 
\begin{align*}
 \va : \Gr A & \longrightarrow \Gr \M_n(A)(\ol \delta), \\
Q &  \longmapsto Q \otimes_A A^n(\ol \delta)
\end{align*}
form equivalences of categories and commute with suspensions, i.e, $\psi \mathcal T_{\alpha}=\mathcal T_{\alpha} \psi$, $\alpha \in \Gamma$. 
\end{proposition}

\begin{proof}
One can check that there is a $\Gamma$\!-graded $A$-bimodule isomorphism 
\begin{align}\label{whuipp}
f : \; A^n(\ol\de)\otimes_{\M_n(A)(\ol\de)} A^n(-\ol \de)  &  \longrightarrow A, \\
(a_1 , \ldots , a_n ) \otimes \begin{pmatrix}
b_1 \\ \vdots \\ b_n
\end{pmatrix} & \longmapsto  a_1
b_1 + \cdots + a_n b_n  \notag
\end{align}
with
\begin{align*}
f^{-1}: A  & \longrightarrow A^n(\ol\de)\otimes_{\M_n(A)(\ol\de)} A^n(-\ol\de),\\
a & \longmapsto (a , 0 , \ldots , 0) \otimes \begin{pmatrix}
1 \\ 0\\ \vdots \\ 0
\end{pmatrix}.
\end{align*}

Moreover, there is a $\Gamma$\!-graded $\M_n(A)(\ol \de)$-bimodule isomorphism 
\begin{align}\label{whuipp2}
g : \; A^n(-\ol\de) \otimes_{A} A^n(\ol\de) &  \longrightarrow \M_n(A)(\ol\de),\\
\begin{pmatrix}
a_1 \\ \vdots \\ a_n
\end{pmatrix}
\otimes
(b_1 , \ldots , b_n )
& \longmapsto
\begin{pmatrix}
a_{1} b_1 & \cdots & a_{1} b_n  \\
\vdots  &   & \vdots  \\
a_{n} b_1  & \cdots & a_{n} b_n
\end{pmatrix} \notag
\end{align}
with
\begin{align*}
g^{-1} : \; \M_n(A)(\ol \de) & \lra A^n(-\ol \de) \otimes_{A} A^n( \ol \de), \\
(a_{i,j}) & \longmapsto
\begin{pmatrix}
a_{1,1} \\ a_{2,1} \\ \vdots \\ a_{n,1}
\end{pmatrix}
\otimes (1,0,\dots, 0)
+ \cdots +
\begin{pmatrix}
a_{1,n\phantom{-1}} \\ \vdots\phantom{-1} \\ a_{n-1,n} \\ a_{n,n\phantom{-1}}
\end{pmatrix}
\otimes (0,0,\dots, 1).
\end{align*}

Now using~(\ref{whuipp}) and~(\ref{whuipp2}),  it follows easily that $\varphi \psi$ and $\psi \varphi$ are equivalent to identity functors. The general fact that for $\alpha \in \Gamma$, \[(P\otimes Q) (\alpha)= P(\alpha)\otimes Q =P \otimes Q(\alpha),\] (see~\S\ref{grtensie}) shows that the suspension functor commutes with $\psi$ and $\phi$. 
\end{proof}

Since the functors $\phi$ and $\psi$, being tensor functors,  preserve the projectivity, and send finitely generated modules to finitely generated modules, we immediately  get the following corollary which we will use it to compute graded $K_0$ of matrix algebras (see Example~\ref{gaul1}).  

\begin{corollary}\label{grmorita}
Let $A$ be a $\Gamma$\!-graded ring and let $\ol \delta = (\de_1 , \ldots ,
\de_n)$, where $\de_i \in \Ga$, $1\leq i \leq n$. Then the functors
\begin{align*} 
\psi : \Pgrp \M_n(A)(\ol \delta)  & \longrightarrow \Pgrp  A, \\
P & \longmapsto P  \otimes_{\M_n(A)(\ol \delta)} A^n(-\ol \delta)
\end{align*}
and 
\begin{align*}
 \va : \Pgrp A & \longrightarrow \Pgrp  \M_n(A)(\ol \delta), \\
Q &  \longmapsto Q \otimes_A A^n(\ol \delta)
\end{align*}
form equivalences of categories and commute with suspensions, i.e, $\psi \mathcal T_{\alpha}=\mathcal T_{\alpha} \psi$, $\alpha \in \Gamma$. 
\end{corollary}

\begin{example}\label{hygfdbte65} \index{matrix units} 
 Let $\e_{ii}$, $1\leq i \leq n$, be the matrix unit in $\M_n(A)(\overline \de)$. Then $\e_{ii}\M_n(A)(\overline \de)$ is the graded finitely generated projective right $\M_n(A)(\overline \de)$-module. It is easy to see that there is a graded $\M_n(A)(\overline \de)$-module isomorphism 
\begin{align*} 
\e_{ii}\M_n(A)(\overline \de) &\longrightarrow A(\delta_1-\delta_i)\oplus A(\delta_2-\delta_i)\oplus \dots \oplus A(\delta_n-\delta_i), \label{gnytuhf43}\\
\e_{ii}X &\longmapsto  (x_{i1},x_{i2},\dots,x_{in}). \notag
\end{align*}
Thus by Proposition~\ref{instancegg}, the module $\e_{ii}\M_n(A)(\overline \de)$ in $\Gr \M_n(A)(\ol \delta)$ corresponds to $A(-\delta_i)$ in $\Gr A$ as follows. 
\begin{multline*}
 \e_{ii}\M_n(A)(\overline \de) \otimes_{\M_n(A)(\ol \delta)} A^n(-\ol \delta)  \cong_{\gr}  \\
A(\delta_1-\delta_i)\oplus A(\delta_2-\delta_i)\oplus \dots \oplus A(\delta_n-\delta_i)    \otimes_{\M_n(A)(\ol \delta)} A^n(-\ol \delta)  \cong_{\gr}  \\
 A^n(\ol \delta)(-\delta_i)\otimes_{\M_n(A)(\ol \delta)} A^n(-\ol \delta) \cong_{\gr}
  \big( A^n(\ol \delta)\otimes_{\M_n(A)(\ol \delta)} A^n(-\ol \delta)\big)(-\delta_i)\cong_{\gr}\\ A(-\delta_i).
  \end{multline*}
\end{example}

\begin{example}\label{hhmehe}\scm[Strongly graded is not a Morita invariant property]
\vspace{0.2cm}

Let $A$ be a strongly $\Gamma$\!-graded ring. One can see that $\M_n(A)(\alpha_1,\dots,\alpha_n)$, $\alpha_i \in \Gamma$, is also strongly graded. However the converse is not true. For example, let $K$ be a field and consider the $\mathbb Z$-graded ring $K[x^2,x^{-2}]$ which has  the support $2\mathbb Z$.  By Proposition~\ref{instancegg},   
$A=\M_2(K[x^2,x^{-2}])(0,1)$ is graded Morita equivalent to $B=K[x^2,x^{-2}]$. However, one can easily see that $A$ is strongly graded  whereas $B$ is not. 

One can also observe that although $A$ and $B$ are graded Morita equivalent, the support set of $A$ is $\mathbb Z$, whereas the support of $B$ is $2\mathbb Z$ (see~(\ref{kkjjhhs})). 
\end{example}

\begin{remark}\scm[Restricting the equivalence to a subgroup of grade group]
\vspace{0.2cm}

Let $A$ be a $\Gamma$\!-graded ring and let $\ol \delta = (\de_1 , \ldots ,
\de_n)$, where $\de_i \in \Ga$, $1\leq i \leq n$. Suppose $\Omega$ is a subgroup of $\Gamma$ such that $\Gamma_A \subseteq \Omega$ and $\{\delta_1,\dots,\delta_n\} \subseteq \Omega$. 
Since the support of $\M_n(A)(\ol \delta)$ is a subset of $\Omega$ (see~(\ref{kkjjhhs})), we can naturally consider $A$ and $\M_n(A)(\ol \delta)$ as $\Omega$-graded rings (see Remark~\ref{presipe}). Since the $\Gamma$\!-isomorphisms (\ref{whuipp}) and (\ref{whuipp2}) in the proof of Proposition~\ref{instancegg} can be considered as $\Omega$-isomorphism, 
we can easily see that  $\Gr[\Omega] \M_n(A)(\ol \delta)$ is equivalent to $\Gr[\Omega] A$ as well.  The converse of this statement is valid for the general graded Morita Theory, see Remark~\ref{maziho}.  
\end{remark}

\begin{remark}\scm[The Morita equivalence for non-abelian grade group]
\vspace{0.2cm}

As mentioned in the Introduction, most of the results of this book can be carried over to the setting of non-abelian grade groups as well. In some cases this will be done by adjusting the arrangements and shifts. Below we adjust the arrangements so that Proposition~\ref{instancegg} holds for the non-abelian grade group as well.

Let $\Gamma$ be a (non-abelian) group and $A$ be a $\Gamma$\!-graded ring. Further let $M$ and $N$ be $\Gamma$\!-graded left and right $A$-modules, respectively. Define the $\de$-shifted graded left and right $A$-modules $M(\de)$ and $(\de)N$, respectively,  as \index{shifted}
\begin{align*}
M(\de) & =\bigoplus_{\ga \in \Ga} M(\de)_\ga, \text{ where }M(\de)_\ga =
M_{\ga\de},\\
(\de)N & =\bigoplus_{\ga \in \Ga} (\de)N_\ga, \text{ where }(\de)N_\ga =
N_{\de\ga}.
\end{align*}
One can easily check that $M(\delta)$ and $(\delta)N$ are left and right graded $A$-modules, respectively, $M(\delta)(\gamma)=M(\gamma \delta)$ and $(\gamma)(\delta)N=(\delta\gamma)N$. Note that with this arrangement, for $\ol \delta = (\de_1 , \ldots ,\de_n)$ and $\ol \de^{-1}=(\de_1^{-1} , \ldots ,\de_n^{-1})$, $(\ol \de) A^n$ is an $A-\M_n(A)(\ol\de)$-bimodule and 
$A^n(\ol\de^{-1})$ is an $\M_n(A)(\ol\de)-A$-bimodule, where $\M_n(A)(\ol\de)$ is as in Remark~\ref{fre298}. Moreover, 
the isomorphisms~(\ref{whuipp}) and~(\ref{whuipp2}) take the form
\begin{align*}
f : \; A^n(\ol\de^{-1})\otimes_{\M_n(A)(\ol\de)} (\ol \de) A^n  &  \longrightarrow A, \\
(a_1 , \ldots , a_n ) \otimes \begin{pmatrix}
b_1 \\ \vdots \\ b_n
\end{pmatrix} & \longmapsto  a_1
b_1 + \cdots + a_n b_n  \notag
\end{align*}
and 
\begin{align*}
g : \; (\ol\de) A^n \otimes_{A} A^n(\ol\de^{-1}) &  \longrightarrow \M_n(A)(\ol\de),\\
\begin{pmatrix}
a_1 \\ \vdots \\ a_n
\end{pmatrix}
\otimes
(b_1 , \ldots , b_n )
& \longmapsto
\begin{pmatrix}
a_{1} b_1 & \cdots & a_{1} b_n  \\
\vdots  &   & \vdots  \\
a_{n} b_1  & \cdots & a_{n} b_n
\end{pmatrix}. \notag
\end{align*}
Thus we can write Proposition~\ref{instancegg} in the non-abelian grade group setting. 
\end{remark}

\section{Graded generators}

We start with a categorical definition of graded generators. In this section, as usual, the modules are (graded) right $A$-modules. 

\begin{definition}\label{grgen}
Let $A$ be a $\Gamma$\!-graded ring. A graded $A$-module $P$ is called a \emph{graded generator} \index{graded generator} if whenever $f:M\rightarrow N$ is a nonzero graded $A$-module homomorphism, then there exists $\alpha \in \Gamma$ and a graded homomorphism $g:P(\alpha) \rightarrow M$ such that $fg:P(\alpha) \rightarrow N$ is a nonzero map. 
\end{definition}

Let $A$ be a $\Gamma$\!-graded ring and $P$ be a graded right $A$-module. Define the \emph{graded trace ideal} of $P$ as follows \index{graded trace ideal}
\[\boxed{\Tr^{\gr} (P):=\Big \{ \, \sum_i f_i(p_i) \mid f_i \in \Hom(P,A)_\alpha, \alpha \in \Gamma, \, p_i \in P^h \, \Big \}.} \]

One can check that $\Tr^{\gr}(P)$ is a graded two-sided ideal of $A$.  The following theorem provides a set theoretic way to define a graded generator.

\begin{theorem}\label{grgenth}
For any graded $A$-module $P$, the following are equivalent:

\begin{enumerate}[\upshape(1)]

\item $P$ is a graded generator;

\item $\Tr^{\gr}(P)=A$;

\item $A$ is a direct summand of a finite direct sum  $\bigoplus_i P(\alpha_i)$, where $\alpha_i \in \Gamma$;

\item $A$ is a direct summand of a direct sum  $\bigoplus_i P(\alpha_i)$, where $\alpha_i \in \Gamma$;

\item Every graded $A$-module $M$ is a homomorphic image of  $\bigoplus_i P(\alpha_i)$, where $\alpha_i \in \Gamma$. 
\end{enumerate}
\end{theorem}

\begin{proof}
(1) $\Rightarrow$ (2) First note that 
\[\Tr^{\gr} (P)=\Big \{ \, \sum_i f_i(P(\alpha)) \mid f_i \in \Hom_{A \rGr} \big(P(\alpha),A\big),   \alpha \in \Gamma \, \Big \}. \]
Suppose that $\Tr^{\gr}(P) \not = A$. Then the graded canonical projection  \[f:A\rightarrow A/ \Tr^{\gr}(P)\] is not a zero map. Since $P$ is graded generator, there is 
$g \in \Hom_{A \rGr} (P(\alpha),A)$ such that $fg$ is not zero. But this implies $g(P(\alpha)) \not \subseteq \Tr^{\gr} (P)$ which is a contradiction. 

(2) $\Rightarrow$ (3) Since $\Tr^{\gr}(P)=A$, one can find $g_i \in \Hom_{A \rGr}( P(\alpha_i),A)$, $\alpha_i \in \Gamma$, $1\leq i \leq n$ such that $\sum_{i=1}^n g_i(P(\alpha_i))=A$. Consider the graded $A$-module epimorphism 
\begin{align*}
\bigoplus_{i=1}^n P(\alpha_i) &\longrightarrow A,\\
(p_1,\dots,p_n)&\longmapsto \sum_{i=1}^n g_i(p_i).
\end{align*}
 Since $A$ is graded projective, this map splits. Thus \[\bigoplus_{i=1}^n P(\alpha_i)\cong_{\gr} A\oplus Q,\] for some graded $A$-module $Q$. This gives (3). 

(3) $\Rightarrow$ (4) This is immediate. 

(4) $\Rightarrow$ (5) Since any module is a homomorphic image of a graded free $A$-module, there is a graded epimorphism $\bigoplus_j A(\alpha_j') \rightarrow M$. By (3) there is a graded epimorphism $\bigoplus_i P(\alpha_i)  \rightarrow A$, so an epimorphism \[\bigoplus_i P(\alpha_i+\alpha_j')  \rightarrow A(\alpha_j').\] Therefore  
\[\bigoplus_j \bigoplus_i P(\alpha_i+\alpha_j') \rightarrow \oplus_j A(\alpha_j') \rightarrow M\] is a graded epimorphism. 

(5) $\Rightarrow$ (1) Let $f:M \rightarrow N$ be a nonzero graded homomorphism. By (4) there is an epimorphism $\bigoplus_i P(\alpha_i) \rightarrow M$. So the composition $\bigoplus_i P(\alpha_i) \rightarrow M \rightarrow N$ is not zero. This immediately implies there is an $i$ such that the composition $P(\alpha_i) \rightarrow M \rightarrow N$ is not zero. This gives (1). 
\end{proof}

\begin{example}\label{BITedu}
Let $A$ be a graded simple ring and $P$ be a graded projective $A$-module. Since $\Tr^{\gr}(P)$ is a graded two-sided ideal, one can easily show, using Proposition~\ref{grgenth}(2), that $P$ is a graded generator. 

\end{example}

Recall that in the category of right  $A$-modules, $\Modd A$, a generator is defined as in Definition~\ref{grgen} by dropping all the graded adjective.  Moreover a similar theorem as in~\ref{grgenth} can be written in the ungraded case, by considering $\Gamma$ to be the trivial group (see~\cite[Theorem~18.8]{lam}). In particular $P$ is a generator if and only if $\Tr(P)=A$, where 
\[\Tr (P):=\Big \{ \, \sum_i f_i(p_i) \mid f_i \in \Hom_A(P,A), \, p_i \in P \, \Big \}. \]

\begin{remark}\label{eximans}
Considering $\Gr A$ as a subcategory of $\Modd A$ (which is clearly not a full subcategory), one can define  generators for $\Gr A$. In this case one can easily see that $\bigoplus_{\gamma \in \Gamma} A(\gamma)$ is a generator for the category $\Gr A$. Recall that, in comparison, $A$ is a generator for $\Modd A$ and  $A$ is a graded generator for $\Gr A$. 
\end{remark}

An $A$-module $P$ is called a \emph{progenerator} \index{progenerator module} if it is both a finitely generated projective and a generator, \ie there is an $n\in \mathbb N$ such that $A^n\cong P\oplus K$ and $P^n\cong A\oplus L$, where $K$ and $L$ are $A$-modules. Similarly, a graded $A$-module is called a \emph{graded progenerator} \index{graded progenerator module} if it is both a graded finitely generated projective and a graded generator. 

 Reminiscent of the case of graded projective modules (Proposition~\ref{grprojectivethm}), we have the following relation between generators and graded generators. 

\begin{theorem}\label{prograded}
Let $P$ be a finitely generated  $A$-module.  Then $P$ is a graded generator if and only if $P$ is graded and is a generator. 
\end{theorem}
\begin{proof}
Suppose $P$ is a graded generator. By Theorem~\ref{grgenth}, $\Tr^{\gr}(P)=A$ which implies that there are graded homomorphisms $f_i$ (of possibly different degrees) in  $\Hom_{A \rGr}( P(\alpha_i),A)$ and $p_i \in P^h$, such that $\sum_i f_i(p_i)=1$. This immediately implies $\Tr(P)$, being an ideal, is $A$. Thus $P$ is  a generator. 

Conversely, suppose $P$ is graded and is a generator. Thus there are homomorphisms $f_i$ in $\Hom(P,A)$ and $p_i \in P$ such that  $\sum_i f_i(p_i)=1$. Since $P$ is finitely generated, by Theorem~\ref{crazhorn}, $f_i$ can be written as a sum of graded homomorphisms, and $p_i$ as sum of homogeneous elements in $P$. This shows $1\in \Tr^{\gr}(P)$. Since $\Tr^{\gr}(P)$ is an ideal, $\Tr^{\gr}(P)=A$ and so $P$ is a graded generator by Theorem~\ref{grgenth}. 
\end{proof}

\section{General graded Morita equivalence} \label{meinghto}

Let $P$ be a right $A$-module. Consider the ring $B=\Hom_A(P,P)$. Then $P$ has a natural $B\!-\!A$-bimodule structure. The actions of $A$ and $B$ on $P$ are defined by  $p.a=pa$ and $g.p=g(p)$, respectively, where $g\in B$, $p\in P$ and $a\in A$. Consider the dual  $P^*=\Hom_A(P,A)$. Then $P^*$ has a natural $A\!-\!B$-bimodule structure. The actions of $A$ and $B$ on $P^*$  defined by $(a.q)(p)=aq(p)$ and $q.g=q \circ g$, respectively, where $g\in B$, $q\in P^*$, $p\in P$ and $a\in A$. Moreover, one defines 
\begin{align}\label{bsbsbs1}
\phi:P^* \otimes_B P &\longrightarrow A,\\
q\otimes p &\longmapsto q(p), \notag
\end{align}
 and 
 \begin{align}\label{bsbsbs2}
  \psi:P\otimes_A P^* &\longrightarrow B\\
  p\otimes q &\longmapsto pq, \notag
  \end{align}
  where $pq(p')=p(q(p'))$. One can check that $\phi$ is a $A\!-\!A$-bimodule homomorphism and $\psi$ is a $B\!-\!B$-bimodule homomorphism. We leave it to the reader to check these and that with the actions introduced above, $P$ has a $B\!-\!A$-bimodule structure and $P^*$ has a $A\!-\!B$-bimodule structure (see~\cite[\S18C]{lam}). These compatibility conditions amount to the fact that 
the \emph{Morita ring}  \index{Morita ring} 
\begin{equation}\label{messiah}
M=\left(\begin{matrix} A & P^*\\ P & B\phantom{*} \end{matrix}\right),
\end{equation} with the matrix multiplication has  the associativity property. In fact $M$ is a formal matrix ring as defined in Example~\ref{atsusan}. 
\index{formal matrix ring}

As part of the Morita theory, one proves that when $P$ is a generator, then $\phi$ is an isomorphism. Similarly, if $P$ is finitely generated and projective, then $\psi$ is an isomorphism (see~\cite[\S18C]{lam}). Putting these facts together, it is an easy observation that 
\begin{equation}\label{fgts6}
-\otimes_A P^*:\Modd A\rightarrow \Modd B, \quad \text{  and  }  \quad -\otimes_B P:\Modd B \rightarrow \Modd A, 
\end{equation}
 are inverse of each other and so these two categories are (Morita) equivalent.  \index{graded formal matrix ring}

If $P$ is a graded finitely generated  right $A$-module, then by Theorem~\ref{crazhorn}, $B=\End_A(P,P)$ is also a graded ring and $P^*$ a graded left $A$-module. In fact, one can easily check that with the actions defined above, $P$ is a graded $B\!-\!A$-bimodule, $P^*$ is a graded $A-B$-module and similarly $\phi$ and $\psi$ are graded $A-A$ and $B-B$-module homomorphisms, respectively. The Morita ring $M$ of~(\ref{messiah}) is a graded formal matrix ring (see Example~\ref{gratsusan}), with 
\begin{equation}\label{peacenn}
M_{\alpha}=\left(\begin{matrix} A_{\alpha} & P^*_{\alpha}\\ P_{\alpha} & B_{\alpha} \end{matrix}\right)=\left(\begin{matrix} A_{\alpha} & \Hom_{\Gr A}(P,A(\alpha))\\ P_{\alpha} & \Hom_{\Gr A}(P,P(\alpha)) \end{matrix}\right).
\end{equation}

We demonstrate here that $P^*$ is a graded $A\!-\!B$-bimodule and leave the others which are similar and routine to the reader. Recall that \[P^*=\bigoplus_{\alpha\in \Gamma} \Hom_A(P,A)_{\alpha}.\]  Let $a\in A_{\alpha}$ and 
$q\in P^*_{\beta}=\Hom_A(P,A)_{\beta}$. Then $a.q\in P^*$, where $(a.q)(p)=aq(p)$. If $p \in P_{\gamma}$, then one easily sees that $(a.q)(p) \in A_{\alpha+\beta+\gamma}$.  This shows that $aq \in \Hom_A(P,A)_{\alpha+\beta}=P^*_{\alpha+\beta}$. On the other hand, if 
$q\in P^*_{\alpha}$ and $g \in \Hom(P,P)_{\beta}$, then $q.g \in P^*$, where $q.g(p)=qg(p)$. So if $p \in P_{\gamma}$ then $(q.g)(p)=qg(p) \in  A_{\alpha+\beta+\gamma}$. So $q.g \in P^*_{\alpha+\beta}$ as well. 

\begin{remark}\scm[A natural groupoid graded ring]\label{meinbed} \index{groupoid grading}
\vspace{0.2cm}

Consider the Morita ring $M$ of~(\ref{messiah}) with homogeneous components~(\ref{peacenn}). Further, consider the following homogeneous elements 
\begin{equation*}
m_\gamma=\begin{pmatrix}
a_\gamma & 0  \\0 & 0
\end{pmatrix} \in M_\gamma, \,\,\,  
m_\delta=\begin{pmatrix}
0 &0  \\ p_\delta & 0
\end{pmatrix}  \in M_\delta. 
\end{equation*}
By definition $m_\gamma m_\delta \in M_{\gamma +\delta}$. However $m_\gamma m_\delta=0$. It would be more informative to recognise this from the outset. Thus we introduce a groupoid grading on the Morita ring $M$ in~(\ref{messiah}) with a groupoid graded structure which gives a ``finer'' grading than one defined in~(\ref{peacenn}) (see Remark~\ref{ggtgt} for the groupoid graded rings). Consider the groupoid  $2\times \Gamma\times 2$ and for $\gamma \in \Gamma$, set
\begin{equation*}
M_{(1,\gamma,1)}=\begin{pmatrix}
A_\gamma & 0 \\0 & 0
\end{pmatrix}, 
M_{(1,\gamma,2)} =\begin{pmatrix}
0 & P^*_\gamma \\0 & 0
\end{pmatrix}, 
M_{(2,\gamma,1)}=\begin{pmatrix}
0 & 0 \\P_\gamma & 0
\end{pmatrix},   
M_{(2,\gamma,2)} =\begin{pmatrix}
0 & 0 \\0 & B_\gamma
\end{pmatrix}.
\end{equation*}
Clearly for $\gamma \in \Gamma$ 
\[ M_\gamma= \bigoplus_{1\leq i,j\leq 2} M_{(i,\gamma,j)},\]
and 
\[M= \bigoplus_{g\in 2\times \Gamma\times 2} M_g.\]
This makes $M$ a $2\times \Gamma\times 2$-groupoid graded ring. 
\end{remark}

By Theorem~\ref{prograded} and Proposition~\ref{grprojectivethm}, if $P$ is a graded finitely generated projective and graded generator, then $P$ is finitely generated projective and a generator. Since the grading is preserved under the tensor products, the restriction of the functors $-\otimes_A P^*$ and $-\otimes_B P$ of~(\ref{fgts6}) induce an equivalence 
\begin{equation}\label{fgts688}
-\otimes_A P^*:\Gr A\longrightarrow \Gr B, \quad \text{  and  }  \quad -\otimes_B P:\Gr B \longrightarrow \Gr A. 
\end{equation}
Moreover, these functors commute with suspensions. Thus we get a commutative diagram where the vertical maps are forgetful functor (see~\S\ref{forgetful}), 
\begin{equation}
\xymatrix{
\Gr A \ar[rr]^{-\otimes_A P^*} \ar[d]_{U}&& \Gr B \ar[d]^{U}\\
\Modd A \ar[rr]^{-\otimes_A P^*}  && \Modd B.
}
\end{equation}
We call $-\otimes_A P^*$ a graded equivalence functor (see Definition~\ref{grdeffsa}).

\begin{example}\label{idempogr}
Let $A$ be a graded ring and $e$ be a \emph{full homogeneous idempotent} \index{full homogeneous idempotent} of $A$, \ie $e$ is a homogeneous element, $e^2=e$ and $AeA=A$. Clearly $e$ has  degree zero. 
Consider $P=eA$. One can readily see that $P$ is a graded right progenerator. Then $P^*=\Hom_A(eA,A)\cong_{\gr}Ae$ as graded left $A$-modules and $B=\End_A(eA)\conggr eAe$ as graded rings (see~\S\ref{idemptis}).  The maps $\phi$ and $\psi$ described above as a part of Morita context ((\ref{bsbsbs1}) and~(\ref{bsbsbs2})), takes the form $\phi:Ae\otimes_{eAe}eA \rightarrow A$ and $\psi:eA\otimes_A Ae \rightarrow eAe$ which are graded isomorphisms. Thus we get an (graded) equivalence between $\Gr A$ and $\Gr eAe$ which lifts to an (graded) equivalence between $\Modd A$ and $\Modd eAe$, as it is shown in the diagram below.
\begin{equation}\label{hgpian}
\xymatrix{
\Gr A \ar[rr]^-{-\otimes_A Ae} \ar[d]_{U}&& \Gr eAe \ar[d]^{U}\\
\Modd A \ar[rr]^-{-\otimes_A Ae}  && \Modd eAe.
}
\end{equation}
\end{example}

Before stating the general graded Morita equivalence, we need to make some definitions. Recall from \S\ref{pregr529} that for $\alpha \in \Gamma$,  the $\alpha$-suspension functor $\mathcal T_\alpha: \Gr A\rightarrow \Gr A$, $M \mapsto M(\alpha)$  is an isomorphism with the property
$\mathcal T_\alpha \mathcal T_\beta=\mathcal T_{\alpha + \beta}$, $\alpha,\beta\in \Gamma$.

\begin{definition} \label{grdeffsa} 
Let $A$ and $B$ be $\Gamma$\!-graded rings. 

\begin{enumerate}

\item A functor $\phi:\Gr A \rightarrow \Gr B$ is called a \emph{graded functor} \index{graded functor} if $\phi \mathcal T_{\alpha} = \mathcal T_{\alpha} \phi$. 

\medskip 

\item A graded functor $\phi:\Gr A \rightarrow \Gr B$ is called a \emph{graded equivalence} \index{graded equivalence} if there is a graded functor $\psi:\Gr B \rightarrow \Gr A$ such that $\psi \phi \cong 1_{\Gr A}$ and $\phi \psi \cong 1_{\Gr B}$. 

\medskip

\item If there is a graded equivalence between $\Gr A$ and $\Gr B$, we say $A$ and $B$ are \emph{graded equivalent} or \emph{graded Morita equivalent} and we write \index{graded Morita equivalent}
$\Gr A \approx_{\gr} \Gr B$, or $\Gr[\Gamma] A \approx_{\gr} \Gr[\Gamma] B$ to emphasis the categories are $\Gamma$\!-graded.

\medskip 

\item A functor $\phi': \Modd A \rightarrow \Modd B$ is called a \emph{graded functor} \index{graded functor} if  there is a graded functor $\phi: \Gr A \rightarrow \Gr B$ such that the following diagram, where the vertical functors are forgetful functors (see~\S\ref{forgetful}), is commutative.
\begin{equation}\label{njhysi}
\xymatrix{
\Gr A \ar[r]^-{\phi} \ar[d]_{U}& \Gr B \ar[d]^{U}\\
\Modd A \ar[r]^-{\phi'}  & \Modd B.
}
\end{equation}

The functor $\phi$ is called an \emph{associated graded functor} of $\phi'$. \index{associated graded functor}
\medskip 

\item A functor $\phi:\Modd A \rightarrow \Modd B$ is called a \emph{graded equivalence} \index{graded equivalence} if it is graded and an equivalence. 
\end{enumerate}
\end{definition}

Definition~\ref{grdeffsa} of graded functors is formulated for the category of (graded) right modules. A similar definition can be written for the category of (graded) left modules. We will see that the notion of graded equivalence is a left-right symmetric 
(see Remark~\ref{hgyhyh1}).

\begin{remark}
Note that although we require the graded functors commute with the suspensions, we don't require the natural transformations between these functors have these properties. 
\end{remark}

\begin{example}
The equivalence between $\Modd A$ and $\Modd eAe$ in Example~\ref{idempogr} is a graded equivalence as it is demonstrated in Diagram~\ref{hgpian}. 
\end{example}

If $Q$ is an object in $\Gr A$, then we denote $U(Q)\in \Modd A$ also by $Q$, forgetting its graded structure (see~\S\ref{forgetful}). Also when working with graded matrix rings, say $\M_n(A)(\overline \delta)$, we write simply $\Modd \M_n(A)$ when considering the category of (ungraded) modules over this matrix algebra.
These should not cause any confusion in the text. 


\begin{example}\label{hfjdla}
Proposition~\ref{instancegg} can be written for the category of $A$-modules, which gives us a ungraded version of Morita equivalence. We have the commutative diagram 
\begin{equation*}
\begin{split}
\xymatrix{
\Gr A \ar[rr]^-{-\otimes_A A^n(\overline \delta)} \ar[d]_{U}&& \Gr \M_n(A)(\overline \delta) \ar[d]^{U}\\
\Modd A \ar[rr]^-{-\otimes_A A^n}  && \Modd \M_n(A).
}
\end{split}
\end{equation*}
which shows that the functor \[-\otimes_A A^n:\Modd A \rightarrow \Modd \M_n(A)\] is a graded equivalence. 

\end{example}

We are in a position to state the main Theorem of this chapter. 
\begin{theorem}\label{grmorim} 
Let $A$ and $B$ be two $\Gamma$\!-graded rings. 
Let $\phi:\Gr A \rightarrow \Gr B$ be a graded equivalence. Then there is a graded equivalence $\phi':\Modd A \rightarrow \Modd B$ with an associated graded functor  isomorphic to $\phi$. Indeed, there is a graded $A\!-\!B$-bimodule $Q$ such that $\phi\cong -\otimes_A Q$ and consequently the following diagram commutes. 
\begin{equation*}
\begin{split}
\xymatrix{
\Gr A \ar[rr]^-{-\otimes_A Q} \ar[d]_{U}&& \Gr B\ar[d]^{U}\\
\Modd A \ar[rr]^-{-\otimes_A Q}  && \Modd B.
}
\end{split}
\end{equation*}
\end{theorem}

\begin{proof}

Suppose the graded functor $\psi: \Gr B \rightarrow \Gr A$ is an inverse of the functor $\phi$ with \[f:\psi\phi \cong 1_{\Gr A}.\] 

Since $B$ is a graded finitely generated projective and a graded generator in $\Gr B$, it follows that $P=\psi(B)$ is a graded finitely generated projective  and a graded generator in $\Gr A$. Thus by Theorem~\ref{prograded} and Proposition~\ref{grprojectivethm}, $P$ is a finitely generated projective and a generator in $\Modd A$ as well. This shows that $\Hom_A(P,-):\Modd A \rightarrow \Modd B$ is a graded equivalence. We will show that $\phi \cong  \Hom_A(P,-)$ on the category of $\Gr A$. 

By Theorem~\ref{crazhorn}, $\Hom_B(B,B)=\bigoplus_{\alpha \in \Gamma}\Hom_B(B,B)_{\alpha}$, 
and by~(\ref{googoo1}) we can write $$\Hom_B(B,B)_{\alpha}=\Hom_{\Gr B}(B,B(\alpha)).$$ 
Applying $\psi$ on each of these components, since $\psi$ is a graded functor, we get a group homomorphism 
\begin{equation}
  \Hom_B(B,B)=\bigoplus_{\alpha \in \Gamma}\Hom_B(B,B)_{\alpha} \stackrel{\psi}{\longrightarrow}
\bigoplus_{\alpha \in \Gamma}\Hom_A(P,P)_{\alpha}=\Hom_A(P,P).
\end{equation} 
One can immediately see that this in fact gives a graded isomorphism of rings between $\Hom_B(B,B)$ and $\Hom_A(P,P)$. 

For any $b \in B$ consider the right $B$-module homomorphism 
\begin{align*}
\eta_b:B&\longrightarrow B\\
x&\longmapsto bx.
\end{align*}
Then the regular representation map $\eta:B \rightarrow \Hom_B(B,B)$, $\eta(b)=\eta_b$, is a graded isomorphism of rings. 
Thus we have a graded isomorphism of rings 
\begin{equation}\label{jjhhgg}
B\stackrel{\eta}{\longrightarrow} \Hom_B(B,B) \stackrel{\psi}{\longrightarrow}\Hom_A(P,P),
\end{equation} where $P$ is a graded $A$-progenerator.  

Now since for any graded $A$-module $X$, \[\Hom_A(P,X)\] is a graded right $\Hom_A(P,P)$-module, the isomorphisms~\ref{jjhhgg} induces a graded $B$-module structure on $\Hom_A(P,X)$. Namely, for homogeneous elements $b\in B$ and $t\in \Hom_A(P,X)$
 \begin{equation}\label{luke}
t.b=t \psi(\eta_b)
\end{equation}
which extends linearly to all elements.

We show that $\phi \cong \Hom_A(P,-)$. Let $X$ be a graded $A$-module. 
Then 
\begin{multline}\label{tetea}
\phi(X) \conggr \\ \Hom_B(B,\phi(X))=\bigoplus_{\alpha \in \Gamma}\Hom_B(B,\phi(X))_{\alpha}  \stackrel{\psi}{\longrightarrow}
\bigoplus_{\alpha \in \Gamma}\Hom_A(P,\psi\phi(X))_{\alpha}\\  \stackrel{\Hom(1,f)}{\longrightarrow} \bigoplus_{\alpha \in \Gamma}\Hom_A(P,X)_{\alpha}=\Hom_A(P,X).
\end{multline}
This shows that $\phi(X)$ is isomorphic to $\Hom_A(P,X)$ as a graded abelian groups. We need to show that this isomorphism, call it $\Theta$, is a  $B$-module isomorphism. Let $z\in \phi(X)$ and $b \in B$ be homogeneous elements. Since \[\phi(X) \conggr \Hom_B(B,\phi(X)),\] for $z\in \phi(X)$,  we denote the corresponding homomorphism in $\Hom_B(B,\phi(X))$ also by $z$.  
Then

\[\Theta(z.b)=\Theta(z \ \eta_b)=\Hom(1,f)\psi(z  \ \eta_b)= f \psi(z)\psi(\eta_b).\]
But 
\[\Theta(z).b=\big(\Hom(1,f)\psi(z)\big).b=f \psi(z)\psi(\eta_b) \qquad  (\text{by (\ref{luke})}). \] 
Thus $\phi \cong  \Hom_A(P,-)$ on $\Gr A$. Now consider \[\phi'= \Hom_A(P,-):\Modd A \rightarrow \Modd B.\] This gives the first part of the Theorem. 

Setting $P^*=\Hom_A(P,A)$, one can easily check that for any graded right $A$-module $M$, the map 
\begin{align*}
M\otimes_A P^* &\longrightarrow \Hom_A(P,M),\\
m\otimes q &\longmapsto (m\otimes q)(p)=mq(p),
\end{align*}
where $m\in M$, $q\in P^*$ and $p\in P$, is a graded $B$-homomorphism (recall that by~(\ref{jjhhgg}), $B\conggr \End_A(P)$).  Since $P$ is a graded progenerator, by Theorem~\ref{prograded}, it a progenerator, which in return gives that the above homomorphism is in fact an isomorphism (see~\cite[Remark~18.25]{lam}).
Thus $\phi \cong \Hom_A(P,-) \cong -\otimes_A P^*$. This gives the second part of the theorem. 
\end{proof}

\begin{theorem}\label{grmorim11} Let $A$ and $B$ be two $\Gamma$\!-graded rings. The following are equivalent:

\begin{enumerate}[\upshape(1)]

\item $\Modd A$ is graded equivalent to $\Modd B$;


\item $\Gr A$ is graded equivalent to $\Gr B$;

\item $B\cong_{\gr} \End_A(P)$ for a graded $A$-progenerator $P$;

\item $B\cong_{\gr} e \M_n(A)(\overline \delta) e$ for a full homogeneous idempotent $e \in \M_n(A)(\overline \delta)$, where $\overline \delta=(\delta_1,\dots,\delta_n)$, $\delta_i \in \Gamma$.


\end{enumerate}
 \end{theorem}
\begin{proof}
(1) $\Rightarrow$ (2) Let $\phi:\Modd A \rightarrow \Modd B$ be a graded equivalence. Using~(\ref{njhysi}), it follows that $\phi(A)=P$ is a graded right $B$-module. Also a similar argument as in~(\ref{jjhhgg}) shows that $P$ is a graded left $A$-module. Since $\phi$ is an equivalence, from the  (ungraded) Morita theory  it follows that $\phi\cong -\otimes_A P$ with an inverse $-\otimes_B P^*$.  Since the tensor product respects the grading, the same functor $\phi$ induces a graded equivalence between $\Gr A$ and $\Gr B$. 

(2) $\Rightarrow$ (3) This is~(\ref{jjhhgg}) in the proof of Theorem~\ref{grmorim}. 

(3) $\Rightarrow$ (4) Since $P$ is a graded finitely generated projective $A$-module, \[P\oplus Q \cong_{\gr} A^n(-\overline \delta),\] where $\overline \delta=(\delta_1,\dots,\delta_n)$ and $n \in \mathbb N$.   Let \[e\in \End_A(A^n(-\overline \delta))\cong_{\gr}\M_n(A)(\overline \delta)\] be the graded homomorphism  which sends $Q$ to zero and acts as identity on $P$. 
Thus \[e\in  \End_A(A^n(-\overline \delta))_0 \cong \M_n(A)(\overline \delta)_0\] and $P=eA^n(-\overline \delta)$. Define the map 
\[ \theta: \End_A(P) \longrightarrow e  \End_A(A^n(-\overline \delta)) e
\] by $\theta(f)|_{P}=f$ and $\theta(f)|_{Q}=0$. Since $e$ is homogeneous of degree zero, it is straightforward to see that this is a graded isomorphism of rings (which preserve the identity). Thus 
\[B\cong_{\gr} \End_A(P) \conggr e  \End_A(A^n(-\overline \delta)) e \conggr e\M_n(A)(\overline \delta)e.\]
We are left to show that $e$ is full. 
 By~\cite[Exercise~2.8]{lam}, \[\M_n(A)e\M_n(A)=\M_n(\Tr(P)),\] where $P=eA^n$. Since $P$ is graded progenerator, it is a finitely generated projective $A$-module and a generator (see Theorem~\ref{prograded}), thus $\Tr(P)=A$ and therefore $e$ is a (homogeneous) full idempotent.

(4) $\Rightarrow$ (1) Example~\ref{idempogr} shows that there is a graded equivalence between $\Modd \M_n(A)$ and $\Modd e\M_n(A)e$. On the other hand, Example~\ref{hfjdla} shows that there is a graded equivalence between $\Modd \M_n(A)$ and $\Modd A$. This finishes the proof. 
\end{proof}

\begin{example}\scm[$\Gr A\cong \Gr B$ does not imply $\Gr A \conggr \Gr B$]
\vspace{0.2cm}

One can easily construct examples of two $\Gamma$\!-graded rings $A$ and $B$ such that the categories $\Gr A$ and $\Gr B$ are equivalent, but not graded equivalent, \ie  the equivalent functors do not commute with suspensions (see Definition~\ref{grdeffsa}). Let $A$ be a $\Gamma$\!-graded ring and $\phi:\Gamma \rightarrow \Aut(A)$ be a group homomorphism. Consider the group ring $A[\Gamma]$ and the skew group ring $A\star_{\phi} \Gamma$ (see \S\ref{khgfewa1}). These rings are strongly $\Gamma$\!-graded, and thus by Dade's Theorem~\ref{dadesthm}, $\Gr A[\Gamma]$ and $\Gr A\star_{\phi} \Gamma$ are equivalent to $\Modd A$ and thus 
$\Gr A[\Gamma]\cong \Gr A\star_{\phi} \Gamma.$ However one can easily show that these two graded rings are not necessarily graded equivalent. 
\end{example}

\begin{remark}\label{hgyhyh1}\hfill 
\begin{enumerate}

\item The graded Morita theory has a left-right symmetry property. Indeed, starting from the category of graded left modules, one can prove a similar statement as in Proposition~\ref{grmorim11}, which in turn, part (4) is independent of the left-right assumption. This shows that $\Gr A$ is graded equivalent to $\Gr B$ if and only if 
$A \rGr$ is graded equivalent to $B \rGr $

\item Proposition~\ref{grmorim11}(3) shows that if all graded finitely generated projective $A$-modules are graded free, then $\Modd A$ is graded equivalence to $\Modd B$ if and only if $B\cong_{\gr} \M_n(A)(\overline \delta)$ for some $n \in \mathbb N$ and 
$\overline \delta=(\delta_1,\dots,\delta_n)$, $\delta_i \in \Gamma$. 
\end{enumerate}
\end{remark}

\begin{remark}\label{maziho} \scm[$\operatorname{Gr^{\Omega}-} A \approx_{\gr} \operatorname{Gr^{\Omega}-} B$ implies 
$\operatorname{Gr^{\Gamma}-} A \approx_{\gr} \operatorname{Gr^{\Gamma}-} B$]
\vspace{0.2cm}

Recall from Definition~\ref{grdeffsa} that we write $\Gr[\Gamma] A \approx_{\gr} \Gr[\Gamma] B$, if there is a graded equivalence between the categories of $\Gamma$\!-graded $A$-modules $\Gr[\Gamma] A$ and $\Gamma$\!-graded $B$-modules $\Gr[\Gamma] B$. 

Let $A$ and $B$ be $\Gamma$\!-graded rings and  $\Omega$ a subgroup of $\Gamma$ such that $\Gamma_A,\Gamma_B \subseteq \Omega \subseteq \Gamma$. Then $A$ and $B$ can be naturally considered as $\Omega$-graded ring. If $\Gr[\Omega] A \approx_{\gr} \Gr[\Omega] B$, then by Theorem~\ref{grmorim11} there is a $\Omega$-isomorphism 
$\phi:B\cong_{\gr} e \M_n(A)(\overline \delta) e$ for a full homogeneous idempotent $e \in \M_n(A)(\overline \delta)$, where $\overline \delta=(\delta_1,\dots,\delta_n)$, $\delta_i \in \Omega$. Since $\Omega \subseteq \Gamma$, another application of Theorem~\ref{grmorim11} shows that $\Gr[\Gamma] A \approx_{\gr} \Gr[\Gamma] B$.  One can also use Theorem~\ref{mhgft42} to obtain this statement. 
\end{remark}

\begin{remark} \scm[$\operatorname{Gr^{\Gamma}-} A \approx_{\gr} \operatorname{Gr^{\Gamma}-} B$ implies 
$\operatorname{Gr^{\Gamma/\Omega}-} A \approx_{\gr} \operatorname{Gr^{\Gamma/\Omega}-} B$]
\vspace{0.2cm}

Let $A$ and $B$ be two $\Gamma$\!-graded rings. Theorem~\ref{grmorim11} shows the equivalence  $\Gr A \approx_{\gr} \Gr B$ induces an equivalence $\Modd A \approx \Modd B$. Haefner in~\cite{haef} observed that $\Gr A \approx_{\gr} \Gr B$ induces other equivalences between different ``layers'' of grading. We briefly recount this result here.  
 
Let $A$ be a $\Gamma$\!-graded ring  and let $\Omega$ be subgroup of $\Gamma$. Recall from~\S\ref{mconfi1} and~\S\ref{bill100} that $A$ can be considered at $\Gamma/\Omega$-graded ring.  Recall also that the category of $\Gamma/\Omega$-graded $A$-modules, denoted by $\Gr[\Gamma/\Omega] A$, consists of the $\Gamma/\Omega$-graded $A$-modules as objects  and $A$-module homomorphisms $\phi: M\rightarrow N$ which are grade-preserving in the sense that 
$\phi(M_{\Omega+ \alpha}) \subseteq  N_{\Omega + \alpha}$ for all $\Omega +\alpha \in \Gamma/\Omega$ as morphisms of the category. In two extreme cases $\Omega=0$ and $\Omega= \Gamma$ we have  $\Gr[\Gamma/\Omega] A=\Gr A$ and  $\Gr[\Gamma/\Omega] A= \Modd A$, respectively. 

In~\cite{haef} Haefner shows that, for any two $\Gamma$\!-graded equivalent rings $A$ and $B$ and for any subgroup $\Omega$ of $\Gamma$, there are equivalences between the categories $\Gr[\Gamma/\Omega] A$ and $\Gr[\Gamma/\Omega] B$.  In fact, Haefner works with an arbitrary (non-abelian) group $\Gamma$ and any subgroup $\Omega$. In this case one needs to adjust the definitions as follows. 

Let $\Gamma/\Omega$ denote a set of (right) cosets of $\Omega$ (we use the multiplication notation here). 
An $\Gamma/\Omega$-graded right $A$-module $M$ is defined as a right $A$-module $M$ with an internal direct sum decomposition 
$M=\bigoplus_{\Omega \alpha \in \Gamma/\Omega}M_{\Omega \alpha}$, where each $M_{\Omega \alpha}$ is an additive subgroup of $M$ such that  $M_{\Omega \alpha}A_{\beta} \subseteq  M_{\Omega \alpha\beta}$ for all $\Omega \alpha \in \Gamma/\Omega$ and $\beta \in \Gamma$. The decomposition is called an $\Gamma/\Omega$-grading of $M$. With the abuse of notation we denote the category of $\Gamma/\Omega$-graded right $A$-modules with $\Gr[\Gamma/\Omega] A$. 
Then in~\cite{haef} it was shown that, $\Gr[\Gamma] A \approx_{\gr} \Gr[\Gamma] B$ implies $\Gr[\Gamma/\Omega] A \approx_{\gr} \Gr[\Gamma/\Omega] B$.

Consider the truncation of $A$ at $\Omega$, \ie $A_\Omega =\oplus_{\gamma\in \Omega} A_\gamma$ (see~\S\ref{mconfi1}). 
As usual, let $\Modd A_\Omega$ denote the category of right $A_\Omega$-modules.
If $A$ is strongly $\Gamma$\!-graded, one can show that the categories  $\Gr[\Gamma/\Omega] A$ and $\Modd A_\Omega$ are equivalent via the functors truncation $(-)_\Omega : \Gr[\Gamma/\Omega] A \rightarrow  \Modd A_\Omega $  and induction 
$-\otimes_{A_\Omega}A : \Modd A_\Omega \rightarrow  \Gr[\Gamma/\Omega] A$ (see ~\cite[Lemma~7.3]{haef}). Again, in the case that the subgroup $\Omega$ is trivial, this gives the equivalences of $\Gr A$ and $\Modd A_0$ (see Theorem~\ref{dadesthm}).  

\end{remark}

\begin{remark} \label{bizet} \scm[The case of $\Gr A \approx \Modd R$]
\vspace{0.2cm}

If a $\Gamma$\!-graded ring $A$ is strongly graded, then by Theorem~\ref{dadesthm}, $\Gr A$ is equivalent to $\Modd A_0$. In~\cite{menini}, the conditions under which $\Gr A$ is equivalent to $\Modd R$, for some ring $R$ with identity, are given. It was shown that, among other things, $\Gr A\approx \Modd R$, for some ring $R$ with identity, if and only if $\Gr A \approx_{\gr} \Gr B$, for a strongly $\Gamma$\!-graded ring $B$, if and only if there is a finite subset $\Omega$ of $\Gamma$ such that for any $\tau \in \Gamma$, 
\begin{equation}\label{rgeusi}
A_0=\sum_{\gamma\in \Omega} A_{\tau-\gamma}A_{\gamma-\tau}.
\end{equation}
Moreover, it was shown that if $\Gr A\approx \Modd R$ then $A\rGr \approx R\rModd$ and $A_0\cong e \M_n(R)e$ for some $n\in \mathbb N$, and  some idempotent $e\in \M_n(R)$. 

If the grade group $\Gamma$ is finite, then~(\ref{rgeusi}) shows that there is a ring $R$ such that $\Gr A$ is equivalent to $\Modd R$. In fact, Cohen and Montgomery in~\cite{cohenmont} construct a so called \emph{smash product}  \index{smash product} $A  \# \mathbb Z[\Gamma]^*$ and show that $\Gr A$ is equivalent to $\Modd A  \# \mathbb Z[\Gamma]^* $. In~\cite{nats1} it was shown that $A  \# \mathbb Z[\Gamma]^*$ is isomorphic to the ring $\End_{\Gr A}\big(\bigoplus_{\alpha \in \Gamma}A(\alpha)\big)$ (see Example~\ref{vpnnabil}). 
\end{remark}

\begin{remark}\label{bbhiidw} \scm[The case of $\operatorname{Gr^{\Gamma}-} A \approx  \operatorname{Gr^{\Lambda}-} B$]
\vspace{0.2cm}

One can consider several variations under which two categories are equivalent. For example, for a $\Gamma$\!-graded ring $A$ and  
$\Lambda$-graded ring $B$, the situation when the categories  $\Gr[\Gamma] A$ and $\Gr[\Lambda] B$ are equivalent is investigated in~\cite{del1,del2} (See also Remark~\ref{shanbei}). 

Moreover, for two $\Gamma$\!-graded rings $A$ and $B$, when $\Gr A$ is equivalent to $\Gr B$ (not necessarily respect the shift) has been studied in~\cite{sierra,zhang96}. 
\end{remark}

\chapter{Graded Grothendieck Groups} \label{ggg} 

Starting from a ring, the isomorphism classes of finitely generated (projective) modules with direct sum form an abelian monoid. The ``enveloping group" of this monoid is called the Grothendieck group of the ring. If the ring comes equipped with an extra structure, this structure should pass to its modules and thus should be reflected in the Grothendieck group. For example, if the ring has an involution, then the Grothendieck group has a $\Z_2$-module structure. If the ring has a coring structure, or it is a Hopf algebra, then the Grothendieck group becomes a ring with involution thanks to the co-multiplication and antipode of the Hopf algebra.   \index{Hopf algebra}

For a $\Gamma$\!-graded ring $A$, one of the main aim of this book is to study the Grothendieck group constructed from the graded projective modules of $A$, called the graded Grothendieck group, $K^{\gr}_0(A)$. In fact, $K^{\gr}_0(A)$ is not just an abelian group but it also has an extra  $\mathbb Z[\Gamma]$-module structure. As we will see throughout this book, this extra structure carries substantial information about the graded ring $A$. In~\S\ref{podong} we construct in detail the graded Grothendieck groups using the concept of group completions. Here we give a brief overview of different equivalent constructions.

For an abelian monoid $V$, we denote by $V^+$ the group completion of $V$. This gives a left adjoint functor to the forgetful functor from the category of abelian groups to abelian monoids. 
When the monoid $V$ has a $\Gamma$\!-module structure, where $\Gamma$ is a group, then $V^+$ inherits a natural $\Gamma$\!-module structure, or equivalently,  $\Z[\Gamma]$-module structure We study this construction in~\S\ref{podong}. \index{adjoint functor}

There is also a more direct way to construct $V^+$ which we recall here.  Consider the set of symbols $\big \{\, [m] \mid m \in V \, \big \}$ and let $V^+$ be the free abelian group generated by this set  modulo the relations $[m]+[n]-[m+n]$, $m,n\in V$.  There is a natural (monoid) homomorphism $V\longrightarrow V^+$, $m\mapsto [m]$, which is universal. Using the universality, one can show that the group $V^+$ obtained here coincides with the one constructed above using the group completion. 

Now let $\Gamma$ be a group which acts on a monoid $V$. Then $\Gamma$ acts in a natural way on the free abelian group generated by symbols $\big \{[m] \mid m \in V \big \}$. Moreover the subgroup generated by relations $[m]+[n]-[m+n]$, $m,n\in V$ is closed under this action. Thus $V^+$ has a natural structure of $\Gamma$\!-module, or equivalently, $V^+$ is a $\Z[\Gamma]$-module.

For a ring $A$ with identity and a finitely generated projective (right) $A$-module $P$, let $[P]$ denote the class of $A$-modules isomorphic to $P$. Then   
the set 
\begin{equation}\label{zhongshan}
\mathcal V(A)=\big \{ \, [P] \mid  P  \text{ is a finitely generated projective A-module} \, \big \}
\end{equation}
 with addition $[P]+[Q]=[P\bigoplus Q]$ forms an abelian monoid. The \emph{Grothendieck group} of $A$, denoted by $K_0(A)$, is by definition $\mathcal V(A)^+$. \index{Grothendieck group}

The graded Grothendieck group of a graded ring is constructed similarly, by using graded finitely generated projective modules everywhere in the above process. Namely,  for a $\Gamma$\!-graded ring $A$ with identity and a graded finitely generated projective (right) $A$-module $P$, let $[P]$ denote the class of graded $A$-modules graded isomorphic to $P$. Then the monoid  \index{$\mathcal V^{\gr}$,  monoid of graded projective modules}
\begin{equation}\label{zhongshan1}
\mathcal V^{\gr}(A)=\big \{ \, [P] \mid  P  \text{ is graded finitely generated projective A-module} \, \big \}
\end{equation}
has a $\Gamma$\!-module structure defined as follows: for $\ga \in \Gamma$ and $[P]\in \mathcal V^{\gr}(A)$, $\ga .[P]=[P(\ga)]$. The group $\mathcal V^{\gr}(A)^+$ is called the \emph{graded Grothendieck group} and is denoted by $K^{\gr}_0(A)$, which 
as  the above discussion shows is a $\Z[\Gamma]$-module. \index{graded Grothendieck group}

The above construction of the graded Grothendieck groups carries over to the category of graded right $A$-modules. Since for a graded finitely generated right $A$-module $P$, the dual module $P^*=\Hom_A(P,A)$ is a graded left $A$-module, and furthermore, if $P$ is projective, then also $P^{**}\cong_{\gr} P$, passing to the dual defines an equivalence between the category of graded finitely generated projective right $A$-modules and  
the category of graded finitely generated projective left $A$-modules, \ie 
\[A \Pgrl  \approx_{\gr} \Pgrp A.\] This implies that constructing a graded Grothendieck group, using the graded left-modules, call it $K^{\gr}_0(A)_l$,  is isomorphic $K^{\gr}_0(A)_r$, the group constructed using graded right $A$-modules. Moreover defining the action of $\Gamma$ on generators of $K^{\gr}_0(A)_l$, by $\alpha [P]=[P(-\alpha)]$ and extend it to the whole group, and defining the action of  $\Gamma$ on $K^{\gr}_0(A)_r$, in the usual way, by $\alpha [P]=[P(\alpha)]$, then~(\ref{terrace1}) shows that these two groups are $\mathbb Z[\Gamma]$-module isomorphic.

As emphasised, one of the main differences between a graded Grothendieck group and the ungraded version is that the former has an extra module structure.  For instance, in Example~\ref{upst}, we will see that for the $\mathbb Z$-graded ring $A=K[x^n,x^{-n}]$, where $K$ is a field and $n\in \mathbb N$, 
\[ K_0^{\gr}(A)\cong \bigoplus_n \mathbb Z,\] which  is a $\mathbb Z[x,x^{-1}]$-module, with the action of $x$ on 
$(a_1,\dots,a_n) \in \bigoplus_n \mathbb Z$ is as follows: \[x (a_1,\dots,a_n)= (a_n,a_1,\dots,a_{n-1}).\]

In this chapter we study graded Grothendieck groups in detail. We will calculate these groups (or rather modules) for certain types of graded rings, including the graded division rings~(\S\ref{ghgwbsia}), graded local rings~(\S\ref{attila1}) and path algebras~(\S\ref{kleaviti}). 
In Chapter~\ref{ultriuy} we use the graded Grothendieck groups to classify the so called graded ultramatricial algebras. We will show that certain information encoded in the graded Grothendieck group can be used to give a complete invariant for such algebras.

In general, the category of graded finitely generated projective modules 
is an exact category with the usual notion of (split) short exact
sequence. The Quillen $K$-groups of this category (see \cite[\S2]{quillen} for the construction of $K$-groups of an exact category) is denoted by 
$K^{\gr}_i(A)$, $i\geq 0$. The group $\Gamma$ acts on this category
 via $(\alpha,P) \mapsto P(\alpha)$. By functoriality of $K$-groups this equips $K_i^{\gr}(A)$
with the structure of a $\mathbb Z[\Gamma]$-module. When $i=0$, Quillen's construction coincides with the above construction via the group completion.  In Chapter~\ref{waraya} we compare the graded versus ungraded Grothendieck groups as well as the higher $K$-groups. 

For a comprehensive study of ungraded Grothendieck groups see~\cite{lamsiu,magurn,rosenberg}.

\section{The graded Grothendieck group $K^{\gr}_0$} \label{podong}

\subsection{Group completions}\label{hghgti1}

\index{$V^{+}$, group completion}
Let $V$ be a monoid and $\Gamma$ be a group which acts on $V$. The \emph{group completion} \index{group completion} of $V$ (this is also called the \emph{Grothendieck group} \index{Grothendieck group} of $V$) has a natural $\Gamma$\!-module structure. We recall the construction here. 
Define a relation $\sim$ on $V\times V$ as follows:  $(x_1,y_1) \sim (x_2,y_2)$ if there is a $z\in V$ such that 
\begin{equation}\label{ppslaki}
x_1+y_2+z=y_1+x_2+z.
\end{equation}
This is an equivalence relation. We denote the equivalence classes of $V\times V/\sim$ by $V^+:=\big \{\, [(x,y)] \mid (x,y)\in V\times V\, \big \}$. Define 
\[ [(x_1,y_1)]+[(x_2,y_2)]=[(x_1+x_2,y_1+y_2)],\]
\[\alpha [(x,y)]=[(\alpha x,\alpha y)]. \]
One can easily check that these operations are well-defined, $V^+$ is an abelian group and further it is a $\Gamma$\!-module. 
The map 
\begin{align}\label{pogfris}
\phi: V &\longrightarrow V^+\\
x &\longmapsto [(x,0)] \notag
\end{align} 
is a $\Gamma$\!-module homomorphism and $\phi$ is \emph{universal}, \ie if there is a $\Gamma$\!-module $G$ and a $\Gamma$\!-module homomorphism $\psi:V\rightarrow G$, then there exists a unique $\Gamma$\!-module homomorphism $f:V^+\rightarrow G$, defined by 
$f([x,y])=\psi(x)-\psi(y)$, such that 
$f\phi=\psi$. 

We record the following properties of $V^+$ whose proofs are easy and are left to the reader. 

\begin{lemma}\label{shealing}
Let $V$ be a monoid, $\Gamma$ a group which acts on $V$, and let 
\begin{align*}
\phi: V &\longrightarrow V^+,\\
x &\longmapsto [(x,0)],
\end{align*}
 be the universal homomorphism. 

\begin{enumerate}[\upshape(1)]

\item For $x,y\in V$, if $\phi(x)=\phi(y)$ then there is a $z\in V$ such that $x+z=y+z$ in $V$.

\item Each element of $V^+$ is of the form $\phi(x)-\phi(y)$ for some $x,y\in V$. 

\item $V^+$ is generated by the image of $V$ under $\phi$ as a group. 

\end{enumerate}
\end{lemma}

 \begin{example}
 Let $G$ be a group and $\mathbb N$ be the monoid of nonnegative integers. Then $\mathbb N[G]$ is a monoid equipped by the natural action of $G$. Its group completion is $\mathbb Z[G]$ which has a natural $G$-module structure. This will be used in Proposition~\ref{k0grof} to calculate the graded Grothendieck group of graded fields. 
 \end{example}

 \begin{example}\scm[A nonzero monoid whose group completion is zero]
 \vspace{0.2cm}
 
 Let $V$ be a monoid with the trivial module structure (\ie $\Gamma$ is trivial). Suppose that $V\backslash \{0\}$ is an abelian group. Then one can check that 
 $V^+\cong V\backslash \{0\}$. 
 
 Now let $V=\{0,v\}$, with $v+v=v$ and $0$ as the trivial element. One can check that with this operation $V$ is a monoid and $V\backslash \{0\}$ is an abelian group which is a trivial group. In general, the group completion of any monoid with a zero element, \ie an element $z$ such that $x+z=z$, 
 where $x$ is any element of the monoid, is a trivial group. 
 \end{example}

\begin{example}\label{ppurdeu}\scm[The type, IBN and the monoid of projective modules] \index{IBN} \index{type of a ring}
\vspace{0.2cm}

Let $A$ be a ring (with trivial grading) and $\mathcal V(A)$ the monoid of the isomorphism classes of finitely generated projective $A$-modules. Then one can observe that $A$ has IBN if and only if the submonoid generated by $[A]$ in $\mathcal V(A)$ is isomorphic to $\mathbb N$. Moreover, a ring $A$ has type $(n,k)$ if and only if the submonoid generated by $[A]$ is isomorphic to a free monoid generated by a symbol $v$ subject to $nv=(n+k)v$ (see~\S\ref{gtr5654}).

\end{example}
 
 \subsection{$K^{\gr}_0$-group}  \index{$K^{\gr}_0(A)$, graded Grothendieck group}

Let $A$ be a $\Gamma$\!-graded ring (with identity as usual) and let $\mathcal V^{\gr}(A)$ denote the monoid of graded isomorphism classes of graded finitely generated   projective modules over $A$ with the direct sum as the addition operation. For a graded finitely generated projective $A$-module $P$, we denote the graded isomorphism class of $P$ by $[P]$ which is an element of $\mathcal V^{\gr}(A)$ (see~(\ref{zhongshan1})).  Thus for $[P], [Q] \in \mathcal V^{\gr}(A)$, we have $[P]+[Q]=[P\oplus Q]$. 
Note that for $\alpha \in \Gamma$,  the  $\alpha$-suspension functor $\mathcal T_\alpha:\Gr A\rightarrow \Gr A$, $M \mapsto M(\alpha)$  is an isomorphism with the property
$\mathcal T_\alpha \mathcal T_\beta=\mathcal T_{\alpha + \beta}$, $\alpha,\beta\in \Gamma$.
Moreover, $\mathcal T_\alpha$ restricts to $\Pgrp A$, the  category of graded finitely generated  
projective $A$-modules. Thus the abelian group $\Gamma$ acts on $\mathcal V^{\gr}(A)$ via 
\begin{equation}\label{hhffwwq}
(\alpha, [P]) \mapsto  [P(\alpha)].
\end{equation}

The \emph{graded Grothendieck group}, \index{graded Grothendieck group} $K_0^{\gr}(A)$, is defined as the group completion of  the monoid $\mathcal V^{\gr}(A)$ which naturally inherits 
the $\Gamma$\!-module structure via (\ref{hhffwwq}).  This makes $K_0^{\gr}(A)$ a $\Z[\Gamma]$-module. In particular if $A$ is a $\mathbb Z$-graded then $K_0^{\gr}(A)$ is a
$\mathbb Z[x,x^{-1}]$-module. This extra structure plays a crucial role in the applications of graded Grothendieck groups. 

For a graded finitely generated projective $A$-module $P$, we denote the image of $[P] \in \mathcal V^{\gr}(A)$ under the natural homomorphism $\mathcal V^{\gr}(A) \rightarrow K^{\gr}_0(A)$ by $[P]$ again (see~(\ref{pogfris})). When the ring has graded stable rank $1$, this map is injective (see Corollary~\ref{hghtt1}).

\begin{example}\label{pogbiaue}
 
 In the following ``trivial'' cases one can determine the graded Grothendieck group based on the Grothendieck group of the ring. 
 \begin{description}
 \item[Trivial group grading:] Let $A$ be a ring and $\Gamma$ be a trivial group. Then $A$ is a $\Gamma$\!-graded ring in an obvious way, $\Gr A=\Modd A$ and $K^{\gr}_0(A)=K_0(A)$ as a $\mathbb Z[\Gamma]$-module.  However here $\mathbb Z[\Gamma]\cong \mathbb Z$ and $K^{\gr}_0(A)$ does not have any extra module structure. This shows that the statements we prove in the graded setting will cover the results in the classical ungraded setting, by considering rings with trivial gradings from the outset. 

 \medskip   \index{trivial grading}
 \item[Trivial grading:] Let $A$  be a ring and $\Gamma$ be a group. Consider $A$ with the trivial $\Gamma$\!-grading, \ie $A$ is concentrated in degree zero (see~\S\ref{jdjthu}). Then by Corollary~\ref{hgfeyes} and Remark~\ref{enjoythemomentt}, $\Pgrp A =\bigoplus_{\Gamma} \Prr A$.  Consequently, 
  $K^{\gr}_0(A)\cong \bigoplus _\Gamma K_0(A)$. Considering the shift which induces a $\mathbb Z[\Gamma]$-module structure (see~(\ref{mugeinizmir})), we have 
  \[K^{\gr}_0(A)\cong K_0(A)[\Gamma]\] as $\mathbb Z[\Gamma]$-module. 
 
 \end{description}
 \end{example}

\begin{remark}\scm[The group completion of all projective modules is trivial]
\vspace{0.2cm}

One reason in restricting ourselves to the finitely generated projective modules is that the group completion of the monoid of the isomorphism classes of (all) projective modules gives a trivial group. To see this, first observe that for a monoid $V$, $V^+$ is the trivial group if and only if for any $x,y \in V$, there is a $z\in V$ such that $x+z=y+z$. Now consider graded projective modules $P$ and $Q$. Then for $M=\bigoplus_{\infty} (P\oplus Q)$, we have $P\oplus M\cong Q\oplus M$. Indeed, 
\begin{align*}
Q \oplus M & \cong  Q \oplus (P\oplus Q) \oplus (P \oplus Q) \oplus \cdots\\
& \cong  (Q\oplus P) \oplus (Q\oplus P) \oplus \cdots \\
& \cong  (P \oplus Q) \oplus (P \oplus Q) \oplus \cdots\\
& \cong P \oplus (Q \oplus P) \oplus (Q\oplus P) \oplus \cdots \\
& \cong P \oplus M.
\end{align*}
\end{remark}

\begin{lemma}\label{rhodas}
Let $A$ be a $\Gamma$\!-graded ring.

\begin{enumerate}[\upshape(1)]

\item Each element of  $K^{\gr}_0(A)$ is of the form $[P]-[Q]$ for some graded finitely generated projective $A$-modules $P$ and $Q$. 

\item Each element of  $K^{\gr}_0(A)$ is of the form $[P]-[A^n(\overline \alpha)]$ for some graded finitely generated projective $A$-module $P$  and some $\overline \alpha=(\alpha_1,\dots,\alpha_n)$.
 
\item  Let $P$, $Q$ be graded finitely generated projective $A$-modules. Then $[P]=[Q]$ in $K^{\gr}_0(A)$ if and only if $P\oplus A^n(\overline \alpha) \conggr Q\oplus A^n(\overline \alpha)$, for some $\overline \alpha=(\alpha_1,\dots,\alpha_n)$. 
\end{enumerate}
\end{lemma}

\begin{proof}
(1) This follows immediately from Lemma~\ref{shealing}(2) by considering $\mathcal V^{\gr}(A)$ as the monoid and the fact that $[P]$ also represent the image of $[P] \in \mathcal V^{\gr}(A)$ in $K^{\gr}_0(A)$. 

(2) This follows from (1) and the fact that for a graded finitely generated projective $A$-module $Q$, there is a graded finitely generated projective module $Q'$ such that $Q\oplus Q'\conggr A^n(\overline \alpha)$, for some $\overline \alpha=(\alpha_1,\dots,\alpha_n)$  (see Proposition~\ref{grprojectivethm} and~(\ref{medickiue})). 

(3) Suppose $[P]=[Q]$ in $K^{\gr}_0(A)$. Then by Lemma~\ref{shealing}(1) (for $V=\mathcal V^{\gr}(A)$) there is a graded finitely generated projective $A$-module $T$ such that $P\oplus T \conggr Q\oplus T$. Since $T$ is graded finitely generated projective,  there is an $S$ such that $T\oplus S\conggr A^n(\overline \alpha)$, for some $\overline \alpha=(\alpha_1,\dots,\alpha_n)$ (see~(\ref{medickiue})). Thus 
\[ P\oplus A^n(\overline \alpha) \conggr  P\oplus T \oplus S \conggr Q\oplus T \oplus S\conggr Q\oplus A^n(\overline \alpha).\]

Since $K^{\gr}_0(A)$ is a group, the converse is immediate. 
\end{proof}

For graded finitely generated projective $A$-modules $P$ and $Q$, we say 
$P$ and $Q$ are \emph{graded stably  isomorphic}  if $[P]=[Q]$ in $K^{\gr}_0(A)$ or equivalently by Lemma~\ref{rhodas}(3), if 
$P\oplus A^n(\overline \alpha) \conggr Q\oplus A^n(\overline \alpha)$, for some $\overline \alpha=(\alpha_1,\dots,\alpha_n)$.  \index{graded stably isomorphic}

\begin{corollary}\label{hghtt1}
Let $A$ be a $\Gamma$\!-graded ring with the graded stable rank $1$. Then the natural map $\mathcal V^{\gr}(A) \rightarrow K^{\gr}_0(A)$ is injective. 
\end{corollary}
\begin{proof}
Let $P$ and $Q$ be graded finitely generated projective $A$-modules such that $[P]=[Q]$ in $K^{\gr}_0(A)$.  Then by Lemma~\ref{rhodas}(3), $P\oplus A^n(\overline \alpha) \conggr Q\oplus A^n(\overline \alpha)$, for some $\overline \alpha=(\alpha_1,\dots,\alpha_n)$. Now by Corollary~\ref{jussiej}, $P\cong_{\gr} Q$. Thus $[P]=[Q]$ in $\mathcal V^{\gr}(A)$. 
\end{proof}

\subsection{$K_0^{\gr}$ of strongly graded rings}\label{jijigogo} \index{strongly graded ring}
\index{0-component ring of a graded ring} \index{$\Prr A$, category of finitely generated projective $A$-modules}

Let $A$ be a strongly $\Gamma$\!-graded ring. By Dade's Theorem~\ref{dadesthm} and Remark~\ref{cafejen}, the functor $(-)_0:\Pgrp A\rightarrow \Prr A_0$, $M\mapsto M_0$, is an additive functor with an  inverse $-\otimes_{A_0} A: \Prr A_0 \rightarrow \Pgrp A$ so that it induces an equivalence between the category of graded finitely generated projective $A$-modules and the category of finitely generated projective $A_0$-module.   
 This implies
that 
\begin{equation}\label{hgyta2}
\boxed{K_0^{\gr}(A)  \cong K_0(A_0).}
\end{equation} (In fact, this implies that $K_i^{\gr}(A)  \cong K_i(A_0)$, for all $i\geq 0$, where $K_i^{\gr}(A)$ and  $K_i(A_0)$ are Quillen's $K$-groups. These groups will be discussed in~\S\ref{waraya}.) Moreover, since $A_\alpha \otimes_{A_0}A_\beta \cong A_{\alpha +\beta}$ as $A_0$-bimodule, the functor $\mathcal T_\alpha$ on $\mbox{gr-}A$ induces a functor on the level of
$\mbox{mod-}A_0$, $\mathcal T_\alpha:\mbox{mod-}A_0 \rightarrow \mbox{mod-}A_0$, $M\mapsto  M \otimes_{A_0} A_\alpha$ such that $\mathcal T_\alpha \mathcal T_\beta\cong\mathcal T_{\alpha +\beta}$, $\alpha,\beta\in \Gamma$, so that the following diagram is commutative up to isomorphism. 
\begin{equation}\label{veronaair}
\xymatrix{
\Pgrp A \ar[r]^{\mathcal T_\alpha} \ar[d]_{(-)_0}& \Pgrp A \ar[d]^{(-)_0}\\
\Prr A_0 \ar[r]^{\mathcal T_\alpha}  & \Prr A_0
}
\end{equation}
Therefore $K_i(A_0)$ is also a $\mathbb Z[\Gamma]$-module and
\begin{equation} \label{dade}
K_i^{\gr}(A) \cong K_i(A_0),
\end{equation} as $\mathbb Z[\Gamma]$-modules.

Also note that if $A$ is a graded commutative ring then the isomorphism~(\ref{dade}), for $i=0$, is a ring isomorphism.

\begin{example}\scm[$K^{\gr}_0$ of crossed products] \label{gptwnow} \index{crossed product ring}
\vspace{0.2cm}
\vspace{0.2cm}

Let $A$ be a $\Gamma$\!-graded crossed product ring. Thus $A={A_0}_{\psi}^{\phi}[\Gamma]$ (see~\S\ref{khgfewa1}). 
By Proposition~\ref{crossedproductstronglygradedprop}(3), $A$ is strongly graded, and so $\Gr A  \approx \Modd A_0$ and $K^{\gr}_0(A)\cong K_0(A_0)$ (see~(\ref{hgyta2})). 

On the other hand, by Corollary~\ref{andrewmathas}(4), the restriction of the shift functor on $\Pgrp  A$, $\mathcal T_\alpha:\Pgrp A\rightarrow \Pgrp A$, is isomorphic to the trivial functor. 
This shows that the action of $\Gamma$ on $K^{\gr}_0(A)\cong K_0(A_0)$ (and indeed on all $K$-groups $K^{\gr}_i(A)$) is trivial. 
\end{example}

\begin{example}\label{poiurp}\scm[$K^{\Gamma}_0$ verses $K^{\Gamma/\Omega}_0$]
\vspace{0.2cm}

Let $A$ be a $\Gamma$\!-graded ring and $\Omega$ be a subgroup of $\Gamma$. Recall the construction of $\Gamma/\Omega$-graded ring $A$ 
from\S~\ref{mconfi1}. The canonical forgetful functor \[U:\Gr[\Gamma] A \rightarrow \Gr[\Gamma/\Omega] A\] is an exact functor (see~\S\ref{bill100}). Proposition~\ref{grprojectivethm} guarantees that $U$ restricts to the categories of graded finitely generated projective modules, \ie \[U:\Pgr[\Gamma] A \rightarrow \Pgr[\Gamma/\Omega] A.\] 
This induces a group homomorphism 
\[ \theta:K^{\Gamma}_0(A) \longrightarrow K^{\Gamma/\Omega}_0(A),\]
such that for $\alpha \in \Gamma$ and $a\in K^{\Gamma}_0(A)$, we have  $\theta (\alpha a) =(\Omega+\alpha) \theta(a).$
Here, to distinguish the grade groups, we denote by $K^{\Gamma}_0$ and $K^{\Gamma/\Omega}_0$  the graded Grothendieck groups of $A$ as $\Gamma$ and $\Gamma/\Omega$-graded rings, respectively (we will use this notation again in~\S\ref{waraya}).  


Now let $A$ be a strongly $\Gamma$\!-graded ring. Then $A$ is also a strongly $\Gamma/\Omega$-graded (Example~\ref{mominjianguomen}). 
Using~(\ref{hgyta2}), we have 
\begin{align*}
K^{\Gamma}_0(A) & \cong K_0(A_0),\\
K^{\Gamma/\Omega}_0(A) & \cong K_0(A_\Omega),
\end{align*}
where $A_\Omega=\bigoplus_{\gamma \in \Omega} A_\gamma$. 
\end{example}

\begin{example}\scm[$K^{\Gamma}_0$ verses $K^{\Omega}_0$]\label{tonyat11}
\vspace{0.2cm}

Let $A$ be a $\Gamma$\!-graded ring and $\Omega$ be a subgroup of $\Gamma$ such that $\Gamma_A \subseteq \Omega$. By Theorem~\ref{mhgft42} 
the category $\Gr[\Gamma] A$ is equivalent to $\bigoplus_{\Gamma/\Omega} \Gr[\Omega] A_\Omega$. Since this equivalence preserves the projective modules (Remark~\ref{enjoythemomentt}), we have 
\begin{equation*}
K^{\Gamma}_0(A)\cong \bigoplus_{\Gamma/\Omega}K^{\Omega}_0(A).
\end{equation*}
In fact, the same argument shows that for any (higher) $K$-groups (see~\S\ref{waraya}), if $A$ is a $\Gamma$\!-graded ring such that for $\Omega:=\Gamma_A$, $A_\Omega$ is strongly $\Omega$-graded, then we have 
\begin{equation}\label{diniat6}
\boxed{K^{\Gamma}_i(A)\cong \bigoplus_{\Gamma/\Omega}K^{\Omega}_i(A) \cong \bigoplus_{\Gamma/\Omega}K_i(A_0).}
\end{equation}

Moreover, if $A_\Omega$ is a crossed product, we can determine the action of $\Gamma$ on $K^{\gr}_i$ concretely. First, using Corollary~\ref{andrewmathas}, in (\ref{hnhdheye}) for any $\omega \in \Omega$, 
$M(w)_{\Omega+\alpha_i} \cong M_{\Omega+\alpha_i}$  as graded $A_\Omega$-module. Thus the functor $\overline \rho_\beta$ in (\ref{hnhdheye}) reduces to permutations of categories. Representing $\bigoplus_{\Gamma/\Omega}K_i(A_0)$ by the additive group of the group ring 
$K_i(A_0)[\Gamma/\Omega]$, by Theorem~\ref{mhgft42}, the $\mathbb Z[\Gamma]$-module structure of $K^{\Gamma}_i(A)$, can be descried as $K_i(A_0)[\Gamma/\Omega]$ with the natural $\mathbb Z[\Gamma]$-module structure.

This can be used to calculate the (lower) graded $K$-theory of graded division algebras (see Proposition~\ref{k0grof}, Remark~\ref{hghghgt4} and Example~\ref{lobbyat12}). 
\end{example}

\subsection{The reduced graded Grothendieck group $\widetilde {K^{\gr}_0}$}\label{pearlf} 
\index{$\widetilde{K^{\gr}_0}(A)$, reduced graded Grothendieck group}

For a ring $A$, the element $[A]\in K_0(A)$ generates the ``obvious'' part of the Grothendieck group, \ie all the elements presented by $\pm [A^n]$ for some $n\in \mathbb N$.  
The \emph{reduced Grothendieck group}, \index{reduced Grothendieck group} denoted by $\widetilde {K_0}(A)$ is the quotient of $K_0(A)$ by this cyclic subgroup. There is a unique ring homomorphism $\Z \rightarrow A$, which induces the group homomorphism $K_0(\Z) \rightarrow K_0(A)$. Since the ring of integers $\Z$ is a PID and finitely generated projective modules over a PID ring
are free of unique rank (see~\cite[Theorem~1.3.1]{rosenberg}), we have 
\begin{equation}
\boxed{\widetilde {K_0}(A)=\coker \big (K_0(\Z) \rightarrow K_0(A)\big). }
\end{equation}

This group appears quite naturally. For example when $A$ is a Dedekind domain,  $\widetilde {K_0}(A)$ coincides with $C(A)$, the class group of $A$ (see~\cite[Chap.~1,\S4]{rosenberg}). 

We can thus write an exact sequence, 
\[ 0\longrightarrow \ker \theta \longrightarrow K_0(\mathbb Z) \stackrel{\theta}{\longrightarrow} K_0(A) \longrightarrow \widetilde {K_0}(A) \longrightarrow 0.\]

For a commutative ring $A$,  one can show that $\ker \theta =0$ and the exact sequence is split, \ie 
\begin{equation}\label{subwaynow}
K_0(A)\cong \Z \bigoplus \widetilde {K_0}(A).
\end{equation}

In fact, $\ker \theta$ determines if a ring has IBN (see~\S\ref{gtr5654}). Namely, $\ker \theta=0$ if and only if 
$A$ has IBN. For, suppose $\ker \theta=0$.  If $A^n\cong A^m$ as $A$-modules, for $n,m \in \mathbb N$, then 
$(n-m)[A]=0$. It follows $n=m$. Thus $A$ has IBN. On the other hand suppose $A$ has IBN. If $n \in \ker \theta$, then we can consider $n\geq 0$. Then $[A^n]=0$ implies $A^{n+k} \cong A^k$, for some $k \in \mathbb N$. Thus $n=0$ and so $\ker \theta=0$.

Here we develop the graded version of reduced Grothendieck groups. Let $A$ be a $\Gamma$\!-graded ring. In this setting the obvious part of $K^{\gr}_0(A)$ is not only the subgroup generated by $[A]$ but by all the shifts of $[A]$, \ie the $\Z[\Gamma]$-module generated by $[A]$. The \emph{reduced graded Grothendieck group} \index{reduced graded Grothendieck group} of the $\Gamma$\!-graded ring $A$, denoted by $\widetilde {K^{\gr}_0}(A)$, is the $\mathbb Z[\Gamma]$-module defined by the quotient of $K^{\gr}_0(A)$ by the submodule generated by $[A]$. Similarly to the ungraded case, considering $\Z$ as a  $\Gamma$\!-graded ring concentrated in degree zero, the unique graded ring homomorphism $\Z \rightarrow A$, indices a $\mathbb Z[\Gamma]$-module homomorphism $K^{\gr}_0(\Z) \rightarrow K^{\gr}_0(A)$ (see Example~\ref{pogbiaue}(2)). Then 
\begin{equation}
\widetilde {K^{\gr}_0}(A)=\coker \big (K^{\gr}_0(\Z) \rightarrow K^{\gr}_0(A)\big). 
\end{equation}

In order to obtain a graded version of the splitting formula~(\ref{subwaynow}), we need to take out a part of $\Gamma$ which its action on $[A]$ via shift would 
not change $[A]$ in $K^{\gr}_0(A)$. 

First note that the $\Z[\Gamma]$-module homomorphism $K^{\gr}_0(\Z) \rightarrow K^{\gr}_0(A)$ can be written as 
\begin{align}\label{mozart1}
\phi:\mathbb Z[\Gamma]&\longrightarrow K^{\gr}_0(A),\\
\sum_\alpha n_\alpha \alpha &\longmapsto \sum_\alpha n_\alpha [A(\alpha)].\notag
\end{align}

Moreover, since by Corollary~\ref{rndcongcori}, $A(\alpha)\conggr A(\beta)$ as right $A$-modules if and only if $\alpha -\beta \in \Gamma_A^*$, the above map induces 
\begin{align*}
\psi:\mathbb Z[\Gamma/\Gamma_A^*]&\longrightarrow K^{\gr}_0(A),\\
\sum_\alpha n_\alpha (\Gamma_A^*+\alpha) &\longmapsto \sum_\alpha n_\alpha [A(\alpha)],
\end{align*}
so that the following diagram is naturally commutative. 
\begin{equation}\label{stonrybridge}
\xymatrix{
\mathbb Z[\Gamma] \ar[rr]^{\phi} \ar[rd]_{\pi} &&K^{\gr}_0(A),\\
& \mathbb Z[\Gamma/\Gamma_A^*] \ar[ur]_{\psi}
}
\end{equation}
Since $\pi$ is surjective, $\widetilde {K^{\gr}_0}(A)$ is also the cokernel of the map $\psi$. 

In Example~\ref{gaul1} we calculate the reduced Grothendieck group of the graded division algebras. In Example~\ref{gaul2} we show that for a strongly graded ring, this group does not necessarily coincide with the reduced Grothendieck group of its 0-component ring.


\subsection{$K_0^{\gr}$ as a $\mathbb Z[\Gamma]$-algebra}\label{hhyyuvy}

Let $A$ be a $\Gamma$\!-graded commutative ring. Then, as in the ungraded case, $K^{\gr}_0(A)$ forms a commutative ring, with multiplication defined on generators by tensor products, \ie $[P].[Q]=[P\otimes_A Q]$ (see~\S\ref{grtensie}). Consider $\mathbb Z[\Gamma]$ as a group ring. To get the natural group ring operations, here we use the multiplicative notation for the group $\Gamma$. 
With this convention, since for any $\alpha, \beta \in \Gamma$, by~(\ref{inenbuild}), $A(\alpha)\otimes_A A(\beta) \conggr A(\alpha\beta)$ , the map 
\begin{align*}
\mathbb Z[\Gamma]&\longrightarrow K^{\gr}_0(A),\\
\sum_\alpha n_\alpha \alpha &\longmapsto \sum_\alpha n_\alpha [A(\alpha)]
\end{align*}  induces a ring homomorphism.  This makes $K^{\gr}_0(A)$ a 
$\mathbb Z[\Gamma]$-algebra. In fact the argument before Diagram~\ref{stonrybridge} shows that $K^{\gr}_0(A)$ can be considered as $\mathbb Z[\Gamma/\Gamma_A^*]$-algebra. 

In Examples~\ref{drring} and \ref{lrring} we calculate these algebras for graded fields and graded local rings.

\section{$K^{\gr}_0$ from idempotents} \label{hhidmi}

One can describe the $K_0$-group of a ring $A$ with identity in terms
of idempotent matrices of $A$. This description is quite helpful as we can work with the  conjugate classes of idempotent matrices instead of isomorphism classes of finitely generated projective modules. For example we will use this description to show that the graded $K_0$ is a continuous map (Theorem~\ref{kcontis}.)  Also, this description allows us to define the monoid $\mathcal V$ for rings without identity in a natural way (see~\S\ref{wedk0}). We briefly recall the construction of $K_0$ from idempotents here.

Let $A$ be a ring with identity. In the following, we can always enlarge matrices of different sizes over $A$ by adding zeros in the lower right hand corner, so that they can be considered in a ring $\M_k(A)$ for a suitable $k\in \mathbb N$. 
This means that we work in the matrix ring $\M_{\infty}(A)=\varinjlim_n \M_n(A)$, where the connecting maps are 
the non-unital ring homomorphism  \index{$\M_{\infty}(A)$}
\begin{align}\label{cafejen25}
\M_n(A)&\longrightarrow \M_{n+1}(A)\\ 
p &\longmapsto 
\left(\begin{matrix} p & 0\\ 0 & 0
\end{matrix}\right). \notag
\end{align}

Any idempotent matrix $p \in \M_n(A)$ (\ie $p^2=p$) \index{idempotent matrix} gives rise to the finitely generated 
projective right $A$-module $pA^n$. On the other hand any finitely generated  projective right 
module $P$ gives rise to an idempotent matrix $p$ such that $pA^n \cong P$. We say two idempotent
matrices $p$ and $q$ are equivalent if (after suitably enlarging
them) there are matrices $x$ and $y$ such that $xy=p$ and $yx=q$.
One can show that $p$ and $q$ are equivalent if and only if they are
conjugate if and only if the corresponding finitely generated   projective modules
are isomorphic.  Therefore $K_0(A)$ can be defined as the group
completion of the monoid of equivalence classes of idempotent
matrices with addition defined by, $[p]+[q]=\Big [\left(\begin{matrix} p & 0\\ 0 & q
\end{matrix}\right) \Big ]$. In fact, this is the definition that one adapts for
$\mathcal V(A)$ when the ring $A$ does not have identity (see for example~\cite[Chapter~1]{mcuntz} or~\cite[p.296]{menal}).

A similar construction can be given in the setting of graded rings.
This does not seem to be documented in literature and we provide the
details here.

Let $A$ be a $\Gamma$\!-graded ring with identity and let $\ol
\alpha=(\alpha_1,\dots,\alpha_n)$, where $\alpha_i \in \Gamma$. Recall that $\M_n(A)(\ol \alpha)$ is a graded ring (see~\S\ref{matgrhe}). In
the following if we need to enlarge a homogeneous matrix $p\in
\M_n(A)(\ol \alpha)$, by adding zeroes in the lower right hand
corner (and call it $p$), then we add zeros in the right hand side of $\ol
\alpha=(\alpha_1,\dots,\alpha_n)$ as well accordingly (and call it
$\ol \alpha$ again) so that $p$ is a homogeneous matrix in
$\M_k(A)(\ol \alpha)$, where $k\geq n$. Recall also the definition of $ \M_k(A)[\ol \alpha][\ol \delta]$ from \S\ref{kjujidw} and note that if $x \in  \M_k(A)[\ol \alpha][\ol \delta]$ and $y \in  \M_k(A)[\ol \de][\ol \alpha]$, then by Lemma~\ref{smallhandy} and (\ref{bequitet}), $xy \in \M_k(-\ol \alpha)_0$ and $yx\in \M_k(-\ol\de)_0$. 

\begin{definition}\label{equison}
Let $A$ be a $\Gamma$\!-graded ring and let $\ol
\alpha=(\alpha_1,\dots,\alpha_n)$ and $\ol
\delta=(\delta_1,\dots,\delta_m)$, where $\alpha_i, \delta_j \in
\Gamma$. Let $p\in \M_n(A)(\ol \alpha)_0$ and $q\in \M_m(A)(\ol
\delta)_0$ be idempotent matrices. Then $p$ and $q$ are \emph{graded equivalent}, \index{graded equivalent idepomtents} denoted by $p\sim q$, if (after suitably
enlarging them) there are $x \in \M_k(A)[-\ol \alpha][-\ol \delta]$ and
$y \in \M_k(A)[-\ol \delta][-\ol \alpha]$ such that $xy=p$ and $yx=q$. Moreover, we say $p$ and $q$ are  \emph{graded conjugate} \index{graded conjugate}
if there is an invertible matrix $g\in \M_{k}(A)[-\ol \delta][-\ol \alpha]$ with $g^{-1} \in \M_{k}(A)[-\ol \alpha][-\ol \delta]$ such that $gpg^{-1}=q$. 
\end{definition}

The relation $\sim$ defined above is an equivalence relation. Indeed, if $p\in \M_n(A)(\ol \alpha)_0$ is a homogeneous idempotent, then
considering \[x=p \in \M_n(A)[-\ol \alpha][-\ol \alpha],\] (see~(\ref{bequitet})) and \[y=1\in  \M_n(A)[-\ol \alpha][-\ol \alpha],\] Definition~\ref{equison} shows that $p\sim p$. Clearly $\sim$ is reflexive. The following trick shows that $\sim$ is also transitive. Suppose $p\sim q$ and $q\sim r$. Then 
$p=xy$, $q=yx$ and $q=vw$ and $r=wv$. Thus $p=p^2=xyxy=xqy=(xv)(wy)$ and 
$r=r^2=wvwv=wqv=(wy)(xv)$. This shows that $p\sim r$.

Let  $p\in \M_n(A)(\ol \alpha)_0$ be a homogeneous idempotent. Then one can easily see that for any $\beta=(\beta_1,\dots,\beta_m)$, 
 the idempotent $\left(\begin{matrix} p & 0\\ 0 & 0 \end{matrix}\right) \in \M_{n+m}(A)(\ol \alpha, \ol \beta)_0$ is graded equivalent to $p$. We call this element an  \emph{enlargement} of $p$. \index{enlargement of an homogeneous idempotent}  

\begin{lemma}\label{ijhlordjes}
Let $p\in \M_n(A)(\ol \alpha)_0$ and $q\in \M_m(A)(\ol \delta)_0$ be idempotent matrices. Then $p$ and $q$ are graded equivalent if and only if some enlargements of $p$ and $q$ are conjugate. 
\end{lemma}
\begin{proof}
Let  $p$ and $q$  be graded equivalent idempotent matrices. By
Definition~\ref{equison}, there are $x' \in \M_k(A)[-\ol \alpha][-\ol
\delta]$ and $y' \in \M_k(A)[-\ol \delta][-\ol \alpha]$ such that
$x'y'=p$ and $y'x'=q$. Let $x=px'q$ and $y=qy'p$. Then
\[xy=px'qy'p=p(x'y')^2p=p\] and similarly $yx=q$. Moreover, 
$x=px=xq$ and $y=yp=qy$. Using Lemma~\ref{smallhandy}, one can check that 
$x \in \M_k(A)[-\ol \alpha][-\ol
\delta]$ and $y \in \M_k(A)[-\ol \delta][-\ol \alpha]$. 
We now use the standard argument, taking into account the shift of the matrices. Consider the matrix
\[\left(\begin{matrix} 1-p & x\phantom{-1}\\ y\phantom{-1}  & 1-q \end{matrix}\right) \in \M_{2k}(A)[-\ol \alpha, -\ol \delta][-\ol \alpha, -\ol \delta].\]
This matrix has order two and thus is its own inverse. Then
\begin{equation*}
\left(\begin{matrix} 1-p & x\phantom{-1}\\ y\phantom{-1} & 1-q \end{matrix}\right) \left(\begin{matrix} p & 0\\ 0 & 0 \end{matrix}\right)
\left(\begin{matrix} 1-p & x\phantom{-1}\\ y\phantom{-1} & 1-q \end{matrix}\right)=\left(\begin{matrix} 1-p & x\phantom{-1}\\ y\phantom{-1} & 1-q \end{matrix}\right)
\left(\begin{matrix} 0 & x\\ 0 & 0 \end{matrix}\right) =\left(\begin{matrix} 0 & 0\\ 0 & q \end{matrix}\right).
\end{equation*}
Moreover, 
\begin{equation*}
\left(\begin{matrix} 0 & 1\\ 1 & 0 \end{matrix}\right) \left(\begin{matrix} 0 & 0\\ 0 & q \end{matrix}\right)
\left(\begin{matrix} 0 & 1\\ 1 & 0 \end{matrix}\right)=\left(\begin{matrix} q & 0\\ 0 & 0 \end{matrix}\right).
\end{equation*}
Now considering the enlargements  of $p$ and $q$ \[\left(\begin{matrix} p & 0\\ 0 & 0 \end{matrix}\right) \in \M_{2k}(A)(\ol \alpha, \ol \delta)_0\] 
 \[\left(\begin{matrix} q & 0\\ 0 & 0 \end{matrix}\right) \in \M_{2k}(A)(\ol \delta, \ol \alpha)_0,\] we have 
 $g  \left(\begin{matrix} p & 0\\ 0 & 0 \end{matrix}\right) g^{-1} = \left(\begin{matrix} q & 0\\ 0 & 0 \end{matrix}\right)$, where 
 \[g= \left(\begin{matrix} 0 & 1\\ 1 & 0 \end{matrix}\right) \left(\begin{matrix} 1-p & x\phantom{-1}\\ y\phantom{-1} & 1-q \end{matrix}\right) 
  \in \M_{2k}(A)[-\ol \delta, -\ol \alpha ][-\ol \alpha,-\ol \delta].\] This shows that some enlargements of $p$ and $q$ are conjugates. Conversely, suppose some enlargements of $p$ and $q$ are conjugates. Thus there is a \[g\in \M_{2k}(A)[-\ol \delta,-\ol \mu][-\ol \alpha, -\ol \beta]\] such that 
  $g\left(\begin{matrix} p & 0\\ 0 & 0 \end{matrix}\right) g^{-1} = \left(\begin{matrix} q & 0\\ 0 & 0 \end{matrix}\right)$, where
$\left(\begin{matrix} p & 0\\ 0 & 0 \end{matrix}\right) \in \M_{k}(A)(\alpha,\beta)_0$ and 
$\left(\begin{matrix} q & 0\\ 0 & 0 \end{matrix}\right) \in \M_{k}(A)(\delta,\mu)_0$. Now setting $x=g\left(\begin{matrix} p & 0\\ 0 & 0 \end{matrix}\right)$ and $y=g^{-1}$, gives that the enlargements of $p$ and $q$ and consequently $p$ and $q$ are graded equivalent. 
\end{proof} 

The following lemma relates graded finitely generated projective modules to homogeneous idempotent matrices which eventually leads to an equivalent definition of the graded Grothendieck group by idempotents. 

\begin{lemma}\label{kzeiogr} 
Let $A$ be a $\Gamma$\!-graded ring.

\begin{enumerate}[\upshape(1)]

\item Any graded finitely generated  projective $A$-module $P$ gives rise to a homogeneous idempotent matrix $p \in \M_n(A)(\ol \alpha)_0$, for some $n \in \mathbb N$ and $\ol
\alpha=(\alpha_1,\dots,\alpha_n)$, such that $P\conggr pA^n(-\ol \alpha)$.

\smallskip 

\item Any homogeneous idempotent matrix  $p\in \M_n(A)(\ol \alpha)_0$ gives rise to a graded finitely generated  projective $A$-module $pA^n(-\ol \alpha).$

\smallskip 

\item Two  homogeneous idempotent matrices  are graded equivalent if and only if the corresponding graded finitely generated  
projective $A$-modules are graded isomorphic.
\end{enumerate}
\end{lemma}
\begin{proof}
(1) Let $P$ be a graded finitely generated   projective (right) $A$-module. Then there is a graded module $Q$ such that 
$P\oplus Q\cong_{\gr} A^n(-\ol \alpha)$ for some $n \in \mathbb N$ and $\ol
\alpha=(\alpha_1,\dots,\alpha_n)$, where $\alpha_i \in \Gamma$  (see~(\ref{medickiue})).
Identify $A^n(-\ol \alpha)$ with $P\oplus
Q$ and define the homomorphism $p \in \End_A(A^n(-\ol \alpha))$ which sends
$Q$ to zero and acts as identity on $P$. Clearly, $p$ is an
idempotent and graded homomorphism of degree $0$. By~(\ref{hgd543p}) $p \in \End_A(A^n(-\ol \alpha))_0$. 
By~(\ref{mmkkhh}), the homomorphism $p$ can be represented by a matrix $p\in  \M_n(A)(\ol \alpha)_0$ acting from the left so, \[P\conggr pA^n(-\ol \alpha).\] 

(2) Let $p\in \M_n(A)(\ol \alpha)_0$, $\ol
\alpha=(\alpha_1,\dots,\alpha_n)$, where $\alpha_i \in \Gamma$. 
Since for any $\gamma \in \Gamma$, $pA^n(-\ol \alpha)_\gamma \subseteq A^n(-\ol\alpha)_\gamma$,  we have \[pA^n(-\ol \alpha) =\bigoplus_{\gamma \in \Gamma} pA^n(-\ol\alpha)_\gamma.\] This shows that $pA^n(-\ol \alpha)$ is a graded finitely generated $A$-module. Moreover, 
$1-p \in \M_n(A)(\ol \alpha)_0$ and
\[A^n(-\ol \alpha) = pA^n(-\ol \alpha) \bigoplus (1-p)A^n(-\ol \alpha).\] Thus $pA^n(-\ol \alpha)$
is a graded finitely generated   projective $A$-module.

\smallskip 

(3) Let  $p\in \M_n(A)(\ol \alpha)_0$ and $q\in \M_m(A)(\ol
\delta)_0$ be graded equivalent idempotent matrices. The first part is similar to the proof of Lemma~\ref{ijhlordjes}. By
Definition~\ref{equison}, there are $x' \in \M_k(A)[-\ol \alpha][-\ol
\delta]$ and $y' \in \M_k(A)[-\ol \delta][-\ol \alpha]$ such that
$x'y'=p$ and $y'x'=q$. Let $x=px'q$ and $y=qy'p$. Then
$xy=px'qy'p=p(x'y')^2p=p$ and similarly $yx=q$. Moreover, 
$x=px=xq$ and $y=yp=qy$. Now the left multiplication by $x$ and $y$
induce graded right $A$-homomorphisms $qA^k(-\ol \delta)\rightarrow pA^k(-\ol \alpha)$ and 
$pA^k(-\ol \alpha)\rightarrow qA^k(-\ol
\delta)$, respectively, which are inverse of each other. Therefore $pA^k(-\ol
\alpha)\cong_{\gr} qA^k(-\ol \delta)$.

On the other hand if $f:pA^k(-\ol \alpha)\cong_{\gr} qA^k(-\ol
\delta)$, then extend $f$ to $A^k(-\ol \alpha)$ by sending $(1-p)A^k(-\ol
\alpha)$ to zero and thus define a map
\[\theta:A^k(-\ol \alpha)=pA^k(-\ol \alpha)\oplus (1-p)A^k(-\ol \alpha) \longrightarrow qA^k(-\ol \delta)\oplus (1-q)A^k(-\ol \delta)=A^k(-\ol \delta).\]
Similarly, extending $f^{-1}$ to $A^k(-\ol \delta)$, we get a map
\[\phi:A^k(-\ol \delta)=qA^k(-\ol \delta)\oplus (1-q)A^k(-\ol \delta) \longrightarrow pA^k(-\ol \alpha)\oplus (1-p)A^k(-\ol \alpha)=A^k(-\ol \alpha)\]
such that $\phi \theta=p$ and $\theta \phi=q$. It follows  that $\theta \in \M_k(A)[-\ol \alpha][-\ol \delta]$, whereas  $\phi \in
\M_k(A)[-\ol \delta][-\ol \alpha]$ (see~\S\ref{kjujidw}). This gives that $p$ and $q$ are graded  equivalent.
\end{proof}

For a homogeneous idempotent matrix $p$ of degree zero, we denote the graded equivalence class of $p$, by $[p]$ (see Definition~\ref{equison})  and we define $[p]+[q]=\Big [\left(\begin{matrix} p & 0\\ 0 & q \end{matrix}\right)\Big ]$. This makes the set of equivalence classes of homogeneous idempotent matrices of degree zero a monoid.  Lemma~\ref{kzeiogr}  shows that this monoid is isomorphic to $\mathcal V^{\gr}(A)$, via $[p]\mapsto [pA^n(-\ol\alpha)]$, where $p\in \M_n(A)(\ol \alpha)_0$ is a homogeneous matrix. Thus 
$K_0^{\gr}(A)$ can be defined as the group completion of this monoid. In fact, this is the definition we adopt for $\mathcal V^{\gr}$ when the graded ring $A$ does not have identity (see~\S\ref{wedk0}). 

\subsection{Stability of idempotents}\label{didood}
Recall that two graded finitely generated projective $A$-modules $P$ and $Q$ are stably isomorphic if $[P]=[Q] \in K^{\gr}_0(A)$ (see Lemma~\ref{rhodas}(3)). Here we describe this stability when the elements of $K^{\gr}_0$ are represented by idempotents.

Suppose $p\in \M_n(A)(\ol \alpha)_0$ and $q \in \M_m(A)(\ol \delta)_0 $ are homogeneous idempotent matrices of degree zero such that $[p]=[q]$ in $K^{\gr}_0(A)$.  As in Definition~\ref{equison}, we add zeros to the lower right hand corner of $p$ and $q$ so that $p,q$ can be considered as matrices in $\M_k(A)$ for some $k \in \mathbb N$. 
Now $[p]$ represents the isomorphism class $[pA^k(-\ol \alpha)]$ and $[q]$ represents $[qA^k(-\ol \delta)]$ in $K^{\gr}_0(A)$ (see Lemma~\ref{kzeiogr}).  Since $[p]=[q]$, $[pA^k(-\ol \alpha)]=[qA^k(-\ol \delta)]$ which by Lemma~\ref{rhodas}(3), implies 
\[pA^k(-\ol \alpha)\oplus A^n(\overline \beta) \conggr qA^k(-\ol \delta)\oplus A^n(\overline \beta),\] for some $n\in \mathbb N$. Again by Lemma~\ref{kzeiogr} this implies $p\oplus I_n=
\left(\begin{matrix} p & 0\\ 0 & I_n \end{matrix}\right)$
is graded equivalent to  $q\oplus I_n=\left(\begin{matrix} q & 0\\ 0 & I_n \end{matrix}\right),$
 where $I_n$ is the identity element of the ring $\M_n(A)$. (This can also be seen directly from Definition~\ref{equison}.) This observation will be used in Theorem~\ref{kcontis} and later in Lemma~\ref{fdpahf2}.  

\subsection{Action of $\Gamma$ on idempotents}\label{beegees}
We define the action of $\Gamma$ on the idempotent matrices as follows: For $\ga \in \Ga$ and $p\in \M_n(A)(\ol \alpha)_0$, $\ga p$ is represented by the same matrix as $p$ but considered in $\M_n(A)(\ol \alpha-\ga)_0$, where 
$\overline \alpha-\ga=(\alpha_1-\ga,\cdots, \alpha_n-\ga)$. Note that if $p \in \M_n(A)(\ol \alpha)_0$ and $q\in \M_m(A)(\ol\delta)_0$ are equivalent, then there are $x \in \M_k(A)[-\ol \alpha][-\ol \delta]$ and
$y \in \M_k(A)[-\ol \delta][-\ol \alpha]$ such that $xy=p$ and $yx=q$ (Definition~\ref{equison}). Since $x \in \M_k(A)[\gamma -\ol \alpha][\gamma -\ol \delta]$
and $y \in \M_k(A)[\gamma -\ol \delta][\gamma -\ol \alpha]$, it follows that $\gamma p$ is equivalent to $\gamma q$. Thus $K^{\gr}_0(A)$ becomes a $\mathbb Z[\Gamma]$-module with this definition.

Now a quick inspection of the proof of Lemma~\ref{kzeiogr} shows that the action of $\Gamma$ is compatible in both definitions of $K^{\gr}_0$. 

Let $A$ and $B$ be $\Gamma$\!-graded rings and $\phi:A\rightarrow B$ be a $\Gamma$\!-graded homomorphism. Using the graded homomorphism $\phi$, one can consider $B$ as a graded 
$A\!-\!B$-bimodule in a natural way (\S\ref{bimref}). Moreover, if $P$ is a graded right $A$-module, then $P\otimes_A B$ is a graded $B$-module (\S\ref{grtensie}). Moreover, if $P$ is finitely generated projective, so is $P\otimes_A B$. 
Thus one can define  a group homomorphism $\ol \phi:K^{\gr}_0(A) \rightarrow K^{\gr}_0(B)$, where $[P]\mapsto [P\otimes_A B]$ and extended to all $K^{\gr}_0(A)$. On the other hand, if 
$p \in \M_n(A)(\ol \alpha)_0$ is an idempotent matrix over $A$  then 
$\phi(p) \in \M_n(B)(\ol \alpha)_0$ is an idempotent matrix over $B$ obtained by applying $\phi$ to each entry of $p$.  This also induces a homomorphism on the level of $K^{\gr}_0$ using the idempotent presentations.  
Since \[pA^n(-\ol\alpha) \otimes_A B \conggr \phi(p)  B^n(-\ol\alpha),\] we have the following commutative diagram 
\begin{equation}\label{v4531}
\xymatrix{
[p] \ar[r] \ar[d]& [\phi(p)] \ar[d]\\
[pA^n(-\ol\alpha) ] \ar[r] & [pA^n(-\ol\alpha) \otimes_A B].
}
\end{equation}

This shows that the homomorphisms induced on $K^{\gr}_0$ by $\phi$ are compatible, whether  using the idempotent presentation or the module presentation for the graded Grothendieck groups. 

\subsection{$K^{\gr}_0$ is a continuous functor}\label{younghan}

Recall the construction of direct limit of graded rings from Example~\ref{penrith123}.  We are in a position to determine their graded Grothendieck groups. 

\begin{theorem}\label{kcontis}
Let $A_i$, $i\in I$, be a direct system of $\Gamma$\!-graded rings and $A=\varinjlim A_i$ be a $\Gamma$\!-graded ring. Then $K^{\gr}_0(A)\cong \varinjlim K^{\gr}_0(A_i)$ as $\mathbb Z[\Gamma]$-modules. 
\end{theorem}
\begin{proof}
First note that $K^{\gr}_0(A_i)$, $i\in I$, is a direct system of abelian groups so $\varinjlim K^{\gr}_0(A_i)$ exists with $\mathbb Z[\Gamma]$-module homomorphisms \[\phi_i:K^{\gr}_0(A_i) \rightarrow \varinjlim K^{\gr}_0(A_i).\] 
On the other hand, for any $i\in I$, there is a map $\psi_i: A_i \rightarrow A$ which induces $\overline \psi_i : K^{\gr}_0(A_i) \rightarrow K^{\gr}_0(A)$. Due to the universality of the direct limit, we have a  $\mathbb Z[\Gamma]$-module homomorphism $\phi:\varinjlim K^{\gr}_0(A_i) \rightarrow K^{\gr}_0(A)$ such that for any $i\in I$, the diagram 
\begin{equation}\label{relaxim}
\xymatrix{
\varinjlim K^{\gr}_0(A_i)  \ar[r]^-{\phi} & K^{\gr}_0(A) \\
K^{\gr}_0(A_i) \ar[u]^{\phi_i} \ar[ur]_{\overline \psi_i} & 
}
\end{equation}
is commutative. We show that $\phi$ is in fact  an isomorphism. We use the idempotent description of $K^{\gr}_0$ group to show this. 
Note that if $p$ is an idempotent matrix over $A_i$ for some $i\in I$, which gives the element $[p]\in K^{\gr}_0(A_i)$, then 
$\overline \psi_i ([p])=[\psi_i(p)]$, where $\psi_i(p)$ is an idempotent matrix over $A$ obtained by applying $\psi_i$ to each entry of $p$.  

Let $p$ be an idempotent matrix over $A=\varinjlim A_i$. Then $p$ has a finite number of entires and each is coming from some $A_i$, $i\in I$. Since $I$ is directed, there is a $j\in I$, such that $p$ is the image of an idempotent matrix in $A_j$. Thus the class $[p]$ in $K^{\gr}_0(A)$ is the image of an element of $K^{\gr}_0(A_j)$.  Since Diagram~\ref{relaxim} is commutative, there is an element in $\varinjlim K_0(A_i)$ which maps to [p] in $K^{\gr}_0(A)$. Since $K^{\gr}_0(A)$ is generated by elements $[p]$, this shows that $\phi$ is surjective. We are left to show that $\phi$ is injective. Suppose $x\in \varinjlim K^{\gr}_0(A_i)$ such that $\phi(x)=0$. Since there is $j\in I$ such that $x$ is the image of an element of $K^{\gr}_0(A_j)$, we have $[p]-[q] \in K^{\gr}_0(A_j)$ such that $\phi_j([p]-[q])=x$, where $p$ and $q$ are idempotent matrices over $A_j$. Again, since Diagram~\ref{relaxim} is commutative, we have $\overline \psi_j([p]-[q])=0$. Thus 
\[ [\psi_j(p)]=\overline \psi_j([p])= \overline \psi_j([q])=[\psi_j(q)] \in K^{\gr}_0(A).\] 
This shows that  $a=\left(\begin{matrix} \psi_j(p) & 0\\ 0 & I_n \end{matrix}\right)$ is equivalent to  $b=\left(\begin{matrix} \psi_j(q) & 0\\ 0 & I_n \end{matrix}\right)$ in $R$ (see \S\ref{younghan}). Thus there are matrices $x$ and $y$ over $R$ such that $xy=a$ and $yx=b$. Since the entires of $x$ and $y$ are finite, one can find $k\geq j$ such that $x$ and $y$ are images of matrices from $R_k$.  Thus 
$\left(\begin{matrix} \psi_{jk}(p) & 0\\ 0 & I_n \end{matrix}\right)$ is equivalent to  $\left(\begin{matrix} \psi_{jk}(q) & 0\\ 0 & I_n \end{matrix}\right)$ in $R_k$. This shows that $\overline \psi_{jk}([p])=\overline \psi_{jk}([q])$, \ie the image of $[p]-[q]$  in $K^{\gr}_0(A_j)$ is zero. Thus $x$ being the image of this element, is also zero. This finishes the proof. 
\end{proof}

\subsection{The Hattori-Stallings (Chern) trace map} \index{trace map}

Recall that, for a ring $A$, one can relate $K_0(A)$ to the Hochschild homology $HH_0(A)=A/[A,A]$, where $[A,A]$ is the subgroup generated by additive commutators $ab-ba$, $a,b \in A$. (Or more generally, if $A$ is a $k$-algebra, where $k$ is a commutative ring, then $HH_0(A)$ is a $k$-module.) The construction is as follows. \index{Hochschild homology} \index{additive commutator subgroup}

Let $P$ be a finitely generated projective right $A$-module. Then as in the introduction of \S\ref{hhidmi}, there is an  idempotent matrix $p\in \M_n(A)$ such that $P\cong pA^n$. Define 
\begin{align*}
T:\mathcal V(A) &\longrightarrow A/[A,A],\\
[P] &\longmapsto [A,A]+\Tr(p),
\end{align*}
where $\Tr$ is the trace map of the matrices.  If $P$ is isomorphic to $Q$ then the idempotent matrices associated to them, call them $p$ and $q$, are equivalent, i.e, $p=xy$ and $q=yx$. This shows that $\Tr(p)-\Tr(q) \in [A,A]$, so the map $T$ is a well-define homomorphism of groups. Since $K_0$ is the universal group completion (\S\ref{hghgti1}), the map $T$ induces a map on the level of $K_0$, which is called $T$ again, \ie \[T:K_0(A) \rightarrow A/[A,A].\] This map is called \emph{Hattori-Stallings trace map} or \emph{Chern map}. \index{Hattori-Stallings trace map} \index{Chern map} 

We carry out a similar construction in the graded setting. Let $A$ be a $\Gamma$\!-graded ring and $P$ be a graded finitely generated projective $A$-module. Then there is a homogeneous idempotent $p \in \M_n(A)(\ol \alpha)_0$, where $\alpha=(\alpha_1,\dots,\alpha_n)$, such that 
$P\conggr pA^n(-\ol \alpha)$. Note that by~(\ref{mmkkhh4}), $\Tr(p) \in A_0$. Set $[A,A]_0:=A_0\cap [A,A]$ and define 
\begin{align*}
T:\mathcal V^{\gr}(A) &\longrightarrow A_0/[A,A]_0,\\
[P] &\longmapsto [A,A]_0+\Tr(p).
\end{align*}
Similar to the ungraded case, Lemma~\ref{kzeiogr} applies to show that $T$ is a well-defined homomorphism. Note that the description of action of $\Gamma$ on idempotents (\S\ref{beegees}) shows that $T([P])=T([P(\alpha)])$ for any $\alpha \in \Gamma$. Again, this map induces a group homomorphism \[T:K^{\gr}_0(A) \rightarrow A_0/[A,A]_0.\] 
Further, the forgetful functor $U:\Gr A \rightarrow \Modd A$ (\S\ref{forgetful}) induces the right hand side of the following commutative diagram, whereas the left hand side is the natural map induces by inclusion $A_0 \subseteq A$, 
\begin{equation*}
\begin{split}
\xymatrix{
K^{\gr}_0(A) \ar[r]^-{T} \ar[d]_{U}& A_0/[A,A]_0 \ar[d]\\
K_0(A) \ar[r]^-{T} & A/[A,A]. & 
}
\end{split}
\end{equation*}

\section{$K^{\gr}_0$ of graded $*$-rings}\label{broods} \index{graded $*$-ring}

Let $A$ be a graded $*$-ring, \ie  $A$ is a $\Gamma$\!-graded ring with an involution ${}^*$ such that for $a\in A_\gamma$, 
$a^* \in A_{-\gamma}$, where $\gamma \in \Gamma$ (\S\ref{involudool}).  \index{Veronese subring} \index{Veronese module} \index{involutary graded ring}

Recall that for a graded finitely generated projective right $A$-module $P$, the dual module $P^*=\Hom_A(P,A)$ is a graded finitely generated projective  left $A$-module. Thus  
${P^*}^{(-1)}$ is a graded right $A$-module (see~(\ref{ontheair3})). Since $P^{**} \cong_{\gr} P$ as graded right $A$-modules, the map 
\begin{equation}\label{hhffwwq4}
[P] \mapsto [{P^*}^{(-1)}],
\end{equation}
 induces a $\Z_2$-action on $\mathcal V^{\gr}(A)$. 

Since $P(\alpha)^* = P^* (-\alpha)$ (see~(\ref{terrace1})), one can easily check that 
\[ {P(\alpha)^*}^{(-1)}= {P^*}^{(-1)} (\alpha). \]
This shows that the actions of $\Gamma$ and $\Z_2$ on $\mathcal V^{\gr}(A)$ commutes. This makes $\mathcal V^{\gr}(A)$ a $\Gamma \!-\! \Z_2$-bimodule or equivalently 
a $\Gamma\times \Z_2$-module. 

The graded Grothendieck group $K_0^{\gr}(A)$, being the group completion of  the monoid $\mathcal V^{\gr}(A)$, naturally inherits 
the $\Gamma$\!-module structure via~(\ref{hhffwwq}) and $\mathbb Z_2$-module structure via~(\ref{hhffwwq4}).  This makes $K_0^{\gr}(A)$ a $\Z[\Gamma]$-module and a $\Z[\Z_2]$-module. In particular if $A$ is a $\mathbb Z$-graded $*$-ring then $K_0^{\gr}(A)$ is a
$\mathbb Z[x,x^{-1}]$-module and a $\Z[x]/(x^2-1)$-module.

\begin{example}\scm[Action of $\Z_2$ on $K_0$]
\begin{enumerate}

\medskip 
\item 
If $A=F\times F$, where $F$ is a field, with ${}^*:(a,b)\mapsto (b,a)$ then the involution induces a homomorphism on the level of $K_0$ as follows, 
\begin{align*}
\overline {{}^*} :\Z\times \Z &\longrightarrow \Z\times \Z,\\
 (n,m) &\longmapsto (m,n).
 \end{align*}

\item Let $R$ be a $*$-ring and consider $\M_n(R)$ as an $*$-ring with the Hermitain transpose. Then there is a $\Z_2$-module isomorphism 
$K_0(R) \cong K_0(\M_n(R))$. 

\item  Let $R$ be a commutative von Neumann regular ring. It is known that any idempotent of $\M_n(R)$ is conjugate with the diagonal matrix with nonzero elements all $1$ \cite{camillio}. Thus if $R$ is a $*$-ring, then the action of $\Z_2$ induced on $K_0(R)$ is trivial. In particular if $F$ is a $*$-field, then the action of $\Z_2$ on $K_0(\M_n(F))$ is trivial (see (2)).

\end{enumerate}
\end{example}

\begin{example}\scm[Action of $\Z_2$ on $K_0$ of Leavitt path algebras is trivial]\label{hyhy6543}
\vspace{0.2cm}

Let $E$ be a finite graph and $A=\LL_K(E)$ be a Leavitt path algebra associated to $E$. Recall from Example~\ref{bfghrtd2} that 
$A$ is equipped with a graded $*$-involution. We know that the finitely generated projected $A$-modules are generated by modules of the form $uA$, where $u \in E^0$. To show that $\Z_2$  acts trivially on $\mathcal V(A)$ (and thus on $K_0(A)$), it is enough to show that ${uA}^*=\Hom_A(uA,A)$ as a right $A$-module is isomorphic to $uA$. Consider the map 
\begin{align*}
\Hom_A(uA,A) &\longrightarrow uA,\\
f &\longmapsto u {f(u)}^*.
\end{align*}
It is easy to see that this map is a right $A$-module isomorphism. 

\end{example}

\section{Relative $K^{\gr}_0$-group}\label{relkgp}

Let $A$ be a ring and $I$ be a two-sided ideal of $A$. The canonical epimorphism $f:A\rightarrow A/I$ induces a homomorphism on the level of $K$-groups. For example on the level of Grothendieck groups, we have 
\begin{align*}
\ol f:K_0(A) &\longrightarrow K_0(A/I),\\
[P]-[Q] &\longmapsto  [P/PI]-[Q/QI].
\end{align*}
  In order to complete this into a long exact sequence, one needs to introduce the relative $K$-groups, $K_i(A,I)$, $i\geq 0$. This is done for lower $K$-groups by Bass and Milnor and the following long exact sequence has been established (see~\cite[Theorem~4.3.1]{rosenberg}),  
\begin{multline*}
K_2(A) \longrightarrow K_2(A/I) \longrightarrow  K_1(A,I)  \longrightarrow K_1(A) \longrightarrow K_1(A/I) \longrightarrow\\  K_0(A,I) \longrightarrow  K_0(A) \longrightarrow K_0(A/I).
\end{multline*}

In this section, we define the relative  $K^{\gr}_0$ group and establish the sequence only on the level of graded Grothendieck groups. As we will see this requires a careful arrangements of the degrees of the matrices.

Let $A$ be a $\Gamma$ graded ring and $I$ be a graded two-sided ideal of $A$. Define the  \emph{graded double ring} of $A$ along $I$ \index{graded double ring} 
\[D^{\Gamma}(A,I)=\big \{\, (a,b) \in A\times A \mid a-b \in I \,\big \}.\] One can check that $D^{\Gamma}(A,I)$ is a  $\Gamma$\!-graded ring with 
\[D^{\Gamma}(A,I)_\gamma =  \big \{ \, (a,b) \in A_\gamma \times A_\gamma \mid a-b \in I_\gamma \, \big \},\] for any $\gamma \in \Gamma$.  

Let \[\pi_1,\pi_2: D^{\Gamma}(A,I) \longrightarrow A\] be the projections of the first and second components of $D^{\Gamma}(A,I)$ to $A$, respectively. The maps $\pi_i$, $i=1,2$, are $\Gamma$\!-graded ring homomorphisms and induce $\mathbb Z[\Gamma]$-module homomorphisms 
\[\ol \pi_1,\ol \pi_2: K^{\gr}_0(D^{\Gamma}(A,I)) \longrightarrow K^{\gr}_0(A).\]  Define the  \emph{relative graded Grothendieck group} \index{relative graded Grothendieck group} of the ring $A$ with respect to $I$ as 
\begin{equation}\label{hgmne2}
\boxed{K^{\gr}_0(A,I) := \ker \big( K^{\gr}_0(D^{\Gamma}(A,I)) \stackrel{\ol \pi_1}{\longrightarrow} K^{\gr}_0(A)\big).}
\end{equation}

The restriction of $\ol \pi_2$ to $K^{\gr}_0(A,I)$ gives a $\mathbb Z[\Gamma]$-module homomorphism \[\ol \pi_2: K^{\gr}_0(A,I) \longrightarrow K^{\gr}_0(A).\] The following theorem relates these groups. 

\begin{theorem}
Let $A$ be a $\Gamma$\!-graded ring and $I$ be a graded two-sided ideal of $A$. Then there is an exact sequence of $\mathbb Z[\Gamma]$-modules 
\[ K^{\gr}_0(A,I)  \stackrel{\ol \pi_2}{\longrightarrow} K^{\gr}_0(A) \stackrel{\ol f}{\longrightarrow} K^{\gr}_0(A/I).\]
\end{theorem}
\begin{proof}
We will use the presentation of $K^{\gr}_0$ by idempotents (\S\ref{hhidmi}) to prove the theorem. If $p$ is a matrix over the ring $A$, we will denote by $\ol p$ the image of $p$ under the canonical graded homomorphism $f:A\rightarrow A/I$. 
Let $[p]-[q]\in K^{\gr}_0(A,I).$ By the construction of $D^{\Gamma}(A,I)$, $p=(p_1,p_2)$, where $p_1,p_2$ are homogeneous idempotent matrices of $A$ such that $p_1-p_2$ is a matrix over $I$ (in fact over $I_0$), namely $\ol p_1 =\ol p_2$. Similarly $q=(q_1,q_2)$, where $q_1,q_2$ are homogeneous idempotent matrices and $\ol q_1 =\ol q_2$. Moreover, since \[[(p_1,p_2)]-[(q_1,q_2)]\in K^{\gr}_0(A,I),\] by the definition of the relative $K$-group, $[p_1]-[q_1]=0$. So 
\begin{equation}\label{plksq2}
\ol f([p_1]-[q_1])=[\ol p_1]-[\ol q_1]=0
\end{equation}
Taking into account that $\ol p_1 =\ol p_2$ and $\ol q_1 =\ol q_2$ we have 
\[\ol f \ol \pi_2([p]-[q])=\ol f\ol \pi_2([(p_1,p_2)]-[(q_1,q_2)])=[\ol p_2]-[\ol q_2]=[\ol p_1]-[\ol q_1]=0.\] 
This shows that $\Im(\ol \pi_2) \subseteq \ker \ol f$. 

Next we show that $ \ker \ol f \subseteq \Im(\ol \pi_2)$. Let $p\in \M_n(A)(\ol \alpha)_0$ and $q\in \M_m(A)(\ol
\delta)_0$ be idempotent matrices with $\ol
\alpha=(\alpha_1,\dots,\alpha_n)$ and $\ol
\delta=(\delta_1,\dots,\delta_m)$, where $\alpha_i, \delta_j \in
\Gamma$. Suppose $x=[p]-[q] \in K^{\gr}_0(A)$ and \[\ol f(x)=\ol f([p]-[q])=[\ol p]-[\ol q]=0.\] Since $[\ol p]=[\ol q]$ in $K^{\gr}_0(A/I)$, there is an $l \in \mathbb N$, such that 
$\ol p\oplus 1_l$ is graded equivalent to $\ol q\oplus 1_l$ in $A/I$ (see~\S\ref{didood}). We can replace $p$ by $p\oplus 1_l$ and $q$ by $q\oplus 1_l$ without changing $x$, and consequently we get that $\ol p$ is graded equivalent to $\ol q$. By Lemma~\ref{ijhlordjes} there is 
an invertible matrix $g \in \M_{2k}(A/I)[-\ol \delta][-\ol \alpha]$ such that 
\begin{equation}\label{joipn}
\ol q=g \ol p g^{-1}.
\end{equation}
 The following standard trick let us lift an invertible matrix over $A/I$ to an invertible matrix over $A$. Consider the invertible matrix 
\begin{equation}\label{jujufe3}
\left(\begin{matrix} g & 1\phantom{^{-1}}\\ 0 & g^{-1} \end{matrix}\right) \in \M_{4k}(A/I)[-\ol \delta,-\ol \alpha][-\ol \alpha,-\ol \delta].
\end{equation}
This matrix is a product of the following four matrices 
\begin{align*}
 \left(\begin{matrix} 1 & g\\ 0 & 1 \end{matrix}\right) \in \M_{4k}(A/I)[-\ol \delta,-\ol \alpha][-\ol \delta,-\ol \alpha], \, & 
\left(\begin{matrix} 1\phantom{^{-1}-} & 0\\ -g^{-1} & 1 \end{matrix}\right) \in \M_{4k}(A/I)[-\ol \delta,-\ol \alpha][-\ol \delta,-\ol \alpha], \\
\left(\begin{matrix} 1 & g\\ 0 & 1 \end{matrix}\right) \in \M_{4k}(A/I)[-\ol \delta,-\ol \alpha][-\ol \delta,-\ol \alpha], \, & 
 \left(\begin{matrix} 0 & -1\\ 1 & 0\phantom{-} \end{matrix}\right)\in \M_{4k}(A/I)[-\ol \delta,-\ol \alpha][-\ol \alpha,-\ol \delta], 
\end{align*}
where, each of them can be lifted from $A/I$ to an invertible matrix over $A$ with the same shift. Thus  
\begin{equation*}
\left(\begin{matrix} g & 1\\ 0 & g^{-1} \end{matrix}\right)= \left(\begin{matrix} 1 & g\\ 0 & 1 \end{matrix}\right)
\left(\begin{matrix} 1\phantom{^{-1}-}  & 0\\ -g^{-1} & 1 \end{matrix}\right)\left(\begin{matrix} 1 & g\\ 0 & 1 \end{matrix}\right)
 \left(\begin{matrix} 0 & -1\\ 1 & 0\phantom{-}  \end{matrix}\right).
\end{equation*}
Let \[h\in \M_{4k}(A)[-\ol \delta,-\ol \alpha][-\ol \alpha,-\ol \delta]\] be a matrix that lifts (\ref{jujufe3}) with 
\[h^{-1} \in \M_{4k}(A)[-\ol \alpha,-\ol \delta][-\ol \delta,-\ol \alpha].\] 

Consider the enlargements of $p$ and $q$ as \[\left(\begin{matrix} p & 0\\ 0 & 0 \end{matrix}\right) \in \M_{4k}(A)(\ol \alpha,\ol \delta)_0\] and 
\[ \left(\begin{matrix} q & 0\\ 0 & 0 \end{matrix}\right) \in \M_{4k}(A)(\ol
\delta,\alpha)_0.\] 
By Lemma~\ref{ijhlordjes}, $h \left(\begin{matrix} p & 0\\ 0 & 0 \end{matrix}\right) h^{-1}$ and $ \left(\begin{matrix} q & 0\\ 0 & 0 \end{matrix}\right)$ are
graded equivalent to $p$ and  $q$, respectively, so replacing them does not change $x$. Replacing $p$ and $q$ with the new representatives, from~(\ref{joipn}), it follows $\ol p=\ol q$. This means $p-q$ is a matrix over $I$, so $(p,q)$ is an idempotent matrix over $D^\Gamma(A,I)$. Now 
\[ [(p,p)]-[(p,q)] \in K^{\gr}_0(D^{\Gamma}(A,I))\] which maps to $x$ under $\ol \pi_2$. This completes the proof. 
\end{proof}

\section{$K^{\gr}_0$ of non-unital rings}\label{wedk0}

Let $A$ be a ring which does not necessarily have identity.  Let $R$ be a ring with identity such that $A$ is a two-sided ideal of $R$. Then $\mathcal V(A)$ is defined as 
\[\mathcal V(A):=\ker \big (\mathcal V(R) \longrightarrow \mathcal V(R/A) \big ).\] 
It is easy to see that 
\[
\mathcal V(A) = \big \{ \, [P] \mid P \text{ is a finitely generated projective $R$-module and } PA=P \, \big \},  
\]
where $[P]$ is the class of $R$-modules isomorphic to $P$ and addition is defined via direct sum as before (compare this with~(\ref{zhongshan})). One can show that this definition is independent of the choice of the ring $R$ by interpreting $P$ as an idempotent matrix of $\M_k(A)$ for a suitable $k\in \mathbb N$.  

In order to define $K_0(A)$ for a non-unital ring $A$, consider the  \emph{unitisation ring} \index{unitisation ring}  $\tilde A= \Z\times A$. The addition is component-wise and multiplication is defined as follows
\begin{equation}\label{gtai3}
(n,a)(m,b)=(nm,ma+nb+ab),
\end{equation}
where $m,n\in \Z$ and $a,b \in A$. 
This is a ring with $(1,0)$ as the identity element and $A$ a two-sided ideal of $\tilde A$ such that $\tilde A/A\cong \Z$. The canonical epimorphism $\tilde A\rightarrow \tilde A/A$ gives a natural homomorphism on the level of $K_0$ and then $K_0(A)$ is defined as 
\begin{equation}\label{thicold11}
K_0(A) :=\ker\big(K_0(\tilde A) \longrightarrow K_0(\tilde A /A)\big).
\end{equation}

This construction extends the functor $K_0$ from the category of rings with identity, to the category of rings (not necessarily with identity) to the category of abelian groups. 

\begin{remark}
Notice that, for a ring $A$ without unit,  we didn't define $K_0(A)$ as the group completion of $\mathcal V(A)$ (as defined in the unital case). This is because the group completion is not in general a left exact functor. For a non-unital ring $A$, let $R$ be a ring with identity containing $A$ as a two-sided ideal. Then we have the following exact sequence of $K$-theory:
\[K_1(R) \longrightarrow K_1(R/A) \longrightarrow  K_0(A) \longrightarrow  K_0(R) \longrightarrow K_0(R/A),\]
(see~\S\ref{relkgp} and~\cite[Theorem~2.5.4]{rosenberg}). There is yet another construction, the  \emph{relative Grothendieck group of $R$ and $A$}, $K_0(R,A)$, which one can show is isomorphic to $K_0(A)$ (see~\cite[\S7]{arafacchini} for  comparisons between 
the groups $K_0(A)$ and $\mathcal V(A)^+$ in the ungraded setting). 
\end{remark} 

A similar construction can be carried over to the graded setting as follows. Let $A$ be a $\Gamma$\!-graded ring which does not necessarily have identity (Remark~\ref{parknan}). Let $R$ be a $\Gamma$\!-graded ring with identity such that $A$ is a graded two-sided ideal of $R$. For example, consider $\tilde A=\mathbb Z \times A$ with multiplication given by (\ref{gtai3}).
 Moreover, $\tilde A=\mathbb Z \times A$ is a $\Gamma$\!-graded with 
 \begin{align}\label{thurattila}
\tilde A_0=\mathbb Z \times A_0, &\\
 \tilde A_\gamma=0\times A_\gamma, &\,\,\,\,\text{   for } \gamma\not =0. \notag
\end{align}

Define
\begin{multline*}
\mathcal V^{\gr}(A) = \big \{ \, [P] \mid P \text{ is a graded finitely generated projective $R$-module } \\ \text{ and } PA=P\,  \big \},  
\end{multline*}
where $[P]$ is the class of graded $R$-modules, graded isomorphic to $P$ and addition is defined via direct sum. 
Parallel to the ungraded setting we define
\begin{equation}\label{ytytew2}
\mathcal V^{\gr}(A):=\ker \big (\mathcal V^{\gr}(R) \longrightarrow \mathcal V^{\gr}(R/A) \big ).
\end{equation}

As in the proof of Lemma~\ref{kzeiogr}(1), since  $P\oplus Q \cong R^n(-\ol \alpha)$, this gives an idempotent matrix $p$ in $\M_n(R)(\ol \alpha)_0$. However, since $PA=P$, we have \[P\oplus QA=PA\oplus QA   \cong A^n(-\ol \alpha).\] This shows that $p\in  \M_n(A)(\ol \alpha)_0$. On the other hand, if \[p\in  \M_n(A)(\ol \alpha)_0\subseteq \M_n(R)(\ol \alpha)_0\] is an idempotent, then 
$pR^n(-\ol \alpha)$ is a graded finitely generated projective $R$-module such that 
\[pR^n(-\ol \alpha) A= pA^n(-\ol \alpha)=pR^n(-\ol \alpha).\] So $[pR^n(-\ol \alpha)] \in \mathcal V^{\gr}(A)$. Now a repeat  of Lemma~\ref{kzeiogr} shows that $\mathcal V^{\gr}(A)$ is isomorphic to the monoid of equivalence classes of graded idempotent
matrices over $A$ as in \S\ref{hhidmi}. This shows that the construction of $\mathcal V^{\gr}(A)$ is independent of the choice of the graded ring $R$.  It also shows that if $A$ has identity, the two constructions (using the graded projective $A$-modules versus the graded projective $R$-modules) coincide.

To define the graded Grothendieck group for non-unital rings, we use a similar approach as in~(\ref{thicold11}): Let $A$ be a ring (possibly without identity) and 
let $\tilde A=\mathbb Z \times A$ with the multiplication defined in~(\ref{gtai3}), so that $A$ is a graded two-sided ideal of $\tilde A$.  
The graded canonical epimorphism $\tilde A\rightarrow \tilde A/A$ gives a natural homomorphism on the level of $K^{\gr}_0$ and then $K^{\gr}_0(A)$ is defined as 
\begin{equation}\label{thicold}
\boxed{K^{\gr}_0(A) :=\ker\big(K^{\gr}_0(\tilde A) \longrightarrow K^{\gr}_0(\tilde A/A)\big).}
\end{equation}
Thus $K^{\gr}_0(A)$ is a $\mathbb Z[\Gamma]$-module. Since the graded homomorphism 
\begin{align*}
\phi:\tilde A &\longrightarrow \tilde A/A\cong \mathbb Z, \\
(n,a)&\longmapsto n
\end{align*}
splits, we obtain the split exact sequence 
\[0\longrightarrow K^{\gr}_0(A) \longrightarrow K^{\gr}_0(\mathbb Z\times A ) \longrightarrow K^{\gr}_0(\mathbb Z)\longrightarrow 0.\] 
By Example~\ref{pogbiaue}, $K^{\gr}_0(\mathbb Z)\cong \mathbb Z[x,x^{-1}]$ and  thus we get  
\[K^{\gr}_0(\tilde A) \cong K^{\gr}_0(A) \oplus \mathbb Z[x,x^{-1}].\]

Interpreting this in the language of the reduced $K^{\gr}_0$ (\S\ref{pearlf}), we have 
\[ \widetilde {K^{\gr}_0}(\tilde A) = K^{\gr}_0 (A).\]

Let $R$ and $S$ be non-unital graded rings and $\phi: R\rightarrow S$ a non-unital graded homomorphism. Then 
\begin{align*}
\tilde \phi: \tilde R &\longrightarrow \tilde S,\\
(r,n) &\longmapsto (\phi(r),n),
\end{align*}
 is a unital graded homomorphism and the following commutative diagram shows that there is an order preserving $\mathbb Z[\Gamma]$-module homomorphism between their $K^{\gr}_0$-groups. 
\begin{equation}\label{yhftgrte5421}
\xymatrix{
K^{\gr}_0(R) \ar@{^{(}->}[r] \ar@{.>}[d]^{\overline{\phi}} & K^{\gr}_0(\tilde R) \ar[r] \ar[d]^{\overline{\tilde \phi}} & K^{\gr}_0(\tilde R/R) \ar@{=}[d] \\ 
K^{\gr}_0(S) \ar@{^{(}->}[r] & K^{\gr}_0(\tilde S) \ar[r]  & K^{\gr}_0(\tilde S/S)
}
\end{equation}

This construction extends the functor $K^{\gr}_0$ from the category of graded rings with identity, to the category of graded rings (not necessarily with identity) to the category of $\mathbb Z[\Gamma]$-module. 

In Example~\ref{gor381}, we calculate the graded Grothendieck group of the non-unital ring of (countable) square matrices with finite number of nonzero entires using the idempotent representation. 

\begin{remark}\scm[Unitisation of a graded ring with identity]
\vspace{0.2cm}

If $\Gamma$\!-graded ring $A$ has identity, then the ring $\tilde A=\mathbb Z\times A$ defined by the multiplication~(\ref{gtai3}) is graded isomorphic to the cartesian product ring $\mathbb Z \times A$, where $\mathbb Z$ is a $\Gamma$\!-graded ring concentrated in degree zero (see Example~\ref{noisest}). 
Indeed the map 
\begin{align*}
\tilde A &\longrightarrow \mathbb Z \times A\\
(n,a) &\longmapsto (n,a+n 1_A)
\end{align*}
is a graded ring isomorphism of unital rings. This shows that if $A$ has an identity, the two definitions of $K^{\gr}_0$ for $A$ coincides. Note also that the unitisation ring $\tilde A$ is never  a strongly graded ring. 
\end{remark}

Similar to the ungraded setting (see from example~\cite[Theorem~1.5.9]{rosenberg}), one can prove that for a graded ideal $A$ of the graded ring $R$, \[K^{\gr}_0(R,A) \cong K^{\gr}_0(A)\] as a $\mathbb Z[\Gamma]$-modules. This shows that $K^{\gr}_0(R,A)$ depends only on the structure of the non-unital ring $A$.



\subsection{Graded inner automorphims} \label{kidenh}

Let $A$ be a graded ring and $f:A\rightarrow A$ an inner automorphism defined by $f(a)=r a r^{-1}$, where $r$ is a homogeneous
and invertible element of $A$ of degree $\delta$. Clearly $f$ is a graded automorphism. Then considering
$A$ as a graded left $A$-module via $f$, it is easy to observe, for any graded right $A$-module $P$, that there is a   
graded right $A$-module isomorphism 
\begin{align*}
P(-\delta) &\longrightarrow P \otimes_f A,\\
p &\longmapsto p \otimes r,
\end{align*}
(with the inverse $p\otimes a \mapsto p r^{-1}a $). 
This induces an isomorphism between the functors \[-\otimes_f A :\Pgrp A \longrightarrow \Pgrp A\] and 
$\delta$-suspension functor \[\mathcal T_\delta:\Pgrp A \longrightarrow \Pgrp A.\] Recall that a Quillen's $K_i$-group, $i\geq 0$,  is a functor from the category of exact categories with exact functors to the category of abelian groups (see~\S\ref{waraya}). Moreover, isomorphic functors induce the same map on the $K$-groups~\cite[p.19]{quillen}. Thus $K^{\gr}_i(f)=K^{\gr}_i(\mathcal T_\delta)$. Therefore if $r$ is a homogeneous element of degree zero, \ie $\delta=0$, then 
$K^{\gr}_i(f):K^{\gr}_i(A)\rightarrow K^{\gr}_i(A)$ is the identity map.

\section{$K^{\gr}_0$ is a pre-ordered module} \label{pregg51}
 
\subsection{$\mathbf \Gamma$\!-pre-ordered modules}  \label{pregg5}

An abelian group $G$ is called a \emph{pre-ordered abelian group} \index{pre-ordered abelian group} if there is a relation, denoted by $\geq$, on $G$ which is reflexive and transitive and it respects the group structure. It follows that the set $G_{+}:=\{\, x\in G \mid x\geq 0\,\}$ forms a monoid. Conversely, any monoid $C$ in $G$, induces a pre-ordering on $G$, by defining $x\geq y$ if $x-y\in C$. It follows that with this pre-ordering $G_{+}=C$, which is called the \emph{positive cone} of the ordering.
(Note that $G_{+}$ should not be confused by $G^{+}$ used for the group completion in~\S\ref{podong}.) \index{$G_{+}$, the cone of the group $G$} \index{positive cone of ordering}  \index{cone of ordering}

\begin{example}\label{windy1}
Let $V$ be a monoid and $V^+$ its completion (see~\S\ref{podong}). There exist a natural homomorphism $\phi:V \rightarrow V^+$ (see~(\ref{pogfris})), which makes $V^+$ a pre-ordered abelian group with the image of $V$ under this homomorphism as a positive cone of $V^+$. 
\end{example}

Let $G$ has a pre-ordering. An element $u\in G$ is called an \emph{order-unit} \index{order-unit} if $u\geq 0$ and for any $x\in G$, there is $n\in \mathbb N$, such that $nu \geq x$. 

For a ring $R$, by Example~\ref{windy1}, the Grothendieck group $K_0(R)$ is a pre-ordered abelian group. Concretely, consider the set of isomorphism classes of finitely generated projective $R$-modules in $K_0(R)$. This set forms a monoid which is considered the positive cone of $K_0(R)$ and is denoted by $K_0(R)_+$. This monoid induces a pre-ordering on $K_0(R)$. We check that with this pre-ordering $[R]$ is an order-unit. Let $u \in K_0(R)$. Then $u=[P]-[Q]$, where $P,Q$ are finitely generated projective $R$-modules. But there is a finitely generated projective module $P'$ such that $P\oplus P' \cong R^n$ as right $R$-modules, for some $n \in \mathbb N$. Then $n[R] \geq u$. Indeed, 
\begin{multline*}
n[R]-u=n[R]-[P]+[Q]=[R^n]-[P]+[Q]=\\ [P\oplus P']- [P]+[Q]=[P']+[Q] \in K_0(R)_{+}.
\end{multline*}

\begin{example}
Suppose $R$ is a ring such that $K_0(R)\not = 0$, but the order unit $[R]=0$. Since for any finitely generated projective module $P$, there is a finitely generated projective module $Q$ and 
$n\in \mathbb N$ such that $P\oplus Q\cong R^n$, it follows that \[-[P]=-n[R]+[Q]= [Q] \in K_0(R)_{+}.\] Thus  $K_0(R)=K_0(R)_{+}$. 
Therefore it is possible that $x\in K_0(R)$ with $x>0$ and $0>x$ simultaneously.  One instance of such a ring is constructed in Example~\ref{disini}
\end{example}

The Grothendieck group of a ring as a pre-ordered group is studied extensively in~\cite[\S15]{goodearlbook}. In particular it was established that 
 for the so called ultramatricial  algebras $R$, the abelian group $K_0(R)$ along with its positive cone $K_0(R)_+$ and the order-unit $[R]$ is a complete invariant (see~\cite[Theorem~15.26]{goodearlbook} and the introduction to \S\ref{ultriuy}). This invariant is also called the \emph{dimension group} \index{dimension group} in the literature as it coincides with an invariant called the dimension group by Elliot~\cite{elliot} to classify such algebras.  

Since we will consider the graded Grothendieck groups, which have an extra $\mathbb Z[\Gamma]$-module structure,  we need to adopt the above definitions on ordering to the graded setting.  For this reason, here we define the category of $\Gamma$\!-pre-ordered modules.

Let $\Gamma$ be a group and $G$ be a (left) $\Gamma$\!-module. Let $\geq$ be a reflexive and transitive relation on $G$ which respects the monoid and the module structures, \ie for $\gamma \in \Gamma$ and $x,y,z \in G$, if $x\geq y$, then $x+z\geq y+z$ and $\gamma x \geq \gamma y$. We call $G$ a \emph{$\Gamma$\!-pre-ordered module}. \index{$\Gamma$-pre-ordered module} We call $G$ a \emph{pre-ordered module} \index{pre-ordered module} when $\Gamma$ is clear from the context. The \emph{positive cone}  of $G$ is defined as $\{x \in G \mid x\geq 0\}$ and denoted by $G_{+}$.  The set $G_{+}$ is a $\Gamma$\!-submonoid of $G$, \ie a submonoid which is closed under the action of $\Gamma$.  In fact, $G$ is a $\Gamma$\!-pre-ordered module if and only if there exists a $\Gamma$\!-submonoid of $G$. (Since $G$ is a $\Gamma$\!-module, it can be  considered as a $\mathbb Z[\Gamma]$-module.)   \index{$G_{+}$, the cone of the group $G$}
An element $u \in G_{+}$ is called an \emph{order-unit} \index{order-unit} if for any $x\in G$, there are $\alpha_1,\dots,\alpha_n \in \Gamma$, $n \in \mathbb N$, such that 
\begin{equation}\label{meinchnhfr} \index{positive cone of ordering}  \index{cone of ordering}
\sum_{i=1}^n \alpha_i  u \geq x.
\end{equation}  
As usual, in this setting, we only consider homomorphisms which preserve the pre-ordering, \ie a $\Gamma$\!-homomorphism $f:G\rightarrow H$, such that $f(G_{+}) \subseteq H_{+}$. We refer to these homomorphisms as the \emph{order preserving homomorphisms}. We denote by $\mathcal P_{\Gamma}$ the category of pointed $\Gamma$\!-pre-ordered modules with pointed order preserving homomorphisms, \ie the objects are the pairs  $(G,u)$, where $G$ is a  $\Gamma$\!-pre-ordered module and $u$ is an order-unit, and $f:(G,u)\rightarrow (H,v)$ is an order preserving $\Gamma$\!-homomorphism such that $f(u)=v$. We sometimes write 
$f:(G,G_+,u)\rightarrow (H,H_+,v)$ to emphasis the ordering of $f$.
The focus of this book is to study the graded Gorthendieck group as a functor from the category of graded rings to the category $\mathcal P_{\Gamma}$ \index{order preserving homomorphism}

Note that when $\Gamma$ is a trivial group, we are in the classical setting of pre-ordered abelian groups. When $\Gamma=\mathbb Z$ and so $\mathbb Z[\Gamma]=\mathbb Z[x,x^{-1}]$, we simply write $\mathcal P$ for $\mathcal P_{\mathbb Z}$.

\begin{example}\label{pofsgtwn}\scm[$K^{\gr}_0(A)$ as a $\Gamma$\!-pre-ordered module]
\vspace{0.2cm} \index{$K_0^{\gr}(A)_+$, the positive cone of $K_0^{\gr}$}

Let $A$ be a $\Gamma$\!-graded ring. Then $K^{\gr}_0(A)$ is a $\Gamma$\!-pre-ordered module with  the set of isomorphic classes of  graded finitely generated  projective right $A$-modules  as the positive cone of ordering. This monoid is denoted by $K^{\gr}_0(A)_+$ which is 
the image of $\mathcal V^{\gr}(A)$ under the natural homomorphism $\mathcal V^{\gr}(A)$. We consider $[A]$ as an order-unit of this group. Indeed, if $x\in K^{\gr}_0(A)$, then by Lemma~\ref{rhodas}(1)  there are graded finitely generated  projective modules $P$ and $P'$ such that $x=[P]-[P']$. But there is a graded module $Q$ such that $P\oplus Q \cong A^n(\overline \alpha)$, where $\overline \alpha =(\alpha_1,\dots,\alpha_n)$, $\alpha_i \in \Gamma$ (see~(\ref{medickiue})).  Now 
\[ [A^n(\overline \alpha)] -x = [P]+[Q]-[P]+[P']=[Q]+[P']=[Q\oplus P'] \in K^{\gr}_0(A)_{+}.\] This shows that $\sum_{i=1}^n \alpha_i [A]=[A^n(\overline \alpha)]\geq x$. 
\end{example}

In~\S\ref{ultriuy} we will show that for the so called graded ultramatricial  algebras $R$, 
the $\mathbb Z[\Gamma]$-module  $K^{\gr}_0(R)$ along with its positive cone $K^{\gr}_0(R)_+$ and the order-unit $[R]$ is a complete invariant  (see Theorem~\ref{ujhosmr}). We denote this invariant $\big(K^{\gr}_0(R),K^{\gr}_0(R)_+,[R]\big)$ and we call it the \emph{graded dimension module} of $R$.  \index{graded dimension module}

\begin{theorem}\label{infroncov}
Let $A$ and $B$ be $\Gamma$\!-graded rings. If $A$ is graded Morita equivalent to $B$, then there is an order preserving $\mathbb Z[\Gamma]$-module isomorphism 
\[K^{\gr}_0(A) \cong K^{\gr}_0(B).\]
\end{theorem} 
\begin{proof}
By Theorem~\ref{grmorim}, there is a graded $A\!-\!B$-bimodule $Q$ such that the  functor $-\otimes_A Q:\Gr A \rightarrow \Gr B$ is a graded equivalence. The restriction to $\Pgrp B$ induces a graded equivalence 
$-\otimes_A Q:\Pgrp A \rightarrow \Pgrp B$. This in return induces an order preserving isomorphism $\phi:K^{\gr}_0(A) \rightarrow K^{\gr}_0(B)$ which the is also a $\Z[\Gamma]$-module isomorphism by Equation~\ref{izmir27june}.
\end{proof}

If $\Gamma$\!-graded rings $A$ and $B$ are graded Morita equivalent, \ie $\Gr A\approx \Gr B$, then by Theorem~\ref{grmorim}, $\Modd A\approx \Modd B$ and thus  not only by Theorem~\ref{infroncov}, $K^{\gr}_0(A) \cong K^{\gr}_0(B)$, but also $K_0(A) \cong K_0(B)$ as well. 

In a more  specific case, we have the following. 

\begin{lemma}\label{nnbok}
Let $A$ be a $\Gamma$\!-graded ring and $P$ be a $A$-graded progenerator.  Let $B=\End_A(P)$. Then $[P]$ is an order-unit in $K^{\gr}_0(A)$ and there is an order preserving $\mathbb Z[\Gamma]$-module isomorphism 
\[\big (K^{\gr}_0(B),[B] \big ) \cong \big(K^{\gr}_0(A),[P] \big ). \]
\end{lemma}
\begin{proof}
Since $P$ is a graded generator, by Theorem~\ref{grgenth}, there are $\alpha_i \in \Gamma$, $1\leq i \leq n$,  such that $[A] \leq \sum_i \alpha_i [P]$. Since $[A]$ is an order-unit, this immediately implies that $[P]$ is an order-unit. By~(\ref{fgts688}), the functor $-\otimes_B P:\Gr B \rightarrow \Gr A$ is a graded equivalence. The restriction to $\Pgrp B$ induces a graded equivalence 
$-\otimes_B P:\Pgrp B \rightarrow \Pgrp A$. This in return induces an order preserving $\mathbb Z[\Gamma]$-module isomorphism $\phi:K^{\gr}_0(B) \rightarrow K^{\gr}_0(A)$, with 
\[\phi([Q]-[L])=[Q\otimes_B P] -[L\otimes_B P],\] where $Q$ and $L$ are graded finitely generated projective $B$-modules. In particular $\phi([B])=[P]$. This completes the proof.  
\end{proof}

\section{$K^{\gr}_0$ of graded division rings}\label{ghgwbsia} \index{graded division ring}

In the case of graded division rings, using the description of graded free modules, one can compute the graded Grothendieck group completely.

\begin{proposition} \label{k0grof}
Let $A$ be a $\Gamma$\!-graded division ring with the support the subgroup $\Gamma_A$. Then 
the monoid of  isomorphism classes of $\Gamma$\!-graded finitely generated  projective $A$-modules is isomorphic to $\mathbb N[\Gamma /\Gamma_A]$.
Consequently  there is a canonical $\mathbb Z[\Gamma]$-module isomorphism
\begin{align*}
K_0^{\gr}(A)  \, \, &\longrightarrow \mathbb Z[\Gamma /\Gamma_A]\\
[A^n(\delta_1,\dots,\delta_n)] & \longmapsto \sum_{i=1}^n \Gamma_A+\delta_i.
\end{align*}
 In particular if $A$ is a (trivially graded) division ring, $\Gamma$ a group and $A$ considered as 
a $\Gamma$\!-graded division ring concentrated in
degree zero, then $K_0^{\gr}(A) \cong \mathbb Z[\Gamma]$ and  $[A^n(\delta_1,\dots,\delta_n)]\in K_0^{\gr}(A)$
corresponds to
 $\sum_{i=1}^n \delta_i$ in $\mathbb Z[\Gamma]$.
\end{proposition}
\index{complete set of coset representative}

\begin{proof}
By Proposition~\ref{rndconggrrna}, $A(\delta_1) \cong_{\gr} A(\delta_2)$ as graded $A$-module if and only if $\delta_1-\delta_2 \in \Gamma_A$. 
Thus any graded free module of rank $1$ is graded isomorphic to some $A(\delta_i)$, where $\{\delta_i\}_{i \in I}$ is a complete set of coset representative of the subgroup $\Gamma_A$ in $\Gamma$, \ie $\{\Gamma_A+\delta_i, i\in I\}$ represents $\Gamma/\Gamma_A$. Since any graded finitely generated  module $M$ over $A$ is graded free (Proposition~\ref{gradedfree}),  it follows that 
\begin{equation}\label{m528h}
M \cong_{\gr} A(\delta_{i_1})^{r_1} \oplus \dots \oplus A(\delta_{i_k})^{r_k},
\end{equation}
where $\delta_{i_1},\dots \delta_{i_k}$ are distinct elements of the coset representative. Now suppose 
\begin{equation}\label{m23gds}
M \cong_{\gr} A(\delta_{{i'}_1})^{{s}_1} \oplus \dots \oplus A(\delta_{{i'}_{k'}})^{s_{k'}}.
\end{equation}
Considering the $A_0$-module $M_{-\delta_{i_1}}$, from~(\ref{m528h}) we have $M_{-\delta_{i_1}}\cong A_0^{r_1}$. Comparing this with $M_{-\delta_{i_1}}$ from~(\ref{m23gds}), it follows that one of $\delta_{{i'}_j}$, $1\leq j \leq k'$, say, $\delta_{{i'}_1}$, has to be $\delta_{i_1}$ and so $r_1=s_1$ as $A_0$ is a division ring.  Repeating  the same argument for each $\delta_{i_j}$, $1\leq j \leq k$, we see  
 $k=k'$, $\delta_{{i'}_j}=\delta_{i_j}$ and $r_j=s_j$, for all $1\leq j \leq k$ (possibly after suitable permutation).  Thus any graded finitely generated  projective $A$-module can be written uniquely as $M \cong_{\gr} A(\delta_{i_1})^{r_1} \oplus \dots \oplus A(\delta_{i_k})^{r_k}
$,  where $\delta_{i_1},\dots \delta_{i_k}$ are distinct elements of the coset representative. The (well-defined) map 
\begin{align}
\mathcal V^{\gr}(A) & \rightarrow \mathbb N[\Gamma/\Gamma_A] \label{attitude}\\
[A(\delta_{i_1})^{r_1} \oplus \dots \oplus A(\delta_{i_k})^{r_k}
] & \mapsto r_1(\Gamma_A+\de_{i_1})+\dots+ r_k(\Gamma_A+\de_{i_k}), \notag
\end{align}
gives a $\mathbb N[\Gamma]$-monoid isomorphism between the monoid of  isomorphism classes of $\Gamma$\!-graded finitely generated  projective $A$-modules  $\mathcal V^{\gr} (A)$ and  $\mathbb N[\Gamma /\Gamma_A]$. The rest of the proof follows easily. 
\end{proof}

\begin{remark}\label{hghghgt4}
One can use Equation~\ref{diniat6} to calculate the graded $K$-theory of graded division algebras as well. One can also compute this group via $G^{\gr}_0$-theory of Artinian rings (see Corollary~\ref{ooppoo}).
\end{remark}

\begin{example}\label{upst}
Using Proposition~\ref{k0grof}, we calculate the graded $K_0$ of two types of graded fields and we determine the action of $\mathbb Z[x,x^{-1}]$ on these groups. These are graded fields obtained from  Leavitt path algebras of acyclic and $C_n$-comet graphs, respectively (see Theorem~\ref{gra1} and~\ref{grCn1}). 
\begin{enumerate}
\item Let $K$ be a field. Consider $A=K$ as a $\mathbb Z$-graded field with  support $\Gamma_A=0$, \ie $A$ is concentrated in degree $0$.  By Proposition~\ref{k0grof}, $K_0^{\gr}(A)\cong \mathbb Z[x,x^{-1}]$ as a $\mathbb Z[x,x^{-1}]$-module. With this presentation $[A(i)]$ corresponds to $x^i$ in $\mathbb Z[x,x^{-1}]$. 

\smallskip 

\item Let $A=K[x^n,x^{-n}]$ be a $\mathbb Z$-graded field with $\Gamma_A=n\mathbb Z$. By Proposition~\ref{k0grof}, 
\[ K_0^{\gr}(A)\cong \mathbb Z \big [ \mathbb Z / n \mathbb Z \big ]\cong \bigoplus_n \mathbb Z,\] is a $\mathbb Z[x,x^{-1}]$-module. The action of $x$ on 
$(a_1,\dots,a_n) \in \bigoplus_n \mathbb Z$ is \[x (a_1,\dots,a_n)= (a_n,a_1,\dots,a_{n-1}).\]
With this presentation $[A]$ corresponds to $(1,0,\dots,0)$ in $\bigoplus_n \mathbb Z$. Moreover, the map \[U:K^{\gr}_0(A) \longrightarrow K_0(A)\] induced by the forgetful functor (\S\ref{forgetful}), gives a group homomorphism 
\begin{align*}
U:\bigoplus_n \mathbb Z &\longrightarrow \Z, \\
(a_i) &\longmapsto \sum_i a_i.
\end{align*}
In fact, the sequence 
\begin{equation*}
\bigoplus_n \mathbb Z  \stackrel{f}{\longrightarrow} \bigoplus_n \mathbb Z \stackrel{U}{\longrightarrow} \Z \longrightarrow 0, 
\end{equation*}
where $f(a_1,\dots,a_n)=(a_n-a_1,a_1-a_2,\dots,a_{n-1}-a_n)$
is exact. In~\S\ref{vandengg} we systematically related $K^{\gr}_0$ to $K_0$ for certain rings. 

\end{enumerate}
\end{example}

\begin{example}\label{eggrktheory}
Consider the Hamilton quaternion algebra \[\mathbb H = \mathbb R \oplus
\mathbb R i \oplus \mathbb R j \oplus \mathbb R k.\] By Example~\ref{egofgrdivisionrings}, $\mathbb H$ is a  $\mathbb Z_2
\times \mathbb Z_2$-graded  division ring with $\Gamma_{\H}=\mathbb Z_2 \times \mathbb Z_2$. By  Proposition~\ref{k0grof}, 
$K_0^{\gr} (\mathbb H)\cong \mathbb Z$. In fact since $\mathbb H$ is a strongly $\mathbb Z_2
\times \mathbb Z_2$-graded  division ring, one can deduce the result using the Dade's theorem (see~\S\ref{jijigogo}), \ie 
$K_0^{\gr} (\mathbb H)\cong K_0(\mathbb H_0)=K_0(\mathbb R)\cong \mathbb Z$. 
\end{example}

\begin{example}\label{whyso8} \index{associated graded ring}
Let $(D,v)$ be a valued division algebra, where $v:D^*\rightarrow \Gamma$ is the valuation homomorphism. 
By  Example~\ref{grdiviwad}, there is a $\Gamma$\!-graded division algebra $\gr(D)$ associated to $D$, where $\Gamma_{\gr(D)}=\Gamma_D$.
By  Proposition~\ref{k0grof}, \[K^{\gr}_0(\gr(D))\cong \mathbb Z[\Gamma/\Gamma_D].\]
\end{example}

The following example generalises Example~\ref{eggrktheory} of the Hamilton quaternion algebra 
$\H$ as a $\Z_2 \times \Z_2$-graded ring.

\begin{remark}
We saw in Example~\ref{egofgrdivisionrings} that $\mathbb H$ can
be considered as a $\mathbb Z_2$-graded division ring. So
$\mathbb H$ is also strongly $\mathbb Z_2$-graded, and \[K_0^{\gr}
(\mathbb H) \cong K_0 (\mathbb H_0) = K_0 (\mathbb C) \cong \mathbb
Z.\] Then $Z(\mathbb H)= \mathbb R,$ which we can consider as a
trivially $\mathbb Z_2$-graded field, so by
Proposition~\ref{k0grof}, $K_0^{\gr} (\mathbb R) =
\mathbb Z \oplus \mathbb Z$. We note that for both grade groups,
$\Z_2$ and $\Z_2 \times \Z_2$, we have $K_0^{\gr} (\H) \cong \Z$, but
the $K_0^{\gr} (\R)$ are different. So the graded $K$-theory of a
graded ring depends not only on the ring, but also on its grade
group.
\end{remark}

\begin{example}\label{eggrktheorynotsameasktheory}
Let $K$ be a field and let $R= K[x^2 , x^{-2}]$. Then $R$ is a
$\mathbb Z$-graded field, with  support $2\mathbb Z$, where $R$ can be written as $R=
\bigoplus_{n \in \mathbb Z} R_n$, with $R_n = Kx^n$ if $n$ is even
and $R_n = 0$ if $n$ is odd.
Consider the shifted graded matrix ring $A = \M_3 (R)(0,1,1)$, which
has support $\mathbb Z$. Then we will show that $A$ is a graded
central simple algebra over $R$.


\medskip 
It is clear that the centre of $A$ is $R$, and $A$ is finite
dimensional over $R$. Recall that a graded ideal is generated by homogeneous elements. If $J$ is a nonzero graded ideal of $A$, then 
using the elementary matrices, we can show that $J=A$ (see
\cite[\S III, Exercise~2.9]{hungerford}), so $A$ is graded simple.

By Theorem~\ref{grCn1}, $A$ is the Leavitt path algebra of the following graph. 
\vspace{0.25cm}
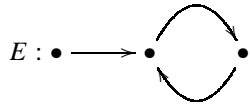
\begin{equation*}
{\def\labelstyle{\displaystyle}
\xymatrix{ 
E:   \bullet \ar[r] & \bullet \ar@/^1.5pc/[r] & \bullet \ar@/^1.5pc/[l] &
}}
\end{equation*}

\vspace{0.25cm}
\noindent Thus by Theorem~\ref{sthfin3} $A$ is a strongly $\mathbb Z$-graded ring.

Here, we also show that $A$ is a strongly $\mathbb Z$-graded
ring by checking  the conditions of  Proposition~\ref{crossedproductstronglygradedprop}. Since  $\mathbb Z$ is finitely generated, it
is sufficient to show that $I_3 \in A_1 A_{-1}$ and $I_3 \in A_{-1} A_1$. We have 
\begin{multline*}
I_3 =
\begin{pmatrix}
0 & 1 & 0 \\
0 & 0 & 0 \\
0 & 0 & 0
\end{pmatrix} 
\begin{pmatrix}
0 & 0 & 0 \\
1 & 0 & 0 \\
0 & 0 & 0
\end{pmatrix}
+  \begin{pmatrix}
0\phantom{^{2}} & 0 & 0 \\
x^2 & 0 & 0 \\
0\phantom{^{2}} & 0 & 0
\end{pmatrix} 
\begin{pmatrix}
0 & x^{-2} & 0 \\
0 & 0\phantom{^{-2}} & 0 \\
0 & 0\phantom{^{-2}} & 0
\end{pmatrix} \! 
+  \\
\begin{pmatrix} 
0\phantom{^{2}} & 0 & 0 \\
0\phantom{^{2}} & 0 & 0 \\
x^2  & 0 & 0
\end{pmatrix} 
\begin{pmatrix}
0 & 0 &  x^{-2} \\
0 & 0 & 0\phantom{^{-2}} \\
0 & 0 & 0\phantom{^{-2}}
\end{pmatrix} 
\end{multline*}
and 
\begin{multline*}
 I_3 =  
\begin{pmatrix}
0 &  x^{-2}  & 0 \\
0 & 0\phantom{^{-2}} & 0 \\
0 & 0\phantom{^{-2}} & 0
\end{pmatrix} 
\begin{pmatrix}
0\phantom{^{2}} & 0 & 0 \\
x^2 & 0 & 0 \\
0\phantom{^{2}} & 0 & 0
\end{pmatrix} 
+  \begin{pmatrix}
0 & 0 & 0 \\
1 & 0 & 0 \\
0 & 0 & 0
\end{pmatrix} 
\begin{pmatrix}
0 & 1 & 0 \\
0 & 0 & 0 \\
0 & 0 & 0
\end{pmatrix} 
+ \\ \begin{pmatrix}
0 & 0 & 0 \\
0 & 0 & 0 \\
1 & 0 & 0
\end{pmatrix}  
\begin{pmatrix}
0 & 0 & 1 \\
0 & 0 & 0 \\
0 & 0 & 0
\end{pmatrix}.
\end{multline*}

As in the previous examples, using
the Dade's theorem (see~\S\ref{jijigogo}), we have $K_0^{\gr} (A) \cong
K_0 (A_0)$. Since $R_0 = K$, by~(\ref{mmkkhh4}) there is a ring isomorphism
$$
A_0 = 
\begin{pmatrix} R_0\phantom{-} & R_1 & R_1 \\
R_{-1} & R_0 & R_0  \\
R_{-1} & R_0 & R_0
\end{pmatrix}=
\begin{pmatrix} K & 0 & 0 \\
0 & K & K  \\
0 & K & K
\end{pmatrix}
\cong K \times \M_2(K).
$$
Then
$$
K_0^{\gr} (A) \cong K_0 (A_0) \cong K_0 (K) \oplus K_0
\big(\M_2(K)\big) \cong \mathbb Z \oplus \mathbb Z,
$$
since $K_0$ respects Cartesian products and Morita equivalence. Note
that  
\[K_0(A)=K_0 \big( \M_3 (R)(0,1,1)
\big) \cong K_0 (R) = K_0 \big( K[x^2 , x^{-2}] \big)\cong \mathbb Z.\] 
The isomorphism $K_0(K[x^2 , x^{-2}]) \cong \mathbb Z$ comes from 
the fundamental theorem of algebraic $K$-theory
\cite[Theorem~3.3.3]{rosenberg} (see also \cite[p.~484]{magurn}), \[K_0
\big( K[x , x^{-1}] \big) \cong K_0 (K) \cong \mathbb Z,\] and that $K[x^2 , x^{-2}] \cong
K[x , x^{-1}]$ as rings. So the $K$-theory of $A$ is isomorphic to
one copy of $\mathbb Z$, which is not the same as the graded
$K$-theory of $A$.
\end{example}

\begin{example}\label{gaul1}\scm[Reduced $K^{\gr}_0$ of graded central simple algebras]
\vspace{0.2cm}

Recall from~\S\ref{pearlf}, that for a $\Gamma$\!-graded ring $A$, $\widetilde{K^{\gr}_0}$ is the cokernel of the homomorphism 
\begin{align*}\label{mozart1}
\phi:\mathbb Z[\Gamma]&\longrightarrow K^{\gr}_0(A),\\
\sum_\alpha n_\alpha \alpha &\longmapsto \sum_\alpha n_\alpha [A(\alpha)]. 
\end{align*}

Let $A$ be a $\Gamma$\!-graded division ring. We calculate $\widetilde{K^{\gr}_0}(\M_n(A))$, where $\M_n(A)$ is a $\Gamma$\!-graded ring (with no shift). The $\mathbb Z[\Gamma]$-module homomorphism $\phi$ above takes the form 
\begin{align*}
\mathbb Z[\Gamma]&\longrightarrow K^{\gr}_0(\M_n(A))\stackrel{\cong}{\longrightarrow} K^{\gr}_0(A)\stackrel{\cong}{\longrightarrow} \Z[\Gamma/\Gamma_A],\\
\alpha& \longmapsto [\M_n(A)(\alpha)]\longmapsto \bigoplus_n [A(\alpha)] \longmapsto n (\Gamma_A+\alpha),
\end{align*}
where the second map is induced by the Morita theory (see Proposition~\ref{grmorita}) and the third map is induced by Proposition~\ref{k0grof}. The cokernel of the composition maps gives $\widetilde{K^{\gr}_0}(\M_n(A))$, which one can immediately calculate 
\[\widetilde{K^{\gr}_0}(\M_n(A))\cong \frac{\Z}{n\Z}\big[\frac{\Gamma }{\Gamma_A}\big].\] 
In particular, for a graded division ring $A$, we have 
\[\widetilde{K^{\gr}_0}(A)\cong 0. \]
\end{example}

\begin{example}\label{gor381}
Let $K$ be a field.  Consider the following sequence of graded matrix rings  
\begin{equation}\label{plosneid1}
K\subseteq \M_2(K)(0,1) \subseteq \M_3(K)(0,1,2) \subseteq \M_4(K)(0,1,2,3)\subseteq \cdots 
\end{equation}
where the inclusion comes from the non-unital (graded) ring homomorphism of~(\ref{cafejen25}). 
Let $R$ be the graded ring 
\[R=\bigcup_{i=1}^{\infty} \M_{i}(K)(\ol \alpha_i),\]
where $\ol \alpha_i=(0,1,\dots,i-1)$. Note that $R$ is a non-unital ring. (This is a non-unital graded ultramatricial algebra which will be studied in~\S\ref{ultriuy}).  
We calculate $K^{\gr}_0(R)$. Since $K^{\gr}_0$ respects the direct limit (Theorem~\ref{kcontis}), we have 
\[K^{\gr}_0(R)=K^{\gr}_0(\varinjlim R_i)=\varinjlim K^{\gr}_0(R_i),\]
where $R_i=\M_i(K)(\ol\alpha_i)$. The non-unital homomorphism 
\[\phi_i:R_i \rightarrow R_{i+1},\] (see~(\ref{cafejen25})) induces the homomorphism $\ol \phi_i:K^{\gr}_0(R_i) \rightarrow K^{\gr}_0(R_{i+1})$ (see (\ref{yhftgrte5421})). We will observe that each of these $K$-groups is $\mathbb Z[x,x^{-1}]$ and each map $\phi_i$ is identity. First, note that the ring $R_i$, $i\in \mathbb N$, is unital, so using 
the Morita theory (see Proposition~\ref{grmorita}) and Example~\ref{upst}(1), $K^{\gr}_0(R_i)=K^{\gr}_0(K)\cong \mathbb Z[x,x^{-1}]$. 
Switching to the idempotent representation of the graded Grothendieck group, observe that the idempotent matrix $p_i$ having $1$ on the upper left corner and zero everywhere else is in ${R_i}_0=\M_i(K)(0,1,\dots,i-1)_0$ and \[p_iR_i\conggr K\oplus K(1)\oplus \dots \oplus K(i-1)\] (see Example~\ref{hygfdbte65}). Again Proposition~\ref{grmorita} (see also Proposition~\ref{instancegg}) shows that $[p_iR_i]=[K\oplus K(1)\oplus \dots \oplus K(i-1)]$ corresponds to $1 \in \mathbb Z[x,x^{-1}]$. Now $p_{i+1}:=\phi_i(p_i)$ gives again a matrix with $1$ on the upper left corner and zero everywhere else. So 
 $p_{i+1}R_{i+1}\conggr K\oplus K(1)\oplus \dots \oplus K(i)$ and consequently $[p_{i+1}R_{i+1}]=1$. So $\ol \phi_i(1)=1$. Since $\ol \phi_i$ respects the shift, this shows that $\ol\phi_i$ is the identity map. Thus $K^{\gr}_0(R)\cong \Z[x,x^{-1}]$. 

A similar argument shows that for the graded ring 
\[S=\bigcup_{i=1}^{\infty} \M_{i}(K)(\ol \alpha_i),\]
where $\ol \alpha_i=(0,1,1,\dots,1)$, with one zero and $1$ repeated $i-1$ times, we have an ordered  $\mathbb Z[x,x^{-1}]$-module isomorphism 
\[K^{\gr}_0(R) \cong  K^{\gr}_0(S).\]

Using a similar argument as in the proof of Theorem~\ref{gra1}, one can show that $R$ is isomorphic to the Leavitt path algebra associated to the infinite graph 
\begin{equation*}
\xymatrix{
E: &  \ar@{.>}[r] & \bullet \ar[r] & \bullet \ar[r] &  \bullet   \\
}
\end{equation*}
whereas $S$ is the Leavitt path algebra of a graph $F$ consisting of infinite vertices, one in the middle and the rest are connected to this vertex with an edge. 
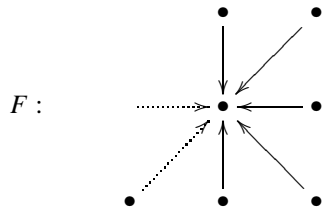
\begin{equation*}
\xymatrix{
&& \bullet \ar[d] & \bullet \ar[dl]\\
F: & \ar@{.>}[r] & \bullet  &  \bullet \ar[l]   \\
&\bullet \ar@{.>}[ru] & \bullet \ar[u] &  \bullet \ar[ul]   
}
\end{equation*}
Clearly the rings $R$ and $S$ are not graded isomorphic, as the support of $R$ is $\mathbb Z$ whereas the support of $S$ is $\{-1,0,1\}$. See Remark~\ref{idontcare} in regard to classification of graded ultramatricial algebras. 
\end{example}

\begin{example}\scm[$K_0^{\gr}$ ring of graded fields]\label{drring}

\vspace{0.2cm}

If $A$ is a $\Gamma$\!-graded field, then by~\S\ref{hhyyuvy}, $K_0^{\gr}(A)$ is a ring. One can easily check that by Proposition~\ref{k0grof} and (\ref{inenbuild}), $K_0^{\gr}(A)\cong \Gamma/\Gamma_A$ as algebras.
\end{example}

\section{$K^{\gr}_0$ of graded local rings}\label{attila1}

Recall from \S\ref{merrychirstmass} that a $\Gamma$\!-graded ring $A$ is called a graded local ring \index{graded local ring} if the two-sided ideal $M$ generate by 
noninvertible homogeneous elements of $A$ is a proper ideal. 

In Proposition~\ref{gdhmeisl}, we will explicitly calculate the graded Grothendieck group of graded local rings. There are two ways to do this. One can prove that a graded finitely generated projective module over a graded local ring is a graded free with a unique rank, as in the case of graded division rings (see \S\ref{mochihge}) and adopt the same approach  to calculate the graded Grothendieck group (see \S\ref{ghgwbsia}). However, there is a more direct way to do so, which we develop here.

Recall the definition of graded Jacobson radical of a graded ring $A$, $J^{\gr}(A)$, from~\S\ref{jghjrye}. We also need a graded version of the Nakayama lemma \index{graded Nakayama Lemma} which holds in this setting as well. Namely, if $J\subseteq J^{\gr}(A)$ is a graded right ideal of $A$ and $P$ is a graded finitely generated right $A$-module, and $Q$ is a graded submodule of $P$ such that $P=Q+PJ$, then $P=Q$ (see~\cite[Corollary~2.9.2]{grrings}).  \index{graded Jacobson radical}

We need the following two lemmas in order to calculate $K_0^{\gr}$ of a graded local ring. 

\begin{lemma}\label{thuphan}
Let $A$ be a $\Gamma$\!-graded ring, $J\subseteq J^{\gr}(A)$ a homogeneous two-sided ideal and $P$ and $Q$ be graded finitely generated projective $A$-modules. 
If $\overline P=P\otimes_A A/J=P/PJ$ and  $\overline Q=Q\otimes_A A/J=Q/QJ$ are isomorphic as graded $A/J$-modules, then $P$ and $Q$ are isomorphic as graded $A$-module. 
\end{lemma}
\begin{proof}
Let $\phi:\overline P \rightarrow \overline Q$ be a graded $A/J$-module isomorphism. Clearly $\phi$ is also a $A$-module isomorphism.  Consider Diagram~\ref{relyhea}. Since $\pi_2$ is an epimorphism, and $P$ is a graded finitely generated projective, there is a graded homomorphism $\psi:P\rightarrow Q$ which makes the diagram commutative. 
 \begin{equation}\label{relyhea}
\xymatrix{
P \ar[r]^{\pi_1} \ar@{.>}[d]_{\psi} & \overline P \ar[d]^{\phi} \ar[r]  &0\\
Q\ar[r]^{\pi_2}  &  \overline Q \ar[r] & 0
}
\end{equation}

We will show that $\psi$ is a graded $A$-module isomorphism. Since $\pi_2 \psi$ is an epimorphism, 
$Q=\psi(P)+\ker\pi_2=\psi(P)+QJ$. The Nakayama lemma then implies that $Q=\psi(P)$, \ie $\psi$ is an epimorphism. Now since $Q$ is graded projective, there is a graded homomorphism $i:Q\rightarrow P$ such that $\psi i=1_Q$ and $P=i(Q)\oplus \ker\psi$. But 
\[\ker\psi \subseteq \ker \pi_2\psi=\ker\phi \pi_1 =\ker\pi_1=PJ.\] 
Thus $P=i(Q)+PJ$. Again the Nakayama lemma implies $P=i(Q)$. Thus $\ker\psi=0$, and so $\psi:P \rightarrow Q$ is an isomorphism. 
\end{proof}

\begin{lemma}\label{akhtar1}
Let $A$ and $B$ be $\Gamma$\!-graded rings and $\phi:A\rightarrow B$ be a graded epimorphism such that $\ker\phi \subseteq  J^{\gr}(A)$. Then the induced homomorphism 
\[\overline \phi :K^{\gr}_0(A) \rightarrow K^{\gr}_0(B)\] is a $\mathbb Z[\Gamma]$-module monomorphism. 
\end{lemma}
\begin{proof}
Let $J=\ker(\phi)$. Without the loss of generality we can assume $B=A/J$. Let $x\in \ker\overline \phi $. Write $x=[P]-[Q]$ for two graded finitely generated projective $A$-modules $P$ and $Q$. So $\phi(x)=[\overline P]-[\overline Q]=0$, where $\overline P=P\otimes_A A/J=P/PJ$ and  $\overline Q=Q\otimes_A A/J=Q/QJ$ are graded $A/J$-modules. 
Since $\phi(x)=0$, by  Lemma~\ref{rhodas}(3), $\overline P \bigoplus B^n (\overline \alpha)\conggr \overline Q \bigoplus B^n (\overline \alpha)$, where $\overline \alpha=(\alpha_1,\dots,\alpha_n)$. So \[\overline{P \bigoplus A^n(\overline \alpha)}\conggr \overline{Q \bigoplus A^n(\overline \alpha)}.\] By Lemma~\ref{thuphan},  
$P \bigoplus A^n(\overline \alpha)\conggr Q \bigoplus A^n(\overline \alpha)$. Therefore $x=[P]-[Q]=0$ in $K^{\gr}_0(A)$. This completes the proof. 
\end{proof}

For a $\Gamma$\!-graded ring $A$, recall that $\Gamma_A$ is the support of $A$ and \[\Gamma^*_A=\{ \alpha \in \Gamma \mid A^*_\alpha \not = \varnothing\}\] is a subgroup of $\Gamma_A$. 

\begin{proposition}\label{gdhmeisl}
Let $A$ be a $\Gamma$\!-graded local ring. Then there is a $\mathbb Z[\Gamma]$-module isomorphism  \[K^{\gr}_0(A) \cong \mathbb Z[\Gamma/\Gamma^*_A].\] 
\end{proposition}
\begin{proof}

Let $M=J^{\gr}(A)$ be the unique graded maximal ideal of $A$. Then by Lemma~\ref{akhtar1} the homomorphism 
\begin{align}\label{jgeakd}
\phi:K^{\gr}_0(A) &\longrightarrow K^{\gr}_0\big(A/M\big) \\
[P] &\longmapsto [P\otimes _A A/M], \notag
\end{align}
is a monomorphism. Since $A/M$ is a graded division ring, by Proposition~\ref{k0grof},  $K^{\gr}_0(A/M)\cong \mathbb Z [\Gamma\big /\Gamma_{A/M}]$. 
Observe that $\Gamma_{A/M}=\Gamma^*_A$. On the other hand, for the graded $A$-module $A(\alpha)$, where $\alpha \in \Gamma$, we have 
\begin{equation}\label{penrithti}
A(\alpha)\otimes_A A/M \conggr (A/M)(\alpha),
\end{equation}
as the graded $A/M$-module. 
Since \[K^{\gr}_0(A/M)\cong \mathbb Z [\Gamma\big /\Gamma^*_A]\] is generated by $[(A/M)(\alpha_i)]$, where $\{ \alpha_i\}_{i\in I} $ is a complete set of coset representative of
 $\Gamma/\Gamma^*_A$ (see Proposition~\ref{k0grof}), from (\ref{jgeakd}) and (\ref{penrithti}) we get $\phi([A(\alpha_i)]=[(A/M)(\alpha_i)]$. So $\phi$ is an epimorphism as well. This finishes the proof.
\end{proof}

\begin{example}\scm[$K_0^{\gr}$ ring of commutative graded local rings]\label{lrring}

\vspace{0.2cm}

If $A$ is a commutative $\Gamma$\!-graded local ring,
 it follows from Proposition~\ref{gdhmeisl} and Example~\ref{drring} that $K_0^{\gr}(A)\cong \Gamma/\Gamma_A^*$ as algebras.
\end{example}

\section{$K^{\gr}_0$ of Leavitt path algebras}\label{kleaviti} \index{Leavitt path algebra}

For a graph $E$, its associated path algebra $\mathcal P(E)$ is a positively graded ring. In~\S\ref{todaytalkbit}, we will use Quillen's theorem on graded $K$-theory of such rings to calculate the graded Grothendieck group of paths algebras (see Theorem~\ref{hgrsiny}).

In this section we calculate the $K_0$ and $K^{\gr}_0$ of Leavitt path algebras. For one thing, they provide very nice examples. We first calculate the ungraded $K_0$ of the Leavitt path algebras in~\S\ref{lktbuy} and in~\S\ref{xuejun}, we determine their $K^{\gr}_0$-groups. 

\subsection{$K_0$ of Leavitt path algebras}\label{lktbuy}

For a Leavitt path algebra $\LL_K(E)$, the monoid $\mathcal V(\LL_K(E))$ is studied in~\cite{amp}. In particular using~\cite[Theorem~3.5]{amp}, one can calculate the Grothendieck group of a Leavitt path algebra  from the adjacency matrix of a graph (see~\cite[p.1998]{aalp}). We present the calculation of the Grothendieck group of a Leavitt path algebra here.

Let $F$ be a free abelian monoid generated by a countable set $X$. The nonzero elements of $F$ can be written as $\sum_{t=1}^n x_t$, where $x_t \in X$. 
Let $I\subseteq \mathbb N$ and $r_i, s_i$ be elements of $F$, where $i\in I$. We define an equivalence relation on $F$ denoted by $\langle r_i=s_i\mid i\in I \rangle$ as follows: Define a binary relation $\rightarrow$ on $F\backslash \{0\}$ by,  \[r_i+\sum_{t=1}^n x_t \rightarrow s_i+\sum_{t=1}^n x_t, i\in I\] and generate the equivalence relation on $F$ using this binary relation. Namely, $a \sim a$ for any $a\in F$ and for $a,b \in F \backslash \{0\}$, $a \sim b$ if there is a sequence \[a=a_0,a_1,\dots,a_n=b\] such that for each $t=0,\dots,n-1$ either $a_t \rightarrow a_{t+1}$ or $a_{t+1}\rightarrow a_t$. We denote the quotient monoid by
$F/\langle r_i=s_i\mid i\in I\rangle$.
Completing the monoid (see~\S\ref{hghgti1}),  one can see that there is a canonical group isomorphism 
\begin{equation} \label{monio}
\Big (\frac{F}{\langle r_i=s_i\mid i\in I \rangle}\Big)^{+} \cong \frac{F^{+}}{\langle r_i-s_i\mid i\in I \rangle}.
\end{equation}

Let $E$ be a graph (as usual we consider only graphs with no sinks) and  $A_E$ be the \emph{adjacency  matrix} \index{adjacency matrix}
$(n_{ij}) \in \mathbb Z^{E^0\oplus E^0}$, where $n_{ij}$ is the
number of edges from $v_i$ to $v_j$. 
Clearly the adjacency matrix depends on the ordering we put on
$E^0$. We usually fix
 an ordering on $E^0$.

 Multiplying the matrix $A_E^t-I$ from the left defines a
homomorphism \[\mathbb Z^{E^0} \longrightarrow
\mathbb Z^{E^0},\] where  $\mathbb Z^{E^0} $ is the direct sum of copies of $\mathbb Z$
indexed by $E^0$. The next
theorem shows that the cokernel of this map gives the Grothendieck
group of Leavitt path algebras. 

\begin{theorem}\label{wke}
Let $E$ be finite graph with no sinks and $\LL(E)$ be the Leavitt path algebra associated to $E$. Then
\begin{equation}
K_0(\LL(E))\cong \coker\big(A_E^t-I:\mathbb Z^{E^0} \longrightarrow \mathbb Z^{E^0}\big).
\end{equation}
\end{theorem}
\begin{proof}
Let $M_E$ be the abelian monoid generated by $\{v \mid v \in E^0 \}$ subject to the relations
\begin{equation}\label{phgqcu}
v=\sum_{\{\alpha\in E^{1} \mid s(\alpha)=v \}} r(\alpha),
\end{equation}
for every $v\in E^0$.  The
relations~(\ref{phgqcu}) can be then written as $A_E^t \overline v_i= I
\overline v_i$, where  $v_i \in E^0$ and
$\overline v_i$ is the $(0,\dots,1,0,\dots)$ with $1$ in the $i$-th
component.   Therefore,
\[M_E\cong \frac{F}{ \langle A_E^t \overline v_i = I\  \overline  v_i , v_i \in E^0  \rangle},\] where $F$ is the free abelian monoid generated by the vertices of $E$.
By~\cite[Theorem~3.5]{amp} there is a natural monoid isomorphism \[\V(\LL_K(E)) \cong M_E.\] So using~(\ref{monio}) we have,
\begin{equation}\label{pajd}
K_0(\LL(E))\cong \V(\LL_K(E))^{+}\cong M_E^{+}\cong \frac{F^+}{ \langle (A_E^t-I) \overline v_i, v_i \in E^0  \rangle}.
\end{equation}
Now $F^{+}$ is $\mathbb Z^{E^0}$ and it is easy to see that the
denominator in~(\ref{pajd}) is the image of $A_E^t-I:\mathbb
Z^{E^0} \longrightarrow \mathbb Z^{E^0}$.
\end{proof}

\begin{example}\label{ffexpnon}
Let $E$ be the following graph. 
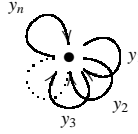
\[
\xymatrix{
   \bullet \ar@{.}@(l,d) \ar@(ur,dr)^{y_{1}} \ar@(r,d)^{y_{2}} \ar@(dr,dl)^{y_{3}} 
\ar@(l,u)^{y_{n}}
}\]

Then the Leavitt path algebra associated to $E$, $\LL(E)$, is the algebra constructed by Leavitt in~(\ref{jh54320}).
By Theorem~\ref{wke}, 
\[K_0(\LL(E)) \cong \mathbb Z \big /(n-1) \mathbb Z.\]  
\end{example}

\begin{example}\label{disini}
Here is an example of a ring $R$ such that $K_0(R)\not=0$ but $[R]=0$. Let $A$ be the Leavitt path algebra associated to the graph 
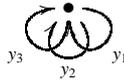
\[
\xymatrix{
   \bullet \ar@(d,l)^{y_{3}}  \ar@(r,d)^{y_{1}} \ar@(dr,dl)^{y_{2}} 
}\]
Thus as a right $A$-module, $A^3\cong A$ (see Example~\ref{levisuji}). By Example~\ref{ffexpnon}, $K_0(A)=\mathbb Z/2\mathbb Z$ and $2[A]=0$. The ring $R=\M_2(A)$ is Morita equivalent to $A$ (using the assignment 
$P_{\M_2(A)} \mapsto P\otimes_{\M_2(A)} A^2$, see Proposition~\ref{instancegg}). Thus $K_0(R)\cong K_0(A)\cong   \mathbb Z/2\mathbb Z$. Under this assignment, $[R]=[\M_2(A)]$ is sent to $2[A]$ which is zero, thus $[R]=0$.
\end{example}

\subsection{Action of $\Z$ on $K^{\gr}_0$ of Leavitt path algebras} \label{xuejun}

Recall that Leavitt path algebras have a natural $\mathbb Z$-graded structure (see~\S\ref{paohdme}).   The graded Grothendieck group as a possible invariant for these algebras was first considered in~\cite{hazgr}. In the case of finite graphs with no sinks, there is a good description of the action of $\mathbb Z$ on the graded Grothendieck group which we give here.

Let $E$ be a finite graph with no sinks. Set $\mathcal  A=\LL(E)$ which is a strongly $\mathbb Z$-graded ring by Theorem~\ref{sthfin}.  For any $u \in E^0$ and $i \in \mathbb Z$, $u\mathcal  A(i)$ is a graded finitely generated projective right $\mathcal  A$-module and any graded finitely generated projective $\mathcal  A$-module is generated by these modules up to isomorphism, \ie 
\begin{equation}\label{roxjhhre}
\boxed{\mathcal V^{\gr}(\mathcal  A)=\Big \langle \, \big [u\mathcal  A(i)\big ]  \mid u \in E^0, i \in \mathbb Z \, \Big \rangle. 
}
\end{equation}
By~\S\ref{podong},   $K_0^{\gr}(\mathcal  A)$ is the group completion of $\mathcal V^{\gr}(\mathcal  A)$. The action of $\mathbb N[x,x^{-1}]$ on $\mathcal V^{\gr}(\mathcal A)$ and thus the action of $\mathbb Z[x,x^{-1}]$ on $K_0^{\gr}(\mathcal  A)$ is defined on generators \[x^j [u\mathcal  A(i)]=[u\mathcal  A(i+j)],\] where $i,j \in \mathbb Z$. We first observe that for $i\geq 0$, 
\begin{equation}\label{hterw}
\boxed{
x[u\mathcal  A(i)]=[u\mathcal  A(i+1)]=\sum_{\{\alpha \in E^1 \mid s(\alpha)=u\}}[r(\alpha)\mathcal  A(i)].
}
\end{equation}

First notice that for $i\geq 0$, $\mathcal  A_{i+1}=\sum_{\alpha \in E^1} \alpha \mathcal  A_i$. It follows \[u \mathcal  A_{i+1}=\bigoplus_{\{\alpha \in E^1 \mid s(\alpha)=u\}} \alpha \mathcal  A_i\] as $\mathcal  A_0$-modules. Using the fact that $\mathcal  A_n\otimes_{\mathcal  A_0}\mathcal  A\cong \mathcal  A(n)$, $n \in \mathbb Z$, and the fact that $\alpha \mathcal  A_i \cong r(\alpha) \mathcal  A_i$ as $\mathcal  A_0$-module,
we get \[u\mathcal  A(i+1) \cong \bigoplus_{\{\alpha \in E^1 \mid s(\alpha)=u\}} r(\alpha) \mathcal  A(i)\] as graded $\mathcal  A$-modules. This gives~(\ref{hterw}).

Recall that for a $\Gamma$\!-graded ring $A$, $K^{\gr}_0(A)$ is a pre-ordered abelian group with  the set of isomorphic classes of  graded finitely generated  projective right $A$-modules  as the positive cone of ordering, denoted by $K^{\gr}_0(A)_+$  (\ie the image of $\mathcal V^{\gr}(A)$ under the natural homomorphism $\mathcal V^{\gr}(A)\rightarrow K^{\gr}_0(A)$). Moreover, $[A]$ is an order-unit. We call the triple, 
$(K^{\gr}_0(A),K^{\gr}_0(A)_+,[A])$ the \emph{graded dimension group} (see~\cite[\S15]{goodearlbook} for some background on dimension groups). \index{graded dimension group}

In~\cite{hazgr} it was conjectured that the graded dimension group is a complete invariant for Leavitt path algebras. Namely,  for graphs  $E$ and $F$,  $\LL(E)\cong_{\gr} \LL(F)$ if and only if 
there is an order preserving $\mathbb Z[x,x^{-1}]$-module  isomorphism 
\begin{equation}\label{ookjh2}
\phi: K_0^{\gr}(\LL(E)) \rightarrow  K_0^{\gr}(\LL(F))
\end{equation}
such that $\phi([\LL(E)]=\LL(F)$. 
 
\subsection{$K^{\gr}_0$ of a Leavitt path algebra via its 0-component ring}\label{sarahbright}
\index{0-component ring of a graded ring}

In this section we calculate the graded Grothendieck group of Leavitt path algebras. 
For a finite graph with no sinks, $\LL(E)$ is strongly graded (Theorem~\ref{sthfin}) and thus $K^{\gr}_0(\LL(E)) \cong K_0(\LL(E)_0)$ (Section~\ref{jijigogo}).
The structure of the ring of homogeneous elements of degree zero, $\LL(E)_0$, is  known. We recall the description of $\LL(E)_0$ in the setting of finite graphs with no sinks (see the proof of
Theorem~5.3 in~\cite{amp} for the general case). We will then use it to calculate $K^{\gr}_0(\LL(E))$. 

Let $A_E$ be the adjacency matrix of $E$. Let $L_{0,n}$ be the
linear span of all elements of the form $pq^*$ with $r(p)=r(q)$ and
$|p|=|q|\leq n$. Then 
\begin{equation}\label{ppooii}
\LL(E)_0=\bigcup_{n=0}^{\infty}L_{0,n},
\end{equation} with
the transition inclusion 
\begin{align}\label{curltoday}
L_{0,n}& \longrightarrow L_{0,n+1},\\
pq^* &\longmapsto \sum_{\{ \alpha | s(\alpha)=v\}} p\alpha (q\alpha)^*, \notag
\end{align}
where $r(p)=r(q)=v$ and extended linearly. Note that since $E$ does not have sinks, 
for any $v\in E_0$, the set $\{ \alpha | s(\alpha)=v\}$ is not
empty.

For a fixed $v \in E^0$, let
$L_{0,n}^v$ be the linear span of all elements of
the form $pq^*$ with $|p|=|q|=n$ and $r(p)=r(q)=v$. Arrange the paths
of length $n$ with the range $v$ in a fixed order
$p_1^v,p_2^v,\dots,p_{k^v_n}^v$, and observe that the correspondence
of  $p_i^v{p_j^v}^*$ to the matrix unit $\e_{ij}$ gives rise to a ring
isomorphism $L_{0,n}^v\cong\M_{k^v_n}(K)$. Moreover, $L_{0,n}^v$, $v\in E^0$ form a direct sum.  This implies that
\[L_{0,n}\cong
\bigoplus_{v\in E^0}\M_{k_n^v}(K),\] where $k_n^v$, $v \in E^0$, 
 is the number of paths of length $n$ with the range
$v$. The inclusion map $L_{0,n}\rightarrow L_{0,n+1}$ of (\ref{curltoday})  can be represented by 
\begin{equation}\label{volleyb}
A_E^t: \bigoplus_{v\in E^0} \M_{k^v_n}(K) \longrightarrow \bigoplus_{v\in E^0}
\M_{k^v_{n+1}}(K).
\end{equation}
  This
means $(A_1,\dots,A_l)\in \bigoplus_{v\in E^0} \M_{k^v_n}(K) $ is sent to
\[\Big (\sum_{j=1}^l n_{j1}A_j,\dots,\sum_{j=1}^l n_{jl}A_j \Big) \in
\bigoplus_{v\in E^0} \M_{k^v_{n+1}}(K),\] where $n_{ji}$ is the number of
edges connecting $v_j$ to $v_i$ and
\[\sum_{j=1}^lk_jA_j=\left(
\begin{matrix}
A_1 &             &             & &\\
       & \ddots   &             &  &\\ 
       &              & A_1            &  & \\
       &              &               & \ddots  &\\
       &         &            &       & A_l\\
       &         &            &       &  & \ddots \\
              &         &            &       &  & & A_l \\    
\end{matrix}
\right)
\]
in which each matrix is repeated $k_j$ times down the leading
diagonal and if $k_j=0$, then $A_j$ is omitted. This shows that $\LL(E)_0$ is an \emph{ultramatricial algebra}, \ie it is isomorphic to the union of an increasing chain of a finite product of matrix algebras over a field $K$. These algebras will be studied in~\S\ref{ultriuy}. \index{ultramatricial algebra}

Writing $\LL(E)_0=\varinjlim_{n} L_{0,n}$, since the Grothendieck group $K_0$ respects the
direct limit, we have \[K_0(\LL(E)_0)\cong
\varinjlim_{n}K_0(L_{0,n}).\] Since $K_0$ of  (Artinian) simple
algebras are $\mathbb Z$, the ring homomorphism \[L_{0,n}\longrightarrow
L_{0,n+1}\] induces the group homomorphism \[\mathbb Z^{E^0} \stackrel{A_E^t}{\longrightarrow}
\mathbb Z^{E^0},\]
where $A_E^t:\mathbb Z^{E^0} \rightarrow \mathbb Z^{E^0}$ is multiplication  from left which is induced by the homomorphism~(\ref{volleyb}). 

For a finite graph $E$ with no sinks, with $n$ vertices and the adjacency matrix $A$, by Theorem~\ref{sthfin}, $K^{\gr}_0(\LL(E))\cong K_0(\LL(E)_0)$. Thus  $K^{\gr}_0(\LL(E))$ is the direct limit  of the ordered direct system 
\begin{equation}\label{thu3}
\mathbb Z^n \stackrel{A^t}{\longrightarrow} \mathbb Z^n \stackrel{A^t}{\longrightarrow}  \mathbb Z^n \stackrel{A^t}{\longrightarrow} \cdots,
\end{equation}
where the ordering in $\mathbb Z^n$ is defined point-wise.

In general,  the direct limit of the system, $\varinjlim_{A} \mathbb Z^n$, where $A\in \M_n(\mathbb Z)$, is an ordered group and  can be described as follows. Consider the pair $(a,k)$, where $a\in \mathbb Z^n$ and $k\in \mathbb N$, and define the equivalence relation $(a,k)\sim (b,k')$ if $A^{k''-k}a=A^{k''-k'}b$ for some $k'' \in \mathbb N$.  Let $[a,k]$ denote the equivalence class of $(a,k)$. Clearly $[A^na,n+k]=[a,k]$. Then it is not difficult to show that the direct limit $\varinjlim_{A} \mathbb Z^n$ is the abelian group consists of equivalent classes $[a,k]$, $a\in \mathbb Z^n$, $k \in \mathbb N$, with addition defined by
\begin{equation}\label{peeme}
[a,k]+[b,k']=[A^{k'}a+A^kb,k+k'].
\end{equation}
The positive cone of this ordered group is the set of elements $[a,k]$, where $a\in {\mathbb Z^+}^n$, $k\in \mathbb N$. 
Moreover, there is automorphism $\delta_A: \varinjlim_{A} \mathbb Z^n \rightarrow \varinjlim_{A} \mathbb Z^n$ defined by $\delta_A([a,k])=[Aa,k]$. 

There is another presentation for  $\varinjlim_{A} \mathbb Z^n$ which is sometimes easier to work with. Consider the set 
\begin{equation}\label{q11}
\Delta_A=\big \{\, v\in  A^n \mathbb  Q^n \mid A^k v \in \mathbb Z^n, \text { for some } k \in \mathbb N\, \big \}.
\end{equation}
 The set $\Delta_A$ forms an ordered abelian group with the usual addition of vectors and  the positive cone  
\begin{equation}\label{q22}
\Delta_A^+=\big \{\, v\in  A^n \mathbb Q^n \mid A^k v\in {\mathbb Z^+}^n, \text { for some } k \in \mathbb N\, \big \}.
\end{equation}
Moreover, there is automorphism $\delta_A:\Delta_A\rightarrow \Delta_A$ defined by $\delta_A(v)=A v$. The map 
\begin{align}\label{kkjjhhga}
\phi:\Delta_A&\rightarrow \varinjlim_{A} \mathbb Z^n\\
v &\mapsto [A^kv,k], \notag
\end{align} 
where $k\in \mathbb N$ such that $A^k v \in \mathbb Z^n$, is an isomorphism which respects the action of $A$ and the ordering, \ie $\phi(\Delta_A^+)= (\varinjlim_{A} \mathbb Z^n)^+$ and $\phi(\delta_A(v))=\delta_A\phi(v)$.

\begin{example}\label{ffexp}
Let $E$ be the following graph. 
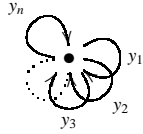
\[
\xymatrix{
   \bullet \ar@{.}@(l,d) \ar@(ur,dr)^{y_{1}} \ar@(r,d)^{y_{2}} \ar@(dr,dl)^{y_{3}} 
\ar@(l,u)^{y_{n}}
}\]

The ungraded $K_0$ of $\LL(E)$ was computed in Example~\ref{ffexpnon}.
The graph $E$ has no sinks, and so by (\ref{thu3}),  
\[K_0^{\gr}(\LL(E)) \cong \varinjlim \mathbb Z,\]  
of the inductive system $\mathbb Z
\stackrel{n}{\longrightarrow} \mathbb Z
\stackrel{n}\longrightarrow \mathbb Z
\stackrel{n}\longrightarrow \cdots$. This gives that 
\[K_0^{\gr}(\LL(E)) \cong \mathbb Z[1/n].\]

\end{example}

\begin{example} \label{gfrt}
For the graph 
\begin{equation*}
E:\,\,\,\,\,\,\,\,\,\,\,\,{\def\labelstyle{\displaystyle}
\xymatrix{
  \bullet \ar@(lu,ld) \ar@/^0.9pc/[r] \ar@/^1.4pc/[r]  & \bullet \ar@/^0.9pc/[l]   &   
}}
\end{equation*}
with the adjacency 
$A_E=\left(
\begin{array}{cc}
 1 & 2 \\
 1 & 0
\end{array}
\right),$
the ring of homogeneous element of degree zero,  $\LL(E)_0$, is the direct limit of the system 
\begin{align*} 
K \oplus K
 \stackrel{A_{E}^t}{\longrightarrow}  \M_2(K) &\oplus \M_2(K)
\stackrel{A_{E}^t}\longrightarrow  \M_4(K)\oplus \M_4(K)
\stackrel{A_{E}^t}\longrightarrow \cdots \\
(a,b)\mapsto \left(
\begin{array}{cc}
 a & 0 \\
 0 & b
\end{array}
\right)
&\oplus
\left(
\begin{array}{cc}
 a & 0 \\
 0 & a
\end{array}
\right)
\end{align*} 
So $K^{\gr}_0(\LL(E))$ is the direct limit of the direct system 
\begin{equation*}
\mathbb Z^2 \stackrel{A_E^t}{\longrightarrow} \mathbb Z^2 \stackrel{A_E^t}{\longrightarrow}  \mathbb Z^2 \stackrel{A_E^t}{\longrightarrow} \cdots,
\end{equation*}
Since $\det(A^t_E)=-2$, one can easily calculate that \[K^{\gr}_0(\LL(E))\cong \mathbb Z[1/2]\bigoplus  \mathbb Z[1/2].\] Moreover 
$[\LL(E)] \in K^{\gr}_0(\LL(E))$ is represented by $(1,1)\in  \mathbb Z[1/2]\bigoplus  \mathbb Z[1/2]$. Adopting~(\ref{q11}) for the description of $K^{\gr}_0(\LL(E))$, since the action of $x$ on $K^{\gr}_0(\LL(E))$ represented by action of $A_E^t$ from the left, we have \[x (a,b)=(a+b,2a).\] Moreover, considering 
~(\ref{q22}) for the positive cone,  $\big({A_E^t}\big)^k(a,b)$ is eventually positive, if $v (a,b) >0$, where $v=(2,1)$ is the Perron eigenvector of $A_E$ (see~\cite[Lemma~7.3.8]{lindmarcus}). It follows that 
\[K^{\gr}_0(\LL(E))^+=\Delta_{A_E^t}^+=\big \{(a,b) \in \mathbb Z[1/2]\oplus  \mathbb Z[1/2] \mid 2a+b > 0\big \} \cup \{(0,0)\}.\]
\end{example}

\begin{example}\label{gaul2} \scm[Reduced $K^{\gr}_0$ of strongly graded rings]
\vspace{0.2cm} \index{0-component ring of a graded ring}

When $A$ is a strongly graded ring, the graded $K$-groups coincide with $K$-groups of its 0-component ring (see~(\ref{dade})). However this example shows that this is not the case for the reduced graded Grothendieck groups. 

Let $A$ be the Leavitt algebra generated by $2n$ symbols (which is associated to a graph with one vertex and $n$-loops) (see~\ref{levisuji}). By Theorem~\ref{sthfin}, this is a strongly graded ring. The homomorphism~\ref{mozart1}, \ie 
\begin{align*}
\phi:\mathbb Z[\Gamma]&\longrightarrow K^{\gr}_0(A),\\
\sum_\alpha n_\alpha \alpha &\longmapsto \sum_\alpha n_\alpha [A(\alpha)].
\end{align*}
takes the form
\begin{align*}
\phi:\mathbb Z[x,x^{-1}]&\longrightarrow K^{\gr}_0(A)\cong \Z[1/n],\\
\sum_i n_i x^i &\longmapsto \sum_i n_i [A(i)].
\end{align*}
This shows that $\phi$ is surjective (see~(\ref{roxjhhre})) and thus $\widetilde {K^{\gr}_0}(A)$ is trivial. On the other hand, by Example~\ref{ffexp}, $K_0(A_0)\cong K^{\gr}_0(A)\cong \Z[1/n]$. But 
\begin{align*}
\phi_0:\mathbb Z &\longrightarrow K_0(A_0)\cong \Z[1/n],\\
n &\longmapsto n[A_0].
\end{align*}

This shows that $\widetilde{K_0}(A_0)$ is a nontrivial torsion group $\Z[1/n] / \Z$. Thus \[\widetilde {K^{\gr}_0}(A)\not \cong \widetilde{K_0}(A_0).\] 
\end{example}

\begin{remark}\scm[$K^{\gr}_0$ of Weyl algebras] \index{Weyl algebra}
\vspace{0.2cm}

Let $A=K(x,y)/\langle xy-yx-1\rangle $ be the Weyl algebra, where $K$ is an algebraically closed field of characteristic $0$. By Example~\ref{weyl}, this is a $\mathbb Z$-graded ring. The graded Grothendieck group of this ring is calculated in~\cite{sierra2}. It is shown that $K^{\gr}_0(A)\cong \bigoplus_{\mathbb Z}\mathbb Z$ (\ie a direct sum of a countably many $\mathbb Z$), $\widetilde {K^{\gr}_0}(A)=0$ and  $K_0(A)=0$. 
\end{remark}

\section{$G_0^{\gr}$ of graded rings}\label{gggzero} \index{$G_0^{\gr}$}
Recall that for a $\Gamma$\!-graded ring $A$ with identity, the graded Grothendieck group $K_0^{\gr}$ was defined as the group completion of the monoid (see~(\ref{zhongshan1})),
\begin{equation*}
\mathcal V^{\gr}(A)=\big \{\, [P] \mid  P  \text{ is graded finitely generated projective A-module} \, \big \}.
\end{equation*}
If instead of isomorphism classes of graded finitely generated projective $A$-modules, we consider the isomorphism classes of all graded finitely generated $A$-modules, the group completion of this monoid is denoted by $G_0^{\gr}(A)$. 

If $A$ is graded Noetherian, then the category of graded finitely generated (right) modules over $A$, $\grr A$, is an abelian category (but not necessarily the category $\Pgrp A$). Several of $K$-theory techniques work only over such categories (such as D\'evissage and localisations (\S\ref{klikhf4})). For this reason, it is beneficial to develop the $G_0^{\gr}$-theory. 

For a $\Gamma$\!-graded ring $A$, $G_0^{\gr}(A)$ can equivalently be defined as the free abelian group generated by isomorphism classes $[M]$, where $M$ is a graded finitely generated right $A$-module, subject to the relation
 $[M]=[K]+[N]$ if there is an exact sequence 
  \[0\longrightarrow K \longrightarrow M \longrightarrow N \longrightarrow 0.\] This definition will be extended to exact categories in~\S\ref{hhyyhhyy6}.

\subsection{$G_0^{\gr}$ of graded Artinian rings}

The theory of composition series for modules is used to calculate the $G_0$ group of  Artinian rings. A similar theory in the graded setting is valid and we briefly recall the concepts.  \index{composition series} \index{graded composition series}

Let $M$ be a nonzero graded right $A$-module. A finite chain of graded submodules of $M$
\[M=M_0\supset M_1 \supset \dots \supset M_n =0 \]
is called \emph{a graded composition series of length $n$} for $M$ if $M_i/M_{i+1}$ is a graded simple $A$-module, where $0\leq i\leq n-1$. The graded simple modules $M_i/M_{i+1}$ are called \emph{graded composition factors} of the series. \index{graded composition factors} Two graded composition series 
\begin{align*}
M&=M_0\supset M_1 \supset \dots \supset M_n =0, \\
M&=N_0\supset N_1 \supset \dots \supset N_p =0 
\end{align*}
are called \emph{equivalent} if $n=p$ and for a suitable permutation $\sigma\in S_n$ 
\[M_i/M_{i+1} \cong N_{\sigma(i)}/N_{\sigma(i)+1},\]
\ie there is a one to one correspondence between composition factors of these two chains such that the corresponding factors are isomorphic as graded $A$-modules. 

We also need the~\emph{graded Jordan-H\"older theorem} which is valid with a similar proof as in the ungraded case. Namely, if a graded module $M$ has a graded composition series, then all graded composition series of $M$ are equivalent. \index{graded Jordan-H\"older theorem}
\index{Jordan-H\"older theorem}

\index{$S_n$, symmetric group on $n$ objects}
\index{permutation group}
\index{symmetric group}

If a module is graded Artinian and Noetherian, then it has a graded composition series. In particular,  if $A$ is a graded right Artinian ring, then any graded finitely generated right $A$-module has a graded composition series. The proofs of these statements are similar to the ungraded case (see for example~\cite[\S11]{anderson}). 

Let $A$ be a $\Gamma$\!-graded right Artinian ring and let $V_1$ be a graded right simple $A$-module. 
Suppose $V_2$ is a graded right simple module which is not graded isomorphic to $V_1(\gamma)$ for any $\gamma \in \Gamma$. Continuing in this fashion, using the graded Jordan-H\"older theorem one can prove that there are a finite number of graded simple modules 
$\{V_1,\dots, V_s \}$ such that any graded simple module is isomorphic to some shift of one of $V_i$. We call $\{V_1,\dots, V_s \}$ a \emph{basic set of graded simple right $A$-modules}. 
\index{basic set of graded simple modules} 

\begin{theorem}\label{artgth} \index{graded Artinian ring}
Let $A$ be a graded Artinian ring and $\{V_1,\dots, V_s \}$ a basic set of graded simple right $A$-modules. Then 
\begin{equation}\label{rainevere}
G^{\gr}_0(A) \cong \bigoplus_{i=1}^s \frac{\big \langle \, V_i(\gamma) \mid \gamma \in \Gamma \, \big \rangle}{
\big \langle \, V_i(\alpha) -V_i(\beta) \mid V_i(\alpha)\cong_{\gr} V_i(\beta), \alpha,\beta \in \Gamma\,  \big \rangle} ,
\end{equation}
as $\mathbb Z [\Gamma]$-modules. Here $\langle \, V_i(\gamma) \mid \gamma \in \Gamma \, \rangle$ is the free abelian group on generators $V_i(\gamma)$ and the action of $\Gamma$ on the generators defined as $\alpha.V_i(\gamma)=V_i(\alpha+\gamma)$.
\end{theorem}
\begin{proof}
Let $M$ be a graded finitely generated $A$-module with a graded composition series 
\begin{equation}\label{comopigt5}
M=M_0\supset M_1 \supset \dots \supset M_n =0, 
\end{equation}
and composition factors $M_i/M_{i+1}$. From the exact sequences 
\[0 \longrightarrow M_{i+1} \longrightarrow M_i \longrightarrow M_i/M_{i+1} \longrightarrow 0,\] we get
\[[M]=[M]-0=\sum_{i=0}^{n-1}\big( [M_i]-[M_{i+1}] \big)=\sum_{i=0}^{n-1}[M_i/M_{i+1}].\]
Since $M_i/M_{i+1}$, $0\leq i \leq n-1$, are graded simple, we can write 
\[[M]=\sum_{t=1}^s\sum_{\gamma \in \Gamma}r_{t,\gamma}(M)[V_t(\gamma)] \in G_0^{\gr}(A),\]
where $r_{t,\gamma}(M)$ is the multiplicity of $V_t(\gamma)$ in the composite factors of $M$ (which are all but a finite number are nonzero). Since any other graded composition series is equivalent to the above composition series, the graded Jordan-H\"older theorem guarantees that $r_{t,\gamma}(M)$ are independent of the choice of the graded composition series. This defines a homomorphism from the free abelian group generated by isomorphism classes of graded finitely generated right $A$-modules to the right hand side of (\ref{rainevere}).

Consider an exact sequence of graded finitely generated modules \[0 \rightarrow K \stackrel{\phi}{\rightarrow} M \stackrel{\psi}{\rightarrow} N \rightarrow 0,\] and the graded composition series 
\begin{align*}
K=K_0\supset K_1 \supset \dots \supset K_n =0, \\
N=N_0\supset N_1 \supset \dots \supset N_p =0. 
\end{align*}
Then one obtains a graded composition series 
\[M=N_0'\supset N_1' \supset \dots \supset N_p' =K_0\supset K_1 \supset \dots \supset K_n =0,\]
 where $N_i'=\psi^{-1}(N_i)$. This shows that the homomorphism above extends to the well-defined homomorphism 
\[f:G^{\gr}_0(A) \longrightarrow \bigoplus_{i=1}^s \frac{\langle \, V_i(\gamma) \mid \gamma \in \Gamma \, \rangle}{
\langle \, V_i(\alpha) -V_i(\beta) \mid V_i(\alpha)\cong_{\gr} V_i(\beta)\,  \rangle}.\]
The fact that $f$ is group isomorphism is not difficult to observe. 

To show that $f$ is a $\mathbb Z[\Gamma]$-module, note that if~(\ref{comopigt5}) is a composition series for $M$, then clearly 
\[M(\alpha)=M_0(\alpha)\supset M_1(\alpha) \supset \dots \supset M_n(\alpha) =0,\] is a composition series for $M(\alpha)$, where $\alpha\in \Gamma$. Since \[M_i(\alpha)/M_{i+1}(\alpha)=M_i/M_{i+1}(\alpha),\] it follows that 
\begin{multline*}
f(\alpha.[M])=f([M(\alpha)])=\sum_{t=1}^s\sum_{\gamma \in \Gamma}r_{t,\gamma}(M)[V_t(\gamma+\alpha)]=\\
\alpha. \sum_{t=1}^s\sum_{\gamma \in \Gamma}r_{t,\gamma}(M)[V_t(\gamma)]=\alpha f([M]).
\end{multline*}

Thus $f$ is $\mathbb Z[\Gamma]$-module as well. 
\end{proof}

We obtain the graded Grothendieck group of a graded division ring (Proposition~\ref{k0grof}) as a consequence of Theorem~\ref{artgth}.

\begin{corollary}\label{ooppoo}
Let $A$ be a $\Gamma$\!-graded division ring with the support $\Gamma_A$. Then 
\[K_0^{\gr}(A) \cong \mathbb Z[\Gamma /\Gamma_A].\]
\end{corollary}
\begin{proof}
First note that $A$ is a graded left and right Artinian ring with $\{A\}$ as a basic set of graded simple modules.  By Proposition~\ref{gradedfree}, $G^{\gr}_0(A)=K^{\gr}_0(A)$. Thus by Theorem~\ref{artgth}
\begin{equation}\label{jhjhuiop}
K^{\gr}_0(A) \cong \frac{\langle \, A(\gamma) \mid \gamma \in \Gamma \, \rangle}{
\langle \, A(\alpha) -A(\beta) \mid A(\alpha)\cong_{\gr} A(\beta)\,  \rangle}. 
\end{equation}
By Corollary~\ref{rndcongcori} $A(\alpha)\cong_{\gr} A(\beta)$ if and only if $\alpha - \beta \in \Gamma_A$. It is now easy to show that~(\ref{jhjhuiop}) reduces to $K_0^{\gr}(A) \cong \mathbb Z[\Gamma /\Gamma_A].$
\end{proof}

\section{Symbolic dynamics and $K_0^{\gr}$}\label{symbolii}

One of the central objects in the theory of symbolic dynamics is a \emph{shift of finite type} (\ie a topological Markov chain). \index{shift of finite type} \index{topological Markov chain}
Every finite directed graph $E$ with no sinks and sources gives rise to a shift of finite type $X_E$ by considering the set of bi-infinite paths and the natural shift of the paths to the left. This is called an \emph{edge shift}. \index{edge shift} Conversely any shift of finite type is conjugate to an edge shift (for a comprehensive introduction to symbolic dynamics see~\cite{lindmarcus}). Several invariant have been proposed in order to classify shifts of finite type, among them Krieger's dimension group.  In this section we see that Krieger's invariant can be expressible as the graded Grothendieck group of a Leavitt path algebra. 

We briefly recall the objects of our interest. Let $\mathcal A$ be a finite alphabet (\ie a finite set). A \emph{full shift space} is defined as \index{full shift space}
\begin{equation*}
\mathcal A^{\mathbb Z} := \big \{ \, (a_i)_{i\in \mathbb Z} \mid a_i \in \mathcal A \, \big \}, 
\end{equation*}
and a \emph{shift map} \index{shift map} $\sigma :  \mathcal A^{\mathbb Z} \rightarrow \mathcal A^{\mathbb Z}$ is defined as 
\begin{equation*}
\sigma \big ((a_i)_{i\in \mathbb Z}\big)=(a_{i+1})_{i\in \mathbb Z}.
\end{equation*}
Moreover, a \emph{subshift} \index{subshift} $X\subseteq \mathcal A^{\mathbb Z}$ is a closed $\sigma$-invariant subspace of $\mathcal A^{\mathbb Z}$. 

Given a finite graph $E$ (see~\S\ref{paohdme3} for terminologies related to graphs), a \emph{subshift of finite type associated to $E$} is defined as 
\index{subshift of finite type associated to a graph} 
\begin{equation*}
X_{E}:=\big \{ \, (e_i)_{i\in \mathbb Z} \in (E^1)^{\mathbb Z} \mid r(e_i)=s(e_{i+1}) \, \big \}. 
\end{equation*}

We say $X_E$ is essential if the graph $E$ has no sinks and sources. Moreover, $X_E$ is called \emph{irreducible} \index{irreducible subshift} if the adjacency matrix $A_E$ is irreducible.
For a square nonnegative integer matrix $A$, we denote by $X_A$ the subshift of finite type associated to the graph with the adjacency matrix $A$.  
Finally, two shifts of finite type $X_A$ and $X_B$ are called \emph{conjugate} \index{conjugate of subshifts} (or \emph{topologically conjugate of subshifts})  \index{topologically conjugate of subshifts} and denoted by $X_A\cong X_B$, if there exists a homeomorphism 
$h:X_A\rightarrow X_B$ such that $\sigma_B\, h=h\, \sigma_A$.

The notion of the shift equivalence for matrices was introduced by Williams \cite{williams} (see also~\cite[\S7]{lindmarcus}) in an attempt to provide a computable machinery for determining the conjugacy between two shifts of finite type. Two square nonnegative integer matrices $A$ and $B$ are called \emph{elementary shift equivalent}, \index{elementary shift equivalent} and denoted by $A\sim_{ES} B$, if there are nonnegative matrices $R$ and $S$ such that $A=RS$ and $B=SR$. 
The equivalence relation $\sim_S$  on square nonnegative integer matrices generated by elementary shift equivalence is called \emph{strong shift equivalence}. \index{strong shift equivalence} 
\begin{example}
Let  
\(A={\left(
\begin{array}{cc}
 1 & 2 \\
 1 & 0
\end{array}
\right)}\) 
and 
\(A^T={\left(
\begin{array}{cc}
 1 & 1 \\
 2 & 0
\end{array}
\right)}\). We show that $A$ is strongly shift equivalent to $A^T$. We have  
\begin{align*}
A=\left(
\begin{array}{ccc}
 1 & 1 & 0 \\
 0 & 0 & 1
\end{array}
\right)
&
\left(
\begin{array}{cc}
 1 & 1 \\
 0 & 1 \\
 1 & 0
\end{array}
\right)\\
\left(
\begin{array}{cc}
 1 & 1 \\
 0 & 1 \\
 1 & 0
\end{array}
\right)
&
\left(
\begin{array}{ccc}
 1 & 1 & 0 \\
 0 & 0 & 1
\end{array}
\right)=
\left(
\begin{array}{ccc}
 1 & 1 & 1 \\
 0 & 0 & 1 \\
 1 & 1 & 0
\end{array}
\right)
=E_1
\end{align*}
Moreover, we have 
\begin{align*}
E_1=
\left(
\begin{array}{ccc}
 1 & 1 & 1 \\
 0 & 0 & 1 \\
 1 & 1 & 0
\end{array}
\right)=
\left(
\begin{array}{ccc}
 0 & 1 & 1 \\
 1 & 0 & 0 \\
 0 & 0 & 1
\end{array}
\right)
&
\left(
\begin{array}{ccc}
 0 & 0 & 1 \\
 0 & 0 & 1 \\
 1 & 1 & 0
\end{array}
\right)\\
\left(
\begin{array}{ccc}
 0 & 0 & 1 \\
 0 & 0 & 1 \\
 1 & 1 & 0
\end{array}
\right)
&
\left(
\begin{array}{ccc}
 0 & 1 & 1 \\
 1 & 0 & 0 \\
 0 & 0 & 1
\end{array}
\right)=\left(
\begin{array}{ccc}
 0 & 0 & 1 \\
 0 & 0 & 1 \\
 1 & 1 & 1
\end{array}
\right)=E_2
\end{align*}
Finally, 
\begin{align*}
E_2=
\left(
\begin{array}{ccc}
 0 & 0 & 1 \\
 0 & 0 & 1 \\
 1 & 1 & 1
\end{array}
\right)=
\left(
\begin{array}{cc}
 1 & 0 \\
 1 & 0 \\
 1 & 1
\end{array}
\right)  
&
\left(
\begin{array}{ccc}
 0 & 0 & 1 \\
 1 & 1 & 0
\end{array}
\right)\\
\left(
\begin{array}{ccc}
 0 & 0 & 1 \\
 1 & 1 & 0
\end{array}
\right)
&
\left(
\begin{array}{cc}
 1 & 0 \\
 1 & 0 \\
 1 & 1
\end{array}
\right) =\left(
\begin{array}{cc}
 1 & 1 \\
 2 & 0
\end{array}
\right)=A^T
\end{align*}
This shows that \[A\sim_{ES} E_1 \sim_{ES} E_2 \sim_{ES} A^T.\]
Thus $A\sim_S A^T$. 
\end{example}

Besides elementary and strongly shift equivalence, there is a weaker notion, called shift equivalence defined as follows. 
The nonnegative integer square matrices $A$ and $B$ are called \emph{shift equivalent} \index{shift equivalence} if there are nonnegative matrices $R$ and $S$ such that $A^l=RS$ and $B^l=SR$, for some $l\in \mathbb N$, and 
$AR=RB$ and $SA=BS$. Clearly the strongly shift equivalence implies the shift equivalence, but the converse, an open question for almost 20 years, does not hold~\cite{lindmarcus}.  

Before stating Williams' main theorem, we need to recall the concept of out-splitting and in-splitting of a graph. 

\begin{definition} \index{out-spiltting of a graph}
Let $E=(E^0,E^1,r,s)$ be a finite graph. For each $v \in E^0$
which is not a sink, partition $s^{-1}(v)$ into disjoint nonempty
subsets $\mathcal{E}_v^1, \ldots , \mathcal{E}_v^{m(v)}$, where
$m(v) \geq 1$. (If $v$ is a sink then put $m(v)=0$.) Let
$\mathcal{P}$ denote the resulting partition of $E^1$. We form the
\emph{out-split graph} $E_s(\mathcal{P})$ \emph{from} $E$ \emph{using}
$\mathcal{P}$ as follows: Let
\begin{eqnarray*}
E_s(\mathcal{P})^0 &=& \big \{\, v^i \mid v \in E^0, 1 \leq i \leq m(v) \, \big \}
\cup \{\, v \mid m(v) = 0\,\},\\
E_s(\mathcal{P})^1 &=& \big \{\, e^j \mid e \in
E^1, 1 \leq j \leq m(r(e)) \, \big \} \cup \{\, e \mid m(r(e)) = 0\, \},
\end{eqnarray*}
and define $r_{E_s(\mathcal{P})}$, $s_{E_s(\mathcal{P})}: E_s(\mathcal{P})^1 \rightarrow E_s(\mathcal{P})^0$
for $e \in \mathcal{E}_{s(e)}^i$ and $1\leq j \leq m(r(e))$, by
\begin{eqnarray*}
s_{E_s(\mathcal{P})}(e^j) = s(e)^i &\text{ and }& s_{E_s(\mathcal{P})}(e) = s(e)^i,\\
r_{E_s(\mathcal{P})}(e^j) = r(e)^j &\text{ and }& r_{E_s(\mathcal{P})}(e) = r(e).\\
\end{eqnarray*}
\end{definition}

\begin{definition} \index{in-spiltting of a graph}
Let $E=(E_0,E_1,r,s)$ be a finite graph. For each $v \in E^0$ which is not a source, 
partition the set $r^{-1}(v)$ into disjoint nonempty subsets $\mathcal{E}^v_1, \ldots , \mathcal{E}^v_{m(v)}$,
 $m(v) \geq 1$. (If $v$ is a source then put $m(v)=0$.)
Let
$\mathcal{P}$ denote the resulting partition of $E^1$. We form the
\emph{in-split graph} $E_r(\mathcal{P})$ \emph{from} $E$ \emph{using}
$\mathcal{P}$ as follows: Let
\begin{eqnarray*}
E_r(\mathcal{P})^0 &=& \big \{ \, v_i \mid v \in E^0, 1 \leq i \leq m(v)\, \big  \}
\cup \{\, v \mid m(v) = 0\, \},\\
E_r(\mathcal{P})^1 &=& \{\, e_j \mid e \in
E^1, 1 \leq j \leq m(s(e)) \, \} \cup \{\, e \mid m(s(e)) = 0\, \},
\end{eqnarray*}
and define $r_{E_r(\mathcal{P})}$, $s_{E_r(\mathcal{P})}: E_r(\mathcal{P})^1 \rightarrow E_r(\mathcal{P})^0$
for $e \in \mathcal{E}^{r(e)}_i$ and $1\leq j \leq m(s(e))$, by 
\begin{eqnarray*}
s_{E_r(\mathcal{P})}(e_j) = s(e)_j &\text{ and }& s_{E_r(\mathcal{P})}(e) = s(e),\\
r_{E_r(\mathcal{P})}(e_j) = r(e)_i &\text{ and }& r_{E_r(\mathcal{P})}(e) = r(e)_i.\\
\end{eqnarray*}
\end{definition}

\begin{example}\scm[Out-splitting of a graph]\label{outsplitexample}
\vspace{0.2cm}

Consider the graph
$$
{E}:\quad {
\def\labelstyle{\displaystyle}
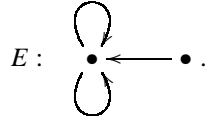
\xymatrix{ \bullet \uloopr{}\dloopr{} & \bullet \ar[l]}}.
$$
Let $\mathcal{P}$ be the partition of the edges of $E$ containing only one edge in each partition.
Then the out-split graph of $E$ using $\mathcal{P}$ is
$${
\def\labelstyle{\displaystyle}E_s(\mathcal P):\quad \xymatrix{ & \bullet
\ar[ddl] \ar[ddr]& \\ & & \\
{\bullet} \ar@(ul,dl) \ar@/^1pc/ [rr] & & {\bullet} \ar@(ur,dr) \ar@/^1pc/ [ll]  }}$$
\vspace{.2truecm}
\end{example}

\begin{theorem}[{\sc Williams}~\cite{williams,lindmarcus}]
Let $A$ and $B$ be two square nonnegative integer matrices and let $E$ and $F$ be two essential graphs.

\begin{enumerate}[\upshape(1)]

\item  $X_A$ is conjugate to $X_B$ if and only if $A$ is strongly shift equivalent to $B$. 

\item  $X_E$ is conjugate to $X_F$ if and only if $E$ can be obtained from $F$ by a sequence of in/out-splittings and their converses. 

\end{enumerate}

\end{theorem}

Krieger in~\cite{krieger} defined an invariant for classifying the irreducible shifts of finite type up to shift equivalence. Later Wagoner systematically used this invariant to relate it with higher $K$-groups. Surprisingly,  Krieger's dimension group  and  Wagoner's dimension module in symbolic dynamics turn out to be expressible as the graded Grothendieck groups of Leavitt path algebras. Here, we briefly describe this relation.

In general, a nonnegative integral $n\times n$ matrix $A$ gives rise to a stationary system. This in turn gives a direct system of order free abelian groups with $A$ acting as an order preserving group homomorphism as follows
\[\mathbb Z^n \stackrel{A}{\longrightarrow} \mathbb Z^n \stackrel{A}{\longrightarrow}  \mathbb Z^n \stackrel{A}{\longrightarrow} \cdots,
\]
where the ordering in $\mathbb Z^n$ is defined point-wise (\ie the positive cone is $\mathbb N^n$). The direct limit of this system, $\Delta_A:= \varinjlim_{A} \mathbb Z^n$, (i.e, the $K_0$ of the stationary system,) along with its positive cone, $\Delta^+$, and the automorphism which induced by $A$ on the direct limit, 
$\delta_A:\Delta_A \rightarrow \Delta_A$, is the invariant considered by Krieger, now known as \emph{Krieger's dimension group}. \index{Krieger's dimension group} Following~\cite{lindmarcus}, we denote this triple by $(\Delta_A, \Delta_A^+, \delta_A)$. 

The following theorem was proved by Krieger (\cite[Theorem~4.2]{krieger}, and~\cite[Theorem~7.5.8]{lindmarcus}, see also~\cite[\S7.5]{lindmarcus} for a detailed algebraic treatment). 

\begin{theorem}
Let $A$ and $B$ be two square nonnegative integer matrices. Then $A$ and $B$ are shift equivalent if and only if 
\[(\Delta_A, \Delta_A^+, \delta_A) \cong (\Delta_B, \Delta_B^+, \delta_B).\]
\end{theorem}

 Wagoner noted that the induced structure on $\Delta_A$ by the automorphism $\delta_A$ makes $\Delta_A$ a $\mathbb Z[x,x^{-1}]$-module which was systematically used in~\cite{wago1,wago2} (see also~\cite[\S3]{boyle}).

Recall that the graded Grothendieck group of a $\mathbb Z$-graded ring has a natural $\mathbb Z[x,x^{-1}]$-module structure  and 
the following observation (Theorem~\ref{mmpags}) shows that the graded Grothendieck group of the Leavitt path algebra associated to a matrix $A$ coincides with the Krieger dimension group of the shift of finite type associated to $A^t$, \ie the graded dimension group of a Leavitt path algebra coincides with Krieger's dimension group,
\[\big(K_0^{\gr}(\LL(E)),(K_0^{\gr}(\LL(E))^+\big ) \cong (\Delta_{A^t}, \Delta_{A^t}^+).\] 
This will provide a link between the theory of Leavitt path algebras and symbolic dynamics.  \index{Leavitt path algebra}
\begin{theorem}\label{mmpags}
Let $E$ be a finite graph with no sinks with the  adjacency matrix $A$. Then there is an isomorphism 
$\phi:K^{\gr}_0(\LL(E)) \longrightarrow \Delta_{A^t}$ such that $\phi(x \alpha)=\delta_{A^t}\phi(\alpha)$, $\alpha\in \LL(E)$, $x\in \mathbb Z[x,x^{-1}]$ and $\phi(K^{\gr}_0(\LL(E))^+)=\Delta_{A^t}^+$. 
\end{theorem}
\begin{proof}
Since $\LL(E)$ is strongly graded by Theorem~\ref{sthfin}, there is an ordered isomorphism \[K^{\gr}_0(\LL(E)) \rightarrow K_0(\LL(E)_0).\] 
Thus  the ordered group $K^{\gr}_0(\LL(E))$ coincides with the ordered group $\Delta_{A^t}$ (see (\ref{thu3})). We only need to check that their module structures are compatible. It is enough to  show that the action of $x$ on $K^{\gr}_0$ coincides with the action of $A^t$ on $K_0(\LL(E)_0)$, \ie $\phi(x \alpha)=\delta_{A^t} \phi(\alpha)$. 

Set $\mathcal A=\LL(E)$. Since graded finitely generated projective modules are generated by $u\mathcal A(i)$, where $u\in E^0$ and $i \in \mathbb Z$, it suffices to show that $\phi(x [u\mathcal  A])=\delta_{A^t}\phi([u \mathcal A])$. 
Since the image of $u\mathcal A$ in $K_0(\mathcal A_0)$ is $[u\mathcal A_0]$, and $\mathcal A_0= \bigcup_{n=0}^{\infty}L_{0,n},$ (see~(\ref{ppooii})) using the presentation of $K_0$ given in~(\ref{peeme}), we have 
\[\phi([u\mathcal  A])=[u\mathcal  A_0]=[uL_{0,0},1]=[u,1].\] Thus 
\[\delta_{A^t}\phi([u\mathcal  A])=\delta_{A^t}([u,1])=[A^t u,1]=\sum_{\{\alpha \in E^1 \mid s(\alpha)=u\}}[r(\alpha),1].\]

On the other hand, 
\begin{multline}
\phi(x[u\mathcal  A])=\phi([u\mathcal  A(1)])=\phi\Big(\sum_{\{\alpha \in E^1 \mid s(\alpha)=u\}}[r(\alpha)\mathcal  A]\Big)=\\ 
\sum_{\{\alpha \in E^1 \mid s(\alpha)=u\}}[r(\alpha)\mathcal  A_0]=
\sum_{\{\alpha \in E^1 \mid s(\alpha)=u\}}[r(\alpha)L_{0,0},1]=\sum_{\{\alpha \in E^1 \mid s(\alpha)=u\}}[r(\alpha),1].
\end{multline}
Thus $\phi(x [u\mathcal  A])=\delta_{A^t}\phi([u \mathcal A])$. This finishes the proof. 
 \end{proof}

It is easy to see that two matrices $A$ and $B$ are shift equivalent if and only if $A^t$ and $B^t$ are shift equivalent. Combining this with Theorem~\ref{mmpags} and the fact that Krieger's dimension group is a complete invariant for shift equivalent we have the following corollary.

\begin{corollary}\label{h99}
Let $E$ and $F$ be finite graphs with no sinks and $A_E$ and $A_F$ be their adjacency matrices, respectively. Then
$A_E$ is shift equivalent to $A_F$ if and only if there is an order preserving $\mathbb Z[x,x^{-1}]$-module  isomorphism
$K_0^{\gr}(\LL(E)) \cong K_0^{\gr}(\LL(F))$.
\end{corollary}

\section{$K^{\gr}_1$-theory}\label{hhyyhhyy6}

For a $\Gamma$\!-graded ring $A$,  the category of  finitely generated
 $\Gamma$\!-graded projective right $A$-modules, $\Pgr[\Gamma] A$, is an exact category. Thus using Quillen's $Q$ construction
 one defines \[K_i^{\gr}(A):=  K_i(\Pgr[\Gamma] A), \quad i\geq 0.\] 
 Moreover, the shift functors (\ref{RBAday1}) \index{shift functor} induce auto-equivalences (in fact, automorphisms) $\mathcal T_\alpha:\Pgrp A \rightarrow \Pgrp A$. These in return give a group homomorphism  
 $\Gamma \rightarrow \Aut (K_i^{\gr}(A))$ or equivalently, a $\mathbb Z[\Gamma]$-module on $K_i^{\gr}(A)$. 
 This general construction will be used in~\S\ref{waraya}. 
 
 In this section we concretely construct the graded $K_1$-group, using a graded version of Bass' construction of $K_1$-group.  For a concise introduction of the (ungraded) groups $K_0$ and $K_1$ see~\cite{lamsiu}.
 
 \index{$K_0, K_1$ of an exact category}
 
 \begin{definition}{\sc $K_0$ and $K_1$ of an exact category}\label{hgygtw2}
 
 Let $\mathcal P$ be an \emph{exact category}, \index{exact category} \ie a full additive subcategory of an abelian category $\mathcal A$ such that, if 
 \[0\longrightarrow P_1 \longrightarrow P \longrightarrow P_2 \longrightarrow 0,\]
 is an exact sequence in $\mathcal A$ and $P_1,P_2 \in \mathcal P$, then $P \in \mathcal P$ (\ie $\mathcal P$ is closed under extension). 
 Moreover, we assume $\mathcal P$ has a small skeleton, \ie $\mathcal P$ has a full subcategory $\mathcal P_0$ which is small and $\mathcal P_0 \hookrightarrow \mathcal P$ is an equivalence. 
 
 The groups $K_0(\mathcal P)$ and $K_1(\mathcal P)$ are defined as follows. 
 \begin{enumerate}
 \item $K_0(\mathcal P)$ is the free abelian group generated by objects of $\mathcal P_0$, subject to the relation
 $[P]=[P_1]+[P_2]$ if there is an exact sequence 
  \[0\longrightarrow P_1 \longrightarrow P \longrightarrow P_2 \longrightarrow 0,\] in $\mathcal P$. 
  
  \item $K_1(\mathcal P)$ is  the free abelian group generated by pairs $(P,f)$, where $P$ is an object of $\mathcal P_0$ and $f \in \Aut(P)$,
subject to the relations 
 \[ [P,f]+[P,g]=[P,fg],\]
 and 
\[ [P,f]=[P_1,g]+[P_2,h], \] 
if there is a commutative diagram in $\mathcal P_0$
 \begin{equation*}
\xymatrix{
0 \ar[r] & P_1 \ar[r]^{i} \ar[d]^{g}& P \ar[r]^{\pi} \ar[d]^{f} & P_2 \ar[r] \ar[d]^{h} & 0 \\
0 \ar[r] & P_1 \ar[r]^{i} & P \ar[r]^{\pi}  & P_2 \ar[r]  & 0.
}
\end{equation*}
 \end{enumerate}
 \end{definition}

Note that from the relations of $K_1$ it follows
\begin{align*}
[P,fg]&=[P,gf],\\
[P,fg]&=[P\oplus P, f\oplus g].
\end{align*}

If $\Gamma$ is a group and $\mathcal T_\alpha:\mathcal P\rightarrow \mathcal P$, $\alpha \in \Gamma$, are auto-equivalences such that $\mathcal T_\beta \mathcal T_\alpha \cong \mathcal  T_{\alpha+\beta}$, then $K_0(\mathcal P)$ and $K_1(\mathcal P)$ have a $\Gamma$\!-module structures.

 One can easily see that $K_0^{\gr}(A)=K_0(\Pgr[\Gamma] A)$. Following Bass, one defines 
 \[K^{\gr}_1(A)=K_1(\Pgr[\Gamma] A).\] 

 Since the category $\Gr A$ is an abelian category and $\Pgrp A$ is an exact category, the main theorems of $K$-theory are valid for the grade Grothendieck group, such as D\'evissage, Resolution theorem and the localisation exact sequences of $K$-theory (see \S\ref{klikhf4} and~\cite[Chapter~3]{rosenberg} and~\cite{quillen,weibelk}).

 \begin{example}
 Let $F$ be a field and $F[x_1,\dots,x_r]$ be the polynomial ring with $r$ variables, which is considered as a $\mathbb Z$-grade ring with support $\mathbb N$, namely, $\deg(x_i)=1, 1\leq i \leq r$.  Then we will prove in~\S\ref{waraya} (see Theorem~\ref{quillen11}) that  
 \[ K^{\gr}_1(F[x_1,\dots,x_r])=F^*\otimes_{\mathbb Z}\mathbb Z[x,x^{-1}].\] 
 \end{example}

 \begin{example}\scm[$K^{\gr}_1$ of graded division algebras]\label{lobbyat12}
 \vspace{0.2cm}
 
 Let $A$ be a $\Gamma$\!-graded division ring. Then $A_0$  is a division ring and $\Omega:=\Gamma_A$  is a group (\S\ref{khgfewa1}). 
 By (\ref{diniat6}), for $i=1$, and the description of 
 $K_1$-group of division ring due to Dieudonn\'{e} (\cite[\S~20]{draxl}), we have 
 \begin{equation}\label{walktony9}
K^{\Gamma}_1(A)\cong \bigoplus_{\Gamma/\Omega}K^{\Omega}_1(A) \cong \bigoplus_{\Gamma/\Omega}K_1(A_0)= \bigoplus_{\Gamma/\Omega} A^*_0/[A^*_0,A^*_0],
\end{equation}
where $A^*_0$ is the group of invertible elements of division ring $A_0$ and $[A^*_0,A^*_0]$ is the multiplicative commutator subgroup.  \index{multiplicative commutator subgroup}
Representing $ \bigoplus_{\Gamma/\Omega}K_1(A_0)$ as the additive group of the group ring $K_1(A_0)[\Gamma/\Omega]$, (see Example~\ref{tonyat11}) the action of $\Gamma$ can be described as follows: for $\beta \in \Gamma$,
\[\beta  \, \Big ( \bigoplus_{\Omega+\alpha \in \Gamma/\Omega} K_1(A_0)(\Omega+\alpha) \Big )=  \bigoplus_{\Omega+\alpha \in \Gamma/\Omega} K_1(A_0)(\Omega+\alpha+\beta).\]
 
As a concrete example, let $A=K[x^n,x^{-n}]$ be a $\mathbb Z$-graded field with $\Gamma_A=n\mathbb Z$. Then by~(\ref{walktony9}), 
\[ K_1^{\gr}(A)\cong \bigoplus_n K^*,\] is a $\mathbb Z[x,x^{-1}]$-module. The action of $x$ on 
$(a_1,\dots,a_n) \in \bigoplus_n  K^*$ is \[x (a_1,\dots,a_n)= (a_n,a_1,\dots,a_{n-1}).\]

Compare this with the computation of $K^{\gr}_0(A)$ in Example~\ref{upst}(2). 
 
  \end{example}
 
 \begin{remark}\scm[The matrix description of $K^{\gr}_1$-group]
 \vspace{0.2cm}
 
 Let $A$ be a strongly $\Gamma$\!-graded ring. Then  the map 
 \begin{align}\label{ghzuhi}
 K^{\gr}_1(A) &\longrightarrow K_1(A_0),\\
 [P,f] &\longmapsto [P_0,f_0] \notag
 \end{align}
 is an isomorphism of groups (see~\S\ref{dadestmal}).  Here for a graded isomorphism $f:P\rightarrow P$, we denote by $f_\alpha$, the restriction of $f$ to $P_\alpha$, where $\alpha \in \Gamma$. Note that $K^{\gr}_1(A)$ is a $\mathbb Z[\Gamma]$-graded module, where the action of $\Gamma$ on the generators is defined by $\alpha [P,f]=[P(\alpha),f]$. 
  
 Since $K_1(A_0)$ has a matrix description (\cite[Theorem~3.1.7]{rosenberg}), from (\ref{ghzuhi}) we get a matrix representation 
 \[K^{\gr}_1(A) \cong K_1(A_0) \cong \GL(A_0) /E(A_0).\]
 We don't know whether for an arbitrary graded ring, one can give a matrix description for  $K^{\gr}_1(A)$. For some work in this direction see~\cite{zhang2013}. 
 
 \end{remark} 

\chapter{Graded Picard Groups} \label{picle}

Let $A$ be a commutative ring. If $M$ is a finitely generated projective $A$-module of constant rank $1$, then there is an $A$-module $N$ such that $M\otimes_A N\cong A$. In fact this is an equivalent condition. The module $M$ above is called an \emph{invertible module}. \index{invertible module} The isomorphism classes of invertible modules with the tensor product form an abelian group, denoted by $\Pic(A)$ and called the \emph{Picard group} \index{Picard group} of $A$. On the other hand, since $A$ is commutative, $K_0(A)$ is a ring with the multiplication defined by the tensor product and one can prove that there is an exact sequence 
\begin{equation}\label{lrss083}
1 \longrightarrow \Pic(A) \stackrel{\phi}\longrightarrow K_0(A)^*,
\end{equation}
where  $K_0(A)^*$ is the group of invertible elements of $K_0(A)$ and  $\phi([A])=[A]$.

When $A$ is a graded commutative ring, a parallel construction, using the graded modules, gives the graded Picard group $\Pic^{\gr}(A)$. As one expects when $A$ is strongly graded commutative ring then, using Dade's theorem~\ref{dadesthm} one immediately gets 
\begin{equation}\label{jhepee}
\Pic^{\gr}(A) \cong \Pic(A_0).
\end{equation}

However when $A$ is a noncommutative ring, the situation is substantially more involved. One needs to define the invertible bimodules in order to define the Picard group. Moreover, since for bimodules $M$ and $N$, $M\otimes_A N$ is not necessarily isomorphic to $N\otimes_A M$ as bimodules, the Picard group is not abelian. 
Even when $A$ is strongly graded, the identities such as (\ref{jhepee}) does not hold in the noncommutative setting (see~\cite{haefrio}). 

By Morita theory, an auto-equivalence of  $\Modd A$ gives rise to an invertible $A$-bimodule. This shows that the isomorphism classes of auto-equivalence of $\Modd A$ under composition form a group which is isomorphic to $\Pic(A)$ (\cite[Theorem~18.29 and Corollary 18.29]{lam}). 

\section{$\Pic^{\gr}$ of a graded commutative ring}\label{ffggyy1}   \index{$\Pic^{\gr}$, graded Picard group}

Let $A$ be a commutative $\Gamma$\!-graded ring. A graded $A$-module $P$ is called a \emph{graded invertible} module \index{graded invertible module} if there is a graded module $Q$ such that 
$P\otimes_A Q\conggr A$ as graded $A$-modules. It is clear that if $P$ is a graded invertible module, so is $P(\alpha)$ for any $\alpha \in \Gamma$. The \emph{graded Picard group}, \index{graded Picard group} $\Pic^{\gr}(A)$, is defined as the set of graded isomorphism classes of graded invertible $A$-modules with tensor product as multiplication. This is a well-defined binary operation and makes 
$\Pic^{\gr}(A)$ an abelian group with the isomorphism class of $A$ as identity element. Since a graded invertible module is an invertible module, we have a group homomorphism 
\begin{align*}
\Pic^{\gr}(A)&\longrightarrow \Pic(A),\\
[P]&\longmapsto [P],
\end{align*} where $[P]$ represents the isomorphism class of $P$ in either group. 

Since for any $\alpha, \beta \in \Gamma$, $A(\alpha)\otimes_A A(\beta) \conggr A(\alpha+\beta)$ (see~\S\ref{grtensie}), the map 
\begin{align}\label{campbell}
\phi:\Gamma &\longrightarrow \Pic^{\gr}(A),\\
 \alpha &\longmapsto [A(\alpha)]\notag
\end{align} 
is a group homomorphism. We use this map in the next lemma to calculate the graded Picard group of  graded fields.

For a $\Gamma$\!-graded ring $A$, recall from~\S\ref{scrosshg} that $\Gamma_A$ is the support of $A$ and $\Gamma^*_A=\{ \alpha \in \Gamma \mid A^*_\alpha \not = \varnothing\}$. Moreover, for a graded field $A$, $\Gamma_A$ is a subgroup of $\Gamma$. 

\begin{proposition}\label{peenei}
Let $A$ be a $\Gamma$\!-graded field with support $\Gamma_A$. Then \[\Pic^{\gr}(A)\cong \Gamma/\Gamma_A.\] 
\end{proposition}
\begin{proof}
Consider the map $\phi:\Gamma \rightarrow \Pic^{\gr}(A)$ from (\ref{campbell}). Since any graded invertible module is graded projective and graded projective modules over graded fields are graded free  (Proposition~\ref{gradedfree}), it follows that the graded invertible modules must be of the form $A(\alpha)$ for some $\alpha\in \Gamma$. This shows that $\phi$ is an epimorphism. Now if $\phi(\alpha)=[A(\alpha)]=[A]$, then $A(\alpha) \conggr A$, which implies by Corollary~\ref{rndcongcori} that $\alpha \in \Gamma_A$. 
Conversely, if $\alpha \in \Gamma_A$, then there is an element of degree $\alpha$, which has to be invertible, as $A$ is a graded field. Thus $A(\alpha)\conggr A$, again by Corollary~\ref{rndcongcori}. This shows that the kernel of $\phi$ is $\Gamma_A$. This completes the proof. 
\end{proof}

The graded Grothendieck group of a graded local ring was determined in Proposition~\ref{gdhmeisl}. Here we determine its graded Picard group.

\begin{proposition}
Let $A$ be a commutative $\Gamma$\!-graded local ring with support $\Gamma_A$. Then \[\Pic^{\gr}(A)\cong \Gamma/\Gamma_A^*.\] 
\end{proposition}
\begin{proof}
Let $M$ be the unique graded maximal ideal of $A$. By Lemma~\ref{thuphan}, if for graded projective $A$-modules $P$ and $Q$, $\overline P=P/PM$ is isomorphic to  
$\overline Q=Q/QM$ as graded $A/M$-modules, then $P$ is isomorphic to $Q$ as graded $A$-modules. This immediately implies that the natural map 
\begin{align*}
\phi:\Pic^{\gr}(A) & \longrightarrow \Pic^{\gr}(A/M),\\ 
[P] &\longmapsto [\overline P]
\end{align*}
is a monomorphism. But by the proof of Proposition~\ref{peenei}, any graded invertible $A/M$-module is of the form $(A/M)(\alpha)$ for some $\alpha \in \Gamma$. Since \[\phi([A(\alpha)])=[(A/M)(\alpha)],\] $\phi$ is an isomorphism. 
Since $\Gamma_{A/M}=\Gamma^*_A$, by Proposition~\ref{peenei}, $\Pic^{\gr}(A/M)=\Gamma/\Gamma_A^*$. This completes the proof. 
\end{proof}

In the next theorem (Theorem~\ref{rahalife}) we will be using the calculus of exterior algebras in the graded setting. Recall that if $M$ is an $A$-module, the $n$-th \emph{exterior power}  \index{exterior power}   of $M$ is the quotient of  the tensor product of $n$ copies of $M$ over $A$, denoted by $\bigotimes^n M$ (or $T_n(M)$ as in Example~\ref{stevevanzandt}), by the submodule generated by $m_1\otimes \dots \otimes m_n$, where $m_i=m_j$ for some $1\leq i\not = j \leq n$. The $n$-th exterior power of $M$ is denoted by $\bigwedge^n M$. We set $\bigwedge^0M=A$ and clearly $\bigwedge^1M=M$. 

If $A$ is a commutative $\Gamma$\!-graded ring and $M$ is a graded $A$-module, then $\bigotimes^n M$ is a graded $A$-module (\S\ref{grtensie}) and the submodule generated by $m_1\otimes \dots \otimes m_n$, where $m_i=m_j$ for some $1\leq i\not = j \leq n$ coincides with the submodule generated by 
$m_1\otimes \dots \otimes m_n$, where all $m_i$ are homogeneous and $m_i=m_j$ for some $1\leq i\not = j \leq n$, and
$m_1\otimes \dots \otimes m_n+m'_1\otimes \dots \otimes m'_n$, where all $m_i$ and $m'_i$ are homogeneous, $m_i=m'_i$ for all $1\leq i  \leq n$ except for two 
indices $i,j$, where $i\not = j$ and $m_i=m'_j$ and $m_j=m'_i$. Thus this submodule is a graded submodule of  $\bigotimes^n M$ and therefore $\bigwedge^n M$ is a graded $A$-module as well. We will use the isomorphism 
\begin{equation}\label{hchihying}
\bigwedge^n(M\oplus N) \conggr \bigoplus_{r=0}^n \Big ( \bigwedge^r M \otimes_A \bigwedge^{n-r} N \Big),
\end{equation}
which is valid in the ungraded setting, and one also checks that it respects the grading.  

\begin{lemma}\label{bachnie}
Let $A$ be a commutative $\Gamma$\!-graded ring. Then 
\begin{equation}\label{poolis}
\bigwedge^n A^m(\alpha_1,\dots,\alpha_m) \conggr \bigoplus_{1\leq i_1<i_2<\dots<i_n \leq m} A(\alpha_{i_1}+\alpha_{i_2}+\dots+\alpha_{i_n}).
\end{equation}
\end{lemma}
\begin{proof}
We prove the lemma by induction on $m$. For $m=1$ and $n=1$ we clearly have $\bigwedge^1 A(\alpha_1)\conggr A(\alpha_1)$. For $n\geq 2$, since $\bigwedge^2 A=0$ it follows that $\bigwedge^n A(\alpha)=0$. This shows that~(\ref{poolis}) is valid for $m=1$. Now by induction, by~(\ref{hchihying}), since $\bigwedge^n A(\alpha) =0$ for any  $\alpha \in \Gamma$ and $n\geq 2$, we have 
\begin{align*}
\bigwedge^n & A^m(\alpha_1,\dots,\alpha_m)  \conggr \bigwedge^n \Big ( A(\alpha_1) \oplus A^{m-1}(\alpha_2,\dots,\alpha_m)\Big )\\
&\conggr A\otimes_A \bigwedge^n A^{m-1}(\alpha_2,\dots,\alpha_m) \oplus  A(\alpha_1) \otimes_A \bigwedge^{n-1} A^{m-1}(\alpha_2,\dots,\alpha_m)\\
\\
&\conggr \bigoplus_{2\leq i_1<i_2<\dots<i_n \leq m} A(\alpha_{i_1}+\alpha_{i_2}+\dots+\alpha_{i_n}) \, \oplus\\
& \qquad \qquad \qquad \qquad A(\alpha_1) \otimes_A  \bigoplus_{2\leq i_1<i_2<\dots<i_{n-1} \leq m} A(\alpha_{i_1}+\alpha_{i_2}+\dots+\alpha_{i_n}) \\
\\
&\conggr \bigoplus_{2\leq i_1<i_2<\dots<i_n \leq m} A(\alpha_{i_1}+\alpha_{i_2}+\dots+\alpha_{i_n}) \, \oplus\\
&  \qquad \qquad \qquad \qquad  \bigoplus_{2\leq i_1<i_2<\dots<i_{n-1} \leq m} A(\alpha_1+\alpha_{i_1}+\alpha_{i_2}+\dots+\alpha_{i_n}) \\
\\
& \conggr \bigoplus_{1\leq i_1<i_2<\dots<i_n \leq m} A(\alpha_{i_1}+\alpha_{i_2}+\dots+\alpha_{i_n}).\qedhere
\end{align*}
\end{proof}

  The following corollary is immediate and will be used in Theorem~\ref{rahalife}.

\begin{corollary}\label{bachnie2}
Let $A$ be a commutative $\Gamma$\!-graded ring. Then 
\begin{equation*}
\bigwedge^n A^n(\alpha_1,\dots,\alpha_n) \conggr  A(\alpha_{1}+\alpha_{2}+\dots+\alpha_{n}).
\end{equation*}
\end{corollary}

Recall from \S\ref{hhyyuvy} that if $A$ is a $\Gamma$\!-graded commutative ring, then $K^{\gr}_0(A)$ is a commutative ring. Denote by $K^{\gr}_0(A)^*$ the group of invertible elements of this ring. The following theorem establishes the graded version of the exact sequence~(\ref{lrss083}). 

\begin{theorem}\label{rahalife}
Let $A$ be a  commutative $\Gamma$\!-graded  ring. Then there is an exact sequence, \[1 \longrightarrow \Pic^{\gr}(A) \stackrel{\phi}\longrightarrow K_0^{\gr}(A)^*,\] where $\phi([P])=[P]$. 
\end{theorem} 
\begin{proof}
It is clear that $\phi: \Pic^{\gr}(A) \rightarrow K_0^{\gr}(A)^*$ is a well-defined group homomorphism. We only need to show that $\phi$ is injective. Suppose $\phi([P])=\phi([Q])$. Thus $[P]=[Q]$ in $K^{\gr}_0(A)$. By Lemma~\ref{rhodas}, $P\oplus A^n(\overline \alpha) \conggr Q\oplus A^n(\overline \alpha)$, where $\overline \alpha=(\alpha_1,\dots,\alpha_n)$. Since $P$ and $Q$ are graded invertible, they are in particular invertible, and so are of constant rank $1$. Thus $\bigwedge^n(P)=0$, for $n\geq 2$. Now by~(\ref{hchihying})
\begin{align*}
 \bigwedge^{n+1}\big (P\oplus A^n(\overline \alpha) \big)&\conggr  \bigoplus_{i+j=n+1}\bigwedge^i A^n(\overline \alpha)  \otimes \bigwedge^j P\\
 &\conggr \bigwedge^n A^n(\overline \alpha) \otimes \bigwedge^1P\\
 \\
 &\conggr A(\alpha_1+\dots+\alpha_n)\otimes P\conggr P(\alpha_1+\dots+\alpha_n),
 \end{align*}
 thanks to Corollary~\ref{bachnie2}. Similarly \[\bigwedge^{n+1}\big (Q\oplus A^n(\overline \alpha) \big)\conggr Q(\alpha_1+\dots +\alpha_n).\] Thus 
 \[P(\alpha_1+\dots +\alpha_n)\conggr Q(\alpha_1+\dots +\alpha_n)\] which implies $P\conggr Q$. So $\phi$ is an injective map. 
 \end{proof}

\section{$\Pic^{\gr}$ of a graded noncommutative ring}

When $A$ is a noncommutative ring, the definition of the (graded) Picard group is more involved 
(see~\cite[Chapter~2]{bass},~\cite{basst},~\cite[\S55]{curtisra}, and~\cite{frohlich} for ungraded Picard groups of noncommutative rings). Note that if $P$  are $Q$ are 
$A\!-\!A$-bimodules then $P\otimes_A Q$ is not necessarily isomorphic to $Q\otimes_A P$ as $A$-bimodules. This is an indication that the  Picard group, in this setting, is not necessarily an abelian group. 

 Let $A$ and $B$ be $\Gamma$\!-graded rings and $P$ be a graded $A\!-\!B$-bimodule. Then $P$ is called a \emph{graded invertible} $A\!-\!B$-bimodule, if there is a graded $B\!-\!A$-bimodule $Q$ such that $P\otimes_B Q\conggr A$ as $A\!-\!A$-bimodules and $Q\otimes_A P \conggr B$ 
 as $B\!-\!B$-bimodules and the following diagrams are commutative.  \index{graded invertible bimodule}

\begin{equation}\label{verm1111}
{\def\labelstyle{\displaystyle}
\xymatrix{
P\otimes_B Q\otimes_A P \ar[r] \ar[d]& A \otimes_A P \ar[d]\\
P\otimes_B B \ar[r] & P
}}\qquad 
{\def\labelstyle{\displaystyle}
\xymatrix{
Q\otimes_A P\otimes_B Q \ar[r] \ar[d]& B \otimes_B Q \ar[d]\\
Q\otimes_A A \ar[r] & Q.
}}
\end{equation}

As in the commutative case (\S\ref{ffggyy1}), for a noncommutative graded ring $A$, the \emph{graded Picard group}, $\Pic^{\gr}(A)$, \index{graded Picard group}  is defined as the set of graded isomorphism classes of graded invertible $A\!-\!A$-bimodules with tensor product as multiplications.  The graded isomorphism class of the graded bimodule $P$ is denoted by $[P]$. Since $P$ is invertible, it has an inverse $[Q]$ in $\Pic^{\gr}(A)$ and Diagram~\ref{verm1111} guarantees that $\big([P][Q]\big)[P]=[P]\big([Q][P]\big)=[P]$. 

\begin{theorem}\label{hgy7d53} \index{graded Morita equivalent} 
Let $A$ and $B$ be $\Gamma$\!-graded rings. If $A$ is graded Morita equivalent to $B$, then 
\[
\Pic^{\gr}(A) \cong  \Pic^{\gr}(B). 
\]
\end{theorem}
\begin{proof}
Let $\Gr A \approx_{\gr} \Gr B$. Then there is a graded equivalence $\phi:\Gr A\rightarrow \Gr B$ with an inverse $\psi:\Gr B \rightarrow \Gr A$.
By Theorem~\ref{grmorim} (and its proof), $\psi(B)=P$ is a graded $B\!-\!A$-bimodule, $\phi\cong -\otimes_A P^*$ and $\psi\cong-\otimes_B P$. Now one can easily check that the map 
\begin{align*}
\Pic^{\gr}(A) &\longrightarrow  \Pic^{\gr}(B),\\
[M]  &\longmapsto [P\otimes_A M\otimes_A P^*]
\end{align*}
is an isomorphism of groups. 
\end{proof}

In fact, in view of the fact that the Picard group coincides with the group of isomorphism classes of auto-equivalences of a given module category, Theorem~\ref{hgy7d53} can be established directly. 

Let $A$ be a $\Gamma$\!-graded ring. Consider the group $\Aut^{\gr}(A)$ of all the graded automorphisms and its subgroup $\Inn_{A_0}^{\gr}(A)$ which consists of  graded inner automorphisms induced by the homogeneous elements of degree zero as follows. 
\[ \Inn_{A_0}^{\gr}(A) =\big \{\, f:A\longrightarrow A \mid f(x)=uxu^{-1}, u \in A_0^* \, \big \}.\] 
The graded $A$-bimodule structure on $A$  induced by $f,g\in \Aut^{\gr}(A)$ will be denoted  by ${}_f A_g$. Namely, $A$ acts on ${}_f A_g$ as follows, $a.x.b=f(a)xg(b)$. 
 For $f,g,h \in \Aut^{\gr}(A)$, one can prove the following graded $A$-bimodule isomorphisms. 
\begin{align}\label{gdn22}
{}_fA_g & \conggr {}_{hf}A_{hg} \\
{}_fA_1 \otimes _A {}_gA_1 & \conggr {}_{fg}A_{1} \notag \\
 {}_fA_1 \otimes _A {}_1A_f & \conggr   {}_1A_f \otimes _A {}_fA_1 \conggr A.  \notag
\end{align}

The following theorem provides two exact sequences  between the groups $\Gamma$, $\Inn_{A_0}^{\gr}(A)$ and $\Aut^{\gr}(A)$ and the graded Picard group. The first sequence belongs to the graded setting, whereas the second sequence is a graded version of a similar result in the ungraded setting  (see~\cite[Theorem~55.11]{curtisra}). 

\index{centre of a ring} \index{$C(A)$, centre of $A$}
Recall that $C(A)$ stands for the centre of the ring $A$ which is a graded subring of $A$ (when the grade group $\Gamma$ is abelian, see Example~\ref{millarexi}). Moreover, it is easy to see that the map 
\begin{align}\label{jujuj1}
\phi: \Gamma &\longrightarrow \Pic^{\gr}(A),\\
\alpha &\longmapsto [A(\alpha)]\notag
\end{align} 
(which was considered in the case of commutative graded rings in~(\ref{campbell})) is well-defined and is a homomorphism. 

\begin{theorem}\label{picadoli}
Let $A$ be a $\Gamma$\!-graded ring. Then the following sequences are exact. 

\begin{enumerate}[\upshape(1)]

\item The sequence  $1\longrightarrow \Ga_{C(A)}^* \longrightarrow \Gamma \stackrel{\phi}\longrightarrow \Pic^{\gr}(A),$ where $\phi(\alpha)=[A(\alpha)]$. 

\item The sequence $1\longrightarrow \Inn_{A_0}^{\gr}(A) \longrightarrow \Aut^{\gr}(A) \stackrel{\phi}{\longrightarrow} \Pic^{\gr}(A),$
where $\phi(f)=[{}_fA_1]$. 
\end{enumerate}
\end{theorem}

\begin{proof}

(1) Consider the group homomorphism $\phi$ defined in~(\ref{jujuj1}). 
If $u \in C(A)^*$ is a homogeneous element of degree $\alpha$, then the map $\psi: A\rightarrow A(\alpha)$, $a\mapsto ua$, is a graded $A$-bimodule isomorphism. This shows that $ \Ga_{C(A)}^* \subseteq \ker(\phi)$. On the other hand if $\phi(\alpha)=[A(\alpha)]=1_{\Pic^{\gr}(A)}=[A]$, then there is a graded $A$-bimodule isomorphism $\psi:A\rightarrow A(\alpha)$. 
 From this it follows that there is an invertible homogeneous element  $u\in C(A)$ of degree $\alpha$ such that $\psi(x)=ux$ (see also Corollary~\ref{rndcongcori} and Proposition~\ref{rndconggrrna}).  This completes the proof. 

(2) The fact that the map $\phi$ is well-defined and is a homomorphism follows from (\ref{gdn22}). We only need to show that $\ker\phi$ coincides with $\Inn_{A_0}^{\gr}(A)$.
Let $f\in \Inn_{A_0}^{\gr}(A)$, so that $f(x)=uxu^{-1}$, for any $x\in A$, where $u \in A^*_0$. Then the $A$-graded bimodules ${}_fA_1$ and $A$ are isomorphic. Indeed for the map 
\begin{align*}
\theta:{}_fA_1&\longrightarrow A,\\
x&\longmapsto u^{-1}x
\end{align*}
 we have
\[ \theta(a.x.b)=\theta(uau^{-1}xb)=au^{-1}xb=a.\theta(x).b,\]
which gives a graded $A$-bimodule isomorphisms. This shows that \[\Inn_{A_0}^{\gr}(A) \subseteq \ker \phi.\] 
Conversely, suppose $\phi(f)=[{}_fA_1]=[A]$. Thus there is a graded $A$-bimodule isomorphism $\theta:{}_fA_1\rightarrow A$ such that 
\begin{equation}\label{arum}
 \theta(a.x.b)=\theta(f(a)xb)=a\theta(x)b.
 \end{equation}
 Since $\theta$ is bijective, there is a $u\in A$ such that $\theta(u)=1$. Also, since $\theta$ is graded and $1\in A_0$ it follows that $u\in A_0$. But $A=\theta(A)=\theta(A1)=Au$ and similarly $A=uA$. This implies $u\in A_0^*$. 
Plugging $x=a=1$ in (\ref{arum}), we have $\theta(b)=ub$, for any $b\in A$. Using this identity, by plugging $x=b=1$ in (\ref{arum}), we obtain $uf(a)=au$. Since $u$ is invertible, we get $f(a)=u^{-1}au$, so $f\in \Inn_{A_0}^{\gr}(A)$ and we are done. 
\end{proof}

For a graded noncommutative ring $A$ the problem of determining the Picard group $\Pic(A)$ is difficult. There are several cases in the literature where $\Picent(A):=\Pic_R(A)$ is determined. Here $R$ is a centre of $A$ and $\Pic_R(A)$ is a subgroup of $\Pic(A)$ consisting of all isomorphism classes of invertible bimodules $P$ of $A$ which are centralised by $R$, (\ie $rp=pr$, for all $r\in R$ and $p\in P$) \ie $P$ is a $A\otimes_R A^{\op}$-left module
(see~\cite{frohlich}, and~\cite[\S 55]{curtisra}). For a graded division algebra $A$ (\ie a graded division ring which is finite dimensional over its centre) we will be able to determine $\Pic^{\gr}(A)$. In fact, since graded division algebras are graded Azumaya algebras, we first determine $\Pic^{\gr}_R(A)$, where $A$ is a graded Azumaya algebra over its centre $R$. The approach follows the idea in Bass~\cite[\S3, Corollary~4.5]{basst}, where it is shown that for an Azumaya algebra $A$ over $R$, the Picard group $\Pic_R(A)$ coincides with $\Pic(R)$. \index{Azumaya algebra}

Recall that a graded algebra $A$ over a commutative graded ring $R$ is called a \emph{graded Azumaya algebra} \index{graded Azumaya algebra} if $A$ is a graded faithfully projective $R$-module and the natural map 
\begin{align*}
\phi:A\otimes_R A^{\op} &\longrightarrow \End_R(A),\\
a\otimes b &\longmapsto \phi(a\otimes b)(x)=axb,
\end{align*} where $a,x\in A$ and $b\in A^{\op}$, is a graded $R$-isomorphism. This implies that $A$ is an Azumaya algebra over $R$. Conversely, if $A$ is a graded algebra over graded ring $R$ and $A$ is an Azumaya algebra over $R$, then $A$ is faithfully projective as an $R$-module and the natural map $\phi:A\otimes_R A^{\op} \rightarrow \End_R(A)$ is naturally graded. Thus $A$ is also a graded Azumaya algebra. 

A graded division algebra is a graded central simple algebra. We first define these type of rings and show that they are graded Azumaya algebras. 

A graded algebra $A$ over a graded commutative ring $R$ is said to be a 
\emph{graded central simple algebra}\index{graded central simple algebra} over $R$ if $A$ is graded simple ring, \ie $A$ does not have any proper graded two-sided ideals, $C(A) \conggr R$ and $A$ is finite dimensional as an $R$-module. Note that since
$A$ is graded simple, the centre of $A$ (identified with $R$), is a graded field. Thus $A$ is graded
free as a graded module over its centre by
Proposition~\ref{gradedfree}, so the dimension of $A$ over $R$ is
uniquely defined. \index{graded simple ring}

\begin{proposition}\label{gcsaazumayaalgebra}
Let $A$ be a $\Ga$\!-graded central
simple algebra over a graded field $R$. Then $A$ is a graded Azumaya
algebra over $R$.
\end{proposition}

\begin{proof}
Since $A$ is graded free of finite dimension over $R$, it follows
that $A$ is faithfully projective over $R$. Consider the natural
graded $R$-algebra homomorphism $\psi: A \otimes_R A^{\op} \ra
\End_R (A)$ defined by $\psi(a\otimes b)(x)=axb$ where $a,x \in A$,
$b \in A^{\op}$. Since graded ideals of $A^{\op}$ coincide with 
graded ideals of $A$, $A^{\op}$ is also graded simple. Thus $A \otimes A^{\op}$ is also
graded simple (see~\cite[Chapter~2]{tignolwadsworth}, so $\psi$ is injective. Hence the map is surjective
by dimension count, using Theorem~\ref{grdimensionprop}. This shows
that $A$ is an Azumaya algebra over $R$, as required.
\end{proof}

Let $A$ be a graded algebra over a graded commutative ring $R$. For a graded $A$-bimodule $P$ centralised by $R$, define \[P^A=\{\,p \in P \mid ap=pa, \text{ for all } a\in A \, \}.\] 
Denote by $\Gr A_R \rGr$ the category of graded $A$-bimodules centralised by $R$. When $A$ is a graded Azumaya algebra, the functors \index{$\Gr A_R \rGr$}
\begin{align*}
\Gr A_R \rGr &\longrightarrow \Gr R,  & \Gr R &\longrightarrow \Gr A_R \rGr,\\
P &\longmapsto P^A& N &\longmapsto N\otimes_R A
\end{align*}
are inverse equivalence of categories which graded projective modules correspond to graded projective modules and moreover, invertible modules correspond to invertible modules (see~\cite[Proposition~III.4.1]{can} and~\cite[Theorem~5.1.1]{knus} for the ungraded version). 
This immediately implies that 
\begin{equation}\label{mofbe72}
\Pic^{\gr}_R(A)\cong \Pic^{\gr}(R).
\end{equation}

\begin{lemma}\label{campbell3} \index{graded division algebra}
Let $A$ be a $\Gamma$\!-graded division algebra with centre $R$. Then \[\Pic^{\gr}_R(A)\cong \Gamma/\Gamma_R.\] 
\end{lemma}
\begin{proof}
By Proposition~\ref{gcsaazumayaalgebra},  $A$ is a graded Azumaya algebra over $R$. By~(\ref{mofbe72}), $\Pic^{\gr}_R(A)\cong \Pic^{\gr}(R)$. Since $R$ is a graded field,  by Proposition~\ref{peenei}, 
$ \Pic^{\gr}(R)=\Gamma/\Gamma_R$. 
\end{proof}

\begin{example}\label{donjendafe}
Recall from Example~\ref{egofgrdivisionrings} that the quaternion algebra  $\H = \R \+ \R i \+ \R j \+ \R k$
is a graded division algebra having two different gradings, \ie $\Z_2$ and $\Z_2
\times \Z_2$, respectively. Since the centre $\R$ is concentrated in degree $0$ in either grading, by Lemma~\ref{campbell3}, 
$\Pic^{\gr}_{\R}(\H)=\mathbb Z_2$ or $\Pic^{\gr}_{\R}(\H)= \mathbb Z_2\oplus \mathbb Z_2$, depending on the grading. This also shows that in contrast to the graded Grothendieck group, $\Pic^{\gr}_{\R}(\H)\not \cong \Pic(\H_0)$, although $\H$ is a strongly graded ring in either grading. 
\end{example}

\begin{example}\label{whyso9} \index{associated graded ring}
Let $(D,v)$ be a tame valued division algebra over a henselian field $F$, where $v:D^*\rightarrow \Gamma$ is the valuation homomorphism. By  Example~\ref{grdiviwad}, there is a $\Gamma$\!-graded division algebra $\gr(D)$ associated to $D$  with the centre $\gr(F)$, where $\Gamma_{\gr(D)}=\Gamma_D$ and $\Gamma_{\gr(F)}=\Gamma_F$. Since by Lemma~\ref{gcsaazumayaalgebra}, graded division algebras are graded Azumaya, by Lemma~\ref{campbell3}, \[\Pic^{\gr}_{\gr(F)}\big(\gr(D)\big)=\Gamma/\Gamma_F.\]
\end{example}

Let $A$ be a strongly graded $\Gamma$\!-graded ring. Then for any $\alpha\in \Gamma$, $A_\alpha$ is a finitely generated projective invertible $A_0$-bimodule (see Theorem~\ref{mozsace}) and the map 
\begin{align}\label{redflaximel}
\psi:\Gamma&\longrightarrow \Pic(A_0),\\
\alpha&\longmapsto [A_\alpha] \notag
\end{align}
is a group homomorphism.  We then have a natural commutative diagram (see Theorem~\ref{picadoli}) 
\begin{equation}\label{redflaxim}
\begin{split}
\xymatrix{
1 \ar[r] &  C(A) \cap A_\alpha^* \ar@{^{(}->}[d] \ar[r] & \Gamma  \ar@{=}[d] \ar[r]^-{\phi} & \Pic^{\gr}(A) \ar[d]^{(-)_0}\\
1 \ar[r] &  C_A(A_0) \cap A_\alpha^* \ar[r] &  \Gamma \ar[r]^-{\psi}  &  \Pic(A_0)
}
\end{split}
\end{equation}

\begin{remark}\scm[Relating $\Pic^{\gr}(A)$ to $\Pic(A_0)$ for a strongly graded ring $A$] \index{strongly graded ring}
\vspace{0.2cm}

Example~\ref{donjendafe} shows that for a strongly graded ring $A$, $\Pic^{\gr}(A)$ is not necessarily isomorphic to $\Pic(A_0)$. 
However one can relate these two groups with an exact sequence as follows.
\[  1\longrightarrow H^1(\Gamma, Z(A_0)^*) \longrightarrow \Pic^{\gr}(A) \stackrel{\Psi}{\longrightarrow} \Pic(A_0)^{\Gamma} \longrightarrow H^2(\Gamma, Z(A_0)^*). \]
Here $\Gamma$ acts on $\Pic(A_0)$ by \[\gamma [P]=[A_\gamma \otimes_{A_0} P \otimes_{A_0} A_{-\gamma}],\] so the notation 
$\Pic(A_0)^{\Gamma}$ refers to the group of $\Gamma$\!-invariant elements of $\Pic(A_0)$. Moreover, $Z(A_0)^*$ denotes the units of the centre of $A_0$ and $H^1,H^2$  denote the first and second cohomology groups. The map $\Psi$ is defined by $\Psi([P])=[P_0]$ (see~\cite{marcus1} for details. Also see~\cite{beatrio1}). 

\end{remark}

We include a result from~\cite{ibno} which describes if $A_0$ has IBN, then $A$ has gr-IBN, based on a condition on $\Pic(A_0)$. In general, one can produce  an example of strongly graded ring $A$ such that $A_0$ has IBN whereas $A$ is a non-IBN ring. 

For a ring $R$, define \[\Pic_n(R)=\big \{\, [X] \in \Pic(R) \mid X^n \cong R^n \text{ as right $R$-module} \, \big \}.\] This is a subgroup of $\Pic(R)$ and $\Pic_n(R)$ and $\Pic_m(R)$ are subgroups of $\Pic_{nm}(R)$. Thus 
\[\Pic_{\infty}(R)=\bigcup_{n\geq 1}\Pic_n(R)\] is a subgroup of $\Pic(R)$.

\begin{proposition}
Let $A$ be a strongly $\Gamma$\!-graded ring such that $\{[A_\alpha]\mid \alpha \in \Gamma\} \subseteq \Pic_{\infty}(A_0)$. If $A_0$ has IBN then $A$ has gr-IBN. Moreover, if $\Gamma$ is finite, then $A$ has IBN if and only if $A_0$ has IBN. 
\end{proposition}

\begin{proof}
Suppose that $A^n(\overline \alpha)\conggr A^m(\overline \beta)$ as graded right $A$-modules, where $\overline \alpha=(\alpha_1,\dots,\alpha_n)$ and $\overline \beta=(\beta_1,\dots,\beta_m)$. Using Dade's theorem~\ref{dadesthm}, we have 
\begin{equation}\label{younguk}
A_{\alpha_1}\oplus\dots\oplus A_{\alpha_n} \cong A_{\beta_1}\oplus \dots\oplus A_{\beta_m},
\end{equation}
 as right $A_0$-modules. Since each $[A_{\alpha_i}]$ and $[A_{\beta_j}]$ are in some $\Pic_{n_i}(A_0)$ and $\Pic_{n_j}(A_0)$, respectively, there is a large enough $t$ such that $[A_{\alpha_i}], [A_{\beta_j}] \in \Pic_t(A_0)$, for all $1\leq i\leq n$ and $1\leq j \leq m$. Thus $A_{\alpha_i}^t\cong A_0^t$ and $A_{\beta_j}^t\cong A_0^t$, for all $i$ and $j$. Replacing this into (\ref{younguk}) we get $A_0^{nt}\cong A_0^{mt}$. Since $A_0$ has IBN, it follows $n=m$. 

Now suppose $\Gamma$ is finite. If $A_0$ has IBN, then we will prove that $A$ has IBN. From what we proved above, it follows that $A$ has gr-IBN. If $A^n\cong A^m$, then $F(A^n)\conggr F(A^m)$, where $F$ is the adjoint to the forgetful functor (see~\S\ref{forgetful}). Since $F(A^n)=F(A)^n$ and $F(A)=\bigoplus_{\gamma \in \Gamma}A(\gamma)$, using the fact that $A$ has gr-IBN, it follows immediately that $n=m$. Conversely, suppose $A$ has IBN. If $A_0^n\cong A_0^m$, then tensoring with $A$ over $A_0$ we have 
\[A^n\cong_{\gr} A_0^n \otimes_{A_0} A \cong A_0^m \otimes_{A_0} A \cong_{\gr}A^m.\] So $n=m$ and we are done. 
\end{proof}

Here is another application of Picard group, relating the ideal theory of $A_0$ to $A$. This proposition is taken from~\cite[Theorem~3.4]{oystae}. Recall the group homomorphism 
\begin{align*}
\psi:\Gamma &\longrightarrow \Pic(A_0)\\
\alpha&\longmapsto [A_\alpha],
\end{align*}

 from~(\ref{redflaximel}). For nonzero element $a \in A$, define the \emph{length} of $a$, $l(a)=n$, where $n$ is the number of nonzero homogeneous elements appear in the decomposition of $a$ into the homogeneous sum. \index{length of an element}

\index{simple ring}
\begin{proposition}\label{balmainmine} \index{Picard group}
Let $A$ be a strongly $\Gamma$\!-graded ring such that the homomorphism $\psi:\Gamma \rightarrow \Pic(A_0)$ is injective. If $A_0$ is a simple ring then so is $A$. 
\end{proposition}
\begin{proof}
Let $I$ be a nonzero ideal of $A$ and $x$ a nonzero element of $I$ which has the minimum length among all elements of $I$.  Suppose $x_\alpha$ is the homogeneous element of degree $\alpha$ appearing in the decomposition of $x$ into homogeneous sum. If $x_\alpha\not =0 $, then multiplying $x$ with a homogeneous element of degree $-\alpha$, we can assume that $x_0 \not = 0$. Since $A_0$ is a simple ring, we have $\sum_i y_i x_0 z_i =1$, where $y_i,z_i \in A_0$. Consider $\sum_i y_i x z_i \in I$, and note that its length is equal to the length of x. Replacing $x$ with this element, we can assume $x_0=1$. If $x=x_0=1$ then $I=A$ and we are done. Otherwise, suppose there is $x_\alpha\not = 0$ in the decomposition of $x$. Note that for each $b \in A_0$, 
$bx-xb \in I$ with $l(bx-xb) < l(x)$. The minimality of $x$ implies that $bx=xb$ and thus $b x_\alpha=x_\alpha b$ for any $b\in A_0$, \ie  $x_\alpha \in C_A(A_0)$. We show that $x_\alpha \in A_\alpha^*$. Note that $A_{-\alpha}x_\alpha$  and $x_\alpha A_{-\alpha}$ are nonzero (ungraded) two-sided ideals of $A_0$. Again since $A_0$ is simple, this implies $x_\alpha$ is invertible. Therefore $x_\alpha \in C_A(A_0) \cap A_\alpha^*$. This is a contradiction with the assumption that 
$\psi$ is injective (see~\ref{redflaxim}). This finishes the proof. 
\end{proof}

   \chapter[Classification of Graded Ultramatricial Algebras]{Graded Ultramatricial Algebras\\ Classification via $K^{\gr}_0$}\label{ultriuy}
 
 Let $F$ be a field. An $F$-algebra is called an  \emph{ultramatricial} \index{ultramatricial algebra} algebra, if it is isomorphic to the union of an increasing chain of a finite product of matrix algebras over $F$. When $F$ is the field of complex numbers, these algebras are also called  \emph{locally semisimple} \index{locally semisimple algebra} algebras (or LS-algebras for short), as they are isomorphic to a union of a chain of semisimple $\mathbb C$-algebras. An important example of such rings is the group ring $\mathbb C[S_{\infty}]$, where $S_{\infty}$ is the infinite symmetric group. 
These rings appeared in the setting of $C^*$-algebras and then von Neumann regular algebras in the work of Grimm, Bratteli, Elliott, Goodearl, Handelman and many others after them.  \index{$S_{\infty}$, infinite symmetric group} \index{permutation group} \index{LS-algebras}

Despite their simple constructions,  the study of ultramatricial algebras are far from over.  As it is noted in~\cite{vershik} ``The current state of the theory of LS-algebras and its applications should be considered as the initial one; one has discovered the first fundamental facts and noted a general circle of questions. To estimate it in perspective, one must consider the enormous number of diverse and profound examples of such algebras. In addition one can observe the connections with a large number of areas of mathematics.''

One of the sparkling examples of the Grothendieck group as a complete invariant is in the setting of ultramatricial algebras (and AF $C^*$-algebras). It is by now a classical result that the $K_0$-group along with its positive cone ${K_0}_+$ (the dimension group) and the position of identity  is a complete invariant for such algebras 
(\cite[Bratteli-Elliott Theorem~15.26]{goodearlbook}). To be precise, let $R$ and $S$ be (unital) ultramatricial algebras. Then $R\cong S$ if and only if there is an order isomorphism $\big(K_0(R),[R]\big)\cong \big(K_0(S),[S]\big)$, that is, an isomorphism from $K_0(R)$ to $K_0(S)$ which sends the positive cone onto the positive cone and $[R]$ to $[S]$.  To emphasis the ordering, this isomorphism is also denoted by 
$\big(K_0(R),K_0(R)_+,[R]\big)\cong \big(K_0(S),K_0(S)_+,[S]\big)$ (see~\S\ref{pregg5}).

\index{AF-algebras}  \index{dimension group}

The theory has also been worked out for the non-unital ultramatricial algebras (see Remark~\ref{idontcare} and~\cite[Chapter XII]{goodhand}). 
Two valuable surveys on ultramatricial algebras and their relations with other branches of mathematics are~\cite{vershik, zalesski}. The lecture notes by Effros~\cite{effros}  also gives an excellent detailed account of this theory. 
 
 In this section we initiate the graded version of this theory. We define the graded ultramatricial algebras. We then show that the graded Grothendieck group, equipped with its module structure and its ordering is a complete invariant for such algebras (Theorem~\ref{ujhosmr}).  When the grade group considered to be trivial, we retrieve the Bratteli-Elliott Theorem (Corollary~\ref{marriotcitywall}). 
 
When the grading is present, $K^{\gr}_0$ seems to capture more details than $K_0$.  The following example demonstrates  this: the Grothendieck group classifies matricial algebras, however it is not a complete invariant if we enlarge the class of algebras to include matrices over Laurent rings.  
Consider the following graphs and their associated Leavitt path algebras (see Theorems~\ref{gra1} and \ref{grCn1}). 
 \begin{equation*}
\xymatrix{
F & \bullet \ar[r] & \bullet \ar[r] & \bullet, & \LL_K(F)\cong \M_3(K),\\
E&  \bullet \ar[r]  & \bullet \ar@/^1.5pc/[r] & \bullet, \ar@/^1.5pc/[l] & \LL_K(E)\cong \M_3(K[x,x^{-1}]).\\
}
\end{equation*}

\vspace{0.25cm}

\noindent Since $K_0$ is Morita invariant, one can calculate that 
$$\Big(K_0( \LL(F)),K_0( \LL(F))_{+}, [ \LL(F)]\Big ) \cong \big(\mathbb Z, \mathbb N, 3\big),$$
and 
$$\Big(K_0(\LL(E)), K_0(\LL(E))_{+}, [\LL(E)]\Big ) \cong \big(\mathbb Z, \mathbb N, 3\big).$$
 However  $$\M_3(K) \not \cong \M_3(K[x,x^{-1}]).$$

But as we will prove in this Chapter, $K^{\gr}_0$ is a complete invariant for this class of algebras (Theorem~\ref{catgrhsf}).

\section{Graded matricial algebras}\label{grmat43}

We begin with the graded version of matricial algebras. \index{graded matrix ring}

\begin{definition}\label{ggthu}
Let  $A$ be a $\Gamma$\!-graded field. A $\Gamma$\!-\emph{graded matricial} \index{graded matricial algebra} $A$-algebra is a graded $A$-algebra of the form \[\M_{n_1}(A)(\overline \delta_1) \times \dots \times   \M_{n_l}(A)(\overline \delta_l),\] where $\overline \delta_i=(\delta_{1}^{(i)},\dots,\delta_{n_i}^{(i)})$, $\delta_{j}^{(i)} \in \Gamma$, $1\leq j \leq n_i$ and $1\leq i \leq l$. 
\end{definition}

If $\Gamma$ is a trivial group, then 
 Definition~\ref{ggthu} reduces to the definition of matricial algebras (\cite[\S15]{goodearlbook}). Note that if $R$ is a graded matricial $A$-algebra, then $R_0$ is a matricial $A_0$-algebra (see~\S\ref{wadi}). 

In general, when two graded finitely generated projective modules represent the same element in the graded Grothendieck group, then they are not necessarily graded isomorphic but rather graded stably isomorphic (see Lemma \ref{rhodas}(3)). However for graded matricial algebras one can prove that the graded stably isomorphic modules are in fact graded isomorphic. Later in Lemma~\ref{fdpahf2} we show this is also the case in the larger category of  graded ultramatricial algebras.

\begin{lemma}\label{fdpahf} \index{graded Morita equivalent} 
Let $A$ be a $\Gamma$\!-graded field and $R$ be a $\Gamma$\!-graded matricial $A$-algebra.  Let $P$ and $Q$ be graded finitely generated  projective $R$-modules. Then $[P]=[Q]$ in $K^{\gr}_0(R)$, if and only if  $P \cong _{\gr} Q$. 
\end{lemma}

\begin{proof}
Since the functor $K^{\gr}_0$ respects the direct sum, it suffices to prove the statement for a graded matricial algebra of the form 
$R=\M_{n}(A)(\overline {\delta})$. Let $P$ and $Q$ be graded finitely generated  projective $R$-modules such that $[P]=[Q]$ in $K^{\gr}_0(R)$. By Proposition~\ref{grmorita},  $R=\M_{n}(A)(\overline {\delta})$ is graded Morita equivalent to $A$. So there are equivalent functors $\psi$ and $\phi$ such that $\psi\phi\cong 1$ and $\phi \psi \cong 1$, which also induce an isomorphism $K^{\gr}_0(\psi):K^{\gr}_0(R) \rightarrow K^{\gr}_0(A)$ such that $[P]\mapsto [\psi(P)]$. Now since $[P]=[Q]$, it follows $[\psi(P)]=[\psi(Q)]$  in $K^{\gr}_0(A)$. But since $A$ is a graded field, by the proof of Proposition~\ref{k0grof}, any graded finitely generated  projective $A$-module can be written uniquely 
as a direct sum of appropriate shifted $A$. Writing $\psi(P)$ and $\psi(Q)$ in this form, the homomorphism~(\ref{attitude}) shows that $\psi(P)\cong_{\gr} \psi(Q)$. Now applying the functor $\phi$ to this we obtain $P\cong_{\gr} Q$.
\end{proof}

Let $A$ be a $\Gamma$\!-graded field (with the support $\Gamma_A$) and let $\mathcal C$ be a category consisting of $\Gamma$\!-graded matricial $A$-algebras as objects and  
$A$-graded algebra homomorphisms  as morphisms. We consider the quotient category $\mathcal C^{\out}$ obtained from $\mathcal C$ by identifying homomorphisms which differ up to a degree zero graded inner automorphim. That is, the graded homomorphisms $\phi, \psi \in \Hom_{\mathcal C}(R,S)$ represent the same morphism in $\mathcal C^{\out}$ if there is an inner automorpshim $\theta:S\rightarrow S$, defined by $\theta(s)=xsx^{-1}$, where $\deg(x)=0$, such that $\phi=\theta\psi$. The following theorem shows that $K^{\gr}_0$ as a graded dimension group (see Example~\ref{pofsgtwn}) ``classifies'' the category of $\mathcal C^{\out}$. 
This is a graded analog of a similar result for matricial algebras (see~\cite[Lemma~15.23]{goodearlbook}). Recall from \S\ref{pregg5} that $\mathcal P$ is a category with objects consisting of the pairs  $(G,u)$, where $G$ is a  $\Gamma$\!-pre-ordered module and $u$ is an order-unit, and $f:(G,u)\rightarrow (H,v)$ is an order preserving $\Gamma$\!-homomorphism such that $f(u)=v$.
\index{quotient category} \index{positive cone of ordering} \index{order preserving homomorphism}

\begin{theorem}\label{catgrhsf}
Let $A$ be a $\Gamma$\!-graded field and $\mathcal C^{\out}$ be the category consisting of $\Gamma$\!-graded matricial $A$-algebras as objects and  
$A$-graded algebra homomorphisms modulo graded inner automorphisms as morphisms.  Then \[K^{\gr}_0: \mathcal C^{\out}\rightarrow \mathcal P\] is a fully faithful functor. Namely,

\begin{enumerate}[\upshape(1)]

\item (well-defined and faithful) For any graded matricial $A$-algebras $R$ and $S$ and $\phi,\psi\in \Hom_{\mathcal C}(R, S)$,  we have  $\phi(r)=x\psi(r)x^{-1}$, $r\in R$, for some invertible homogeneous element $x$ of degree $0$ in $S$, if and only if $K^{\gr}_0(\phi)=K^{\gr}_0(\psi)$. 

\smallskip

\item(full) For any graded matricial $A$-algebras $R$ and $S$ and  the morphism 
\[f:\big(K^{\gr}_0(R),[R]\big)\rightarrow \big(K^{\gr}_0(S),[S]\big)\] in $\mathcal P$, there is $\phi\in \Hom_{\mathcal C}(R, S)$ such that $K^{\gr}_0(\phi)=f$.

\smallskip

\end{enumerate} 

\end{theorem}

\begin{proof}
(1) (well-defined) Let $\phi,\psi\in \Hom_{\mathcal C}(R, S)$  such that $\phi=\theta\psi$, where $\theta(s)=xsx^{-1}$ for some invertible homogeneous element $x$ of $S$ of degree $0$. Then \[K^{\gr}_0(\phi)=K^{\gr}_0(\theta\psi)=K^{\gr}_0(\theta)K^{\gr}_0(\psi)=K^{\gr}_0(\psi),\] since $K^{\gr}_0(\theta)$ is the identity map (see \S\ref{kidenh}). 

(faithful) The rest of the proof is similar to the ungraded version with an extra attention given to the grading (cf. \cite[p.218]{goodearlbook}). Let $K^{\gr}_0(\phi)=K^{\gr}_0(\psi)$. Let $R=\M_{n_1}(A)(\overline \delta_1) \times \dots \times   \M_{n_l}(A)(\overline \delta_l)$ and set $g_{jk}^{(i)}=\phi(\e^{(i)}_{jk})$ and 
$h_{jk}^{(i)}=\psi(\e^{(i)}_{jk})$ for  $1\leq i \leq l$ and $1\leq j,k \leq n_i$, where $\e^{(i)}_{jk}$ are the graded matrix units basis for $\M_{n_i}(A)$. Since $\phi$ and $\psi$ are graded homomorphism, 
\[\deg(\e^{(i)}_{jk})=\deg(g^{(i)}_{jk})=\deg(h^{(i)}_{jk})=\delta^{(i)}_j-\delta^{(i)}_k.\]

Since $\e^{(i)}_{jj}$ are pairwise graded orthogonal idempotents (of degree $0$)  in $R$  and  
$\sum_{i=1}^l \sum_{j=1}^{n_i} \e^{(i)}_{jj}=1$ (\ie graded full matrix units),  the same is also the case for $g^{(i)}_{jj}$ and $h^{(i)}_{jj}$. 
Then 
\[[g^{(i)}_{11}S]=K^{\gr}_0(\phi)([\e^{(i)}_{11}R])=K^{\gr}_0(\psi)([\e^{(i)}_{11}R])=[h^{(i)}_{11}S].\]
By Lemma~\ref{fdpahf}, $g^{(i)}_{11}S \cong_{\gr} h^{(i)}_{11}S$. Thus there are homogeneous elements $x_i$ and $y_i$ of degree $0$ such that $x_iy_i=g^{(i)}_{11}$ and $y_ix_i=h^{(i)}_{11}$ (see \S\ref{idemptis}). Let 
\[x=\sum_{i=1}^{l}\sum_{j=1}^{n_i}g^{(i)}_{j1}x_ih^{(i)}_{1j},\, \,  \text{ and }  \, \, 
y=\sum_{i=1}^{l}\sum_{j=1}^{n_i}h^{(i)}_{j1}y_ig^{(i)}_{1j}.\] Note that $x$ and $y$ are homogeneous elements of degree zero. One checks easily that $xy=yx=1$. Now for  $1\leq i \leq l$ and $1\leq j,k \leq n_i$, we have 
\begin{align*}
x h^{(i)}_{jk}&=\sum_{s=1}^{l}\sum_{t=1}^{n_s}g^{(s)}_{t1}x_sh^{(s)}_{1t}h^{(i)}_{jk}\\
& = g^{(i)}_{j1}x_ih^{(i)}_{1j}h^{(i)}_{jk}=g^{(i)}_{jk}g^{(i)}_{k1}x_ih^{(i)}_{1k}\\
&= \sum_{s=1}^{l}\sum_{t=1}^{n_s}g^{(i)}_{jk}g^{(s)}_{t1}x_sh^{(i)}_{1t}= g^{(i)}_{jk}x
\end{align*}
Let $\theta:S\rightarrow S$ be the graded inner automorphism $\theta(s)=xsy$. Then \[\theta\psi(e^{(i)}_{jk}))=
xh^{(i)}_{jk}y=g^{(i)}_{jk}=\phi(e^{(i)}_{jk}).\] Since $\e^{(i)}_{jk}$, $1\leq i \leq l$ and $1\leq j,k \leq n_i$, 
form a homogeneous $A$-basis for $R$,  $\theta\psi=\phi$. 

\smallskip

(2) Let $R=\M_{n_1}(A)(\overline \delta_1) \times \dots \times   \M_{n_l}(A)(\overline \delta_l).$
Consider \[R_i= \M_{n_i}(A)(\overline \delta_i), 1\leq i \leq l.\] 
Each $R_i$ is a graded finitely generated   projective $R$-module, so  $f([R_i])$ is in the positive cone of $K^{\gr}_0(S)$, \ie there is a graded finitely generated  projective $S$-module $P_i$ such that $f([R_i])=[P_i]$. Then 
\[[S]=f([R])=f([R_1]+\dots+[R_l])=[P_1]+\dots+[P_l]=[P_1\oplus\dots\oplus P_l].\] Since $S$ is a graded matricial algebra, by Lemma~\ref{fdpahf}, $P_1\oplus\dots\oplus P_l \cong_{\gr} S$ as right $S$-modules. So there are homogeneous orthogonal idempotents $g_1,\dots,g_l$ in $S$ such that $g_1+\dots+g_l=1$ and $g_iS\cong_{\gr} P_i$ (see \S\ref{idemptis}). 
Note that each of $R_i=\M_{n_i}(A)(\overline \delta_i)$ is a graded simple algebra. Set $\overline \delta_i=(\delta_1^{(i)},\dots,\delta_n^{(i)})$ (here $n=n_i$). Let $\e_{jk}^{(i)}$, $1\leq j,k\leq n$, be the graded matrix units of $R_i$ and consider 
the graded finitely generated   projective (right) $R_i$-module, \[V=\e_{11}^{(i)} R_i=A(\de_1^{(i)} - \de_1^{(i)}) \oplus  A(\de_2^{(i)}  - \de_1^{(i)}) \oplus \cdots \oplus
A(\de_n^{(i)} - \de_1^{(i)}).\] Then~(\ref{pjacko}) shows that \[R_i \cong_{\gr} V(\de_1^{(i)} - \de_1^{(i)})\oplus V(\de_1^{(i)}  - \de_2^{(i)}) \oplus \dots \oplus V(\de_1^{(i)} - \de_n^{(i)}),\] as graded $R$-module. Thus 
\begin{multline}\label{lkheohg}
[P_i]=[g_iS]=f([R_i])=f([V(\de_1^{(i)} - \de_1^{(i)})])+f([V(\de_1^{(i)} - \de_2^{(i)})])+\dots\\+f([V(\de_1^{(i)} - \de_n^{(i)})]).
\end{multline}
There is a graded finitely generated  projective $S$-module $Q$ such that \[f([V])=f([V(\de_1^{(i)} - \de_1^{(i)})])=[Q].\] Since $f$ is a $\mathbb Z[\Gamma]$-module homomorphism, for $1\leq k \leq n$, 
\begin{multline*}
f([V(\de_1^{(i)} - \de_k^{(i)})])=f((\de_1^{(i)} - \de_k^{(i)})[V])=(\de_1^{(i)} - \de_k^{(i)})f([V])=\\(\de_1^{(i)} - \de_k^{(i)})[Q]=[Q(\de_1^{(i)} - \de_k^{(i)})].
\end{multline*}

From ~(\ref{lkheohg}) and Lemma~\ref{fdpahf} now follows 
\begin{equation}\label{bvgdks}
g_iS\cong_{\gr} Q(\de_1^{(i)} - \de_1^{(i)})\oplus Q(\de_1^{(i)}  - \de_2^{(i)}) \oplus \dots \oplus Q(\de_1^{(i)} - \de_n^{(i)}).
\end{equation}
Let \[g^{(i)}_{jk}\in \End(g_iS)\cong_{\gr} g_iSg_i\]  map the $j$-th summand of the right hand side of~(\ref{bvgdks}) to its $k$-th summand and everything else to zero. Observe that $\deg(g^{(i)}_{jk})=\de^{(i)}_j-\de^{(i)}_k$ and $g^{(i)}_{jk}$, $1\leq j,k\leq n$, form  the matrix units. Moreover, $g^{(i)}_{11}+\dots+g^{(i)}_{nn}=g_i$ and \[g^{(i)}_{11}S=Q(\delta^{(i)}_1-\delta^{(i)}_1)=Q.\] Thus $[g^{(i)}_{11}S]=[Q]=f([V])=f([e_{11}^{(i)}R_i])$.

Now for any $1\leq i\leq l$, define the $A$-algebra homomorphism 
\begin{align*}
R_i&\longrightarrow g_iSg_i,\\
e_{jk}^{(i)} &\longmapsto g_{jk}^{(i)}.
\end{align*}
 This is a graded homomorphism, and induces a graded homomorphism $\phi:R\rightarrow S$ such that 
$\phi(\e_{jk}^{(i)})= g_{jk}^{(i)}$. Clearly \[K_0^{\gr}(\phi)([\e_{11}^{(i)}R_i])=[\phi(\e_{11}^{(i)})S]=[g_{11}^{(i)}S]=f([\e_{11}^{(i)}R_i]),\] for $1\leq i \leq l$. Now $K^{\gr}_0(R)$ is generated by $[\e_{11}^{(i)}R_i]$, $1\leq i \leq l$, as $\mathbb Z[\Gamma]$-module. This implies that $K^{\gr}_0(\phi)=f$. 
\end{proof}
 
\begin{remark}\label{helphgnr}
 In Theorem~\ref{catgrhsf}, both parts (1) and (2) are valid when the ring $S$ is a graded ultramatricial algebra as well. 
  In fact, in the proofs of (1) and (2), the only property of $S$ which is used is that if $[P]=[Q]$ in $K^{\gr}_0(S)$, then $P\cong_{\gr} Q$, where 
$P$ and $Q$ are  graded finitely generated projective $S$-modules. Lemma~\ref{fdpahf2} below shows that this is the case for ultramatricial algebras. 
\end{remark}

\begin{lemma}\label{fdpahf2}
Let $A$ be a $\Gamma$\!-graded field and $R$ be a $\Gamma$\!-graded ultramatricial $A$-algebra.  Let $P$ and $Q$ be graded finitely generated  projective $R$-modules. Then $[P]=[Q]$ in $K^{\gr}_0(R)$, if and only if  $P \cong _{\gr} Q$. 
\end{lemma}
\begin{proof}
We will use the description of $K^{\gr}_0$ based on homogeneous idempotents to prove this Lemma (see~\S\ref{hhidmi}). Let $p$ and $q$ be homogeneous idempotents matrices over $R$ corresponding to the graded finitely generated $R$-modules $P$ and $Q$, respectively (see Lemma~\ref{kzeiogr}(2)). Suppose $[P]=[Q]$ in $K^{\gr}_0(R)$. We will show that $p$ and $q$ are graded equivalent in $R$, which by Lemma~\ref{kzeiogr}(3), implies that $P\conggr Q$.
Since by Definition~\ref{kjhwal}, $R=\bigcup_{i\in I} R_i$, there is $j\in I$ such that $p$ and $q$ are homogeneous idempotent  matrices over $R_j$. 
But since $[p]=[q]$ in $K^{\gr}_0(R)$, there is an $n \in \mathbb N$, such that $p\oplus 1_n$ is graded equivalent to $q\oplus 1_n$ in $R$ 
(see~\S\ref{didood}). So there is $k\in I$, $k\geq j$, such that  $p\oplus 1_{n}$ and $q\oplus 1_{n}$ are graded equivalent in $R_k$. 
Thus $[p\oplus 1_{n}]= [q\oplus 1_{n}]$ in $K^{\gr}_0(R_k)$. This implies $[p]=[q]$  in $K^{\gr}_0(R_k)$. Since $R_k$ is a graded matricial algebra, by Lemma~\ref{fdpahf2}, $p$ is graded equivalent to $q$ in $R_k$. So $p$ is graded equivalent to $q$ in $R$ and consequently $P\cong_{\gr}Q$ as $R$-module. The converse is immediate. 
\end{proof}
 
\section{Graded ultramatricial algebras, classification via $K^{\gr}_0$}

 The direct limit of a direct system of graded rings is a ring with a graded structure (see Example~\ref{penrith123}). In this section, we study a particular case of such graded rings, namely the direct limit of graded matricial algebras. 
 Recall from Definition~\ref{ggthu} that a $\Gamma$\!-graded matricial algebra over the graded field $A$ is of the form 
 \[\M_{n_1}(A)(\overline \delta_1) \times \dots \times   \M_{n_l}(A)(\overline \delta_l),\] for some shifts  $\overline \delta_i$, $1\leq i \leq l$.

 \begin{definition}\label{kjhwal}
 Let $A$ be a $\Gamma$\!-graded field. A ring $R$ is called a $\Gamma$\!-\emph{graded ultramatricial} \index{graded ultramatricial algebra} $A$-algebra if $R=\bigcup_{i=1}^{\infty} R_i$, where $R_1 \subseteq R_2 \subseteq \cdots$ is a sequence of graded matricial $A$-subalgebras. 
 Here the inclusion respects the grading, \ie ${R_i}_{\alpha} \subseteq {R_{i+1}}_{\alpha}$. Moreover, under the inclusion $R_i \subseteq R_{i+1}$, we have $1_{R_i}=1_{R_{i+1}}$.
 \end{definition}
 
 Clearly $R$ is also a $\Gamma$\!-graded $A$-algebra with $R_\alpha= \bigcup_{i=1}^{\infty} {R_i}_{\alpha}$. If $\Gamma$ is a trivial group, then 
 Definition~\ref{kjhwal} reduces to the definition of ultramatricial algebras (\cite[\S15]{goodearlbook}). Note that if $R$ is a graded ultramatricial $A$-algebra, then $R_0$ is a ultramatricial $A_0$-algebra (see~\S\ref{wadi}).

 
 \begin{example}\label{hfoxrw}
Let $A$ be a ring. We identify $\M_n(A)$ as a subring of  $\M_{2n}(A)$ under the monomorphism 
\begin{equation}\label{fridayatmaxi}
X\in \M_n(A) \longmapsto \left(\begin{matrix} X & 0\\ 0 & X 
\end{matrix}\right) \in \M_{2n}(A).
\end{equation}
With this identification, we have a sequence \[\M_2(A) \subseteq \M_4(A) \subseteq \cdots.\] 

Now let $A$ be a $\Gamma$\!-graded field, $\alpha_1,\alpha_2 \in \Gamma$  and consider the sequence of graded subalgebras 
\[\M_2(A)(\alpha_1,\alpha_2) \subseteq \M_4(A)(\alpha_1,\alpha_2,\alpha_1,\alpha_2) \subseteq \cdots,\]
with the same embedding as (\ref{fridayatmaxi}).  
Then \[R=\bigcup_{i=1}^{\infty} \M_{2^i}(A)\big((\alpha_1,\alpha_2)^{2^{i-1}}\big),\] where $(\alpha_1,\alpha_2)^k$ stands for $k$ copies of $(\alpha_1,\alpha_2)$,  is a graded ultramatricial algebra.  
 \end{example}

\begin{example}
Let $A=K[x^3,x^{-3}]$ be a $\mathbb Z$-graded field with support $3\Z$, where $K$ is a field.  Consider the following sequence of graded matricial algebras with the embedding as in (\ref{fridayatmaxi}).
\begin{equation}\label{hgqq2}
\M_2(A)(0,1) \subseteq \M_4(A)(0,1,1,2) \subseteq \M_8(A)(0,1,1,2,0,1,1,2)\subseteq \cdots 
\end{equation}

Let $R$ be a graded ultramatricial algebra constructed as in Example~\ref{hfoxrw}, from the union of matricial algebras of sequence~\ref{hgqq2}.  We calculate $K^{\gr}_0(R)$. Since $K^{\gr}_0$ respects the direct limit (Theorem~\ref{kcontis}), we have 
\[K^{\gr}_0(R)=K^{\gr}_0(\varinjlim R_i)=\varinjlim K^{\gr}_0(R_i),\]
where $R_i$ corresponds to the $i$-th algebra in the sequence~\ref{hgqq2}. Since all the matricial algebras $R_i$ are strongly graded, we get 
\[\varinjlim K^{\gr}_0(R_i)=\varinjlim K_0(R_{i_0}).\]

Recall that (see Proposition~\ref{wadi}) if 
\begin{equation*}\label{ytsnf}
T\cong_{\gr} \M_{m}\big(K[x^n,x^{-n}] \big)\big (p_1,\dots,
p_m\big),
\end{equation*}
then letting $d_l$, $0 \leq l \leq n-1$, be the number of $i$ such that
$p_i$ represents $\overline l$ in $\mathbb Z/n \mathbb Z$, we have
\begin{equation*}
T \cong_{\gr} \M_{m}\big(K[x^n,x^{-n}] \big)\big (0,\dots,0,1,\dots,1,\dots,
n-1,\dots,n-1\big),
\end{equation*}
where $0\leq l \leq n-1$ occurring $d_l$ times. It is now easy to see
\begin{equation*}
T_0 \cong \M_{d_0}(K)\times \dots \times \M_{d_{n-1}}(K).
\end{equation*}

Using this, the 0-component rings of the sequence~\ref{hgqq2} takes the form  
\begin{align*}
K\oplus K &\subseteq K\oplus \M_2(K) \oplus K \subseteq \M_2(K) \oplus \M_4(K) \oplus \M_2(K) \\
(x,y) &\mapsto (x, \left(\begin{matrix} y & 0\\ 0 & x \end{matrix}\right),y)   \\
&\qquad \quad (x,y,z) \mapsto (\left(\begin{matrix} x & 0\\ 0 & x \end{matrix}\right),\left(\begin{matrix} y & 0\\ 0 & y \end{matrix}\right),\left(\begin{matrix} z & 0\\ 0 & z \end{matrix}\right)).
\end{align*}
Thus 
\[K^{\gr}_0(R)\cong \varinjlim K_0(R_{i_0})\cong \mathbb Z[1/2] \oplus \mathbb Z[1/2] \oplus \mathbb Z[1/2]. \]
\end{example} 

We are now in a position to classify the graded ultramatricial algebras via  $K^{\gr}_0$-group. The following theorem shows that the 
\emph{dimension module} \index{dimension module} is a complete invariant for the category of graded ultramatricial algebras. 
 
 \begin{theorem}\label{ujhosmr} \index{order preserving homomorphism}
 Let $R$ and $S$ be $\Gamma$\!-graded ultramatricial algebras over a graded field $A$. Then $R \conggr S$ as graded $A$-algebras if and only if there is an order preserving $\mathbb Z[\Gamma]$-module isomorphism
 \[\big (K^{\gr}_0(R),[R] \big ) \cong \big (K^{\gr}_0(S),[S]\big ). \]
 \end{theorem}
 \begin{proof}
 One direction is clear.  Let $R=\bigcup_{i=1}^{\infty} R_i$ and $S=\bigcup_{i=1}^{\infty} S_i$, where $R_1 \subseteq R_2 \subseteq \cdots$ and $S_1 \subseteq S_2 \subseteq \cdots$ are sequences of graded matricial $A$-algebras. Let $\phi_i:R_i\rightarrow R$ and $\psi_i:S_i\rightarrow S$, $i\in \mathbb N$, be inclusion maps. 
 In order to show that the isomorphism \[f:\big (K^{\gr}_0(R),[R] \big ) \longrightarrow \big (K^{\gr}_0(S),[S]\big )\]
 between the graded Grothendieck groups give rise to an isomorphism between the rings  $R$ and $S$, we will find a sequence $n_1<n_2<\cdots$ of positive numbers and graded $A$-module injections $\rho_k:R_{n_k}\rightarrow S$ such that $\rho_{k+1}$ is an extension of $\rho_k$ and 
 $\bigcup_{k=1}^{\infty} R_{n_k}=R$. To achieve this we repeatedly use  Theorem~\ref{catgrhsf} (and Remark~\ref{helphgnr}) (\ie a ``local'' version of this theorem) and the fact that since $R_n$, $n \in \mathbb N$, is a finite dimensional $A$-algebra, for any $A$-graded homomorphism $\rho:R_n \rightarrow S$, we have $\rho(R_n) \subseteq S_i$, for some positive number $i$. Throughout the proof, for simplicity, we write 
 $\overline \theta$ for the $\mathbb Z[\Gamma]$-homomorphism $K^{\gr}_0(\theta)$ induced by a graded $A$-algebra homomorphism $\theta:R\rightarrow S$.

 We first prove two auxiliary facts. 
 
 {\bf I.} If $\sigma:S_k\rightarrow R_n$ is a graded $A$-algebra homomorphism such that 
 \begin{equation}\label{pokji}
 \overline \phi_n \overline \sigma={\bar f}^{-1}\overline \psi_k,
 \end{equation}
 (see Diagram~\ref{thelosj})  then there exist an integer $j>k$ and a graded $A$-algebra homomorphism $\rho:R_n\rightarrow S_j$ such that $\psi_j\rho\sigma=\psi_k$ and $\overline \psi_j \overline \rho =\overline f \, \overline \phi_n$.  
 \begin{equation}\label{thelosj}
 {\def\labelstyle{\displaystyle}
\xymatrix{
S_k \ar[rr]^{\sigma} \ar[dd]_{\psi_k}&& R_n  \ar[dd]^{\phi_n} \ar@{.>}[ld]^{\rho}\\
& S_j \ar@{.>}[ld]^{\psi_j}\\
S  && R.
}}\qquad \quad
{\def\labelstyle{\displaystyle}
\xymatrix{
K^{\gr}_0(S_k) \ar[rr]^{\overline \sigma} \ar[dd]_{\overline \psi_k}&& K^{\gr}_0(R_n)  \ar[dd]^{\overline \phi_n} \ar@{.>}[ld]^{\overline \rho}\\
& K^{\gr}_0(S_j) \ar@{.>}[ld]^{\overline \psi_j}\\
K^{\gr}_0(S) \ar[rr]^{\bar f^{-1}}  && K^{\gr}_0(R).
}}
\end{equation}
 
\noindent {\it Proof of I.} Consider  $\overline f \, \overline \phi_n :K^{\gr}_0(R_n)  \rightarrow K^{\gr}_0(S)$. By Theorem~\ref{catgrhsf} (and Remark~\ref{helphgnr}), there is an $A$-graded homomorphism $\rho':R_n\rightarrow S$ such that $\overline {\rho'}= \overline f \, \overline \phi_n$. Since $R_n$ (a matricial $A$-algebra) is a finite dimensional $A$-algebra, $\rho'(R_n) \subseteq S_i$ for some $i$. Thus $\rho'$ gives a graded homomorphism $\rho'':R_n \rightarrow S_i$ such that $\psi_i \rho ''=\rho'$ (recall that $\psi_i$ is just an inclusion). Moreover 
 \begin{equation}\label{start2}
 \overline \psi_i  \overline {\rho''} =\overline {\rho'}= \overline f \, \overline \phi_n.
 \end{equation}
 Then, using~(\ref{pokji}), \[\overline \psi_i  \overline {\rho''}  \overline \sigma =\overline f \, \overline \phi_n  \overline \sigma =\overline \psi_k.\] Theorem~\ref{catgrhsf} implies that there is a graded inner automorphism $\theta$ of $S$ such that 
 \begin{equation}\label{kjanr}
  \theta \psi_i \rho '' \sigma=\psi_k.
  \end{equation}
  The restriction of $\theta$ on $S_i$ gives $\theta|_{S_i}:S_i \rightarrow S$. Since $S_i$ is finite dimensional $A$-algebra, it follows that there is $j$ such that $\theta(S_i) \subseteq S_j$. So $\theta$ gives a graded homomorphism $\theta':S_i \rightarrow S_j$ such that $\psi_j \theta'=\theta \psi_i$. 
  Set 
  \begin{equation}\label{satgabriel}
  \rho:=\theta' \rho'':R_n \rightarrow S_j.
  \end{equation}
   We get, using~(\ref{kjanr}) 
 \[\psi_j\rho \sigma=\psi_j\theta'\rho''\sigma=\theta\psi_i\rho''\sigma=\psi_k.\] This gives the first part of I. 
 Using Theorem~\ref{catgrhsf},~(\ref{start2}) and~(\ref{satgabriel}) we have 
 \[\overline \psi_j \overline \rho= \overline\psi_j\overline\theta'\overline\rho''=\overline\theta \, \overline\psi_i \overline\rho''=\overline\psi_i \overline\rho'' =\overline f \,\overline\phi_n.\] This completes the proof of I.

The second auxiliary fact we need is the to replace $R_i$'s and $S_i$'s in I as follows. The proof is similar to I. 

{\bf II.}  If $\rho:R_n\rightarrow S_k$ is a graded $A$-algebra such that 
 \[\overline \psi_k \overline \rho=\bar f\, \overline \phi_n.\] Then there exist an integer $m>n$ and a graded $A$-algebra homomorphism $\sigma:S_k\rightarrow R_m$ such that 
 $\phi_m\sigma\rho=\phi_n$ and $\overline \phi_m \overline \sigma ={\bar f}^{-1} \overline \psi_k$.

Now we are in a position to construct positive numbers $n_1<n_2<\cdots$ and graded $A$-algebra homomorphisms $\rho_k:R_{n_k}\rightarrow S$ such that 
\begin{enumerate}

\item $S_k \subseteq \rho_k(R_{n_k})$ and $\overline \rho_k= \bar f \, \overline \phi_{n_k}$, for all $k\in \mathbb N$.

\item $\rho_{k+1}$ is an extension of $\rho_{k}$ and for all $k\in \mathbb N$, i.e, the following diagram commutes.
\begin{equation*}
\xymatrix{
R_{n_k} \ar@{^{(}->}[d] \ar[r]^{\rho_k} & S \\
R_{n_{k+1}} \ar[ru]_{\rho_{k+1}}
}
\end{equation*}
Moreover,  $\rho_k$ is injective for all  $k\in \mathbb N$.
\end{enumerate}

Consider the morphism ${\bar f}^{-1} \overline \psi_1 : K^{\gr}_0(S_1) \rightarrow K^{\gr}_0(R)$. By Theorem~\ref{catgrhsf}, there is a graded $A$-algebra homomorphism  $\sigma':S_1 \rightarrow R$ such that $\overline \sigma'={\bar f}^{-1} \overline \psi_1$. Since $S_1$ is a finite dimensional $A$-algebra, $\sigma'(S_1) \subseteq R_{n_1}$ for some positive number $n_1$. So 
$\sigma'$ gives a graded $A$-algebra homomorphism $\sigma:S_1\rightarrow  R_{n_1}$ such that $\phi_{n_1}\sigma =\sigma'$ and $\overline \phi_{n_1} \overline \sigma ={\bar f}^{-1} \overline \psi_1$. 
Thus $\sigma$ satisfies the conditions of part I. Therefore there is $j>1$ and a graded $A$-algebra homomorphism $\rho:R_{n_1}\rightarrow S_j$ (see Diagram~\ref{hhhh}) such that $\psi_j\rho\sigma=\psi_1$ and $\overline \psi_j \overline \rho=\bar f \, \overline \phi_{n_1}.$
 \begin{equation}\label{hhhh}
 \xymatrix{
S_1 \ar[rr]^{\sigma} \ar[dd]_{\psi_1}&& R_{n_1}  \ar@{.>}[ld]^{\rho}\\
& S_j \ar@{.>}[ld]^{\psi_j}\\
S  
}
\end{equation}

So $\rho_1=\psi_j \rho:R_{n_1} \rightarrow S$ is a graded $A$-homomorphism such that $\rho_1 \sigma=\psi_1$ and $\overline \rho_1= \bar f \, \overline \phi_{n_1}$. 
But 
\begin{equation*}
S_1=\psi_1(S_1)=\rho_1\sigma(S_1)\subseteq \rho_1(R_{n_1}). 
\end{equation*}
This proves (1). We now proceed by induction. Suppose there are $\{n_1,\dots,n_k\}$ and $\{\rho_1,\dots,\rho_k\}$, for some positive integer $k$ such that (1) and (2) above are satisfied.  Since $R_{n_k}$ is a finite dimensional $A$-algebra, there is $i\geq k+1$ such that $\rho_k(R_{n_k}) \subseteq S_i$. So $\rho_k$ gives a graded $A$-homomorphism 
$\rho':R_{n_k}\rightarrow S_i$ such that $\psi_i \rho'=\rho_k$ and $\overline \psi_i \overline{\rho'} =\bar f \, \overline \phi_{n_k}$. By II, there is 
$n_{k+1}>n_k$ and $\sigma:S_i \rightarrow R_{n_{k+1}}$ such that 
$\phi_{n_{k+1}}\sigma \rho'=\phi_{n_k}$ and $\overline \phi_{n_{k+1}} \overline \sigma= {\bar f}^{-1} \overline \phi_i$. 
Since $\phi_{n_{k+1}}\sigma \rho'=\phi_{n_k}$, and $\phi_{n_k}$ (being an inclusion) is injective, $\rho'$ is injective and so $\rho_k=\psi_i\rho'$ is injective. Since
$\sigma:S_i \rightarrow R_{n_{k+1}}$ satisfies conditions of Part I,  there is $j>i$ such that $\rho: R_{n_{k+1}}\rightarrow S_j$ such that $\psi_j \rho \sigma =\psi_i$ and $\overline \psi_j \overline \rho =\bar  f \, \overline \phi_{n_{k+1}}$. Thus 
for $\rho_{k+1}=\psi_j \rho:R_{n_{k+1}} \rightarrow S$ we have  $\overline \rho_{k+1}=\overline \psi_j \overline \rho =\bar f \,\overline   \phi_{n_{k+1}}$. 
Since $\rho_{k+1}\sigma =\psi_j\rho \sigma=\psi_i$ and $i\geq k+1$, we have 
\[S_{k+1}=\psi_i(S_{k+1})=\rho_{k+1}\sigma(S_{k+1}) \subseteq \rho_{k+1}\sigma(S_i) \subseteq \rho_{k+1}(R_{n_{k+1}}).\] 
Finally, since $\phi_{n_{k+1}} \sigma \rho '=\phi_{n_k}$ and $\rho_{k+1}\sigma \rho'=\psi_i \rho'=\rho_k$, it follows that $\rho_{k+1}$ is an extension of $\rho_k$. Then by induction (1) and (2) follows. 

We have $n_1<n_2<\cdots$ and $k\leq n_k$. So $\bigcup R_{n_k}=R$. So $\rho_k$ induces an injection $\rho:R\rightarrow S$ such that $\rho\phi_{n_k}=\rho_k$, for any $k$. 
But 
\[S_k \subseteq \rho_k(R_{n_k}) =\rho (R_{n_k}) \subseteq \rho (R).\] It follows that $\rho$ is an epimorphism as well. This finishes the proof. 
 \end{proof}

\begin{theorem}\label{meinchina} \index{graded Morita equivalent} 
Let $R$ and $S$ be $\Gamma$\!-graded ultramatricial algebras over a graded field $A$.
Then $R$ and $S$ are graded Morita equivalent if and only if there is an order preserving $\mathbb Z[\Gamma]$-module isomorphism 
$K^{\gr}_0(R) \cong K^{\gr}_0(S)$.
\end{theorem}
\begin{proof}
Let $R$ be graded Morita equivalent to $S$. Then by Theorem~\ref{grmorim11}, there is a $R$-graded progenerator $P$, such that 
$S\cong_{\gr} \End_R(P)$. Now using Lemma \ref{nnbok}, we have 
\[\big(K^{\gr}_0(S),[S]\big ) \cong \big(K^{\gr}_0(R),[P]\big ).\] In particular $K^{\gr}_0(R) \cong K^{\gr}_0(S)$ as ordered $\mathbb Z[\Gamma]$-modules.   (Note that the ultramatricial assumptions on rings are not used in this direction.)

For the converse, suppose $K^{\gr}_0(S)\cong K^{\gr}_0(R)$ as ordered $\mathbb Z[\Gamma]$-modules. Denote the image of $[S]\in K^{\gr}_0(S)$ under this isomorphism by $[P]$. Then   
\[\big(K^{\gr}_0(S),[S]\big )  \cong \big (K^{\gr}_0(R),[P]\big ).\] 
Since $[P]$ is an order-unit (see~(\ref{meinchnhfr})), there is $\overline \alpha=(\alpha_1,\dots,\alpha_n)$, $\alpha_i \in \Gamma$, such that 
\[\sum_{i=1}^n \alpha_i [P]=[P^n(\overline \alpha)]\geq [R].\] This means that there is a graded finitely generated projective  $A$-module, $Q$ such that 
$[P^n(\overline \alpha)]=[R\oplus Q]$. 
Since $R$ is a graded ultramatricial algebra, by Lemma~\ref{fdpahf2}, $P^n(\overline \alpha)\conggr R\oplus Q$. Using Theorem~\ref{grgenth}, it follows that $P$ is a graded progenerator. Let $T=\End_R(P)$. Using Lemma~\ref{nnbok}, we get an order preserving $\mathbb Z[\Gamma]$-module isomorphism 
\[\big(K^{\gr}_0(T),[T]\big ) \cong \big(K^{\gr}_0(R),[P]\big ) \cong \big(K^{\gr}_0(S),[S]\big ).\] 
Observe that $T$ is also a graded ultramatricial algebras. By the main Theorem~\ref{ujhosmr}, $\End_R(P) =T \conggr S$. By Theorem~\ref{grmorim11}, this implies that $R$ and $S$ are graded Morita equivalent. 
\end{proof}

Considering a trivial group as the grade group,  we retrieve the classical Bratteli-Elliott Theorem (see the introduction of this chapter). 

\begin{corollary}[The Bratteli-Elliott Theorem]\label{marriotcitywall}

Let $R$ and $S$ be ultramatricial algebras over a field $F$.
Then $R$ and $S$ are Morita equivalent if and only if there is an order preserving isomorphism 
$K_0(R) \cong K_0(S)$. Moreover, $R \cong S$ as $F$-algebras if and only if there is an order preserving isomorphism
 \[\big (K_0(R),[R] \big ) \cong \big (K_0(S),[S]\big ). \]
 \end{corollary}
\begin{proof}
The corollary follows by considering $\Gamma$  to be a trivial group in Theorems~\ref{ujhosmr} and~\ref{meinchina}. 
\end{proof}

\begin{remark}\label{idontcare}
Theorem~\ref{ujhosmr} can also been extend to the non-unital rings. If graded rings $R$ and $S$ are not unital, then instead of the order units in the Theorem~\ref{ujhosmr}, one considers the generating interval \[\big \{\, 0 \leq x \leq [\tilde R] \mid x\in K^{\gr}_0(R) \, \big \},\] where $[\tilde R]$ is the order unit in the group $K^{\gr}_0(\tilde R)$. Then the order preserving map $f: K^{\gr}_0(R) \rightarrow K^{\gr}_0(S)$ should be \emph{contractive}, \ie preserving the generating intervals (see~\cite[Chapter XII]{goodhand}) for the ungraded version of this theorem). \index{contractive}

As a demonstration, in Example~\ref{gor381}, there is an order preserving $\Z[x,x^{-1}]$-module isomorphism $K^{\gr}_0(R) \cong K^{\gr}_0(S)$. However, one calculates that the generating interval for $K^{\gr}_0(R)$ is $\big \{\sum_{i \in \mathbb N} x^i \big\}$, whereas for $S$ the interval is 
$\big \{\delta+nx \mid \delta \in \{0,1\}, n \in \mathbb N\big \}$. Thus there is no contractive order preserving isomorphism between $K^{\gr}_0(R)$ and $K^{\gr}_0(S)$. Therefore by the non-unital version of Theorem~\ref{ujhosmr}, $\LL(F)$ and $\LL(E)$ are not graded isomorphic. 

\end{remark}


\chapter[Graded Versus Ungraded $K$-Theory]{Graded Versus Ungraded\\ (Higher) $K$-Theory}\label{waraya}

Recall that for a $\Gamma$\!-graded ring $A$,  the category of  finitely generated
$\Gamma$\!-graded  projective right $A$-modules is denoted by $\Pgr[\Gamma] A$. This 
is an exact category with the usual notion of (split) short exact
sequence. Thus one can apply Quillen's $Q$-construction~\cite{quillen} to obtain $K$-groups
\[K_i(\Pgr[\Gamma] A),\] for $i\geq 0$, which we denote by $K_i^{\gr}(A)$. If more than one grading is involve (which is the case in this chapter) we denote it  
by $K_i^{\Gamma}(A)$. 
 The group $\Gamma$ acts on the category
$\Pgr[\Gamma] A$ from the right via $(P,\alpha) \mapsto P(\alpha)$. By functoriality of $K$-groups this gives $K_i^{\gr}(A)$ the structure of a right $\mathbb Z[\Gamma]$\nbd-module.

The relation between graded $K$\nbd-groups and ungraded
$K$\nbd-groups is not always apparent. For example consider the
$\mathbb Z$\nbd-graded matrix ring \[A=\mathbb{M}_5(K)(0,1,2,2,3),\] 
where $K$ is a field. Using
graded Morita theory one can show that
$K_0^{\mathbb Z}(A)\cong \mathbb Z[x,x^{-1}]$ (see Example~\ref{upst}), whereas  $K_0(A) \cong \mathbb{Z}$ and 
$K_0(A_0)\cong \mathbb Z \times \mathbb Z \times \mathbb Z\times
\mathbb Z$ (see Example~\ref{fgfdgs}).  

In this chapter we describe a theorem due to van den Bergh~\cite{vandenbergh}  relating the graded $K$-theory to ungraded $K$-theory. For this we need a generalisation of a result of Quillen on relating $K^{\gr}_i$-groups of a positively graded rings to their  ungraded $K_i$-groups (\S\ref{todaytalkbit}). We also need to recall a proof of the fundamental theorem of $K$-theory, \ie for a regular Noetherian ring $A$, $K_i(A[x])\cong K_i(A)$ (\S\ref{shitnome}). We will see that in the proof of this theorem, graded $K$-theory appears in a crucial way. Putting all these together we can relate the graded $K$-theory of regular Noetherian $\mathbb Z$-graded rings to their ungraded $K$-theory (\S\ref{vandengg}). 

As a conseqeunce, we will see that if $A$ is a $\mathbb Z$-graded right regular Noetherian ring, then one has an exact sequence (Corollary~\ref{ght9laks}),
\[ 
\begin{split}
\xymatrix{
K^{\gr}_0(A) \ar[rr]^{[P]\mapsto [P(1)]-[P]} && K^{\gr}_0(A) \ar[r]^{U} &  K_0(A) \ar[r] & 0,
}
\end{split}
\] where $U$ is induced by the forgetful functor (\S\ref{forgetful}). 
This shows that 
\[K^{\gr}_0(A) /\langle [P]-[P(1)] \rangle \cong K_0(A),\]
where $P$ is a graded finitely generated  projective $A$-module. 
Further the action of $\mathbb Z[x,x^{-1}]$ on the quotient group is trivial. This means, as soon as we discard the shift in $K^{\gr}_0$, this group reduces to the usual $K_0$-group. 

However this is not always the case. For example consider the group ring $\mathbb Z[G]$, where $G$ is an abelian group. 
Since $\mathbb Z[G]$ has IBN, the canonical homomorphism $\theta: K_0(\Z) \rightarrow K_0(\mathbb Z[G])$ is injective (see~\S\ref{pearlf}). The augmentation map $\Z[G]\rightarrow \Z, g\mapsto 1$, induces $\vartheta: K_0(\Z[G])\rightarrow K_0(\Z)$, such that $\vartheta \theta=1$. Thus 
\[K_0(\Z[G]) \cong \Z \bigoplus \widetilde{K_0}(\Z[G]). \]
Now since $\Z[G]$ is a group ring (so a crossed product), \[K^{\gr}_0(\Z[G])\cong K_0(\Z)\cong \Z\] and the action of $\Z[x,x^{-1}]$ on the $K^{\gr}_0(\Z[G])$ is trivial (see Example~\ref{gptwnow}). So \[K^{\gr}_0(\Z[G]) / \langle [P]-[P(1)] \rangle \cong \Z.\] This shows that $K^{\gr}_0(\Z[G])$ does not shed any light on the group $K_0(\Z[G])$.
The Grothendieck group of group rings is not easy to compute (see for example~\cite{bassmurthy} and~\cite[Example~2.4]{weibelk}). 

In this chapter we  compare the graded $K$-theory to its ungraded counterpart.  We start with a generalisation of Quillen's theorem.

\section{$K^{\gr}_*$ of positively graded rings}\label{todaytalkbit}

For a $\mathbb Z$-graded ring $A$ with support in $\mathbb N$, in his seminal paper~\cite{quillen},  Quillen calculated the graded $K$-theory of $A$ in terms of $K$-theory of its zero homogeneous component ring. \index{positively graded ring}

\begin{theorem}{\sc \cite[\S3, Proposition]{quillen}} \label{quillen11}
Let $A$ be a $\mathbb Z$-graded ring with support in $\mathbb N$. Then for $i\geq 0$, there is a $\mathbb  Z[t,t^{-1}]$-module isomorphism, 
\begin{equation}\label{lusttti}
K^{\gr}_i(A)\cong K_i(A_0)\otimes_{\mathbb Z} \mathbb Z[t,t^{-1}].
\end{equation}
Moreover, the elements of $K_i(A_0) \otimes t^r$ maps to $K^{\gr}_i(A)$ by $\ol f_r$, where $f_r$ is the exact functor 
\begin{align}\label{macitywall}
f_r:\Prr A_0 &\longrightarrow \Pgrp A, \\
P &\longmapsto (P\otimes_{A_0}A)(-r) \notag. 
\end{align}
\end{theorem}

In Theorem~\ref{thm:main}, we prove a generalised version of this proposition. Theorem~\ref{quillen11}  was used in an essential way to prove the fundamental theorem of $K$-theory (see~\S\ref{bolgona}). In particular, Theorem~\ref{quillen11}  gives a powerful tool to calculate the graded Grothendieck group of positively graded rings. We demonstrate this by calculating $K^{\gr}_0$ of path algebras.  \index{path algebra}

\begin{theorem}\label{hgrsiny}
Let $E$ be a finite graph and $\PP_K(E)$ be the path algebra with coefficients in the field $K$. Then 
\[K^{\gr}_0(\PP_K(E))\cong \bigoplus_{|E^0|}\mathbb Z[t,t^{-1}]. \]
\end{theorem}
\begin{proof}
Recall from~\S\ref{paohdme}, that $\PP_K(E)$ is a graded ring with support $\mathbb N$. Moreover $\PP_K(E)_0=\bigoplus_{|E^0|} K$. Then by~(\ref{lusttti})
\begin{multline*}
K^{\gr}_0(\PP_K(E))\cong K_0(\bigoplus_{|E^0|} K)\otimes_{\mathbb Z} \mathbb Z[t,t^{-1}]\cong  \\\bigoplus_{|E^0|} \mathbb Z \otimes_{\mathbb Z} \mathbb Z[t,t^{-1}]\cong \bigoplus_{|E^0|}\mathbb Z[t,t^{-1}].\qedhere
\end{multline*} 
\end{proof}

The proof of Proposition~\ref{hgrsiny} shows that, in the case of $i=0$, $[A]\in K^{\gr}_i(A)$ is sent to $[A_0]\otimes 1 \in K_i(A_0)\otimes_{\mathbb Z} \mathbb Z[t,t^{-1}]$ under the isomorphism~(\ref{lusttti}). Thus for the graphs 
\medskip
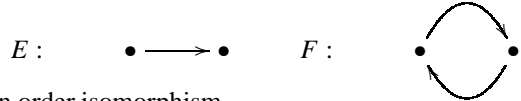
\begin{equation*}
{\def\labelstyle{\displaystyle}
\xymatrix{ 
E: &  \bullet \ar[r] & \bullet  
}}
{\def\labelstyle{\displaystyle}
\qquad 
\xymatrix{ 
F:  &\bullet \ar@/^1.5pc/[r] & \bullet \ar@/^1.5pc/[l] &
}}
\end{equation*}
We have an order isomorphism 
\begin{multline*}
\big( K^{\gr}_0(\PP_K(E),[\PP_K(E)]\big) \cong \big( \mathbb Z[t,t^{-1}] \oplus \mathbb Z[t,t^{-1}], (1,1)\big)\cong \\ \big( K^{\gr}_0(\PP_K(E),[\PP_K(E)]\big).
\end{multline*}
This shows in particular that the ordered group $K^{\gr}_0$ in itself does not classify the path algebras. 

In contrast to other fundamental theorems in the subject, such as the fundamental theorem of $K$-theory (\ie $K_n(R[x,x^{-1}])=K_n(R)\times K_{n-1}(R)$, for  a regular ring $R$), one cannot use an easy induction on~(\ref{lusttti}) to write a similar statement for multivariables rings. For example, it appears that there is no obvious inductive approach to
generalise~(\ref{lusttti}) to $\mathbb Z^m \times G$\nbd-graded
rings. However, by generalising Quillen's argument to take
account of gradings on both sides of the isomorphism~(\ref{lusttti}), such a
procedure becomes feasible. The details have been worked out in~\cite{hazhu} and we present it here.

We will prove the following statement.

\begin{theorem}\label{thm:main}
  Let $G$ be an arbitrary group, and let $A$ be a ${\mathbb Z} \times
  G$\nbd-graded ring with support in ${\mathbb N} \times G$. Then
  there is a $\mathbb Z[{\mathbb Z}\times G]$-module isomorphism
  \[K_i^{{\mathbb Z} \times G}(A) \cong K_i^{G}(A_{(0,G)}) \otimes_{\mathbb Z[G]} \mathbb Z[{\mathbb Z}\times G],\] 
  where $A_{(0,G)}=\bigoplus_{g\in G} A_{(0,g)}$.
\end{theorem}

By a straightforward induction this now implies:

\begin{corollary}\label{jhyhgtewqs}
  For a $\mathbb Z^m \times G$\nbd-graded ring~$A$ with support in
  $\mathbb N^m \times G$ there is a $\mathbb Z[t_1^{\pm1}, \cdots, t_m^{\pm1}]$-module isomorphism
  \[K_i^{\mathbb
    Z^m \times G}(A) \cong K_i^G(A_{(0,G)}) \otimes_{\mathbb Z} \mathbb Z[t_1^{\pm1}, \cdots, t_m^{\pm1}].\]
\end{corollary}
For the trivial group~$G$ this is a direct generalisation of
Quillen's theorem to $\mathbb Z^m$\nbd-graded rings.

As in Quillen's calculation the proof of the Theorem~\ref{thm:main} is based
on a version of Swan's Theorem, modified to the present
situation: it provides a correspondence between isomorphism classes of
$\mathbb Z \times G$\nbd-graded finitely generated  projective
$A$\nbd-modules and of $G$\nbd-graded finitely generated  projective
$A_{(0,G)}$-modules.

\begin{proposition}
  \label{swanii} Let $\Ga$ be a (possibly non-abelian)
  group. Let $A$ be a $\Ga$\!-graded ring, $A_0$ a graded subring
  of~$A$ and $\pi \colon A\rightarrow A_0$ a graded ring homomorphism
  such that $\pi |_{A_0}=1$. (In other words, $A_0$~is a retract
  of~$A$ in the category of $\Ga$\nbd-graded rings.) We denote the
  kernel of~$\pi$ by~$A_+$.

  Suppose that for any graded finitely generated  right $A$\nbd-module
  $M$ the condition $MA_+=M$ implies $M=0$. Then the natural
  functor
  \[S = - \otimes_{A_0} A \colon \Pgr[\Ga] A_0  \longrightarrow
  \Pgr[\Gamma] A\] induces a bijective correspondence between the
  isomorphism classes of graded finitely generated  projective
  $A_0$\nbd-modules and of graded finitely generated  projective
  $A$\nbd-modules. An inverse of the bijection is given by the functor
  \[T = -\otimes_A A_0 \colon \Pgr[\Ga] A \longrightarrow
  \Pgr[\Ga] A_0.\] There is a natural isomorphism $T \circ S
  \cong \mathrm{id}$, and for each $P \in \Pgr[\Ga] A$ a
  noncanonical isomorphism $S \circ T(P) \cong P$. The latter is
  given by
  \begin{equation}
    \label{eq:iso_lemma}
    T(P) \otimes_{A_0} A \longrightarrow P, \quad x \otimes a \longmapsto
    g(x) \cdot a,
  \end{equation}
  where $g$~is an $A_0$\nbd-linear section of the epimorphism $P
  \rightarrow T(P)$.
\end{proposition}

\begin{proof}
  For any graded finitely generated  projective $A_0$\nbd-module~$Q$ we
  have a natural isomorphism $TS(Q) \cong Q$ given by
  \begin{equation}
    \label{eq:nuQ}
    \nu_Q \colon TS(Q) = Q \otimes_{A_0} A \otimes_A {A_0} \longrightarrow
    Q, \quad q \otimes a \otimes a_0 \mapsto q \pi(a)a_0.
  \end{equation}
  We will show that for a graded projective $A$\nbd-module~$P$
  there is a noncanonical graded isomorphism $ST(P)\cong_{\gr}
  P$. The lemma then follows.

  Consider the natural graded $A$\nbd-module epimorphism
  \[f \colon P\longrightarrow T(P)=P\otimes_{A} A_0, \quad p \longmapsto p
  \otimes 1.\] Here $T(P)$ is considered as an $A$\nbd-module via
  the map~$\pi$. Since $T(P)$ is a graded projective $A_0$\nbd-module,
  the map $f$~has a graded $A_0$\nbd-linear section $g \colon T(P)
  \rightarrow P$. This section determines an $A$\nbd-linear graded map
  \[\psi \colon ST(P) = P \otimes_A A_0 \otimes_{A_0} A \longrightarrow P, \quad p \otimes a_0 \otimes a \longmapsto g(p \otimes a_0) \cdot a, \] and we will show that $\psi$ is an isomorphism. First note
  that the map
  \[T(f) \colon T(P) \longrightarrow TT(P), \quad p \otimes a_0 \longmapsto
  f(p) \otimes a_0 = p \otimes 1 \otimes a_0\] is an isomorphism
  (here we consider $T(P)$~as an $A$\nbd-module via~$\pi$). In fact the
  inverse is given by the isomorphism $TT(P) = P \otimes_A A_0 \otimes_A
  A_0 \rightarrow P \otimes_A A_0$ which maps $p \otimes a_0 \otimes b_0$ to
  $p \otimes(a_0 b_0)$.  Tracing the definitions now shows that both
  composites
  
  \[
  \xymatrix{ 
  TST(P) \ar@<0.75ex>[r]^{T(\psi)} \ar@<-0.75ex>[r]_{\nu_{T(P)}} & T(P) \ar[r]^{T(f)}_{\cong}   &TT(P)
  }
  \]

  map $p \otimes a_0 \otimes a \otimes b_0
  \in P \otimes_A A_0 \otimes_{A_0} A \otimes_A A_0 = TST(P)$ to the
  element \[f \big( g(p \otimes a_0) \cdot a \big) \otimes b_0 = p
  \otimes (a_0 \pi(a) b_0) \otimes 1 \in TT(P).\] This implies that
  $T(\psi) = \nu_{T(P)}$, which is an isomorphism.

  The exact sequence
  \begin{equation}
    \label{gfds} 0 \longrightarrow \ker \psi \longrightarrow ST(P)
    \overset{\psi}{\longrightarrow} P \longrightarrow \coker \psi \longrightarrow
    0
  \end{equation}
  gives rise, upon application of the right exact functor~$T$, to an
  exact sequence
  \[TST(P) \overset{T(\psi)}{\longrightarrow} T(P) \longrightarrow T(\coker
  \psi) \longrightarrow 0.\] Since $T(\psi)$ is an isomorphism we have
  $T(\coker \psi) = \coker T(\psi) = 0$. Since $\coker \psi$ is
  finitely generated by~\eqref{gfds} this implies $\coker \psi = 0$
  (note that $T(M) = MA_+$ for every finitely generated
  module~$M$). In other words, $\psi$~is surjective and
  \eqref{gfds}~becomes the short exact sequence \[0 \rightarrow \ker
  \psi \rightarrow ST(P) \overset{\psi}{\rightarrow} P \rightarrow
  0.\] This latter sequence splits since $P$~is projective; this
  immediately implies that $\ker \psi$ is finitely generated, and
  since $T(\psi)$ is injective we also have $T(\ker \psi) = \ker
  T(\psi) = 0$. The hypotheses guarantee $\ker \psi = 0$ now so that
  $\psi$~is injective as well as surjective, and thus is an
  isomorphism as claimed.
\end{proof}


\begin{lemma}\label{dido13}
  \label{hggf}
  Let $G$ and~$\Ga$ be groups, and let $A$ be a $G$\nbd-graded
  ring. Then, considering $A$ as a $\Ga \times G$\nbd-graded ring in a
  trivial way where necessary, the functorial assignment $(M, \gamma)
  \mapsto M(\gamma,0)$ induces a $\mathbb Z[\Ga\times G]$\nbd-module
  isomorphism
  \[K_i^{\Ga \times G}(A) \cong K_i^{G}(A) \otimes_{\mathbb Z[G]}
  \mathbb Z[\Ga \times G].\]
\end{lemma}

\begin{proof}
  Let $P = \bigoplus_{(\gamma,g) \in \Ga \times G} P_{(\gamma,g)}$ be
  a $\Gamma\times G$\nbd-graded finitely generated  projective
  $A$\nbd-module. Since the support of $A$ is $G = 1 \times G$, there
  is a unique decomposition $P=\bigoplus_{\ga \in \Ga} P_\ga$, where
  the $P_\ga = \bigoplus_{g \in G} P_{(\gamma,g)}$ are finitely
  generated $G$\nbd-graded projective  $A$\nbd-modules. This gives a
  natural isomorphism of categories (see Corollary~\ref{zuhoingding} and Remark~\ref{enjoythemomentt})
  \[\Psi \colon \Pgr[\Gamma \times G] A
  \overset{\cong}{\longrightarrow} \bigoplus_{\gamma \in \Gamma}
  \Pgr[G] A.\]
  The natural right action of $\Ga \times G$ on these categories is
  described as follows: for a given module $P \in \Pgr[\Ga \times G]
  A$ as above and elements $(\gamma,g) \in \Ga \times G$ we have
  \[P(\gamma,g)_{(\delta,h)} = P_{(\gamma+\delta,g+h)} \text{ and }
  \Psi(P)_\delta = P_{\gamma + \delta}(g)\] so that $\Psi\big(
  P(\gamma,g) \big) = \Psi(P)(\gamma,g)$. Since $K$-groups respect
  direct sums we thus have a chain of $\mathbb Z[\Ga \times
  G]$\nbd-linear isomorphisms
  \begin{multline*}
    K_i^{\Ga\times G}(A) = K_i(\Pgr[\Gamma \times G] A )\cong
    K_i\big(\bigoplus_{\gamma \in \Gamma} \Pgr[G] A\big) \\
    = \bigoplus_{\ga \in \Ga} K_i^G(A) \cong K_i^G(A) \otimes_{\mathbb
      Z} \mathbb Z[\Ga] = K_i^G(A) \otimes_{\mathbb
      Z[G]} \mathbb Z[\Ga \times G].\qedhere
  \end{multline*}
\end{proof}

We are in a position to prove  Theorem~\ref{thm:main}. 

\begin{proof} [Proof of Theorem~\ref{thm:main}]

Let $A$ be a ${\mathbb Z} \times G$\nbd-graded ring with
support in ${\mathbb N} \times G$. That is, $A$~comes equipped with a
decomposition
\[A=\bigoplus_{\omega\in{\mathbb N}} A_{(\omega,G)} \quad \text{where}
\quad A_{(\omega,G)} = \bigoplus_{g \in G} A_{(\omega,g)}.\] The
ring~$A$ has a ${\mathbb Z} \times G$ graded subring~$A_{(0,G)}$ (with
trivial grading in ${\mathbb Z}$\nbd-direction). The projection map $A
\rightarrow A_{(0,G)}$ is a $\mathbb Z\times G$-graded  ring homomorphism; its
kernel is denoted~$A_+$. Explicitly, $A_+$~is the two-sided ideal
\[A_+=\bigoplus_{\omega > 0} A_{(\omega,G)}.\] We
identify the quotient ring~$A/A_+$ with the subring~$A_{(0,G)}$ via
the projection.

If $P$ is a graded finitely generated  projective $A$\nbd-module, then
$P\otimes_A A_{(0,G)}$ is a finitely generated ${\mathbb Z}\times
G$\nbd-graded projective $A_{(0,G)}$\nbd-module. Similarly, if $Q$ is a
graded finitely generated  projective $A_{(0,G)}$\nbd-module then $Q
\otimes_{A_{(0,G)}} A$ is a $\mathbb Z\times G$- graded finitely generated  projective
$A$\nbd-module. We can thus define functors
\begin{align*}
  T & =-\otimes_A A_{(0,G)} \colon \Pgr[\mathbb Z \times G] A
  \longrightarrow \Pgr[\mathbb Z \times G] A_{(0,G)}, \\
  \text{and} \quad S &= -\otimes_{A_{(0,G)}}A \colon \Pgr[\mathbb Z \times
    G] A_{(0,G)} \longrightarrow \Pgr[\mathbb Z \times G]A.
\end{align*}
Since $T(P) = P/PA_+$ we see that \emph{the support of~$T(P)$ is
  contained in the support of~$P$}.

Observe now that \emph{if $M$ is a finitely generated ${\mathbb Z}\times
  G$-graded $A$-module and $MA_+=M$ then $M=0$}; for if $M \neq 0$
there is a minimal $\omega \in {\mathbb Z}$ such that $M_{(\omega,G)} \neq
0$, but $(MA_+)_{(\omega,G)} = 0$.  It follows from
Proposition~\ref{swanii} that for each graded finitely generated
projective $A$\nbd-module~$P$ there is a noncanonical isomorphism $P
\cong T(P) \otimes_{A_{(0,G)}} A$ as in \eqref{eq:iso_lemma} which
respects the ${\mathbb Z} \times G$\nbd-grading. Explicitly, for a given
$(\omega,g) \in {\mathbb Z} \times G$ we have an isomorphism of
abelian groups
\begin{equation}
  \label{eq:noncan_iso}
  P_{(\omega,g)} \cong \bigoplus_{(\kappa,h)} T(P)_{(\kappa,h)}
  \otimes A_{(-\kappa+\omega,-h+g)};
\end{equation}
the tensor product $T(P)_{(\kappa,h)} \otimes
A_{(-\kappa+\omega,-h+g)}$ denotes, by convention, the
abelian subgroup of $T(P) \otimes_{A_{(0,G)}} A$ generated by
primitive tensors of the form $x \otimes y$ with homogeneous elements
$x \in T(P)$ of degree~$(\kappa,h)$ and $y \in A$ of
degree~$(-\kappa+\omega,-h+g)$.

For a ${\mathbb Z} \times G$\nbd-graded $A$\nbd-module~$P$ write $P =
\bigoplus_{\omega \in {\mathbb Z}} P_{(\omega,G)}$, where $P_{(\omega,G)} =
\bigoplus_{g \in G} P_{(\omega,g)}$. For $\lambda \in {\mathbb Z}$ let
$F^\lambda P$ denote the $A$\nbd-submodule of~$P$ generated by the
elements of $\bigcup_{\omega \leq \lambda} P_{(\omega,G)}$; this is
${\mathbb Z} \times G$\nbd-graded again. As an explicit example, we have
\[F^\lambda A(\omega,g) = \begin{cases} A(\omega,g) & \text{if }
  \lambda \geq -\omega, \\ 0 & \text{else.} \end{cases}\] 

Suppose that $P$ is a graded finitely generated  projective
$A$-module. Since the support of~$A$ is contained in ${\mathbb N}
\times G$ there exists $n \in {\mathbb Z}$ such that $F^{-n}P = 0$ and
$F^{n}P = P$. Write $\nPgr[\mathbb Z \times G] A$ for the full
subcategory of $\Pgr[\mathbb Z \times G] A$ spanned by those
modules~$P$ which satisfy $F^{-n}P = 0$ and $F^{n}P = P$. Then
$\Pgr[\mathbb Z \times G] A$ is the filtered union of the
$\nPgr[\mathbb Z \times G]A$.

Let $P \in \nPgr[\mathbb Z \times G] A$; we want to identify
$F^\lambda P$. By definition, the $A$\nbd-module $F^\lambda P$ is
generated by the elements of~$P_{(\omega, g)}$ for $\omega \leq
\lambda$, with $P_{(\omega,g)}$ having been identified
in~\eqref{eq:noncan_iso}. We remark that the direct summands
in~\eqref{eq:noncan_iso} indexed by $\kappa > \omega$ are trivial as
$A$~has support in ${\mathbb N} \times G$. On the other hand, for
$\omega \geq \kappa$ a given primitive tensor $x \otimes y \in
P_{(\omega,g)}$ with $x \in T(P)_{(\kappa,h)}$ and $y \in
A_{(-\kappa+\omega, -h+g)}$ can always be re-written, using the right
$A$\nbd-module structure of $T(P) \otimes_{A_{(0,G)}} A$, as
\[x \otimes y = (x \otimes 1) \cdot y, \quad \text{where } \quad  x \otimes 1
\in T(P)_{(\kappa,h)} \otimes A_{(0,0)} \subseteq P_{(\kappa, h)}.\] That is, the $A$\nbd-module $F^\lambda P$ is generated by those
summands of \eqref{eq:noncan_iso} with $\kappa = \omega \leq
\lambda$. We claim now that $F^\lambda P$ is isomorphic to
\begin{equation}
  \label{eq:M}
  M^{(\lambda)} = \bigoplus_{\kappa \leq \lambda} T(P)_{(\kappa,-)}
  \otimes_{A_{(0,G)}} A(-\kappa, 0),
\end{equation}
considering $T(P)_{(\kappa,-)}$ as a ${\mathbb Z} \times G$\nbd-graded
$A_{(0,G)}$\nbd-module with support in $\{0\}\times G$.  The
homogeneous components of~$M^{(\lambda)}$ are given by
\[M^{(\lambda)}_{(\omega,g)} = \bigoplus_{\kappa \leq \lambda}
\bigoplus_{h \in G} T(P)_{(\kappa,h)} \otimes
A(-\kappa,0)_{(\omega,-h+g)}.\] Now elements of the form
\[x \otimes 1 \in T(P)_{(\kappa,h)} \otimes
A(-\kappa,0)_{(\kappa,-h+h)} \subseteq M^{(\lambda)}_{(\kappa,h)}\]
clearly form a set of $A$\nbd-module generators for~$M^{(\lambda)}$ so
that, by the argument given above, $F^\lambda P$ and~$M^{(\lambda)}$
have the same generators in the same degrees. The claim follows. 
The module $F^\lambda P$ is finitely generated (\emph{viz.}, by those
generators of~$P$ that have ${\mathbb Z}$\nbd-degree at
most~$\lambda$). Since $T(P)$~is a finitely generated projective
$A_{(0,G)}$\nbd-module so is its summand $T(P)_{(\lambda_k,G)}$;
consequently, $P \mapsto F^kP$~is an endofunctor of $\nPgr[\mathbb Z
  \times G]  A$. It is exact as can be deduced from the
(noncanonical) isomorphism in~\eqref{eq:M}, using exactness of tensor
products.

From the isomorphism $F^k P \cong M^{(\lambda_k)}$, cf.~\eqref{eq:M},
we obtain an isomorphism
\begin{equation}
  \label{eq:Q}
  F^{k+1}P / F^kP \cong T(P)_{(\lambda_k,G)} \otimes_{A_{(0,G)}}
  A(-\lambda_k, G);
\end{equation}
in particular, $F^{k+1}P/F^kP \in \nPgr[\mathbb Z \times G] A $.

The isomorphism~\eqref{eq:Q} depends on the isomorphism~\eqref{eq:M},
and thus ultimately on~\eqref{eq:iso_lemma}. The latter depends on a
choice of a section~$g$ of $P \rightarrow T(P)$. Given another
section~$g_0$ the difference $g-g_0$ has image in $\ker (P \rightarrow
T(P)) = PA_+$. Since $A_+$~consists of elements of positive $\mathbb
Z$\nbd-degree only, this implies that the isomorphism $F^{k+1} P \cong
M^{(\lambda_{k+1})}$ does not depend on~$g$ up to elements in~$F^kP$;
in other words, the quotient $F^{k+1}P/F^kP$ is independent of the
choice of~$g$. Thus the isomorphism~\eqref{eq:Q} is, in fact,
a natural isomorphism of functors.

We are now in a position to perform the $K$\nbd-theoretical
calculations. First define the exact functor
\begin{align*} \Theta_q: \qPgr[\mathbb Z \times G] A_{(0,G)} &
\longrightarrow \qPgr[\mathbb Z \times G]A\\ P = \bigoplus_\omega
P_{(\omega,G)} & \longmapsto \bigoplus_\omega P_{(\omega,G)}
\otimes_{A_{(0,G)}} A(-\omega,G);
\end{align*}
here $\qPgr[\mathbb Z \times G] A_{(0,G)}$ denotes the full
subcategory of $\Pgr[\mathbb Z \times G]  A_{(0,G)}$ spanned by
modules with support in $[-q,q] \times G$, and $P_{(\omega,G)}$~on the
right is considered as a ${\mathbb Z} \times G$\nbd-graded
$A_{(0,G)}$\nbd-module with support in~$\{0\} \times G$.

Next define the exact functor
\begin{align*} \Psi_q: \qPgr[\mathbb Z \times G] A & \longrightarrow
\qPgr[\mathbb Z \times G] A_{(0,G)} \\ P & \longmapsto \bigoplus _{\omega
  \in {\mathbb Z}} T(P)_{(\omega,G)};
\end{align*}
here $T(P)_{(\omega,G)}$~is considered as an $A_{(0,G)}$\nbd-module
with support in~$\{\omega\} \times G$.

Now $\Psi_q \circ \Theta_q \cong \mathrm{id}$; indeed, the composition
sends the summand~$P_{(\omega,G)}$ of~$P$ to the $\kappa$\nbd-indexed
direct sum of
\[T\big(P_{(\omega,G)} \otimes_{A_{(0,G)}} A(-\omega,G)\big)_{(\kappa,G)}
\cong
\begin{cases}
  P_{(\omega,G)} & \text{if } \kappa = \omega, \\ 0 & \text{else.}
\end{cases}
\] In particular, $\Psi_q \circ \Theta_q$ induces the identity on
$K$\nbd-groups. As for the other composition, we have
\[\Theta_q \circ \Psi_q (P) = \bigoplus_{\omega} T(P)_{(\omega,G)}
\otimes_{A_{(0,G)}} A(-\omega,G) \underset{\eqref{eq:Q}}=
\bigoplus_{j=-q}^{q-1} F^{j+1}P / F^jP.\] Since $F^q =
\mathrm{id}$, additivity for characteristic filtrations \cite[p.~107,
Corollary~2]{quillen} implies that $\Theta_q \circ \Psi_q$ induces the
identity on $K$\nbd-groups.

For any $P \in \Pgr[\mathbb Z \times G] A_{(0,G)}$ we have
\[\big (P
\otimes_{A_{(0,G)}} A(-\omega,G)\big )(0,g)=P(0,g)\otimes _{A_{(0,G)}}
A(-\omega,G),\] by direct calculation. Hence the functor $\Theta_q$
induces a $\mathbb Z[G]$\nbd-linear isomorphism on
$K$\nbd-groups. Since $K$\nbd-groups are compatible with direct
limits, letting $q \rightarrow \infty$ yields a $\mathbb
Z[G]$\nbd-linear isomorphism 
\[K_i^{{\mathbb Z} \times G} (A) \cong K_i^{{\mathbb Z}\times G}(A_{(0,G)})\]
and thus, by Lemma~\ref{hggf}, a $\mathbb Z[{\mathbb Z} \times
G]$\nbd-module isomorphism
\[K_i^{{\mathbb Z} \times G}(A) \cong K_i^{G}(A_{(0,G)}) \otimes_{\mathbb Z[G]} \mathbb Z[{\mathbb Z}\times G]. \qedhere\] 
\end{proof}

\section{The fundamental theorem of $K$-theory}\label{shitnome}

\subsection{Quillen's $K$-theory of exact categories}\label{klikhf4}

\index{$K_n$-functors}

Recall that an exact category $\mathcal P$ is a full additive subcategory of an abelian category $\mathcal A$ which is closed under extension (see Definition~\ref{hgygtw2}). An \emph{exact functor} \index{exact functor} $\mathcal F: \mathcal P \rightarrow \mathcal P'$, between exact categories $\mathcal P$ and $\mathcal P'$, is an additive functor which preserves the short exact sequences. 

From an exact category $\mathcal P$, Quillen constructed abelian groups $K_n(\mathcal P)$, $n\geq 0$ (see~\cite{quillen}). 
In fact, for a fixed $n$, $K_n$ is a functor from the category of exact categories, with exact functors as morphisms to the category of abelian groups. If 
$f:\mathcal A \rightarrow \mathcal B$ is an exact functor, then we denote the corresponding group homomorphism by 
$\ol f :K_n(\mathcal A) \longrightarrow K_n(\mathcal B)$. 
\index{$\Prr A$, category of finitely generated projective $A$-modules}

When $\mathcal P=\Prr A$, the category of finitely generated projective $A$-modules,  then $K_0(\mathcal P)=K_0(A)$. Quillen  proved that basic theorems established for $K_0$-group, such as D\'evissage, resolution and localisation theorems (\cite{swan,weibelk}), can be established for these higher $K$-groups.  We briefly mention the main theorems of $K$-theory which will be used in \S\ref{bolgona}. 
The following statements were established by Quillen in~\cite{quillen} (see also~\cite{srinivas,weibelk}). 
\begin{enumerate}

\item{{\bf The Grothendieck group.}} For an exact category $\mathcal P$, the abelian group $K_0(\mathcal P)$ coincides with the construction given in Definition~\ref{hgygtw2}(1). In particular, Quillen's construction gives the Grothendieck group $K_0(A)=K_0(\Prr A)$ and graded Grothendieck group  
$K^{\gr}_0(A)=K_0(\Pgrp A)$, for ungraded and graded rings, respectively.  

\medskip

\item{{\bf D\'evissage.}} Let $\mathcal A$ be an abelian category and $\mathcal B$ be a nonempty full subcategory closed under subobjects, quotient objects and finite products in $\mathcal A$. Thus $\mathcal B$ is an abelian category and the inclusion $\mathcal B \hookrightarrow \mathcal A$ is exact. If any element $A$ in $\mathcal A$, has a finite filtration 
\[0=A_0 \subseteq A_1 \subseteq \dots \subseteq A_n= A,\]
with $A_i/A_{i-1}$ in $\mathcal B$, $1\leq i \leq n$, then the inclusion $\mathcal B \hookrightarrow \mathcal A$ induces isomorphisms $K_n(\mathcal B) \cong K_n(\mathcal A)$, $n\geq 0$. 
 
 \medskip
 
\item {{\bf Resolution.}}
Let $\mathcal M$ be an exact category and $\mathcal P$ be a full subcategory closed under extensions in $\mathcal M$. Suppose that if 
$0\rightarrow M \rightarrow P \rightarrow P' \rightarrow 0$ is exact in $\mathcal M$ with $P$ and $P'$ in $\mathcal P$, then $M$ is in $\mathcal P$. Moreover, 
for every $M$ in $\mathcal M$, there is a finite $\mathcal P$ resolution of $M$,
\[0\longrightarrow P_n \longrightarrow \cdots \longrightarrow P_1 \longrightarrow P_0 \longrightarrow M \longrightarrow 0.\]
Then the inclusion $\mathcal P \hookrightarrow \mathcal M$ induces isomorphisms $K_n(\mathcal P) \cong K_n(\mathcal M)$, $n\geq 0$. 

\medskip

\item{{\bf Localisation.}} Let $\mathcal S$ be a \emph{Serre subcategory} \index{Serre subcategory} of abelian category $\mathcal A$ (\ie $\mathcal S$ is an abelian subcategory of $\mathcal A$, closed under subobjects, quotient objects and extensions in $\mathcal A$), and let $\mathcal A /\mathcal S$ be the quotient abelian category (\cite[Exercise~10.3.2]{weibel}). Then there is a natural long exact sequence 
\[\cdots \longrightarrow K_{n+1}(\mathcal A/\mathcal S) \stackrel{\delta}{\longrightarrow} K_n(\mathcal S) \longrightarrow K_n(\mathcal A) \longrightarrow K_n(\mathcal A/\mathcal S) \longrightarrow \cdots.\]

\medskip

\item{{\bf Exact Sequence of Functors.}} Let $\mathcal P$ and $\mathcal P'$ be exact categories and 
\[0\longrightarrow  f' \longrightarrow f \longrightarrow f'' \longrightarrow 0,\]
be an exact sequence of exact functors from $\mathcal P$ to $\mathcal P'$. Then 
\[\ol f = \ol f' +\ol f'' : K_n(\mathcal P) \longrightarrow K_n(\mathcal P').\] 

\end{enumerate}

\subsection{Base change and transfer functors}\label{monashnick}

Let $A$ be a ring with identity. If $A$ is a right Noetherian ring then the category of finitely generated right $A$-modules, $\modd A$, form an abelian category. Thus one can define the $K$-groups of this category, $G_n(A):=K_n(\modd A)$, $n\in \mathbb N$. If further $A$ is regular, then by the  
resolution theorem (\S\ref{klikhf4}), $G_n(A)=K_n(A)$.   \index{$G_n$-functors}

Let $A\rightarrow B$ be a ring homomorphism. This makes $B$ an $A$-module. If $B$ is a finitely generated $A$-module, then any finitely generated 
$B$-module can be considered as a finitely generated $A$-module. This induces an exact functor, called a \emph{transfer functor},  \index{transfer functor}
 $\modd B \rightarrow \modd A$. Since $K$-theory is a functor from the category of exact categories with exact functors as morphisms, we get a transfer map $G_n(B) \rightarrow G_n(A)$, $n\in \mathbb N$. A similar argument shows that, if 
$A\rightarrow B$ is a graded homomorphism, then we have a transfer map $G^{\gr}_n(B) \rightarrow G^{\gr}_n(A)$, $n\in \mathbb N$.

On the other hand, if under the ring homomorphism  $A\rightarrow B$, $B$ is a flat $A$-module, then the \emph{base change functor} \index{base change functor} $-\otimes_A B:\Modd A \rightarrow \Modd B$ is exact, which restricts to $-\otimes_A B:\modd A \rightarrow \modd B$. This in turn induces the group homomorphism $G_n(A) \rightarrow G_n(B)$, $n\in \mathbb N$.  Note that any ring homomorphism $A\rightarrow B$, induces an exact functor $-\otimes_A B:\Prr A \rightarrow \Prr B$, which induces the homomorphism $K_n(A) \rightarrow K_n(B)$. Similar statements can be written for graded rings $A$ and $B$ and the graded homomorphism $A\rightarrow B$. However, to prove the fundamental theorem of $K$-theory (see~(\ref{schmerz})), we need to use the D\'evissage and localisation theorems, which are only valid for abelian categories. This forces us to work in the abelian category $\grr A$ (assuming $A$ is right Noetherian), rather than the exact category $\Pgrp A$.

\subsection{A localisation exact sequence for graded rings}\label{jensgty4}

Let $A$ be a right regular Noetherian $\Gamma$\!-graded ring, \ie $A$ is graded right regular and graded right Noetherian. We assume the ring is graded regular so that we could work with $K$-theory instead of $G$-theory. We recall the concept of graded regular rings. For a graded right module $M$, the minimum length of all graded projective resolutions of $M$ is defined as the \emph{graded projective dimension} of $M$ and denoted by $\pdim^{\gr}(M)$. \index{graded projective dimension} If $M$ does not admit a finite graded projective resolution, then we set $\pdim^{\gr}(M)=\infty$. The \emph{graded right global dimension} of the graded ring $A$ is the supremum of the graded projective dimension of all the graded right modules over $A$ and denoted by $\gldim^{\gr}(A)$. \index{graded global dimension} \index{global dimension} We say $A$ is a \emph{graded right regular}  if the graded right global dimension is finite. As soon as $\Gamma$ is a trivial group, the above definitions become the standard definitions of projective dimension, denoted by, $\pdim$, and the global dimension, denoted by, $\gldim$, in ring theory. In the case of $\Z$-graded ring, we have the following relation between the dimensions (see~\cite[Theorem~II.8.2]{grrings2}) \index{$\gldim^{\gr}$, graded global dimension} \index{graded right regular ring} \index{graded regular ring}
\begin{equation}
\gldim^{\gr}(A) \leq \gldim(A) \leq 1+\gldim^{\gr}(A).
\end{equation}
This shows, if $A$ is graded regular, then $A$ is regular. This will be used in~\S\ref{vandengg}. 

Let $S$ be a central multiplicative closed subsets of $A$, consisting of homogeneous elements. Then $S^{-1}A$ is a $\Gamma$\!-graded ring (see 
Example~\ref{meinbalmain}). 
Let $\grs A$ be the category of $S$-torsion graded finitely generated right $A$-modules, \ie a graded finitely generated  $A$-module $M$ such that $Ms=0$ for some $s\in S$. 
 Clearly $\grs A$ is a Serre subcategory of the abelian category $\grr A$ and 
 \begin{equation}\label{rohaniii}
 \grr A/ \grs A \approx \grr \, (S^{-1} A),
 \end{equation}
(see~\cite[Exercise~10.3.2]{weibel}).  \index{Serre subcategory}

Now let $S=\{s^n \mid n\in \mathbb N \}$, where $s$ is a homogeneous element in the centre of $A$. Then any module $M$ in $\grs A$ has a filtration 
\[ 0=Ms^k \subseteq \dots \subseteq Ms\subseteq M,\]
and clearly the consecutive quotients in the filtration are $\Gamma$\!-graded $A/sA$-modules. Considering $\grr A/sA$ as a full subcategory of $\grs A$ via the transfer map $A\rightarrow A/sA$ (see~\S\ref{monashnick}),  and using  D\'evissage theorem, we get, for any $n\in \mathbb N$, 
\[K_n(\grs A) \cong K^{\gr}_n(A/sA).\]

Now using the localisation theorem for (\ref{rohaniii}), we obtain a long localisation exact sequence 
\begin{equation}\label{anniversary14june3}
\cdots \longrightarrow K^{\gr}_{n+1}(A_s) \longrightarrow K^{\gr}_n(A/sA) \longrightarrow K^{\gr}_n(A) \longrightarrow K^{\gr}_n(A_s) \longrightarrow \cdots ,
\end{equation} where $A_s:=S^{-1}A$.  

In particular, consider the $\Gamma$\!-graded rings $A[y]$ and $A[y,y^{-1}]$, where $\deg(y)=\alpha$, $\deg(y^{-1})=-\alpha$, $\alpha \in \Gamma$. For $A[y]$ and $S=\{y^n \mid n\in \mathbb N\}$,
the sequence~\ref{anniversary14june3} reduces to 
 \begin{equation}\label{anniversary14june}
\cdots \longrightarrow K^{\gr}_{n+1}(A[y,y^{-1}]) \longrightarrow K^{\gr}_n(A) \longrightarrow K^{\gr}_n(A[y]) \longrightarrow K^{\gr}_n(A[y,y^{-1}]) \longrightarrow \cdots .
\end{equation}
This long exact sequence will be used in~\S\ref{vandengg} in two occasions (with $\deg(y)=1 \in \Z$ and $\deg(y)=(1,-1)\in \Z\times\Z$) to relate $K^{\gr}_n$-groups  to $K_n$-groups.

\subsection{The fundamental theorem}\label{bolgona}

Let $A$ be a right Noetherian and regular ring. Consider the polynomial ring $A[x]$. The fundamental theorem of $K$-theory gives that, for any $n \in \mathbb N$,  \index{Noetherian ring}
\begin{equation}\label{schmerz}
K_n(A[x])\cong K_n(A).
\end{equation}   

In this section we prove that if $A$ is a positively graded $\Z$-graded ring, then 
\[\boxed{K_n(A) \cong K_n(A_0).}\] This, in particular, gives~(\ref{schmerz}). We will follow Gersten's treatment in~\cite{gersten}. The proof shows how the graded $K$-theory is effectively used to prove a theorem on (ungraded) $K$-theory.

Let $A$ be a positively $\mathbb Z$-graded ring, \ie $A=\bigoplus_{i\in \mathbb N} A_i$. Define a $\Z$-grading on the polynomial ring $A[x]$ as follows (see Example~\ref{hyg61}):
\begin{equation}\label{hnhg5}
A[x]_n=\bigoplus_{i+j=n}A_i x^j.
\end{equation}
Clearly this is also  a positively graded ring and $A[x]_0=A_0$.

Throughout the proof of the fundamental theorem we consider two evaluation maps on $A[x]$ as follows.

\begin{enumerate}

\item Consider the evaluation homomorphism at $0$, \ie 
\begin{align}\label{hgyht31}
e_0:A[x] &\longrightarrow A,\\
f(x)&\longmapsto f(0). \notag
\end{align}
This is a graded homomorphism (with the grading defined in~(\ref{hnhg5})), and $A$ becomes a graded finitely generated $A[x]$-module.  This induces an exact transfer functor (see~\S\ref{monashnick})
\begin{align}\label{gbdh98}
i: \grr A &\longrightarrow \grr A[x],\\
M &\longmapsto M \notag
\end{align}
and in turn a homomorphism \[\ol i:G^{\gr}_n(A) \rightarrow G^{\gr}_n(A[x]).\]
Let $M$ be a graded $A$-module. Under the homomorphism \ref{gbdh98}, $M$ is also a grade $A[x]$-module. 
On the other hand, the canonical graded homomorphism $A\rightarrow A[x]$, gives the graded $A[x]$-module $M\otimes_A A[x]\cong_{\gr} M[x]$. Then we have the following short exact sequence of graded $A[x]$-modules,
\begin{equation}\label{bulgim}
0\longrightarrow M[x](-1) \stackrel{x}{\longrightarrow} M[x] \stackrel{e_0}{\longrightarrow} M \longrightarrow 0.
\end{equation}
Note that this short exact sequence is not an split sequence.

\item Consider the evaluation homomorphism at $1$, \ie 
\begin{align}\label{hgyht313}
e_1:A[x] &\longrightarrow A,\\
f(x)&\longmapsto f(1). \notag
\end{align}
Note that this homomorphism is not graded. Clearly $\ker(e_1)=(1-x)$ (\ie the ideal generated by $1-x$). Thus 
$A[x]/(1-x) \cong A$. If $M$ is a (graded) $A[x]$-module, then 
\begin{equation*}
M \otimes_{A[x]} A \cong M \otimes_{A[x]} A[x]/(1-x) \cong M/M(1-x),
\end{equation*}
 is an $A$-module. 
Thus the homomorphism $e_1$ induces the functor
\begin{align}\label{gbg42ki}
\mathcal F:=-\otimes_{A[x]} A : \grr A[x] &\longrightarrow \modd A \\
M &\longmapsto M/M(1-x). \notag
\end{align}
\end{enumerate}

We need the following Lemma (see~\cite[Lemma~II.8.1]{grrings2}). 

\begin{lemma}\label{hghss1} Consider the functor $\mathcal F$ defined in~\eqref{gbg42ki}. Then we have the following. 

\begin{enumerate}[\upshape(1)]

\item $\mathcal F$ is an exact functor. 

\item For a graded finitely generated $A[x]$-module $M$, $\mathcal F(M)=0$ if and only if $Mx^n=0$ for some $n\in \mathbb N$. 
\end{enumerate}
\end{lemma}
\begin{proof}
(1)  Since the tensor product is a right exact functor, we are left to show that if $0\rightarrow M' \rightarrow M$ is exact, then by~(\ref{gbg42ki}),
\[0\longrightarrow M'/M'(1-x) \longrightarrow M/M(1-x),\] is exact, \ie 
$M' \cap M(1-x) =M'(1-x).$ 

Let $m' \in M' \cap M(1-x)$. Then 
\begin{equation}\label{poifrp}
m'=m(1-x),
\end{equation} for some $m\in M$. Since $M'$ and $M$ are graded modules, $m'=\sum m'_i$ and $m=\sum m_i$, with 
$\deg(m'_i)=\deg(m_i)=i$. Comparing the degrees in the Equation~\ref{poifrp}, we have $m_i-m_{i-1}x=m'_i \in M'_i$. 
Let $j$  be the smallest $i$ in the support of $m$ such that $m_j \not \in M'_j$. Then $m_{j-1}x \in M'_j$ which implies that $m_j=m'_j- m_{j-1}x\in M'_j$,  which is a contradiction. Thus $m\in M'$, and so $m' \in M'(1-x).$

(2) The proof is elementary and left to the reader.
\end{proof}

By Lemma~\ref{hghss1}(2), $\ker(\mathcal F)$ is an (abelian) full subcategory of $\grr A[x]$, with objects the graded $A[x]$-modules $M$ such that $Mx^n=0$, for some $n\in \mathbb N$. 
Define $\mathcal C_1$ the (abelian) full subcategory of $\grr A[x]$ with objects the graded $A[x]$-modules $M$ such that $Mx=0$. Thus 
$\mathcal C_1 \hookrightarrow \ker(\mathcal F)$ is an exact functor. Moreover, for any object $M$ in $\ker(\mathcal F)$, we have a finite filtration of graded $A[x]$-modules 
\[0=Mx^n \subseteq Mx^{n-1} \subseteq \dots \subseteq Mx \subseteq M,\] such that
$Mx^i/Mx^{i+1} \in \mathcal C_1$.  Thus by D\'evissage theorem (\S\ref{klikhf4}), we have 
\begin{equation}\label{hnlofde3}
K_n(\mathcal C_1) \cong K_n(\ker(\mathcal F)).
\end{equation}

Next recall the exact functor
$i: \grr A \longrightarrow \grr A[x]$ from~(\ref{gbdh98}).  It is easy to see that this functor induces an equivalence between the categories 
\begin{equation}\label{scute32}
i: \grr A \longrightarrow \mathcal C_1,
\end{equation}
and thus on the level of $K$-groups. Combing this with~(\ref{hnlofde3}) we get 
\begin{equation}\label{hnlofde38}
K^{\gr}_n(A)=K_n(\mathcal C_1) \cong K_n(\ker(\mathcal F)).
\end{equation}
Since by Lemma~\ref{hghss1}, $\mathcal F$ is exact, $\ker(\mathcal F)$ is closed under subobjects, quotients and extensions, i.e, it is a  Serre subcategory of $\grr A[x]$. We next prove that 
\begin{equation}\label{sbfh3}
\grr A[x] /\ker(\mathcal F) \cong \modd A.
\end{equation}

Any finitely generated module over a Noetherian ring is finitely presented. We need the following general lemma to invoke the localisation. 
\begin{lemma}\label{brek66}
Let $M$ be a finitely presented $A$-module. Then there is a graded finitely presented $A[x]$-module $N$ such that $\mathcal F(N) \cong M$, where $\mathcal F$ is the functor defined in~\eqref{gbg42ki}. 
\end{lemma}
\begin{proof}
Since $\mathcal F=-\otimes_{A[x]} A$, for any free $A$-module $M$ of rank $k$ we have $\mathcal F(\bigoplus_k A[x])\cong \bigoplus_k A\cong M$.
Let $M$ be a finitely presented $A$-module. Then there is an exact sequence 
\[\bigoplus_m A \stackrel{f}{\longrightarrow} \bigoplus_n A \longrightarrow M \longrightarrow 0.\]
If we show that there are graded free $A[x]$-modules $F_1$ and $F_2$ and a graded homomorphism $g:F_1\rightarrow F_2$ 
such that the following  diagram  is commutative, 
\begin{equation*}
\xymatrix{
\mathcal F(F_1) \ar[r]^{\mathcal F(g)} \ar[d]_{\cong}& \mathcal F(F_2)\ar[d]^{\cong}\\
\bigoplus_m A \ar[r]^f & \bigoplus_n A,
}
\end{equation*}
then since $\mathcal F$ is right exact, one can complete the diagram 
\begin{equation*}
\xymatrix{
\mathcal F(F_1) \ar[r]^{\mathcal F(g)} \ar[d]_{\cong}& \mathcal F(F_2) \ar[r] \ar[d]^{\cong}& \mathcal F(\coker(g)) \ar[r] \ar@{.>}[d]^{\cong} & 0\\
\bigoplus_m A \ar[r]^f & \bigoplus_n A \ar[r]&  M \ar[r] & 0.
}
\end{equation*}
This shows $\mathcal F(\coker(g)) \cong M$ and we will be done. 

Thus, suppose $f:\bigoplus_m A \rightarrow \bigoplus_n A$ is an $A$-module homomorphism. With the misuse of notation, denote the standard basis for both $A$-modules $A^m$ and $A^n$, by $e_i$. Then 
\[f(e_i)=\sum_{j=1}^n a_{ij} e_j,\]
where $1\leq i \leq m$ and  $a_{ij} \in A$. 
Decomposing  $a_{ij}$ into its homogeneous components, $a_{ij}=\sum_k {a_{ij}}_k$, since the number of $a_{ij}$, $1\leq i \leq m$, $1\leq j \leq n$, are finite, there is $N\in \mathbb N$ such that ${a_{ij}}_N=0$ for all $i,j$. Now consider the graded free $A[x]$-modules $\bigoplus_m A[x](-N)$ and  
$\bigoplus_n A[x]$, with the standard bases (which are of degrees $N$ and $0$ in  $\bigoplus_m A[x](-N)$ and $\bigoplus_n A[x]$, respectively). Define
\begin{align*}
g: \bigoplus_m A[x](-N) & \longrightarrow  \bigoplus_n A[x],\\
e_i & \longmapsto \sum_j \big(\sum_k x^{N-k}{a_{ij}}_k \big)\,  e_j. 
\end{align*}
This is a graded $A[x]$-module homomorphism. Further, $\mathcal F(g)$ amount to evaluation of $g$ at $x=1$ which coincides with the map $f$. This finishes the proof.
\end{proof}

\begin{remark}
The proof of Lemma~\ref{brek66}, in particular, shows that if $f \in \M_n(A)$, then there is a $g\in \M_n(A[x])(\ol \delta)$ such that evaluation of $g$ at $1$, gives $f$.  
\end{remark}

Lemma~\ref{brek66} along with standard results of localisation theory (see~\cite[p. 114, Theorem~5.11]{swan}) implies that 
\[ \grr A[x] /\ker(\mathcal F) \cong \modd A.\]

Now we are ready to apply the localisation theorem (\S\ref{klikhf4}) to the sequence 
\[\ker(\mathcal F) \hookrightarrow \grr A[x] \longrightarrow \grr A[x] /\ker(\mathcal F) \cong \modd A,\]
to obtain a long exact sequence 
\begin{equation*}
\cdots \longrightarrow K_{n+1}(A) \longrightarrow K_n(\ker(\mathcal F)) \longrightarrow K^{\gr}_n(A[x]) \longrightarrow K_n(A) \longrightarrow \cdots. 
\end{equation*}
Now using~(\ref{hnlofde38}) to replace $K_n(\ker(\mathcal F))$, we get
\begin{equation}\label{garbi1}
\cdots \longrightarrow K_{n+1}(A) \longrightarrow K^{\gr}_n(A) \stackrel{\ol i}{\longrightarrow} K^{\gr}_n(A[x]) \longrightarrow K_n(A) \longrightarrow \cdots. 
\end{equation}
Next, we show that the homomorphism $\ol i$ is injective. Note that the homomorphism $\ol i$ on the level of $K$-groups induced by the composition of exact functors (see~(\ref{gbdh98}) and~(\ref{scute32})),
\[
\Pgrp A \hookrightarrow \grr A \longrightarrow \mathcal C_1 \hookrightarrow \ker(\mathcal F) \hookrightarrow \grr A[x],
\]
which sends a graded finitely generated projective $A$-module $M$, to $M$ considered as graded $A[x]$-module. 
Define three exact functors $\phi_i: \Pgrp A \rightarrow \grr A[x]$, $1\leq i \leq 3$,  as follows,
\begin{align}\label{hgtgrete215}
\phi_1(M) &=M[x](-1)\\
\phi_2(M) &=M[x]\notag\\
\phi_3(M) &=M.\notag
\end{align}
Now the exact sequence~\ref{bulgim} and the Exact Sequence of Functors Theorem (\S\ref{klikhf4}) immediately give, 
\begin{equation}\label{jnhuye199}
\ol \phi_3= \ol \phi_2- \ol \phi_1. 
\end{equation}
Note that $\ol \phi_3=\ol i$. 

Now invoking Quillen's Theorem~\ref{quillen11}, from exact sequence~\ref{garbi1}, we get 
\begin{equation}\label{gfnksdb589}
\cdots \longrightarrow K_{n+1}(A) \longrightarrow \mathbb Z[t,t^{-1}]\otimes K_n(A_0) \stackrel{\ol i}{\longrightarrow} \mathbb Z[t,t^{-1}]\otimes K_n(A_0) \longrightarrow K_n(A) \longrightarrow \cdots 
\end{equation}
Further, the maps $\ol \phi_2$ and $\ol \phi_1$ becomes $1\otimes 1$ and $t\otimes 1$, respectively (see~(\ref{macitywall})).  Thus from~(\ref{jnhuye199}) we get 
\[\ol i=(1-t)\otimes 1 : \mathbb Z[t,t^{-1}]\otimes_{\Z} K_n(A_0) \longrightarrow \mathbb Z[t,t^{-1}] \otimes_{\Z} K_n(A_0).\]
This shows that $\ol i$ is an injective map. Thus the exact sequence~\ref{gfnksdb589} reduces to 
\begin{equation}\label{gfnksdb5}
0 \longrightarrow \mathbb Z[t,t^{-1}]\otimes_{\Z} K_n(A_0) \stackrel{\ol i}{\longrightarrow} \mathbb Z[t,t^{-1}]\otimes_{\Z} K_n(A_0) \longrightarrow K_n(A) \longrightarrow 0.
\end{equation}
Again, from the description of $\ol i$ it follows,  
\[ K_n(A) \cong K_n(A_0). \]

\section{Relating $K^{\gr}_*(A)$ to $K_*(A_0)$}\label{quoihes}

If $A$ is a strongly $\Gamma$\!-graded ring, then $\Gr A$ is equivalent to $\Modd A_0$. Thus we have $K^{\gr}_n(A) \cong K_n(A_0)$,  for any $n\in \mathbb N$ (see~\S\ref{jijigogo}). 
In general (for graded Noetherian regular rings) these two groups are related via a long exact sequence which we will establish in this section (see~(\ref{ghgtre54})).  \index{graded Noetherian ring}

Let $A$ be a $\Gamma$\!-graded ring. For a subset $\Omega \subseteq \Gamma$, consider the full subcategory 
$\Gromega A$ of $\Gr A$, of all graded right $A$-modules $M$ as objects such that $M_\omega=0$, for all $\omega \in \Omega$. This is a Serre subcategory of $\Gr A$.  We need the following Lemma.  \index{Serre subcategory}
\index{$\Gromega A$, $\Omega$ a subgroup of $\Gamma$}

\begin{lemma}
Let $A$ be a $\Gamma$\!-graded ring and $\Omega$ a subgroup of $\Gamma$. Then we have
\begin{equation}\label{dhge329}
\Gr[\Gamma] A /\Gromega A \cong \Gr[\Omega] A_\Omega. 
\end{equation}
In particular, 
$\Gr A / \Gro A \cong \Modd A_0$.
\end{lemma}
\begin{proof}
Consider a family $\Sigma$ of morphisms $f$ in $\Gr[\Gamma] A$ such that $\ker(f)$ and $\coker(f)$ are in $\Gromega A$. By definition 
\[
\Gr[\Gamma] A /\Gromega A = \Sigma^{-1} A.
\]
Note that $f:M\rightarrow N$ is in $\Sigma$ if and only if $f|_{M_\Omega}:M_\Omega \rightarrow N_\Omega$ is an $A_\Omega$-module isomorphism. 
Consider the (forgetful) functor 
\begin{align*}
(-)_{\Omega}:\Gr[\Gamma] A &\longrightarrow  \Gr[\Omega] A_\Omega,\\
M &\longmapsto M_\Omega
\end{align*}
 (see~\S\ref{bill100}). Under $(-)_{\Omega}$ the morphisms of $\Sigma$ are sent to invertible morphisms in $\Gr[\Omega] A_\Omega$. Thus by the property of quotient categories, there is an induced functor $\psi$ which makes the following diagram commutative: 
\begin{equation*}
\begin{split}
\xymatrix{
\Gr[\Gamma] A \ar[r]^-{\phi} \ar[d]_{(-)_{\Omega}} &  \Gr[\Gamma] A /\Gromega A \ar@{.>}[dl]^{\psi}\\
 \Gr[\Omega] A_\Omega
}
\end{split}
\end{equation*}

The functor $\phi$ is defined as follows: 

\begin{equation*}
M\stackrel{f}{\longrightarrow} N \longmapsto \xymatrix{
& M \ar[dl]_{\id} \ar[dr]^f\\
M && N
}
\end{equation*}

And the functor $\psi$ is defined as follows

\begin{equation*}
\xymatrix{
& M' \ar[dl]_{s} \ar[dr]^f\\
M && N
} \longmapsto  M_\Omega \stackrel{f_\Omega s^{-1}_\Omega}{\longrightarrow} N_\Omega .
\end{equation*}
One can now check that $\psi$ is an equivalence of categories. This finishes the proof.  
\end{proof}

 
 When $A$ is a (right) regular Noetherian graded ring,  the quotient category identity~(\ref{dhge329}) holds for the corresponding graded finitely generated modules, \ie  $\grr A / \gro A \cong \modd A_0$ and 
the localisation theorem gives a long exact sequence,
\begin{equation}\label{ghgtre54}
\cdots \longrightarrow K_{n+1}(A_0) \stackrel{\delta}{\longrightarrow} K_n(\gro A) \longrightarrow K^{\gr}_n(A) \longrightarrow K_n(A_0) \longrightarrow \cdots.
\end{equation} \index{quotient category}

Note that if $A$ is a strongly graded ring, then $\Gro A=0$, for any $\alpha \in \Gamma$. In particular, the long exact sequence~\ref{ghgtre54} gives 
$K^{\gr}_n(A) \cong K_n(A_0)$,  for any $n\in \mathbb N$.

\section{Relating $K^{\gr}_*(A)$ to $K_*(A)$}\label{vandengg}

In~\cite{vandenbergh}, van den Bergh, following the methods to prove the fundamental theorem of $K$-theory (\S\ref{bolgona}), established a long exact sequence relating 
graded $K$-theory of a ring to ungraded $K$-theory in the case of right regular Noetherian $\Z$-graded rings. In this section we will present van den Bergh's observation.  As a consequence, we will see that in the case of graded Grothendieck group, the shift of modules is all the difference between the $K^{\gr}_0$ and $K_0$-groups (Corollary~\ref{ght9laks}).

Let $A$ be a right regular Noetherian $\mathbb Z$-graded ring. Thus $G^{\gr}$-theory and $G$-theory become $K^{\gr}$-theory and $K$-theory, respectively (see~\S\ref{jensgty4}). Consider the $\Z$-graded ring $A[y]$, where $\deg(y)=1$ and $\Z\times\Z$-graded ring $B=A[s]$, where the homogeneous components defined by $B_{(n,m)}=A_ms^n$. Moreover, consider 
the $\Z\times\Z$-graded rings $B[z]$ and $B[z,z^{-1}]$, where $\deg(z)=(1,-1)$. It is easy to check that 
\begin{enumerate}

\item[(1)] The support of $B$ is $\mathbb N \times \Z$ and $B_{(0,\Z)}=A$ as $\Z$-graded rings.

\item[(2)] The support of $B[z]$ is $\mathbb N \times \Z$ and $B[z]_{(0,\Z)}=A$ as $\Z$-graded rings.

\end{enumerate}

\noindent Moreover 

\begin{enumerate}

\item[(3)] $B[z,z^{-1}]$ is a $\Z\times \Z$-graded ring such that $\Gr[\Z\times \Z] B[z,z^{-1}] \approx \Gr[\Z] A[y]$. 
\begin{proof}
Observe that for any $n \in \Z$, 
\[1 \in B[z,z^{-1}]_{(n,\Z)} B[z,z^{-1}]_{(-n,\Z)}.\]

Now by Example~\ref{yhyhew45},
\[\Gr[\Z\times \Z] B[z,z^{-1}]  \approx \Gr[\Z] B[z,z^{-1}]_{(0,\Z)}.\]
But it is easy to see that $B[z,z^{-1}]_{(0,\Z)} \cong_{\gr} A[y]$ as $\Z$-graded rings. 
\end{proof}
\end{enumerate}

By~\S\ref{jensgty4} (in particular~(\ref{anniversary14june}) for $\deg(z)=(1,-1)$), we have a long exact sequence 
\begin{equation}\label{janedouglas}
\cdots \longrightarrow K^{\gr}_{n+1}(B[z,z^{-1}])  \longrightarrow K^{\gr}_n(B) \stackrel{\ol i}{\longrightarrow} K^{\gr}_n(B[z]) \longrightarrow K^{\gr}_n(B[z,z^{-1}]) \longrightarrow \cdots .
\end{equation}

By Corollary~\ref{jhyhgtewqs}, (1) and (2) above, 
\[K^{\gr}_n(B) = K^{\Z\times \Z}_n(B) \cong  \Z[t,t^{-1}]\otimes K^{\Z}_n(B_{(0,\Z)})  \cong \Z[t,t^{-1}]\otimes K^{\gr}_n(A)\] and 
\[K^{\gr}_n(B[z])= K^{\Z\times \Z }_n(B[z]) \cong \Z[t,t^{-1}]\otimes K^{\Z}_n(B[z]_{(0,\Z)}) \cong \Z[t,t^{-1}]\otimes K^{\gr}_n(A).\] 
Moreover, from (3) we have 
\[K^{\gr}_n(B[z,z^{-1}]) = K^{\Z\times \Z}_n(B[z,z^{-1}]) \cong K^{\Z}_n(A[y]) =K^{\gr}_n(A[y]).\]
So, from the sequence~\ref{janedouglas}, we get  
\begin{multline}\label{gfnksdb5234}
\cdots \longrightarrow K^{\gr}_{n+1}(A[y]) \longrightarrow \mathbb Z[t,t^{-1}]\otimes K^{\gr}_n(A) \stackrel{\ol i}{\longrightarrow} \mathbb Z[t,t^{-1}]\otimes K^{\gr}_n(A) \\ \longrightarrow K^{\gr}_n(A[y]) \longrightarrow \cdots. 
\end{multline}
The rest is similar to the method used to prove the Fundamental Theorem in \S\ref{bolgona}. First note that we have a short exact sequence of $\mathbb Z\times \mathbb Z$-graded $B[z]$-modules  (compare this with (\ref{bulgim}))
\begin{equation}\label{bulgimzz}
0\longrightarrow M[z](-1,1) \stackrel{z}{\longrightarrow} M[z] \stackrel{e_0}{\longrightarrow} M \longrightarrow 0.
\end{equation}
 Here $M$ is a graded $B$-module which becomes a graded $B[z]$-module under the evaluation map $e_0:B[z]\rightarrow S, z\mapsto 0$. 
Again, similar to~(\ref{hgtgrete215}), for $1\leq i \leq 3,$ defining 
\[ \phi_i:\gr B \longrightarrow \gr B[z],\]
by
\begin{align*}
\phi_1(M) &=M[z](-1,1);\\
\phi_2(M) &=M[z];\\
\phi_3(M) &=M,
\end{align*}
the exact sequence~\ref{bulgimzz} and the Exact Sequence of Functors Theorem (\S\ref{klikhf4}) immediately give, 
\begin{equation}\label{jnhuye1}
\ol \phi_3= \ol \phi_2- \ol \phi_1,
\end{equation}
and $\ol \phi_3=\ol i$. 
One can then observe that $\ol i$ is injective, so the exact sequence~\ref{gfnksdb5234} reduces to 
\[
0 \longrightarrow \mathbb Z[t,t^{-1}]\otimes K^{\gr}_n(A) \stackrel{\ol i}{\longrightarrow} \mathbb Z[t,t^{-1}]\otimes K^{\gr}_n(A) \longrightarrow K^{\gr}_n(A[y]) \longrightarrow 0. 
\]
Finally, from this exact sequence it follows that 
\begin{equation}\label{thbbieh7}
K^{\gr}_n(A[y]) \cong K^{\gr}_n(A).
\end{equation}

We are in a position to relate graded $K$-theory to ungraded $K$-theory. 

\begin{theorem}\label{bgungdjr}
Let $A$ be a right regular Noetherian $\mathbb Z$-graded ring. Then there is a long exact sequence 
\begin{equation}\label{gelatobal}
\cdots \longrightarrow K_{n+1}(A)  \longrightarrow K^{\gr}_n(A) \stackrel{\ol i}{\longrightarrow} K^{\gr}_n(A) \stackrel{U}{\longrightarrow} K_n(A) \longrightarrow \cdots.
\end{equation}
Here $\ol i= \overline{\mathcal T_1}  - \overline{\mathcal T_0}=\overline{\mathcal T_1}  - 1$, where $\mathcal T$ is the shift functor~\eqref{RBAday1} and $U$ is the forgetful functor.  
\end{theorem}
\begin{proof}
By~\S\ref{jensgty4} (in particular~(\ref{anniversary14june}) for $\deg(y)=1$), we have a long exact sequence 
\begin{equation*}
\cdots \longrightarrow K^{\gr}_{n+1}(A[y,y^{-1}])  \longrightarrow K^{\gr}_n(A) \stackrel{i}{\longrightarrow} K^{\gr}_n(A[y]) \longrightarrow K^{\gr}_n(A[y,y^{-1}]) \longrightarrow \cdots .
\end{equation*}
By~(\ref{thbbieh7}), \[K^{\gr}_n(A[y])\cong  K^{\gr}_n(A).\] Since $A[y,y^{-1}]$ is strongly graded, by Dade's Theorem~\ref{dadesthm},  \[K^{\gr}_n(A[y,y^{-1}]) \cong  K_n(A).\]
Replacing these into above long exact sequence, the theorem follows. 
\end{proof}

The following is an immediate corollary of the Theorem~\ref{bgungdjr}. It shows that for the graded regular Noetherian rings,  the shift of modules is all the difference between the graded Grothendieck group and the usual Grothendieck group.

\begin{corollary}\label{ght9laks}
Let $A$ be a right regular Noetherian $\mathbb Z$-graded ring. Then
\[K^{\gr}_0(A) / \langle [P(1)] -[P] \rangle \cong K_0(A),\] where $P$ is a graded finitely generated projective $A$-module. 
\end{corollary}
\begin{proof}
The long exact sequence~\ref{bgungdjr}, for $n=0$, reduces to 
\[ 
\begin{split}
\xymatrix{
K^{\gr}_0(A) \ar[rr]^{[P]\mapsto [P(1)]-[P]} && K^{\gr}_0(A) \ar[r]^{U} &  K_0(A) \ar[r] & 0.
}
\end{split}
\] 
The corollary is now immediate. 
\end{proof}

\begin{remark}
Both the fundamental theorem (\S\ref{bolgona}) and Theorem~\ref{bgungdjr} can be written for the case of graded coherent regular rings, as it was demonstrated by Gersten  in~\cite{gersten} for the ungraded case. 
\end{remark}


\printindex

\end{document}